\documentclass{amsart} \usepackage{bbm}
\usepackage[top=2.8cm,left=2.8cm,right=2.8cm,bottom=2.6cm]{geometry}
\usepackage{amsfonts}
\usepackage{amsmath,amsthm,amscd,mathrsfs,amssymb,graphicx,enumerate,
  dsfont}
\usepackage{xypic}

\newtheorem*{claim}{Claim}
\newtheorem*{thm3}{Theorem}
\newtheorem{thm2}{Theorem}
\newtheorem{conj2}{Conjecture}
\newtheorem{thm}{Theorem}[section]
\newtheorem{cor}[thm]{Corollary}
\newtheorem{lem}[thm]{Lemma}
\newtheorem{prop}[thm]{Proposition}
\newtheorem{conj}[thm]{Conjecture}
\theoremstyle{definition}
\newtheorem{defn}[thm]{Definition}
\newtheorem{assm}[thm]{Assumption}
\newtheorem{ex}[thm]{Example}
\newtheorem{rmk}[thm]{Remark}

\DeclareMathOperator{\End}{End} 
 
 \DeclareMathOperator{\rk}{rk}
 \DeclareMathOperator{\Sym}{Sym}
\DeclareMathOperator{\Pic}{Pic}

\newcommand{\C}{\ensuremath\mathds{C}}

\newcommand{\Z}{\ensuremath\mathds{Z}}
\newcommand{\Q}{\ensuremath\mathds{Q}}
\newcommand{\F}{\ensuremath\mathrm{F}}

\newcommand{\fa}{\ensuremath\mathfrak{a}}
\newcommand{\fb}{\ensuremath\mathfrak{b}}

\newcommand{\FF}{\ensuremath\mathcal{F}}

\newcommand{\PP}{\ensuremath\mathds{P}}
\newcommand{\calO}{\ensuremath\mathcal{O}}

\newcommand{\HH}{\ensuremath\mathrm{H}}
\newcommand{\CH}{\ensuremath\mathrm{CH}}

\newcommand{\set}[1]{\left\{#1\right\}}

\begin{document}

\newgeometry{top=1.4cm,left=2.8cm,right=2.8cm,bottom=2.3cm}
\thispagestyle{empty} 
\title{The Fourier transform for certain hyperk\"ahler fourfolds}
\author{Mingmin Shen and Charles Vial}



\thanks{The first author is supported by the Simons Foundation as a
  Simons Postdoctoral Fellow.}

\thanks{The second author is supported by EPSRC Early Career
  Fellowship number EP/K005545/1.}

\address{DPMMS, University of Cambridge, Wilberforce Road, Cambridge
  CB3 0WB, UK}
\email{M.Shen@dpmms.cam.ac.uk, C.Vial@dpmms.cam.ac.uk}


\begin{abstract} Using a codimension-$1$ algebraic cycle obtained from
  the Poincar\'e line bundle, Beauville defined the Fourier transform
  on the Chow groups of an abelian variety $A$ and showed that the
  Fourier transform induces a decomposition of the Chow ring
  $\CH^*(A)$.  By using a codimension-$2$ algebraic cycle representing
  the Beauville--Bogomolov class, we give evidence for the existence
  of a similar decomposition for the Chow ring of hyperk\"ahler
  varieties deformation equivalent to the Hilbert scheme of length-$2$
  subschemes on a K3 surface. We indeed establish the existence of
  such a decomposition for the Hilbert scheme of length-$2$ subschemes
  on a K3 surface and for the variety of lines on a very general cubic
  fourfold.
\end{abstract}

\maketitle

\vspace{-27pt}

\tableofcontents
\restoregeometry

\newpage

\section*{Introduction}

\subsection*{A. Abelian varieties} Let $A$ be an abelian variety of
dimension $d$ over a field $k$. Let $\hat{A} = \mathrm{Pic}^0(A)$ be
its dual and let $L$ be the Poincar\'e line bundle on $A \times
\hat{A}$ viewed as an element of $\CH^1(A\times \hat{A})$. The
\emph{Fourier transform} on the Chow groups with rational coefficients
is defined as
\begin{center}
  $\FF(\sigma)  := p_{2,*}(e^L \cdot p_1^*\sigma)$, \quad for all
  $\sigma \in \CH^i(A)$.
\end{center}
Here, $e^L  := [A\times \hat{A}] + L + \frac{L^2}{2!} + \ldots +
\frac{L^{2d}}{(2d)!}$, and $p_1  : A \times \hat{A} \rightarrow A$ and
$p_2 : A \times \hat{A} \rightarrow \hat{A}$ are the two projections.
The main result of \cite{beauville1} is the following.

\begin{thm3}[Beauville] Let $A$ be an abelian variety of dimension
  $d$. The Fourier transform induces a canonical splitting
  \begin{equation}
    \label{eq fourier abelian} \CH^i(A) = \bigoplus_{s=i-d}^{i} \CH^i(A)_s,
    \quad  \mbox{where} \ \CH^i(A)_s  := \{ \sigma \in \CH^i(A)  :
    \FF(\sigma) \in \CH^{d-i+s}(\hat{A})\}.
  \end{equation}
  Furthermore, this decomposition enjoys the following two
  properties~: \medskip

  (a) $\CH^i(A)_s = \{\sigma \in \CH^i(A) : [n]^*\sigma =
  n^{2i-s}\sigma\}$, where $[n] : A \rightarrow A$ is the
  multiplication-by-$n$ map~;

  (b) $\CH^i(A)_s\cdot \CH^j(A)_r \subseteq \CH^{i+j}(A)_{r+s}$.\qed
\end{thm3}

\noindent Property (a) shows that the Fourier decomposition \eqref{eq
  fourier abelian} is canonical, while Property (b) shows that the
Fourier decomposition is compatible with the ring structure on
$\CH^*(A)$ given by intersection product. It should be mentioned that,
as explained in \cite{beauville1}, \eqref{eq fourier abelian} is
expected to be the splitting of a Bloch--Beilinson type filtration on
$\CH^*(A)$. By \cite{dm}, this splitting is in fact induced by a
Chow--K\"unneth decomposition of the diagonal and it is of
Bloch--Beilinson type if it satisfies the following two properties~:
\medskip

(B) $\CH^i(A)_s
=0$ for all $s<0$~;

(D) the cycle class map $\CH^i(A)_0 \rightarrow \HH^{2i}(A,\Q)$ is
injective for all $i$.\medskip

\noindent Actually, if Property (D) is true for all abelian varieties,
then Property (B) is true for all abelian varieties~; see
\cite{vial2}.

A direct consequence of the Fourier decomposition on the Chow ring of
an abelian variety is the following. First note that $\CH^d(A)_0 =
\langle[0] \rangle$, where $[0]$ is the class of the identity element
$0 \in A$. Let $D_1, \ldots, D_d$ be symmetric divisors, that is,
divisors such that $[-1]^*D_i = D_i$, or equivalently $D_i \in
\CH^1(A)_0$, for all $i$. Then
\begin{equation} \label{eq weaksplitting} D_1\cdot D_2\cdot \ldots
  \cdot D_d = \deg(D_1\cdot D_2\cdot \ldots \cdot D_d) \,[0] \quad
  \mbox{in } \CH^d(A).
\end{equation} 
So far, for lack of a Bloch--Beilinson type filtration on the Chow
groups of hyperk\"ahler varieties, it is this consequence \eqref{eq
  weaksplitting}, and variants thereof, of the canonical splitting of
the Chow ring of an abelian variety that have been tested for certain
types of hyperk\"ahler varieties.  Property \eqref{eq weaksplitting}
for divisors on hyperk\"ahler varieties $F$ is Beauville's weak
splitting conjecture \cite{beauville2}~; it was subsequently
strengthened in \cite{voisin2} to include the Chern classes of $F$.
The goal of this manuscript is to show that a Fourier decomposition
should exist on the Chow ring of hyperk\"ahler fourfolds which are
deformation equivalent to the Hilbert scheme of length-$2$ subschemes
on a K3 surface. Before making this more precise, we first consider
the case of K3 surfaces.\medskip

\subsection*{B. K3 surfaces}
Let $S$ be a complex projective K3 surface. In that case, a
Bloch--Beilinson type filtration on $\CH^*(S)$ is explicit~: we have
$\F^1\CH^2(S)= \F^2\CH^2(S) = \CH^2(S)_{\hom} := \ker \{cl : \CH^2(S)
\rightarrow \HH^4(S,\Q)\}$ and $\F^1\CH^1(S) = 0$. Beauville and
Voisin \cite{bv} observed that this filtration splits canonically by
showing the existence of a zero-cycle $\mathfrak{o}_S \in \CH^2(S)$,
which is the class of any point lying on a rational curve on $S$, such
that the intersection of any two divisors on $S$ is proportional to
$\mathfrak{o}_S$. Let us introduce the Chow--K\"unneth decomposition
\begin{equation} \label{eq CKK3} \pi^0_S := \mathfrak{o}_S \times S,
  \quad \pi^4_S := S \times \mathfrak{o}_S \quad \mbox{and} \quad
  \pi^2_S := \Delta_S - \mathfrak{o}_S \times S - S \times
  \mathfrak{o}_S \quad \mbox{in}\ \CH^2(S\times S).
\end{equation} 
Cohomologically, $\pi_S^0$, $\pi_S^2$ and $\pi^4_S$ are the K\"unneth
projectors on $\HH^0(S,\Q)$, $\HH^2(S,\Q)$, and $\HH^4(S,\Q)$,
respectively. The result of Beauville--Voisin shows that among the
Chow--K\"unneth decompositions of $\Delta_S$ (note that a pair of
zero-cycles of degree $1$ induces a Chow--K\"unneth decomposition for
$S$ and two distinct pairs induce distinct such decompositions) the
symmetric one associated to $\mathfrak{o}_S$, denoted $\Delta_S =
\pi^0_S + \pi^2_S + \pi^4_S$, is special because the decomposition it
induces on the Chow groups of $S$ is compatible with the ring
structure on $\CH^*(S)$.  \medskip

A key property satisfied by $\mathfrak{o}_S$, proved in \cite{bv}, is
that $c_2(S) = 24 \, \mathfrak{o}_S$.  Let $\iota_\Delta : S
\rightarrow S \times S$ be the diagonal embedding. Having in mind that
the top Chern class of the tangent bundle $c_2(S)$ is equal to
$\iota_\Delta^*\Delta_S$, we see that the self-intersection of the
Chow--K\"unneth projector $\pi_S^2$ satisfies
\begin{align*}
  \pi_S^2 \cdot \pi_S^2 & = \Delta_S^{\cdot 2} - 2\, \Delta_S \cdot
  (\mathfrak{o}_S \times S + S \times \mathfrak{o}_S) +
  (\mathfrak{o}_S \times S  + S \times \mathfrak{o}_S)^{\cdot 2} \\
  & = (\iota_\Delta)_*  c_2(S) - 2\,\mathfrak{o}_S \times \mathfrak{o}_S  \\
  &= 22\,\mathfrak{o}_S \times \mathfrak{o}_S.
\end{align*}
We then observe that the action of $e^{\pi_S^2} := S \times S +
\pi^2_S + \frac{1}{2}\pi_S^2 \cdot \pi_S^2 = S \times S + \pi^2_S +
11\,\mathfrak{o}_S \times \mathfrak{o}_S$ on the Chow group
$\CH^*(S)$, called the \emph{Fourier transform} and denoted
$\mathcal{F}$, induces the same splitting as the Chow--K\"unneth
decomposition considered above. Indeed, writing $$\CH^i(S)_s := \{
\sigma \in \CH^i(S) : \mathcal{F}(\sigma) \in \CH^{2-i+s}(S) \},$$ we
also have $$\CH^i(S)_s = (\pi_S^{2i-s})_*\CH^i(S).$$ With these
notations, we then have $\CH^2(S)=\CH^2(S)_0 \oplus \CH^2(S)_2$,
$\CH^1(S)=\CH^1(S)_0$ and $\CH^0(S)= \CH^0(S)_0$, with the
multiplicative property that
$$\CH^2(S)_0 = \CH^1(S)_0 \cdot \CH^1(S)_0 = \langle\mathfrak{o}_S
\rangle.$$

The Beauville--Bogomolov form $q_S$ on a K3 surface $S$ is simply
given by the cup-product on $\HH^2(S,\Q)$~; its inverse $q_S^{-1}$
defines an element of $\HH^2(S,\Q) \otimes \HH^2(S,\Q)$. By the
K\"unneth formula, we may view $q_S^{-1}$ as an element of $\HH^4(S
\times S, \Q)$ that we denote $\mathfrak{B}$ and call the
Beauville--Bogomolov class. The description above immediately gives
that the cohomology class of $\pi^2_S$ in $\HH^4(S \times S, \Q)$ is
$\mathfrak{B}$.  Let us denote $\mathfrak{b} :=
\iota_\Delta^*\mathfrak{B}$ and $\mathfrak{b}_i = p_i^*\mathfrak{b}$,
where $p_i : S \times S \rightarrow S$ is the $i^\mathrm{th}$
projection, $i=1,2$. On the one hand, a cohomological computation
yields $$\mathfrak{B} = [\Delta_S] - \frac{1}{22}(\mathfrak{b}_1 +
\mathfrak{b}_2).$$ On the other hand, the cycle $L := \Delta_S -
\mathfrak{o}_S \times S - S \times \mathfrak{o}_S$, which was
previously denoted $\pi_S^2$, satisfies $\iota_\Delta^*L = c_2(S) - 2
\mathfrak{o}_S = 22 \mathfrak{o}_S$.  Thus not only does $L$ lift
$\mathfrak{B}$ to rational equivalence, but $L$ also lifts the above
equation satisfied by $\mathfrak{B}$ in the sense that $$L=\Delta_S -
\frac{1}{22}(l_1 + l_2)$$ in the Chow group $\CH^2(S \times S)$, where
$l := \iota_\Delta^*L$ and $l_i := p_i^*l$, $i=1,2$. Moreover, a cycle
$L$ that satisfies the equation $L=\Delta_S -
\frac{1}{22}\left((p_1^*\iota_\Delta^*L) +
  (p_2^*\iota_\Delta^*L)\right)$ is unique. Indeed, if $L+
\varepsilon$ is another cycle such that $L +\varepsilon =\Delta_S -
\frac{1}{22}\big(p_1^*\iota_\Delta^*(L+\varepsilon) +
p_2^*\iota_\Delta^*(L+\varepsilon)\big)$, then $\varepsilon = -
\frac{1}{22}(p_1^*\iota_\Delta^*\varepsilon +
p_2^*\iota_\Delta^*\varepsilon)$. Consequently,
$\iota_\Delta^*\varepsilon = -\frac{1}{11}\iota_\Delta^*\varepsilon$.
Thus $\iota_\Delta^*\varepsilon = 0$ and hence $\varepsilon = 0$.
\medskip

We may then state a refinement of the main result of
\cite{bv}.

\begin{thm3}[Beauville--Voisin, revisited]
  Let $S$ be a complex projective K3 surface. Then there exists a
  unique $2$-cycle $L \in \CH^2(S \times S)$ such that $L=\Delta_S -
  \frac{1}{22}(l_1 + l_2)$. Moreover, $L$ represents the
  Beauville--Bogomolov class $\mathfrak{B}$ on $S$ and the Fourier
  transform associated to $L$ induces a splitting of the Chow ring
  $\CH^*(S)$.\qed
\end{thm3}\medskip

\subsection*{C. Hyperk\"ahler varieties of $\mathrm{K3}^{[2]}$-type}
Hyperk\"ahler manifolds are simply connected compact K\"ahler
manifolds $F$ such that $\HH^0(F,\Omega_F^2)$ is spanned by a nowhere
degenerate $2$-form. This class of manifolds constitutes a natural
generalization of the notion of a K3 surface in higher dimensions.
Such manifolds are endowed with a canonical symmetric bilinear form
$q_F : \HH^2(F,\Z) \otimes \HH^2(F,\Z) \rightarrow \Z$ called the
Beauville--Bogomolov form. Passing to rational coefficients, its
inverse $q_F^{-1}$ defines an element of $\HH^2(F,\Q) \otimes
\HH^2(F,\Q)$ and we define the \emph{Beauville--Bogomolov class}
$\mathfrak{B}$ in $\HH^4(F\times F,\Q)$ to be the class corresponding
to $q_F^{-1}$ via the K\"unneth formula. In this manuscript, we will
mostly be concerned with projective hyperk\"ahler manifolds and these
will simply be called \emph{hyperk\"ahler varieties}.  A hyperk\"ahler
variety is said to be of \emph{$\mathrm{K3}^{[n]}$-type} if it is
deformation equivalent to the Hilbert scheme of length-$n$ subschemes
on a K3 surface. A first result is that if $F$ is a hyperk\"ahler
variety of $\mathrm{K3}^{[n]}$-type, then there is a cycle $L \in
\CH^2(F \times F)$ defined in \eqref{def L Markman sheaf} whose
cohomology class is $\mathfrak{B} \in \HH^4(F\times F,\Q)$~; see
Theorem \ref{thm existence of L}.  \medskip

\noindent \textbf{C.1. The Fourier decomposition for the Chow groups.}
Given a hyperk\"ahler variety $F$ of $\mathrm{K3}^{[2]}$-type endowed
with a cycle $L$ with cohomology class $\mathfrak{B}$, we define a
descending filtration $\F^\bullet$ on the Chow groups $\CH^i(F)$ as
\begin{equation} \label{eq filtration gal} \F^{2k+1}\CH^i(F) =
  \F^{2k+2}\CH^i(F)  := \ker \{(L^{4-i+k})_*  : \F^{2k}\CH^i(F)
  \rightarrow \CH^{4-i+2k}(F)\}.
\end{equation}
The cohomological description of the powers of $[L]=\mathfrak{B}$
given in Proposition \ref{prop action of B powers} shows that this
filtration should be of Bloch--Beilinson type.  We then define the
\emph{Fourier transform} as
\begin{center}
  $\FF(\sigma) = (p_2)_*(e^L \cdot p_1^*\sigma)$, \quad for all
  $\sigma \in \CH^*(F)$.
\end{center}
In this work, we ask if there is a canonical choice of a
codimension-$2$ cycle $L$ on $F \times F$ lifting the
Beauville--Bogomolov class $\mathfrak{B}$ such that the Fourier
transform associated to $L$ induces a splitting of the conjectural
Bloch--Beilinson filtration on the Chow group $\CH^*(F)$ compatible
with its ring structure given by intersection product.
We give positive answers in the case when $F$ is either the Hilbert
scheme of length-$2$ subschemes on a K3 surface or the variety of
lines on a cubic fourfold.\medskip

Consider now a projective hyperk\"ahler manifold $F$ of
$\mathrm{K3}^{[2]}$-type and let $\mathfrak{B}$ denote its
Beauville--Bogomolov class, that is, the inverse of its
Beauville--Bogomolov form seen as an element of $\HH^4(F\times F,
\Q)$. We define $\mathfrak{b} :=\iota_\Delta^*\mathfrak{B}$ and
$\mathfrak{b}_i :=p_i^*\mathfrak{b}$. Proposition \ref{prop action of
  B powers} shows that
$\mathfrak{B}$ is uniquely determined up to sign by the following
quadratic equation~:
\begin{equation}\label{eq cohomological equation S2}
  \mathfrak{B}^2=2\,[\Delta_F]
  -\frac{2}{25}(\mathfrak{b}_1+\mathfrak{b}_2) \cdot \mathfrak{B}
  -\frac{1}{23\cdot 25}(2\mathfrak{b}_1^2 -
  23\mathfrak{b}_1\mathfrak{b}_2 + 2 \mathfrak{b}_2^2)
  \quad \mbox{in}\ \HH^8(F\times F,\Q).
\end{equation}
We address the following~; see also Conjecture \ref{conj main general}
for a more general version involving hyperk\"ahler fourfolds whose
cohomology ring is generated by degree-$2$ classes.

 \begin{conj2}\label{conj main}
   Let $F$ be a hyperk\"ahler variety $F$ of $\mathrm{K3}^{[2]}$-type.
   Then there exists a cycle $L \in \CH^2(F \times F)$
   with cohomology class $\mathfrak{B} \in \HH^4(F \times F, \Q)$
   satisfying
   \begin{equation}\label{eq rational equation}
     L^2=2 \, \Delta_F -\frac{2}{25}(l_1+l_2)\cdot L
     -\frac{1}{23\cdot 25}(2l_1^2 - 23l_1l_2 + 2l_2^2)
     \quad \mbox{in}\ \CH^4(F \times F),
 \end{equation}
 where by definition we have set $l  := \iota_\Delta^*L$ and $l_i  :=
 p_i^*l$.
 \end{conj2}

 In fact, we expect the symmetric cycle $L$ defined in \eqref{def L
   Markman sheaf} to satisfy the quadratic equation \eqref{eq rational
   equation}. Moreover, we expect a symmetric cycle $L \in \CH^2(F)$
 representing the Beauville--Bogomolov class $\mathfrak{B}$ to be
 uniquely determined by the quadratic equation \eqref{eq rational
   equation}~; see Proposition \ref{prop L unique} for some evidence.
 From now on, when $F$ is the Hilbert scheme of length-$2$ subschemes
 on a K3 surface, $F$ is endowed with the cycle $L$ defined in
 \eqref{eq L S2} -- it agrees with the one defined in \eqref{def L
   Markman sheaf} by Proposition \ref{prop L agree}. When $F$ is the
 variety of lines on a cubic fourfold, $F$ is endowed with the cycle
 $L$ defined in \eqref{eq L Cubic} -- although we do not give a proof,
 this cycle \eqref{eq L Cubic} should agree with the one defined in
 \eqref{def L Markman sheaf}.\medskip

 Our first result, upon which our work is built, is the following
 theorem~; see Theorem \ref{prop cubic L conjecture} and Theorem
 \ref{thm cubic L conjecture}.
 \begin{thm2} \label{thm2 main} Let $F$ be either the Hilbert scheme
   of length-$2$ subschemes on a K3 surface or the variety of lines on
   a cubic fourfold. Then Conjecture \ref{conj main} holds for $F$.
 \end{thm2}

Let us introduce the cycle $l :=\iota_\Delta^*L \in \CH^2(F)$ -- it
turns out that $l=\frac{5}{6}c_2(F)$ when $F$ is either the Hilbert
scheme of length-$2$ subschemes on a K3 surface or the variety of
lines on a cubic fourfold~; see \eqref{eq l on S2} and \eqref{eq l
  cubic fourfold}. The following hypotheses, together with \eqref{eq
  rational equation}, constitute the key relations towards
establishing a Fourier decomposition for the Chow groups of $F$~:
\begin{align}
  L_*l^2 &= 0\, ;& \label{assumption pre l}  \\
  L_*(l\cdot L_*\sigma) & =25\, L_*\sigma
  &  \mbox{for all} \ \sigma \in \CH^4(F) ; \label{assumption hom}\\
  (L^2)_*(l\cdot(L^2)_*\tau) &= 0 & \mbox{for all}\ \tau \in
  \CH^2(F).\label{assumption 2hom}
\end{align}
Indeed, Theorem \ref{thm main splitting} below shows that in order to
establish the existence of a Fourier decomposition on the Chow groups
of a hyperk\"ahler variety of $\mathrm{K3}^{[2]}$-type, it suffices to
show that the cycle $L$ defined in \eqref{def L Markman sheaf} (which
is a characteristic class of Markman's twisted sheaf \cite{markman})
satisfies \eqref{eq rational equation}, \eqref{assumption pre l},
\eqref{assumption hom} and \eqref{assumption 2hom}. Note that
Properties \eqref{assumption hom} and \eqref{assumption 2hom} describe
the intersection of $l$ with $2$-cycles on $F$.

\begin{thm2} \label{thm main splitting} Let $F$ be a hyperk\"ahler
  variety of K3$^{[2]}$-type. Assume that there exists a cycle $L \in
  \CH^2(F \times F)$ representing the Beauville--Bogomolov class
  $\mathfrak{B}$ satisfying \eqref{eq rational equation},
  \eqref{assumption pre l}, \eqref{assumption hom} and
  \eqref{assumption 2hom}. For instance, $F$ could be either the
  Hilbert scheme of length-$2$ subschemes on a K3 surface endowed with
  the cycle $L$ of \eqref{eq L S2} or the variety of lines on a cubic
  fourfold endowed with the cycle $L$ of \eqref{eq L Cubic}.
  Denote
  \begin{align*} \CH^i(F)_s & := \{\sigma
    \in \CH^i(F)  : \FF(\sigma) \in \CH^{4-i+s}(F)\}.\\
    \intertext{Then the Chow groups of $F$ split canonically as}
    \CH^0(F)&=\CH^0(F)_0 \, ;\\
    \CH^1(F)& = \CH^1(F)_0\, ; \\
     \CH^2(F)&= \CH^2(F)_0 \oplus \CH^2(F)_2\, ; \\
     \CH^3(F)& = \CH^3(F)_0\oplus \CH^3(F)_2\, ;\\
     \CH^4(F) & = \CH^4(F)_0 \oplus \CH^4(F)_2 \oplus \CH^4(F)_4.
 \end{align*}
Moreover, we have
\begin{center}
  $\CH^i(F)_s = \mathrm{Gr}_{\F^\bullet}^{s}\CH^i(F),$ where
  $\F^\bullet$ denotes the filtration \eqref{eq filtration gal}~;
\end{center}
and 
\begin{center} $\sigma$ belongs to $\CH^i(F)_s$ if and only if
  $\FF(\sigma)$ belongs to $\CH^{4-i+s}(F)_s$.
\end{center}
\end{thm2}

We give further evidence that the splitting obtained in Theorem
\ref{thm main splitting} is the splitting of a conjectural filtration
$\F^\bullet$ on $\CH^*(F)$ of Bloch--Beilinson type by showing that it
arises as the splitting of a filtration induced by a Chow--K\"unneth
decomposition of the diagonal~; see Theorem \ref{thm CK}. Note that
our indexing convention is such that the graded piece of the
conjectural Bloch--Beilinson filtration $\CH^i(F)_s =
\mathrm{Gr}_{\F}^{s}\CH^i(F)$ should only depend on the cohomology
group $\HH^{2i-s}(F,\Q)$, or rather, on the Grothendieck motive
$\mathfrak{h}_{\mathrm{hom}}^{2i-s}(F)$.\medskip

The proof that the Chow groups of a hyperk\"ahler variety of
K3$^{[2]}$-type satisfying hypotheses \eqref{eq rational equation},
\eqref{assumption pre l}, \eqref{assumption hom} and \eqref{assumption
  2hom} have a Fourier decomposition as described in the conclusion of
Theorem \ref{thm main splitting} is contained in Theorems \ref{prop
  L2} \& \ref{prop main}. That the Hilbert scheme of length-$2$
subschemes on a K3 surface endowed with the cycle $L$ of \eqref{eq L
  S2} satisfies hypotheses \eqref{eq rational equation},
\eqref{assumption pre l}, \eqref{assumption hom} and \eqref{assumption
  2hom} is Theorem \ref{prop cubic L conjecture}, Proposition
\ref{prop cubic assumption pre l hom} and Proposition \ref{lem
  condition 3}. That the variety of lines on a cubic fourfold endowed
with the cycle $L$ of \eqref{eq L Cubic} satisfies hypotheses
\eqref{eq rational equation}, \eqref{assumption pre l},
\eqref{assumption hom} and \eqref{assumption 2hom} is Theorem \ref{thm
  cubic L conjecture}, Proposition \ref{prop cubic assumption l hom}
and Proposition \ref{prop cubic assumption 2hom}.\medskip

Theorem \ref{thm main splitting} is of course reminiscent of the case
of abelian varieties where the Fourier decomposition on the Chow
groups is induced by the exponential of the Poincar\'e bundle.
Beauville's proof relies essentially on the interplay of the
Poincar\'e line bundle and the multiplication-by-$n$ homomorphisms.
Those homomorphisms are used in a crucial way to prove the
compatibility of the Fourier decomposition with the intersection
product.  The difficulty in the case of hyperk\"ahler varieties is
that there are no obvious analogues to the multiplication-by-$n$
morphisms~; see however Remark \ref{rmk lagrangian}. Still, when $F$
is the variety of lines on a cubic fourfold, Voisin \cite{voisin3}
defined a rational self-map $\varphi :F\dashrightarrow F$
as follows. For a general point $[l]\in F$ representing a line $l$ on
$X$, there is a unique plane $\Pi$ containing $l$ which is tangent to
$X$ along $l$. Thus $\Pi\cdot X=2l+l'$, where $l'$ is the residue
line. Then one defines $\varphi([l])=[l']$. In \S \ref{sec varphi}, we
study the graph of $\varphi$ in depth and completely determine its
class both modulo rational equivalence and homological equivalence. It
turns out that the action of $\varphi$ on the Chow groups of $F$
respects the Fourier decomposition~; see \S \ref{sec varphi fourier
  compatible}. Thus in many respects the rational map $\varphi$ may be
considered as an ``endomorphism'' of $F$. The interplay of $\varphi$
with $L$ is used to prove many features of the Fourier decomposition
on the Chow groups of $F$.\\

\noindent \textbf{C.2. The Fourier decomposition for the Chow ring.}
As in the case of abelian varieties or K3 surfaces, we are interested
in the compatibility of the Fourier decomposition of Theorem \ref{thm
  main splitting} with the ring structure of $\CH^*(F)$. In the
hyperk\"ahler case, this was initiated by Beauville \cite{beauville2}
who considered the sub-ring of $\CH^*(F)$ generated by divisors on the
Hilbert scheme of length-$2$ subschemes on a K3 surface, and then
generalized by Voisin \cite{voisin2} who considered the sub-ring of
$\CH^*(F)$ generated by divisors and the Chern classes of the tangent
bundle when $F$ is either the Hilbert scheme of length-$2$ subschemes
on a K3 surface or the variety of lines on a cubic fourfold~:
\begin{thm3}[Beauville \cite{beauville2}, Voisin \cite{voisin2}] Let
  $F$ be either the Hilbert scheme of length-$2$ subschemes on a K3
  surface, or the variety of lines on a cubic fourfold. Then any
  polynomial expression $P(D_i,c_2(F))$, $D_i \in \CH^1(F)$, which is
  homologically trivial vanishes in the Chow ring
  $\CH^*(F)$.  \end{thm3}
\noindent This theorem implies the existence of a zero-cycle
$\mathfrak{o}_F \in \CH_0(F)$ which is the class of a point such
that $$\langle \mathfrak{o}_F \rangle = \langle c_2(F)^2 \rangle =
\langle c_2(F) \rangle \cdot \CH^1(F)^{\cdot 2} = \CH^1(F)^{\cdot
  4}.$$ This latter result can already be restated, in the context of
our Fourier decomposition, as follows~; see Theorem \ref{thm
  reformulation voisin}.
$$\CH^4(F)_0 = \langle
l^2 \rangle = \langle l \rangle \cdot \CH^1(F)_0^{\cdot 2} =
\CH^1(F)_0^{\cdot 4}.$$ We ask whether \begin{center} $\CH^i(F)_s\cdot
  \CH^j(F)_r \subseteq \ \CH^{i+j}(F)_{r+s}$, \quad for all
  $(i,s),(j,r)$.
\end{center} The following theorem answers this question affirmatively
when $F$ is the Hilbert scheme of length-$2$ subschemes on a K3
surface or the variety of lines on a very general cubic
fourfold. Hilbert schemes of length-$2$ subschemes on K3 surfaces are
dense in the moduli of hyperk\"ahler varieties of
$\mathrm{K3}^{[2]}$-type, and the varieties of lines on cubic
fourfolds form an algebraic component of maximal dimension. Therefore
Theorem \ref{thm main mult} gives strong evidence that a Fourier
decomposition on the Chow ring of hyperk\"ahler varieties of
$\mathrm{K3}^{[2]}$-type should exist.

\begin{thm2} \label{thm main mult} Let $F$ be either the Hilbert
  scheme of length-$2$ subschemes on a K3 surface or the variety of
  lines on a very general cubic fourfold.
  Then
\begin{equation*}
  \CH^i(F)_s\cdot \CH^j(F)_r \subseteq \ \CH^{i+j}(F)_{r+s}, \quad
  \mbox{for all} \ (i,s),(j,r).
 \end{equation*}
 Moreover equality holds except when $\CH^{i+j}(F)_{r+s} =
     \CH^3(F)_2$ or $\CH^2(F)_0$.
\end{thm2}

We actually show Theorem \ref{thm main mult} in a few more cases.
Indeed, it is shown in Theorem \ref{thm birational invariance} that
the existence of a Fourier decomposition on the Chow groups or Chow
ring of a hyperk\"ahler variety of $\mathrm{K3}^{[2]}$-type is a
birational invariant, so that the conclusion of Theorem \ref{thm main
  mult} also holds for any hyperk\"ahler fourfold that is birational
to the Hilbert scheme of length-$2$ subschemes on a K3 surface or to
the variety of lines on a very general cubic fourfold.

The proof of Theorem \ref{thm main mult} uses in an essential way a
Theorem of Beauville--Voisin \cite{bv} on the vanishing of the
``modified diagonal'' of the K3 surface $S$ in the case when $F$ is
the Hilbert scheme $S^{[2]}$, while it uses in an essential way the
rational self-map $\varphi  : F \dashrightarrow F$ constructed by
Voisin \cite{voisin3} when $F$ is the variety of lines on a cubic
fourfold. Theorem \ref{thm main mult} is proved for divisors in
Section \ref{sec mult div}, while it is proved in full generality for
$S^{[2]}$ in Theorem \ref{thm multS2complete} and for the variety of
lines on a very general cubic fourfold in Section \ref{sec mult}. In the $S^{[2]}$ case, we prove in fact a stronger result~: the Fourier decomposition is induced by a \emph{multiplicative Chow--K\"unneth decomposition} of the diagonal in the sense of Definition \ref{def multCK} ; see Theorem \ref{thm multS2complete}.

In addition to stating the multiplicativity property of the Fourier
decomposition, Theorem \ref{thm main mult} also states that
$\CH^2(F)_2 \cdot \CH^2(F)_2 = \CH^4(F)_4$. This equality, which also
holds for the variety of lines on a (not necessarily very general)
cubic fourfold, reflects at the level of Chow groups, as predicted by
the Bloch--Beilinson conjectures, the fact that the transcendental
part of the Hodge structure $\HH^4(F,\Q)$ is a sub-quotient of
$\Sym^2\HH^2(F,\Q)$. A proof in the case of $S^{[2]}$ can be found in
Proposition \ref{prop F4 of S2} and a proof in the case of the variety
of lines on a cubic fourfold can be found in Proposition \ref{prop
  cubic CH0426}.  It is also expected from the Bloch--Beilinson
conjectures that $\CH^2(F)_{0,\hom}=0$~; see Theorem \ref{thm CK}. In
fact, for $F=S^{[2]}$, this would essentially follow from the validity
of Bloch's conjecture for $S$. Although we cannot prove such a
vanishing, a direct consequence of Theorem \ref{thm main mult} is that
\begin{equation*}
  \CH^1(F) \cdot \CH^2(F)_{0,\hom} = 0 \quad \mbox{and} \quad \CH^2(F)_0
  \cdot \CH^2(F)_{0,\hom} = 0.
\end{equation*}
Propositions \ref{prop CH3624hom} \& \ref{prop cubic mult structure1}
show that the above identities also hold when $F$ is the variety of
lines on a (not necessarily very general) cubic fourfold.

Furthermore, we ought to mention that each piece of the decomposition
of $\CH^i(F)$ obtained in Theorem \ref{thm main splitting} is
non-trivial.  Indeed, we have $\CH^4(F)_0=\langle l^2 \rangle$ and $l
\cdot  : \CH^2(F)_2 \rightarrow \CH^4(F)$ is injective with image
$\CH^4(F)_2$~; see Theorems \ref{prop L2} \& \ref{prop
  main}. Therefore, $\CH^4(F)_0 \oplus \CH^4(F)_2$ is supported on a
surface. Since $\HH^4(F,\mathcal{O}_F) \neq 0$, it follows from
Bloch--Srinivas \cite{bs} that $\CH^4(F)_4 \neq 0$. Thus, because
$\CH^2(F)_2 \cdot \CH^2(F)_2 = \CH^4(F)_4$,
$\CH^2(F)_2 \neq 0$ and hence $\CH^4(F)_2 \neq 0$. \\

A direct consequence of Theorems \ref{thm main splitting} \& \ref{thm
  main mult} for zero-dimensional cycles is the following theorem
which is analogous to the decomposition of the Chow group of
zero-cycles on an abelian variety as can be found in \cite[Proposition
4]{beauville1}.

\begin{thm2} \label{thm mult} Let $F$ be either the Hilbert scheme of
  length-$2$ subschemes on a K3 surface or the variety of lines on a
  cubic fourfold. Then
    $$
   \CH^4(F) = \langle l^2 \rangle \oplus \langle l \rangle \cdot
   L_*\CH^4(F) \oplus (L_*\CH^4(F))^{\cdot 2}.
  $$
  Moreover, this decomposition agrees with the Fourier decomposition
  of Theorem \ref{thm main splitting}.
\end{thm2}

Let us point out that, with the notations of Theorem \ref{thm main
  splitting}, $L_*\CH^4(F) = \CH^2(F)_2$.  A hyperk\"ahler variety is
simply connected. Thus its first Betti number vanishes and a theorem
of Rojtman \cite{roitman} implies that $\CH_{\Z}^4(F)_{\hom}$ is
uniquely divisible. Therefore, Theorem \ref{thm mult} can actually be
stated for $0$-dimensional cycles with integral coefficients~:
$$\CH^4_{\Z}(F) = \Z\mathfrak{o}_F \oplus \langle l \rangle \cdot
(L_*\CH^4(F)) \oplus (L_*\CH^4(F))^{\cdot 2}.$$
\medskip

Finally, let us mention the following remark which could be useful to
future work. As it is not essential to the work presented here, the
details are not expounded. Given a divisor $D$ on $F$ with $q_F([D])
\neq 0$ where $q_F$ denotes the quadratic form attached to the
Beauville--Bogomolov form, the cycle $L_D  := L - \frac{1}{q_F([D])}D_1
\cdot D_2,$ where $D_i  :=p_i^*D$, $i=1,2$, defines a \textit{special
  Fourier transform} $\mathcal{F}_D : \CH^*(F)\rightarrow\CH^*(F),
x\mapsto (p_{2})_*( e^{L_D} \cdot p_1^*x)$ such that $\mathcal{F}_D
\circ \mathcal{F}_D$ induces a further splitting of $\CH^*(F)$ which
takes into account $D$. For instance, if $F$ is the variety of lines
on a cubic fourfold and if $g$ is the Pl\"ucker polarization on $F$,
then $\mathcal{F}_g \circ \mathcal{F}_g$ induces an orthogonal
decomposition $\CH^1(F) = \langle g \rangle \oplus
\CH^1(F)_{\mathrm{prim}}$ but also a further decomposition $\CH^3(F)_2
= A \oplus B$, where $\mathcal{F}_g \circ \mathcal{F}_g$ acts as the
identity on $A$ and as zero on $B = g\cdot \CH^2(F)_2$. It turns out
that $\varphi^*$ acts by multiplication by $4$ on $A$ and by
multiplication by $-14$ on $B$.\medskip

\subsection*{D. Hyperk\"ahler varieties} Let $F$ be a hyperk\"ahler
variety.  In general, the sub-algebra of $\HH^*(F \times F,\Q)$
generated by the Beauville--Bogomolov class $\mathfrak{B}$,
$\mathfrak{b}_1$ and $\mathfrak{b}_2$ only ``sees'' the sub-Hodge
structure of $\HH^*(F,\Q)$ generated by $\HH^2(F,\Q)$. Precisely, it
can be checked that the cohomological Fourier transform $\FF$ acts
trivially on the orthogonal complement of the image of $\mathrm{Sym}^n
\, \HH^2(F,\Q)$ inside $\HH^{2n}(F,\Q)$. Therefore, it does not seem
possible to formulate directly an analogous Fourier decomposition, as
that of Theorems \ref{thm main splitting} \& \ref{thm main mult}, for
the Chow ring of those hyperk\"ahler varieties $F$ whose cohomology is
not generated by $\HH^2(F,\Q)$. There are however three questions that
we would like to raise concerning algebraic cycles on hyperk\"ahler
varieties. \medskip
\pagebreak

\noindent \textbf{D.1. Zero-cycles on hyperk\"ahler varieties.} If
$\omega$ is a nowhere degenerate $2$-form on a hyperk\"ahler variety
$F$, then the powers $\omega^n \in \mathrm{Sym}^n \, \HH^2(F,\C)$ of
$\omega$ span the degree-zero graded part of the cohomology of $F$ for
the coniveau filtration.  Hence we can expect that if a canonical
$2$-cycle $L \in \CH^2(F \times F)$ representing the
Beauville--Bogomolov class $\mathfrak{B}$ exists, then the Fourier
transform $\FF :=e^L$ splits a Bloch--Beilinson type filtration on
$\CH_0(F)$.  It is thus tempting to ask whether the decomposition of
Theorem \ref{thm mult} holds for hyperk\"ahler varieties~:

\begin{conj2} \label{conj fourier} Let $F$ be a hyperk\"ahler variety
  of dimension $2d$. Then there exists a canonical cycle $L \in
  \CH^2(F \times F)$ representing the Beauville--Bogomolov class
  $\mathfrak{B}$ which induces a canonical splitting $$\CH^{2d}(F) =
  \bigoplus_{s=0}^{d} \CH^{2d}(F)_{2s},$$ where $\CH^{2d}(F)_{2s} :=
  \{\sigma \in \CH^{2d}(F) : \FF(\sigma) \in \CH^{2s}(F)\}$. Moreover,
  we have
  $$\CH^{2d}(F)_{2s} = \langle l^{d-s}\rangle \cdot
  (L_*\CH^{2d}(F))^{\cdot s}$$ and $$\CH^{2d}(F)_{2s} \supseteq
  P(l,D_1,\ldots,D_r) \cdot (L_*\CH^{2d}(F))^{\cdot s},$$ for any
  degree $2d-2s$ weighted homogeneous polynomial $P$ in $l$ and
  divisors $D_i$, $i=1,\ldots,r$.
\end{conj2}
Note that in the case of hyperk\"ahler varieties of
$\mathrm{K3}^{[n]}$-type, a candidate for a canonical cycle $L$
representing $\mathfrak{B}$ is given by Theorem \ref{thm existence of
  L}.  Note also that, in general, because $[L] = \mathfrak{B}$
induces an isomorphism $\HH^{2d-2}(F,\Q) \rightarrow \HH^2(F,\Q)$ a
consequence of the Bloch--Beilinson conjectures would be that
$L_*\CH^{2d}(F) = \ker \{AJ^2 : \CH^2(F)_{\hom} \rightarrow J^2(F)
\otimes \Q \}$, where $AJ^2$ denotes Griffiths' Abel--Jacobi map
tensored with $\Q$. \medskip

\noindent \textbf{D.2. On the existence of a distinguished cycle $L
  \in \CH^2(F \times F)$.}  Let $F$ be a hyperk\"ahler variety.
Beauville \cite{beauville2} conjectured that the sub-algebra $V_F$ of
$\CH^*(F)$ generated by divisors injects into cohomology via the cycle
class map. Voisin \cite{voisin2} conjectured that if one adds to $V_F$
the Chern classes of the tangent bundle of $F$, then the resulting
sub-algebra still injects into cohomology. The following conjecture,
which is in the same vein as the conjecture of Beauville--Voisin, is
rather speculative but as Theorem \ref{thm2 conj} shows, it implies
the existence of a Fourier decomposition on the Chow ring of
hyperk\"ahler fourfolds of $\mathrm{K3}^{[2]}$-type.  Before we can
state it, we introduce some notations. Let $X$ be a smooth projective
variety, and let the $\Q$-vector space $\bigoplus_n \CH^*(X^n)$ be
equipped with the algebra structure given by intersection product on
each summand. Denote $p_{X^n ;i_1,\ldots, i_k} : X^n \rightarrow X^k$
the projection on the $(i_1,\ldots, i_k)^\mathrm{th}$ factor for
$1\leq i_1<\ldots < i_k \leq n$, and $\iota_{\Delta, X^n} : X
\rightarrow X^n$ the diagonal embedding.  Given cycles
$\sigma_1,\ldots, \sigma_r \in \bigoplus_n \CH^*(X^n)$, we define $V(X
;\sigma_1,\ldots,\sigma_r)$ to be the smallest sub-algebra of
$\bigoplus_n \CH^*(X^n)$ that contains $\sigma_1,\ldots, \sigma_r$ and
that is stable under $(p_{X^n ;i_1,\ldots, i_k})_*$, $(p_{X^n
  ;i_1,\ldots, i_k})^*$, $(\iota_{\Delta, X^n})_*$ and
$(\iota_{\Delta, X^n})^*$.

\begin{conj2}[Generalization of Conjecture \ref{conj
    main}] \label{conj sheaf} Let $F$ be a hyperk\"ahler variety.
  Then there exists a canonical cycle $L \in \CH^2(F \times F)$
  representing the Beauville--Bogomolov class $\mathfrak{B} \in
  \HH^4(F \times F,\Q)$ \linebreak such that the restriction of the
  cycle class map $\bigoplus_n \CH^*(F^n) \rightarrow \bigoplus_n
  \HH^*(F^n,\Q)$ to \linebreak $V(F ;L,c_0(F), c_2(F),\ldots,
  c_{2d}(F),D_1,\ldots,D_r)$, for any $D_i \in \CH^1(F)$, is
  injective.
\end{conj2}
Note that in higher dimensions $l$ and $c_2(F)$ are no longer
proportional in $\HH^4(F,\Q)$. This is the reason why we consider $V(F
;L,c_2(F),\ldots, c_{2d}(F),D_1,\ldots,D_r)$ rather than $V(F
;L,D_1,\ldots,D_r)$. Conjecture \ref{conj sheaf} is very strong~:
Voisin \cite[Conjecture 1.6]{voisin2} had already stated it in the
case of K3 surfaces (in that case $L$ need not be specified as it is
simply given by $\Delta_S - \mathfrak{o}_S\times S - S\times
\mathfrak{o}_S$) and noticed \cite[p.92]{voisinlectures} that it
implies the finite dimensionality in the sense of Kimura \cite{kimura}
and O'Sullivan \cite{o'sullivan2} of the Chow motive of $S$.  The
following theorem reduces the Fourier decomposition problem for the
Chow ring of hyperk\"ahler varieties of $\mathrm{K3}^{[2]}$-type to a
weaker form of Conjecture \ref{conj sheaf} that only involves $L$.

\begin{thm2}[Theorem \ref{thm2 conj repeat}] \label{thm2 conj} Let $F$
  be a hyperk\"ahler variety of $\mathrm{K3}^{[2]}$-type.  Assume that
  $F$ satisfies the following weaker version of Conjecture \ref{conj
    sheaf}~: there exists a cycle $L \in \CH^2(F \times F)$
  representing the Beauville--Bogomolov class $\mathfrak{B}$
  satisfying equation \eqref{eq rational equation}, and the
  restriction of the cycle class map $\bigoplus_n \CH^*(F^n)
  \rightarrow \bigoplus_n \HH^*(F^n,\Q)$ to $V(F ;L)$ is
  injective. Then $\CH^*(F)$ admits a Fourier decomposition as in
  Theorem \ref{thm main splitting} which is compatible with its ring
  structure.
\end{thm2}

This provides an approach to proving the Fourier decomposition for the
Chow ring of hyperk\"ahler varieties of $\mathrm{K3}^{[2]}$-type that
would avoid having to deal with non-generic cycles. Conjecture
\ref{conj sheaf} can be considered as an analogue for hyperk\"ahler
varieties of O'Sullivan's theorem \cite{o'sullivan} which is concerned
with abelian varieties. In the same way that Conjecture \ref{conj
  sheaf} implies the existence of a Fourier decomposition for the Chow
ring of hyperk\"ahler varieties of K3$^{[2]}$-type, we explain in
Section \ref{sec abelian} how Beauville's Fourier decomposition
theorem for abelian varieties can be deduced directly from
O'Sullivan's theorem. \medskip

Finally, as already pointed out, the diagonal $[\Delta_F]$ cannot be
expressed as a polynomial in $\mathfrak{B}$, $\mathfrak{b}_1$ and
$\mathfrak{b}_2$ when $F$ is a hyperk\"ahler variety whose cohomology
ring is not generated by degree-$2$ cohomology classes. Nevertheless,
the orthogonal projector on the sub-Hodge structure generated by
$\HH^2(F,\Q)$ can be expressed as a polynomial in $\mathfrak{B}$,
$\mathfrak{b}_1$ and $\mathfrak{b}_2$.  If one believes in Conjecture
\ref{conj sheaf}, this projector should in fact lift to a projector,
denoted $\Pi$, modulo rational equivalence. In that case, it seems
reasonable to expect $\Pi_*\CH^*(F)$ to be a sub-ring of $\CH^*(F)$
and to expect the existence of a Fourier decomposition with kernel $L$
on the ring $\Pi_*\CH^*(F)$. \medskip

\noindent \textbf{D.3. Multiplicative Chow--K\"unneth decompositions.}
We already mentioned that the Fourier decomposition for the Chow ring
of the Hilbert scheme of length-$2$ subschemes on a K3 surface or the
variety of lines on a very general cubic fourfold is in fact induced
by a Chow--K\"unneth decomposition of the diagonal (\emph{cf.} \S
\ref{sec CK} for a definition). A smooth projective variety $X$ of
dimension $d$ is said to admit a \emph{weakly multiplicative
  Chow--K\"unneth decomposition} if it can be endowed with a
Chow--K\"unneth decomposition $\{\pi^i_X : 0 \leq i \leq 2d\}$ that
induces a decomposition of the Chow ring of $X$. That is,
writing $$\CH_{\mathrm{CK}}^i(X)_s := (\pi_X^{2i-s})_*\CH^i(X),$$ we
have $$\CH^i_{\mathrm{CK}}(X)_s \cdot \CH_{\mathrm{CK}}^j(X)_r
\subseteq \CH_{\mathrm{CK}}^{i+j}(X)_{r+s}, \quad \mbox{for all} \
(i,s),\, (j,r).$$ \medskip

The Chow--K\"unneth decomposition is said to be \emph{multiplicative}
if the above holds at the level of correspondences~; see Definition
\ref{def multCK}.
 
Together with Murre's conjecture (D) as stated in \S \ref{sec CK}, we
ask~:

\begin{conj2} \label{conj2 vanishingc1} Let $X$ be a hyperk\"ahler
  variety. Then $X$ can be endowed with a Chow--K\"unneth
  decomposition that is multiplicative. Moreover, the cycle class map
  $\CH^i(X) \rightarrow \HH^{2i}(X,\Q)$ restricted to
  $\CH_{\mathrm{CK}}^i(X)_0$ is injective.
\end{conj2}

Note that if $A$ is an abelian variety, then $A$ has a multiplicative
Chow--K\"unneth decomposition (\emph{cf.} Example \ref{ex abelian})~;
and that the cycle class map restricted to the degree-zero graded part
of the Chow ring for the Chow--K\"unneth decomposition is injective
was asked by Beauville \cite{beauville1}.  Note also that Conjecture
\ref{conj sheaf} and Conjecture \ref{conj2 vanishingc1} have
non-trivial intersection. Indeed, if $X$ is a hyperk\"ahler variety,
then, provided that $\mathfrak{B}$ is algebraic (which is the case if
$X$ is of K3$^{[n]}$-type by Theorem \ref{thm existence of L}), 
by invoking Proposition \ref{lem diagonal pullback} we see that Conjecture
\ref{conj sheaf} for $X$ is implied by Conjecture \ref{conj2
  vanishingc1} for $X$ and by knowing that the Chern classes $c_i(X)
\in \CH^{i}(X)$ of $X$ belong to $\CH_{\mathrm{CK}}^i(X)_0$. As a
partial converse, Theorem \ref{thm2 conj repeat} shows that a
hyperk\"ahler variety of K3$^{[2]}$-type for which Conjecture
\ref{conj sheaf} holds admits a multiplicative Chow--K\"unneth
decomposition.\medskip

In Section \ref{sec multCK}, we define the notion of multiplicative
Chow--K\"unneth decomposition and discuss its relevance as well as its
links with so-called \emph{modified diagonals} (\emph{cf.} Proposition
\ref{prop multdiag}). In Section \ref{sec multCKX2}, we prove that the
Hilbert scheme of length-$2$ subschemes on a K3 surface not only
admits a weakly multiplicative Chow--K\"unneth decomposition but
admits a \emph{multiplicative} Chow--K\"unneth decomposition (as
predicted by Conjecture \ref{conj sheaf}). More generally, we produce
many examples of varieties that can be endowed with a multiplicative
Chow--K\"unneth decomposition. In \S \ref{sec proof thm6}, we prove~:

\begin{thm2} \label{thm2 CK} Let $E$ be the smallest sub-set of smooth
  projective varieties that contains varieties with Chow groups of
  finite rank (as $\Q$-vector spaces), abelian varieties, symmetric
  products of hyperelliptic curves, and K3 surfaces, and that is stable
  under the following operations~:
  \begin{enumerate}[(i)]
\item if $X$ and $Y$ belong to $E$, then $X\times Y \in E$~;
\item if $X$ belongs to $E$, then $\PP(\mathscr{T}_X) \in E$, where
  $\mathscr{T}_X$ is the tangent bundle of $X$~;
\item if $X$ belongs to $E$, then the blow-up of $X \times X$ along
  the diagonal belongs to $E$~;
\item if $X$ belongs to $E$, then the Hilbert scheme of length-$2$
  subschemes $X^{[2]} \in E$.
  \end{enumerate}
  Let $X$ be a smooth projective variety that is isomorphic to a
  variety in $E$. Then $X$ admits a multiplicative Chow--K\"unneth
  decomposition.
\end{thm2} 

It is intriguing to notice that all the examples that we produce
satisfy $c_i(X) \in \CH_{\mathrm{CK}}^{i}(X)_0$.\medskip

Abelian varieties and hyperk\"ahler varieties are instances of
varieties with vanishing first Chern class. 
In fact a classical theorem of
Beauville \cite{beauville0} and Bogomolov \cite{bogomolov0} states
that every smooth projective variety $X$ with vanishing first Chern
class is, up to finite \'etale cover, isomorphic to the product of an
abelian variety with hyperk\"ahler varieties and Calabi--Yau
varieties. However, one cannot expect varieties with vanishing first
Chern class to have a multiplicative (and even a weakly
multiplicative) Chow--K\"unneth decomposition~: Beauville constructed
Calabi--Yau threefolds $X$ \cite[Example 2.1.5(b)]{beauville2} that do
not satisfy the \emph{weak splitting property}. Precisely, Beauville's 
examples have the
property that there exist two divisors $D_1$ and $D_2$ such that
$D_1\cdot D_2  \neq 0 \in \CH^2(X)$ but such that
$[D_1]\cup [D_2]  = 0 \in \HH^4(X,\Q)$. To conclude, we ask~: what is a
reasonable class of varieties, containing abelian varieties and
hyperk\"ahler varieties, that we can expect to have a multiplicative
Chow--K\"unneth decomposition?

\newpage
\subsection*{Outline}
The manuscript is divided into three parts. We have tried to extract
the key properties that a cycle $L$ representing the
Beauville--Bogomolov class $\mathfrak{B}$ must satisfy in order to
induce a Fourier decomposition on the Chow groups of a hyperk\"ahler
fourfold of $\mathrm{K3}^{[2]}$-type. As a consequence the sections
are not organized in a linear order.\medskip

We first introduce in \S \ref{sec coho fourier} the
Beauville--Bogomolov class $\mathfrak{B}$ and establish in Proposition
1.3 the quadratic equation \eqref{eq cohomological equation S2}. The
core of Part 1 then consists in Theorems \ref{prop L2} \& \ref{prop
  main} and Theorem \ref{thm CK}, where we consider a hyperk\"ahler
variety $F$ of K3$^{[2]}$-type endowed with a cycle $L \in \CH^2(F
\times F)$ representing the Beauville--Bogomolov class $\mathfrak{B}$
that satisfies \eqref{eq rational equation}, \eqref{assumption pre l},
\eqref{assumption hom} and \eqref{assumption 2hom}, and show that the
conclusion of Theorem \ref{thm main splitting} holds for $F$ and that
the resulting Fourier decomposition on the Chow groups of $F$ is in
fact induced by a Chow--K\"unneth decomposition. We then assume that
we have proved the existence of a cycle $L$ representing the
Beauville--Bogomolov class satisfying hypotheses \eqref{eq rational
  equation}, \eqref{assumption pre l}, \eqref{assumption hom} and
\eqref{assumption 2hom} when $F$ is either the Hilbert scheme of
length-$2$ subschemes on a K3 surface or the variety of lines on a
cubic fourfold. We then consider in \S \ref{sec mult div} the
sub-algebra $V_F$ of $\CH^*(F)$ generated by divisors and $l$.  With
the formalism of the Fourier decomposition in mind, we show that the
results of Beauville \cite{beauville2} and Voisin \cite{voisin2} which
state that $V_F$ injects into cohomology via the cycle class map can
be reinterpreted as saying that $V_F$ lies in the degree-zero graded
part of our filtration. Proposition \ref{prop symplectic auto general}
explains how Conjecture \ref{conj fourier} is helpful to understanding
the action of automorphisms on zero-cycles of hyperk\"ahler
varieties. An application to symplectic automorphisms of $S^{[2]}$ is
given. We show in \S 6 that the multiplicativity property of the
Fourier decomposition boils down to intersection-theoretic properties
of the cycle $L$. We deduce in Theorem \ref{thm birational invariance}
that the Fourier decomposition is a birational invariant for
hyperk\"ahler varieties of K3$^{[2]}$-type. This approach is used in
\S \ref{sec abelian} to give new insight on the theory of algebraic
cycles on abelian varieties by showing how Beauville's Fourier
decomposition theorem \cite{beauville1} is a direct consequence of a
recent theorem of O'Sullivan \cite{o'sullivan}. Section 8 introduces
the notion of \emph{multiplicative Chow--K\"unneth decomposition} and
its relevance is discussed. In particular, we relate this notion to
the notion of so-called modified diagonals, give first examples of
varieties admitting a multiplicative Chow-K\"unneth decomposition, and
prove a more precise version of Theorem \ref{thm2 conj}. Part 1 ends
with \S \ref{sec alg B} where a proof of the algebraicity of the
Beauville--Bogomolov class $\mathfrak{B}$ is given for hyperk\"ahler
varieties of K3$^{[n]}$-type.\medskip

Part 2 and Part 3 are devoted to proving Theorem \ref{thm2 main},
Theorem \ref{thm main splitting} and Theorem \ref{thm main mult} for
the Hilbert scheme of length-$2$ subschemes on a K3 surface and for
the variety of lines on a cubic fourfold, respectively. In both cases,
the strategy for proving Theorem \ref{thm2 main} and Theorem \ref{thm
  main splitting} consists in first studying the incidence
correspondence $I$ and its intersection-theoretic properties, and then
in constructing a cycle $L \in \CH^2(F \times F)$ very close to $I$
representing the Beauville--Bogomolov class satisfying hypotheses
\eqref{eq rational equation}, \eqref{assumption pre l},
\eqref{assumption hom} and \eqref{assumption 2hom}.\medskip

Actually, Part 2 begins by considering the incidence correspondence
$I$ for the Hilbert scheme $X^{[2]}$ for any smooth projective variety
$X$. In that generality, we establish in \S \ref{sec I X2} some
equations satisfied by $I$ and deduce in \S \ref{sec mult X2} some
splitting results for the action of $I^2 = I \cdot I$ on
$\CH_0(X^{[2]})$. We then gradually add some constraints on $X$ which
are related to degeneration properties of the modified diagonal. In \S
\ref{sec multCKX2}, we study in depth the notion of multiplicative
Chow--K\"unneth decomposition for $X$ and give sufficient conditions
for it to be stable under blow-up and under taking the Hilbert scheme
of length-$2$ points $X^{[2]}$. The main contribution there is the
proof of Theorem \ref{thm2 CK} given in \S \ref{sec proof thm6}, and a
notable intermediate result is Theorem \ref{thm multCK X2} which shows
in particular that $X^{[2]}$ admits a multiplicative Chow--K\"unneth
decomposition when $X$ is a K3 surface or an abelian variety.  In \S
\ref{sec Fourierdec S2}, we turn our focus exclusively on $S^{[2]}$
for a K3 surface $S$, and prove Theorem \ref{thm2 main} (note the use
of the crucial Lemma \ref{lem lifting hilbert}) and Theorem \ref{thm
  main splitting} in that case.  In \S \ref{section mult S2}, we prove
Theorem \ref{thm main mult} by showing that the Fourier decomposition
of the Chow group $\CH^*(S^{[2]})$ agrees with the multiplicative
Chow--K\"unneth decomposition of Theorem \ref{thm multCK X2} obtained
by considering the multiplicative Chow--K\"unneth decomposition for
$S$ of Example \ref{ex K3}. Finally, in \S \ref{sec L agree}, we prove
that the cycle $L$ constructed in \S \ref{sec Fourier S2} using the
incidence correspondence coincides with the cycle $L$ constructed in
\S \ref{sec alg B} using Markman's twisted sheaf.\medskip

The structure of Part 3 is more straightforward. We define the
incidence correspondence $I$ and Voisin's rational self-map $\varphi$
for $F$, the variety of lines on a cubic fourfold, and study their
interaction when acting on the Chow groups of $F$. This leads in \S
\ref{sec Fourier cubic} to a proof of Theorems \ref{thm2 main} \&
\ref{thm main splitting} in that case.  The goal of \S \ref{sec mult1
  cubic} is to prove that $\CH^4(F)_4 = \CH^2(F)_2 \cdot
\CH^2(F)_2$. While such a statement was only a matter of combinatorics
in the case of $S^{[2]}$, here we have to resort to an analysis of the
geometry of cubic fourfolds~; see Theorem \ref{prop F3 triangle} in
this respect. Finally, in order to complete the proof of Theorem
\ref{thm main mult} in \S \ref{sec mult}, we use the rational map
$\varphi$ in an essential way. For that matter, the main result of \S
\ref{sec varphi} is Proposition \ref{prop key identity of varphi}, a
nice consequence of which is Theorem \ref{thm eigenspace
  decomposition} which shows that the action of $\varphi^*$ on
$\CH^*(F)$ is diagonalizable when $X$ does not contain any plane.  For
the sake of completeness and clarity, we have gathered general results
used throughout Part 3 about cubic fourfolds and about the action of
rational maps on Chow groups in two separate and self-contained
appendices.\medskip

\vspace{10pt}
\subsection*{Conventions} We work over the field of complex numbers.
Chow groups $\CH^i$ are with rational coefficients and we use
$\CH^i_{\Z}$ to denote the Chow groups with integral coefficients. If
$X$ is a variety, $\CH^i(X)_\mathrm{hom}$ is the kernel of the cycle
class map that sends a cycle $\sigma \in \CH^i(X)$ to its cohomology
class $[\sigma] \in \HH^{2i}(X,\Q)$. By definition we set $\CH^j(X)=0$
for $j<0$ and for $j>\dim X$. If $Y$ is another variety and if
$\gamma$ is a correspondence in $\CH^i(X \times Y)$, its transpose
${}^t\gamma \in \CH^i(Y \times X)$ is the image of $\gamma$ under the
action of the permutation map $X \times Y \rightarrow Y \times X$. If
$\gamma_1,\cdots, \gamma_n$ are correspondences in $\CH^*(X \times
Y)$, then the correspondence $\gamma_1 \otimes \cdots \otimes \gamma_n
\in \CH^*(X^n\times Y^n)$ is defined as $\gamma_1 \otimes \cdots
\otimes \gamma_n = \prod_{i=1}^n (p_{i,n+i})^*\gamma_i$, where
$p_{i,n+i} : X^n \times Y^n \rightarrow X \times Y$ is the projection
on the $i^\text{th}$ and $(n+i)^\text{th}$ factors.

\vspace{10pt}
\subsection*{Acknowledgments} We would like to thank Daniel Huybrechts
for his interest and for explaining \cite{markman}. We are also very
grateful to two referees for their detailed comments that have helped
improve the exposition.


\newpage
\begin{large}
\part{The Fourier transform for hyperk\"ahler fourfolds}
\end{large}


\vspace{10pt}
\section{The cohomological Fourier transform}
\label{sec coho fourier}
The goal of this section is to establish equation \eqref{eq
  cohomological equation S2} satisfied by the Beauville--Bogomolov
class $\mathfrak{B}$ of a hyperk\"ahler manifold of
$\mathrm{K3}^{[2]}$-type, and to show that $\mathfrak{B}$ is uniquely
determined, up to sign, by this equation. This is embodied in
Proposition \ref{prop action of B powers}. Along the way, we express
the K\"unneth projectors in terms of $\mathfrak{B}$~; see Corollary
\ref{prop kunneth}. \medskip

We start by defining the Beauville--Bogomolov class $\mathfrak{B}$
seen as an element of $\HH^4(F\times F,\Q)$. In a broader context, let
us first consider a free abelian group $\Lambda$ and the associated
vector space $V=\Lambda\otimes\Q$ (or $\Lambda\otimes\C$). A quadratic
form (or equivalently a symmetric bilinear form) $q$ on $V$ can be
viewed as an element of $\Sym^2(V^*)$. If $q$ is non-degenerate, then
it induces an isomorphism $V\cong V^*$ which further gives an
identification
\[
\Sym^2(V) \cong \Sym^2(V^*).
\]
Under this identification, the element $q\in \Sym^2(V^*)$ gives rise
to an element, denoted $q^{-1}$, of $\Sym^2(V)$. If we choose a basis
$\{v_1,\ldots,v_r\}$ of $V$ and write $A :=\big(q(v_i,v_j)\big)_{1\leq
  i,j \leq r}$, then
\begin{equation}\label{eq defn b inverse}
  q^{-1}=\sum_{1\leq i,j\leq
    r}q_{ij}v_i\otimes v_j\in\Sym^2V,
\end{equation}
where $(q_{ij})_{1\leq i,j \leq r} :=A^{-1}$.\medskip

We will be especially interested in the second cohomology group of a
compact hyperk\"ahler manifold, which carries a canonical bilinear
form by the following theorem of Bogomolov, Beauville and Fujiki.
\begin{thm}[Bogomolov, Beauville, and Fujiki~; see
  \cite{beauville0,fujiki}] \label{thm BBJ}
  Let $F$ be a compact hyperk\"ahler manifold of dimension $2n$. Then
  $\Lambda=\HH^2(F,\Z)$ is endowed with a canonical bilinear form $q_F$
  which satisfies the Fujiki relation
\[
\int_F \alpha^{2n} = \frac{(2n)!}{2^n n!}c_F q_F(\alpha, \alpha)^n,
\quad\forall \alpha\in \HH^2(F,\Z).
\]
\end{thm}
The above bilinear form of Theorem \ref{thm BBJ} is called the
\textit{Beauville--Bogomolov form} and the constant $\frac{(2n)!}{2^n
  n!}c_F$ is called the \textit{Fujiki constant}. In this section, we
deal with cup-product only for cohomology classes of even
degree. Hence  the symbol ``$\cup$'' will frequently be omitted and 
$\alpha\beta$ will stand for $\alpha\cup\beta$.

\begin{defn}\label{defn Fourier hom}
  Let $q_F$ be the Beauville--Bogomolov bilinear form on
  $\HH^2(F,\Z)$. Its inverse $q_F^{-1}$ defines an element of
  $\HH^2(F,\Z) \otimes \HH^2(F,\Z) \otimes \Q$, and  the
  \emph{Beauville--Bogomolov class} $$\mathfrak{B} \in \HH^4(F\times
  F,\Q)$$ is defined to be the image of $0 \oplus q_F^{-1} \oplus 0 \in
  \HH^0(F,\Q) \otimes \HH^4(F,\Q) \oplus \HH^2(F,\Q) \otimes
  \HH^2(F,\Q) \oplus \HH^4(F,\Q) \otimes \HH^0(F,\Q)$ under the
  K\"unneth decomposition.

The \emph{cohomological Fourier transform} is the homomorphism
$$
[\mathcal{F}] :\HH^*(F,\Q) \rightarrow \HH^*(F,\Q),\quad x\mapsto
p_{2*}(e^{\mathfrak{B}}\cup p_1^*x),
$$
where $p_i  : F \times F \rightarrow F$ is the projection onto the
$i^\mathrm{th}$ factor.

We write $\iota_\Delta  : F\hookrightarrow F\times F$ for the diagonal
embedding and we define $$\mathfrak{b}  := \iota_\Delta^*\mathfrak{B}
\in \HH^4(F,\Q) \quad \mbox{and} \quad \mathfrak{b}_i  :=
p_i^*\mathfrak{b} \in \HH^4(F \times F, \Q).$$
\end{defn}
\medskip

From now on, we assume that $\dim F = 4$. In this case, the Fujiki
relation of Theorem \ref{thm BBJ} implies
\begin{equation}\label{eq Fujiki relation}
  \int_F \alpha_1\alpha_2\alpha_3\alpha_4 = c_F \{q_F(\alpha_1,
  \alpha_2) q_F(\alpha_3,\alpha_4)+ q_F(\alpha_1,\alpha_3)
  q_F(\alpha_2,\alpha_4)+ q_F(\alpha_1,\alpha_4)q_F(\alpha_2,\alpha_3)\}
\end{equation}
for all $\alpha_1, \alpha_2, \alpha_3, \alpha_4 \in \Lambda = \HH^2(F,\Z)$.
 The
following proposition describes the action of the powers of
$\mathfrak{B} \in \HH^4(F \times F, \Q)$ on $\HH^*(F,\Q)$, when the
cohomology ring of $F$ is generated by degree-$2$ classes. Note that
by a result of Verbitsky, the cup-product map $\Sym^2(\HH^2(F,\Q))
\rightarrow \HH^4(F,\Q)$ is always injective~; see \cite{bogomolov}.

\begin{prop}\label{prop action of B powers}
  Assume $F$ is a compact hyperk\"ahler manifold of dimension $4$
  which satisfies
\begin{center}
  $ \HH^4(F,\Q)=\Sym^2(\HH^2(F,\Q)) $ \quad and \quad $\HH^3(F,\Q) =
  0$.
\end{center}
Let $r=\dim\HH^2(F,\Q)$ and let $c_F$ be a third of the Fujiki
constant of $F$.  Then the following are true.

\begin{enumerate}[(i)]
\item The following equality holds in $\HH^8(F\times F,\Q)$,
  \begin{equation}\label{eq cohomological equation}
  \mathfrak{B}^2=2c_F[\Delta_F]
  -\frac{2}{r+2}(\mathfrak{b}_1+\mathfrak{b}_2) \mathfrak{B}
  -\frac{1}{r(r+2)}(2\mathfrak{b}_1^2 -
  r\mathfrak{b}_1\mathfrak{b}_2 + 2 \mathfrak{b}_2^2).
\end{equation}

\item The action of $(\mathfrak{B}^k)_*$, $k=0,1,2,3,4$, on
  $\HH^i(F,\Q)$ is zero for $i \neq 8-2k$.

\item The action of $(\mathfrak{B}^2)_*$ on $\HH^4(F,\Q)$ gives an
  eigenspace decomposition
$$
\HH^4(F,\Q) = \langle \mathfrak{b}\rangle \oplus \langle
\mathfrak{b}\rangle^\perp
$$
which is orthogonal for the intersection form and where
$(\mathfrak{B}^2)_*=(r+2)c_F$ on $\langle \mathfrak{b}\rangle$ and
$(\mathfrak{B}^2)_*=2c_F$ on $\langle \mathfrak{b}\rangle^\perp$.

\item $(\mathfrak{B}^k)_* \circ (\mathfrak{B}^{4-k})_*  : \HH^{2k}(F,\Q)
\rightarrow \HH^{8-2k}(F,\Q)\rightarrow\HH^{2k}(F,\Q)$ is
multiplication by $3 (r+2)c_F^2$ when $k=1,3$ and is multiplication by
$3 r (r+2)c_F^2$ when $k=0,4$.

\item $\pm\mathfrak{B}$ is the unique element in $\HH^2(F)\otimes
  \HH^2(F)$ that satisfies the above equation \eqref{eq cohomological
    equation} in the following sense. Assume that $\mathfrak{C}\in
  \HH^2(F)\otimes \HH^2(F)$ satisfies equation \eqref{eq cohomological
    equation} with $\mathfrak{B}$ replaced by $\mathfrak{C}$ and
  $\mathfrak{b}_i$ replaced by $\mathfrak{c}_i$, where
  $\mathfrak{c}_i=p_i^*\mathfrak{c}$ and
  $\mathfrak{c}=\iota_\Delta^*\mathfrak{C}$. Then
  $\mathfrak{C}=\pm\mathfrak{B}$.
\end{enumerate}
\end{prop}

Before proving Proposition \ref{prop action of B powers}, we establish
some auxiliary lemmas~: Lemma \ref{lem dual basis} and Lemma \ref{lem
  b dot B} hold generally for any compact hyperk\"ahler manifold of
dimension $4$, and Lemma \ref{lem cohomology of diagonal} computes the
class of the diagonal $[\Delta_F]$ in terms of the
Beauville--Bogomolov class under the condition that the cohomology of
$F$ is generated by degree-$2$ elements.\medskip

Let $V=\HH^2(F,\C)$, then the Beauville--Bogomolov bilinear form
extends to $V\times V$. We choose an orthonormal basis
$\{e_1,e_2,\ldots, e_r\}$ for $V$, namely $q_F(e_i,e_j)=\delta_{ij}$.
Then we have
\begin{equation}\label{eq B}
  \mathfrak{B} = e_1\otimes e_1 +e_2\otimes e_2 +\cdots +e_r\otimes
  e_r
\end{equation}
and $\mathfrak{b}= e_1^2 + e_2^2 + \cdots +e_r^2$.

By Poincar\'e duality, we have
$$
\HH^6(F,\C) = \HH^2(F,\C)^\vee.
$$
This allows us to define the dual basis $\{e_1^\vee, e_2^\vee,\ldots
,e_r^\vee\}$ of $\HH^6(F,\C)$.

With the above notions, we first note that if $1\leq i,j,k\leq r$ are
distinct then
\begin{equation}\label{eq basis mult 1}
 e_ie_je_k=0, \qquad \text{in }\HH^6(F,\C).
\end{equation}
This is because, by equation \eqref{eq Fujiki relation}, one easily
shows that $e_ie_je_k e_l=0$ for all $1\leq l \leq r$. We also note
that
\begin{align}
  \label{eq basis mult 2}e_i^2e_j &=   c_F\,e_j^\vee, \quad 1 \leq i\neq j\leq r,\\
  \label{eq basis mult 3}e_j^3 &= 3c_F e_j^\vee, \quad 1\leq j\leq r.
\end{align}
The equation \eqref{eq basis mult 2} follows from the following computation
\[
e_i^2e_j e_k = c_F\, \big(q_F(e_i,e_i)q_F(e_j,e_k) +2
q_F(e_i,e_j)q_F(e_i,e_k)\big) =c_F\,\delta_{jk}, \quad i\neq j.
\]
Similarly, equation \eqref{eq basis mult 3} follows from $e_j^3
e_k=3c_F q_F(e_j,e_j)q_F(e_j,e_k)=3c_F\delta_{jk}$.  Combining
equations \eqref{eq basis mult 2} and \eqref{eq basis mult 3}, we have
\begin{equation}\label{eq basis mult 4}
\mathfrak{b} e_i= (r+2)c_F
  e_i^\vee, \quad 1\leq i\leq r.
\end{equation}

\begin{lem}\label{lem dual basis} The following cohomological
  relations hold.
  \begin{enumerate}[(i)]
\item $\mathfrak{B}^2=\sum_{i=1}^{r} e_i^2\otimes e_i^2 + 2\sum_{1\leq i<j
\leq r} e_ie_j\otimes e_ie_j$.
\item $\mathfrak{B}^3 = 3(r+2)c_F^2\sum_{i=1}^{r}e_i^\vee\otimes
  e_i^\vee$.
\item $\mathfrak{B}^4=3r(r+2)c_F^2\,[pt]\otimes [pt]$.
\end{enumerate}
\end{lem}
\begin{proof}
  Statement \emph{(i)} follows by squaring both sides of \eqref{eq
    B}. Statement \emph{(ii)} can be derived as follows.  First we use
  (i) and get
\[
\mathfrak{B}^2(e_k\otimes e_k) = \sum_{i=1^r} e_i^2e_k \otimes
e_i^2e_k + 2\sum_{1\leq i<j\leq r} e_ie_je_k\otimes e_ie_je_k.
\]
We invoke equations \eqref{eq basis mult 1}, \eqref{eq basis mult 2}
and \eqref{eq basis mult 3} to find
\[
\mathfrak{B}^2(e_k\otimes e_k) = (r+6)c_F^2 e_k^\vee \otimes e_k^\vee
+ 2 c_F^2 \sum_{i=1}^r e_i^\vee\otimes e_i^\vee.
\]
By taking the sum as $k$ runs through $1$ to $r$, we get
\emph{(ii)}. Statement \emph{(iii)} can be proved similarly.
\end{proof}

\begin{lem}\label{lem b dot B}
 The following cohomological relations hold.
\begin{enumerate}[(i)]
\item $\mathfrak{b}_1 \mathfrak{B}=(r+2)c_F\sum_{i=1}^{r}
  e_i^\vee\otimes e_i\in\HH^6(F)\otimes\HH^2(F)$.
  \item $\mathfrak{b}_2 \mathfrak{B}=(r+2)c_F\sum_{i=1}^{r}
  e_i\otimes
  e_i^\vee \in\HH^2(F)\otimes\HH^6(F)$.
  \item $(\mathfrak{b}_1)^2=r(r+2)c_F\,[pt]\otimes [F]\in
  \HH^8(F)\otimes
  \HH^0(F)$.
  \item $(\mathfrak{b}_2)^2=r(r+2)c_F\,[F]\otimes [pt]\in
  \HH^0(F)\otimes \HH^8(F)$.
\item $\mathfrak{b}_1\mathfrak{b}_2 = \sum_{i,j=1}^r e_i^2\otimes
  e_j^2 \in \HH^4(F)\otimes \HH^4(F)$.
  \end{enumerate}
\end{lem}
\begin{proof}
  Under the K\"unneth isomorphism, we have
  $\mathfrak{b}_1=\sum_{i=1}^{r}e_i^2\otimes 1$ and
  $\mathfrak{b}_2=\sum_{i=1}^{r} 1\otimes e_i^2$. Hence we get
\[
\mathfrak{b}_1 \mathfrak{B} = \sum_{i,j=1}^{r}e_i^2e_j\otimes e_j
=\sum_{j=1}^{r}(\mathfrak{b} e_j)\otimes e_j = (r+2)c_F
\sum_{j=1}^{r}e_j^\vee\otimes e_j,
\]
which proves \emph{(i)}. Statement \emph{(ii)} can be proved
similarly. Statements \emph{(iii)} and \emph{(iv)} follow from
$\mathfrak{b}^2=r(r+2)c_F[pt]$. Statement \emph{(v)} follows from the
above explicit expressions for $\mathfrak{b}_1$ and $\mathfrak{b}_2$.
\end{proof}

\begin{lem}\label{lem cohomology of diagonal}
  Under the assumption of Proposition \ref{prop action of B powers},
  the cohomology class of the diagonal $\Delta_F$ is given by
\begin{align}
  [\Delta_F] & = \sum_{i=1}^{r} (e_i^\vee\otimes e_i +e_i\otimes
  e_i^\vee) + \frac{1}{2c_F}\sum_{i=1}^{r} e_i^2\otimes e_i^2
  + \frac{1}{c_F} \sum_{1\leq i<j\leq r} e_ie_j\otimes e_ie_j 
  \label{eq deltakunneth}\\
  &\quad - \frac{1}{2(r+2)c_F}\sum_{i,j=1}^r e_i^2\otimes e_j^2
  +[pt]\otimes [F] + [F]\otimes[pt]. \nonumber
\end{align}
\end{lem}
\begin{proof}
  We need to check that the right-hand side of \eqref{eq deltakunneth}
 acts as the identity on
  cohomology. Here we only check this on $\HH^4(F,\Q)$. The other
  cases are easy to check. Note that if $1\leq i<j\leq r$ then
\begin{equation}\label{eq action 3}
  (e_ie_j\otimes e_ie_j)_*(e_ke_l) = (e_ie_je_ke_l) \;e_ie_j =
  c_F(\delta_{ik}\delta_{jl}) e_ie_j,\quad 1\leq k<l\leq r.
\end{equation}
Likewise, we note that
\begin{equation*} 
(\sum_{i,j=1}^r e_i^2\otimes e_j^2)_*(e_ke_l)=0, \quad 1\leq k<l\leq r
\end{equation*}
and that
\begin{equation}\label{eq action 5}
(e_i^2\otimes e_i^2)_*e_ke_l=0,\quad 1\leq k<l\leq r.
\end{equation}
From these, we see that the right-hand side of \eqref{eq deltakunneth}
acts as the identity on the sub-space of $\HH^4(F,\Q)$ spanned by
$\{e_ke_l :1\leq k<l \leq r\}$. Now we consider the action of the
right-hand side of \eqref{eq deltakunneth} on $e_k^2$.  First we note
that
\begin{equation}\label{eq action 1}
(e_ie_j\otimes e_ie_j)_*e_k^2=0
\end{equation}
for all $1\leq i<j\leq r$ and all $1\leq k \leq r$, and also that
\[
(e_i^2\otimes e_i^2)_*e_k^2 =
\begin{cases}
  c_F\; e_i^2, &i\neq k ;\\
  3c_F\; e_i^2, &i=k.
\end{cases}
\]
From the above equation, we get
\begin{equation}\label{eq action 2}
  (\sum_{i=1}^{r} e_i^2\otimes e_i^2)_* e_k^2=3c_F\,e_k^2 +c_F
  \sum_{i\neq k} e_i^2 = 2c_F\,e_k^2 + c_F\,\sum_{i=1}^r e_i^2.
\end{equation}
Furthermore, we note that equation \eqref{eq basis mult 4} implies
that $(\sum_{i=1}^r e_i^2) e_k^2= (r+2)c_F\,[pt]$. It follows
that
\[
 (\sum_{i,j=1}^r e_i^2\otimes e_j^2)_*e_k^2 = (r+2)c_F \sum_{i=1}^re_i^2.
\]
Combining the above computations, we see that the right-hand 
side of \eqref{eq deltakunneth}
 acts
as the identity on the sub-space spanned by $\{e_k^2\}$. Thus it acts
as the identity on the whole cohomology group $\HH^4(F,\Q)$.
\end{proof}

\begin{proof}[Proof of Proposition \ref{prop action of B powers}]
  We may work with complex coefficients and hence use an orthonormal
  basis $\{e_i\}$ as above. Then statement \emph{(i)} follows easily
  from Lemma \ref{lem b dot B} and Lemma \ref{lem cohomology of
    diagonal}.

  Statement \emph{(ii)} follows from the fact that $\mathfrak{B}^k$
  lies in $\HH^{2k}(F)\otimes \HH^{2k}(F)$~; see Lemma \ref{lem dual
    basis}.

  If we combine Lemma \ref{lem dual basis} with equations \eqref{eq
    action 3} and \eqref{eq action 5}, we see that
\[
(\mathfrak{B}^2)_*e_ke_l = 2c_F\, e_ke_l ,\quad 1\leq k<l \leq r.
\]
Similarly, equations \eqref{eq action 1} and \eqref{eq action 2}
give
\[
(\mathfrak{B}^2)_* e_k^2= 2c_F\,e_k^2 + c_F\,\mathfrak{b}.
\]
Note that $\mathfrak{b}=\sum e_i^2$, and therefore
\[
(\mathfrak{B}^2)_*\mathfrak{b} = (r+2)c_F\,\mathfrak{b}.
\]
Since by definition $\langle \mathfrak{b} \rangle^\perp$ is generated
by $e_ie_j$ and $e_i^2-e_j^2$, we easily see that $(\mathfrak{B}^2)_*$
is multiplication by $2c_F$ on $\langle \mathfrak{b}
\rangle^\perp$. This proves statement \emph{(iii)}.

Statement \emph{(iv)} can be proved using Lemma \ref{lem dual
  basis}. For example take $e_i^\vee\in\HH^6(F)$. Then we easily see
that $\mathfrak{B}_*e_i^\vee = e_i$. Using the explicit expression for
$\mathfrak{B}^3$ in Lemma \ref{lem dual basis}, we get
$(\mathfrak{B}^3)_* e_i = 3(r+2)c_F^2 e_i^\vee$. It follows that
$(\mathfrak{B}^3)_*\circ (\mathfrak{B})_*$ is equal to $3(r+2)c_F^2$
on $\HH^6(F,\Q)$. The other equalities are proved similarly.

To prove \emph{(v)}, we assume that $\mathfrak{C}$ satisfies the equation
\eqref{eq cohomological equation}. We write $\mathfrak{C}=\sum
a_{ij}e_i\otimes e_j$ and hence $\mathfrak{c}=\sum a_{ij}e_ie_j$. We
define
the transpose ${}^t\mathfrak{C}  :=\iota^*\mathfrak{C}=\sum
a_{ji}e_j\otimes e_i$, where $\iota$ is the involution on $F\times F$
that switches the two factors. Let $A=(a_{ij})$. Plugging in
$\mathfrak{C}$ into equation \eqref{eq cohomological equation} and
taking the $\HH^4(F)\otimes\HH^4(F)$-component of the equation, we get
\begin{equation}\label{eq 44 equation}
  \mathfrak{C}^2 = 2c_F[\Delta_F]_{4,4} +\frac{1}{r+2}\mathfrak{c}
  \otimes \mathfrak{c},
\end{equation}
where $[\ ]_{4,4}$ means the $\HH^4(F)\otimes\HH^4(F)$-component.
Given the explicit expression for $[\Delta_F]$ in Lemma \ref{lem
  cohomology of diagonal}, we can compare the coefficients of the two
sides of \eqref{eq 44 equation}. When we consider the coefficient of
$e_i^2\otimes e_i^2$, we get
$$
(a_{ii})^2 = 1 -\frac{1}{r+2} +\frac{(a_{ii})^2}{r+2},
$$
which implies that $a_{ii} = \pm 1$. Without loss of generality, we
assume that $a_{ii}=1$ for $1\leq i\leq r'$ and $a_{ii}=-1$ for
$r'+1\leq i\leq r$, where $0\leq r'\leq r$.

We first show that $\mathfrak{C}$ is symmetric, meaning
$\mathfrak{C}={}^t\mathfrak{C}$. Applying $\iota^*$ to equation
\eqref{eq cohomological equation} and taking the difference with the
original equation, we get
$$
\mathfrak{C}^2 - ({}^t\mathfrak{C})^2 = -\frac{1}{2}(\mathfrak{c}_1
+\mathfrak{c}_2) (\mathfrak{C} - {}^t\mathfrak{C}).
$$
Note that the left-hand side is an element in
$\HH^4(F)\otimes\HH^4(F)$ and that $\mathfrak{c}_1 (\mathfrak{C} -
{}^t\mathfrak{C})$ (resp. $\mathfrak{c}_2 (\mathfrak{C} -
{}^t\mathfrak{C})$) is an element in $\HH^6(F)\otimes\HH^2(F)$ (resp.
$\HH^2(F)\otimes \HH^6(F)$). Hence we conclude that
\begin{equation}\label{eq C2 symmetric}
  \mathfrak{C}^2=({}^t\mathfrak{C})^2,\qquad
  \mathfrak{c}_i (\mathfrak{C} -
  {}^t\mathfrak{C}) = 0.
\end{equation}
Comparing the coefficients of $e_i^2\otimes e_je_{j'}$ in
$\mathfrak{C}^2$ and $({}^t\mathfrak{C})^2$, we get
\begin{align}
  \label{eq compare coeff 1} a_{ij}a_{ij'} &= a_{ji}a_{j'i},\quad
  \forall i,j,j'.
\end{align}
By taking $j'=i$ in \eqref{eq compare coeff 1} and by noting that
$a_{ii}=\pm 1$, we get $a_{ij}=a_{ji}$. Hence $\mathfrak{C}$ is
symmetric.

Now we have ${}^tA = A$, and Lemma \ref{lem dual basis} yields
\begin{align*}
  \mathfrak{c}_1 \mathfrak{C} & = \sum a_{ij}a_{kl}e_ie_je_k\otimes e_l\\
  &= c_F \sum a_{ij}a_{kl}(\delta_{ij}e_k^\vee +\delta_{ik}e_j^\vee
  +\delta_{jk}e_i^\vee ) \otimes e_l\\
  &= c_F \sum_{i,j} (\sum_{k=1}^r a_{kk} a_{ij} + 2\sum_{k=1}^r
  a_{ik}a_{kj})e_i^\vee \otimes e_j.
\end{align*}
Note that the left-hand side of equation \eqref{eq cohomological
  equation} is purely in $\HH^4(F)\otimes \HH^4(F)$. Given Lemma
\ref{lem cohomology of diagonal} and considering the $\HH^6(F)\otimes
\HH^2(F)$-component of the right-hand side of  \eqref{eq
  cohomological equation}, we get
\begin{equation}\label{eq A equation}
 2A^2 +(\mathrm{tr}A)A -(r+2)I =0.
\end{equation}
Comparing the coefficients of $e_i^2\otimes e_j^2$ for $i\neq j$ in
equation \eqref{eq 44 equation}, we get
$$
(a_{ij})^2 = -\frac{1}{r+2} + \frac{a_{ii}a_{jj}}{r+2}.
$$
It follows that $a_{ij}=0$ for all $1\leq i,j\leq r'$ or $r'+1\leq
i,j\leq r$ with $i\neq j$. Comparing the coefficients of $e_i^2\otimes
e_je_k$ for $j\neq k$ in equation \eqref{eq 44 equation}, we get
$$
2 a_{ij}a_{ik} = \frac{2 a_{ii}a_{jk}}{r+2}.
$$
If $r'\geq 2$, we take $1\leq i\neq j\leq r'$ and $r'+1\leq k\leq r$
in the above equation. Recalling that $a_{ij}=0$ it follows that
$a_{jk}=0$ and hence we conclude that $A$ is diagonal. If $r'=1$, we
take $r'+1 \leq i\neq j\leq r$ in the above equation, we can still
conclude that $A$ is diagonal. So, in any case, $A$ is a diagonal
matrix with $\pm 1$ coefficients in the diagonal. Then equation
\eqref{eq A equation} implies that $A=\pm I$, namely
$\mathfrak{C}=\pm\mathfrak{B}$.
\end{proof}

\begin{cor} \label{prop kunneth} Let $F$ be a compact hyperk\"ahler
  manifold of dimension $4$ and let $\pi^i_{\mathrm{hom}}$ be the
  class of the K\"unneth projector on $\HH^i(F,\Q)$ in $\HH^8(F\times
  F,\Q)$. Then
\begin{align*}
  \pi_{\hom}^0 &= \frac{1}{c_F r (r+2)} \fb_1^2, \quad
  \pi_{\hom}^2 = \frac{1}{c_F (r+2)} \mathfrak{b}_1
  \mathfrak{B}, \quad
  \pi_{\hom}^6 &= \frac{1}{c_F (r+2)} \mathfrak{b}_2
  \mathfrak{B}, \quad \pi_{\hom}^8 = \frac{1}{c_F r  (r+2)}
  \fb_2^2.
\end{align*}
Moreover, if $F$ satisfies the assumptions of Proposition \ref{prop
  action of B powers}, then equation (\ref{eq cohomological equation})
defines a K\"unneth decomposition of the diagonal and we have
$$  \pi_{\hom}^4 =
\frac{1}{2c_F} \big( \mathfrak{B}^2 - \frac{1}{r+2}\fb_1 \fb_2\big).$$
\end{cor}
\begin{proof}
  This follows from Lemma \ref{lem dual basis} and Lemma \ref{lem b
    dot B}.
\end{proof}

The cohomological Fourier transform is then easy to understand.

\begin{prop} \label{prop coho fourier}
    Let $F$ be a compact hyperk\"ahler manifold of dimension $4$ and let
  $[\mathcal{F}]$ be the cohomological Fourier transform. Then
\begin{enumerate}[(i)]
\item $[\mathcal{F}]\circ[\mathcal{F}]= \frac{r (r+2) c_F^2}{8}$ on
  $\HH^0(F,\Q)$ and on $\HH^8(F,\Q)$~;

\item $[\mathcal{F}]\circ[\mathcal{F}]=\frac{(r+2) c_F^2}{2}$ on
  $\HH^2(F,\Q)$ and on $\HH^6(F,\Q)$~;

\item if $F$ satisfies the assumptions of Proposition \ref{prop action
    of B powers}, $[\mathcal{F}]\circ[\mathcal{F}]$ induces an
  eigenspace decomposition $ \HH^4(F,\Q) = \langle \mathfrak{b}
  \rangle \oplus \langle \mathfrak{b} \rangle^\perp $, where
  $[\mathcal{F}]\circ[\mathcal{F}]=\big( \frac{(r+2)}{2} c_F \big)^2$
  on $\langle \mathfrak{b} \rangle$ and
  $[\mathcal{F}]\circ[\mathcal{F}]=c_F^2$ on $\langle \mathfrak{b}
  \rangle^\perp$. \qed
  \end{enumerate}
\end{prop}

\vspace{10pt}
\section{The Fourier transform on the Chow groups of
hyperk\"ahler  fourfolds} \label{sec fourier chow}

The aim of this section is to exhibit key properties that a cycle $L
\in \CH^2(F \times F)$ representing the Beauville--Bogomolov class
$\mathfrak{B}$ on a hyperk\"ahler variety $F$ of
$\mathrm{K3}^{[2]}$-type should satisfy in order to induce a Fourier
decomposition as in Theorem \ref{thm main splitting} on the Chow
groups of $F$. The two main results here are Theorems \ref{prop L2} \&
\ref{prop main}~; they reduce Theorem \ref{thm main splitting} to
showing that $L$ satisfies hypotheses \eqref{eq rational equation},
\eqref{assumption pre l}, \eqref{assumption hom} and \eqref{assumption
  2hom}. \medskip

Although in this paper we will be mostly considering hyperk\"ahler varieties
which are deformation equivalent to the Hilbert scheme of length-$2$
subschemes on a K3 surface, in which case by \cite{beauville0,
  rapagnetta} we have $c_F=1$ (so that the Fujiki constant is equal to
$3$) and $r=23$, we first formulate a more general version of
Conjecture \ref{conj main}.

\begin{conj}\label{conj main general}
  Let $F$ be a projective hyperk\"ahler manifold of dimension $4$
  whose cohomology ring is generated by $\HH^2(F,\Q)$. Let $r=\dim
  \HH^2(F,\Q)$. Then there exists a canonical cycle $L \in \CH^2(F
  \times F)$ with cohomology class $\mathfrak{B} \in \HH^4(F \times F,
  \Q)$ satisfying
  \begin{equation}\label{eq rational equation2}
    L^2=2 c_F \Delta_F -\frac{2}{r+2}(l_1+l_2)\cdot L
    -\frac{1}{r (r+2)}(2l_1^2 - rl_1l_2 + 2l_2^2)
    \quad \in \CH^4(F \times F),
\end{equation}
where by definition we have set $l  := \iota_\Delta^*L$ and $l_i  :=
p_i^*l$.
\end{conj}

\begin{thm} \label{prop L2} Let $F$ be a hyperk\"ahler variety of
  $\mathrm{K3}^{[2]}$-type. Assume that there exists a cycle $L \in
  \CH^2(F \times F)$ as in Conjecture \ref{conj main} and assume
  moreover that\medskip

  \noindent \eqref{assumption pre l} \hfill $L_*l^2 = 0$ ; \hfill
  {}\medskip

\noindent \eqref{assumption hom}\hfill $L_*(l\cdot L_*\sigma) = 25 \;
L_*\sigma$ for all $\sigma \in \CH^4(F).$ \hfill {} \medskip

\noindent Then the action of $(L^2)_*$ on $\CH^*(F)$ diagonalizes.
Precisely, for $\lambda \in \Q$, writing $$\Lambda^i_\lambda  :=
\{\sigma \in \CH^i(F) : (L^2)_*\sigma = \lambda\sigma\},$$ we have
\begin{align*}
  \CH^4(F) & = \Lambda_0^4 \oplus \Lambda_2^4, \ \text{with} \
  \Lambda_2^4 = \F^4\CH^4(F)  := \ker \{L_*  : \CH^4(F)_{\hom}
  \rightarrow \CH^2(F)\}\, ; \\
  \CH^3(F) & = \Lambda_0^3 \oplus \Lambda_2^3, \
  \text{with} \  \Lambda_2^3 = \CH^3(F)_{\hom}\,  ;\\
  \CH^2(F) &= \Lambda^2_{25} \oplus \Lambda_2^2 \oplus \Lambda_0^2, \
  \text{with} \ \Lambda^2_{25} = \langle l \rangle \ \text{and} \
  \Lambda^2_0 = L_*\CH^4(F) \subseteq \CH^2(F)_{\hom}\,  ;\\
  \CH^1(F) &= \Lambda_{0}^1\,  ;\\
  \CH^0(F) &= \Lambda^0_{0}.
\end{align*}
Moreover, $l \cdot  : \Lambda^2_0 \rightarrow \CH^4(F)$ maps
$\Lambda^2_0$ isomorphically onto $(\Lambda_0^4)_{\hom}$ with inverse
given by $\frac{1}{25}L_*$ and $l \cdot  : \Lambda_0^1 \rightarrow
\CH^3(F)$ maps $\Lambda_0^1$ isomorphically onto $\Lambda_0^3$ with
inverse given by $\frac{1}{25}L_*$. We also have the natural inclusion
$l\cdot\Lambda_2^2\subseteq \Lambda_2^4$.
\end{thm}
\begin{proof} Let us first remark that, by the projection formula, we
  have that for all $\sigma \in \CH^*(F)$ and all $i,j,k \geq 0$
$$(l_1^i\cdot l_2^j\cdot L^k)_* \sigma = l^j \cdot(L^k)_*(l^i\cdot
\sigma).$$
The following lemma will be used throughout the proof of the proposition.
\begin{lem} \label{lem assumption l} Assume that there exists $L \in
  \CH^2(F \times F)$ as in Conjecture \ref{conj main} satisfying
  hypothesis \eqref{assumption pre l}. Then
  \begin{equation}
    \label{assumption l} (L^k)_*(l^j) = 0 \ \mathrm{if} \ k + 2j \neq 4,
    \ \mathrm{and} \ (L^0)_*l^2 = 23\cdot
    25\; [F], \ (L^2)_*l = 25\; l, \ (L^4)_*[F]=3\; l^2.
  \end{equation}
\end{lem}
\begin{proof}[Proof of Lemma \ref{lem assumption l}] Following
  directly from the cohomological description of $L$ are the following
  identities~: $(L^0)_*l^2 = 23\cdot 25l^0$ and $L_*l^j=0$ for $j=0,
  1$. Now, using \eqref{eq rational equation}, we have
\[
(L^2)_*l^j = 2l^j + \left\{ \begin{array}{lll}
    -\frac{2}{23 \cdot 25}(L^0)_*l^2 & = 0 & \mbox{if $j = 0$} ;\\
    + \frac{1}{ 25}l\cdot (L^0)_*l^2 & = 25l & \mbox{if $j = 1$} ;
    \\
    -\frac{2}{23 \cdot 25}l^2\cdot (L^0)_*l^2 & = 0 & \mbox{if $j =
      2$}.
  \end{array} \right.
\] Because $\Delta_F \cdot L = (\iota_\Delta)_*\iota_\Delta^*L =
(\iota_\Delta)_*l$, where $\iota_\Delta  : F \rightarrow F \times F$ is
the diagonal embedding, we have after intersecting \eqref{eq rational
  equation} with $L$~:
\begin{equation}
  \label{eq L3} L^3
  = 2(\iota_\Delta)_*l - \frac{2}{25}(l_1 + l_2)\cdot L^2 -
  \frac{1}{23\cdot 25}(2l_1^2-23l_1l_2 + 2l_2^2)\cdot L,
\end{equation}
which yields $$(L^3)_*l^j = 2l^{j+1} - \frac{2}{25}(L^2)_*(l^{j+1}) -
\frac{2}{25}l \cdot (L^2)_*(l^{j}) = 0 \ \mbox{for all $j$}. $$
Finally, intersecting \eqref{eq rational equation} with $L^2$ gives
\begin{equation}
  \label{eq L4} L^4
  = 2(\iota_\Delta)_*l^2 - \frac{2}{25}(l_1 + l_2)\cdot L^3 -
  \frac{1}{23\cdot 25}(2l_1^2-23l_1l_2 + 2l_2^2)\cdot L^2,
\end{equation}
which, combined with the previous computations, yields the relation
$(L^4)_*l^0 = 3l^2$.
\end{proof}

\noindent The decompositions induced by the action $L^2$ on the Chow
groups of $F$ are established case by case.
\vspace{2mm}\\
\indent\textit{Step 1}  : $\CH^4(F)=\Lambda_2^4\oplus \Lambda_0^4$,
with $\Lambda_2^4 =\mathrm{F}^4\CH^4(F)$.

Let $\sigma\in\CH^4(F)$ and denote $\tau  := l\cdot L_*\sigma$. Note
that the statement \emph{(ii)} of Proposition \ref{prop action of B powers}
implies that $L_*\sigma$ is homologically trivial and hence
$\deg(\tau)=0$. Equation \eqref{eq rational equation} applied to
$\sigma$ gives
\begin{equation}\label{eq L - 2}
  ((L^2)_*-2)\sigma = -\frac{2}{25} l\cdot L_*\sigma
  -\frac{2\deg(\sigma)}{23\cdot 25}\, l^2 =-\frac{2}{25}\tau
  -\frac{2\deg(\sigma)}{23\cdot 25}\, l^2
\end{equation}
On the one hand, by condition (\ref{assumption hom}) we have $l\cdot
L_*\tau =25\tau$ so that
\[
(L^2)_*\tau = 2\tau -\frac{2}{25}l\cdot L_*\tau
-\frac{2\deg(\tau)}{23\cdot 25}\, l^2 = 2\tau -\frac{2}{25}\cdot
25\tau= 0.
\]
On the other hand, we have $(L^2)_*l^2=0$ by Lemma \ref{lem assumption
  l}.

Applying $(L^2)_*$ to equation \eqref{eq L - 2} thus gives
$$
(L^2)_*((L^2)_*-2)\sigma=0,
$$
which yields the required eigenspace decomposition for $\CH^4(F)$.

Let us now show that $\Lambda_2^4 = \mathrm{F}^4\CH^4(F)$. If
$\sigma\in \mathrm{F}^4\CH^4(F)$, then $\tau = l\cdot L_*\sigma = 0$.
Hence we get $(L^2)_*\sigma = 2\sigma$. It follows that
\[
\mathrm{F}^4\CH^4(F)\subseteq \Lambda_2^4.
\]
Conversely, consider $\sigma\in \Lambda_2^4$. Then equation
\eqref{eq L - 2} implies
\[
0 = -\frac{2}{25}\tau -\frac{2\deg(\sigma)}{23\cdot 25}\, l^2.
\]
Since $\tau$ is homologically trivial, we get $\deg(\sigma)=0$, which
further implies $\tau=0$. Then by condition \eqref{assumption hom} we
have $25L_*\sigma =L_*\tau=0$ which means that $L_*\sigma=0$. Hence we
have
$$
\Lambda_2^4\subseteq \mathrm{F}^4\CH^4(F).
$$

Finally, note that, as an immediate consequence of equation \eqref{eq
  L - 2} and of the fact that $\deg((L^2)_*\sigma)=0$, we have
\begin{equation}\label{eq description of W00 hom}
  (\Lambda_0^4)_\mathrm{hom}=  \{\sigma \in \CH^4(F)  : l\cdot
  L_*\sigma =  25\sigma\}.
\end{equation}

\vspace{2mm} \indent\textit{Step 2}  : $\CH^3(F)=\Lambda_2^3\oplus
\Lambda_0^3$, where $\Lambda_2^3=\CH^3(F)_\mathrm{hom}$.

Consider a divisor $\sigma\in\CH^3(F)$. Equation \eqref{eq rational
  equation} applied to $\sigma$ gives
$$
((L^2)_*-2)\sigma = -\frac{2}{25}l\cdot L_*\sigma.
$$
Note that $L_*\sigma$ is a divisor, and recall from Corollary
\ref{prop kunneth} that $L_*$ induces an isomorphism $\HH^6(F,\Q)
\stackrel{\simeq}{\longrightarrow} \HH^2(F,\Q)$ with inverse given by
intersecting with $\frac{1}{25}l$. Because $\CH^1(F)$ injects into
$\HH^2(F,\Q)$ via the cycle class map, we obtain that $L_*\sigma=0$ if
and only if $\sigma$ is homologically trivial.  In other words, we
have $(L^2)_*\sigma = 2 \sigma$ if and only if $\sigma \in
\CH^3(F)_\mathrm{hom}$. Statement \emph{(ii)} of Proposition \ref{prop
  action of B powers} implies that $(L^2)_*\sigma$ is homologically
trivial for all $\sigma\in\CH^3(F)$.  Thus
$$
((L^2)_*-2)(L^2)_*\sigma = 0,
$$
which yields the required eigenspace decomposition. Note for future
reference that
\begin{equation}
  \label{eq action pi6 CH1} \Lambda_0^3 = \{\sigma \in \CH^3(F)  :
  l\cdot L_*\sigma =25\sigma\}.
\end{equation}
\vspace{2mm}\\

\indent\textit{Step 3}  : $\CH^2(F) = \Lambda_{25}^2 \oplus \Lambda_2^2
\oplus \Lambda_{0}^2$, with $\Lambda_{25}^2=\langle l \rangle$,
$\Lambda_0^2=L_*\CH^4(F)=L_* (\Lambda_0^4)_\mathrm{hom}$.
Furthermore, we have $\CH^2(F)_\mathrm{hom} =
(\Lambda_2^2)_\mathrm{hom} \oplus \Lambda_0^2$.

Let $\sigma\in\CH^2(F)$ and denote $\tau  := L_*(l\cdot\sigma)$. We
know from \emph{(ii)} of Proposition \ref{prop action of B powers} that
$L_*\sigma = 0$ and that $\tau$ is homologically trivial. Equation
\eqref{eq rational equation}, together with the identity $L_*l=0$,
then gives
\begin{equation}\label{eq L - 2 on 2-cycles}
((L^2)_*-2)\sigma = -\frac{2}{25}L_*(l\cdot\sigma)
+\frac{1}{25} \left( \int_F [l]\cup[\sigma] \right) \, l.
\end{equation}
In particular, equation \eqref{eq rational equation} applied to $\tau$
gives
\[
((L^2)_*-2)\tau = -\frac{2}{25} L_*(l\cdot\tau) = -\frac{2}{25}
L_*(l\cdot L_*(l\cdot\sigma))= -\frac{2}{25}
(25L_*(l\cdot\sigma))=-2\tau,
\] where the third equality is hypothesis \eqref{assumption hom}.
Thus $(L^2)_*\tau =0$. We then apply $(L^2)_*$ to
equation \eqref{eq L - 2 on 2-cycles} and get, in view of the identity
$(L^2)_*l=25\,l$ of Lemma \ref{lem assumption l},
\begin{equation}\label{eq L dot L - 2}
  (L^2)_*((L^2)_*-2)\sigma =\left( \int_F [l]\cup[\sigma] \right) \, l, \qquad
  \forall \sigma\in\CH^2(F).
\end{equation}
We then apply $((L^2)_*-25)$ to the above equation and get
\[
(L^2)_* ((L^2)_* - 2) ((L^2)_* - 25)\sigma =0, \qquad \forall
\sigma\in\CH^2(F),
\] which yields the required eigenspace decomposition.

It is worth noting that \eqref{eq L dot L - 2} implies that
\begin{equation}\label{eq polynomial on CH2 hom}
(L^2)_*((L^2)_*-2)\sigma=0,\qquad \forall
\sigma\in\CH^2(F)_\mathrm{hom}.
\end{equation}

Let us characterize the eigenspaces $\Lambda_{25}^2$, $\Lambda_2^{2}$
and $\Lambda_{0}^2$.  First we show that $\Lambda_{25}^2$ is spanned
by $l$.  If $\sigma\in \Lambda^2_{25}$, then we have
\[
23\sigma = ((L^2)_*- 2)\sigma
=-\frac{2}{25}\tau+\frac{1}{25} \left( \int_F [l]\cup[\sigma] \right) \, l.
\]
Since $\tau$ is homologically trivial, we see that the homology class
$[\sigma]$ of $\sigma$ is a multiple of the homology class of $l$.
Consequently, there exists a rational number $a$ such that
$\sigma'=\sigma-al$ is homologically trivial and hence, by equation
\eqref{eq polynomial on CH2 hom}, we get
\[
(L^2)_*((L^2)_*-2)\sigma'=0
\]
Since $\sigma'\in\Lambda_{25}^2$, the above equation implies
$\sigma'=0$, namely $\sigma=al$. This proves that
$\Lambda_{25}^2=\langle l \rangle$.

Identity \eqref{eq L - 2 on 2-cycles}, together with Proposition
\ref{prop action of B powers}\emph{(iii)}, implies that $\sigma \in
\CH^2(F)$ satisfies $(L^2)_*\sigma = 2 \sigma$ if and only if
$L_*(l\cdot \sigma) =0$ and $\deg(l\cdot \sigma) = 0$. Thus, since
$\Lambda_2^4 = \F^4\CH^4(F)$, we get
\begin{equation}
  \label{eq action pi6 CH2} \Lambda_2^2 = \{\sigma \in \CH^2(F)  :
  (L^2)_*(l\cdot \sigma) =2l \cdot \sigma\}.
\end{equation}
In particular, this implies the inclusion
$l\cdot\Lambda_2^2\subseteq\Lambda_2^4$.

Finally, we give a description of the space $\Lambda_0^2$. Statement
\emph{(iii)} of Proposition \ref{prop action of B powers} implies that
$0$ is not an eigenvalue of $(L^2)_*$ acting on the cohomology group
$\HH^4(F,\Q)$. Hence all the elements of $\Lambda_0^2$ are
homologically trivial. Thus, if $\sigma\in\Lambda_0^2$, then we have
\[
-2\sigma = ((L^2)_*-2)\sigma = -\frac{2}{25}L_*(l\cdot\sigma).
\]
It follows that $\sigma=\frac{1}{25}L_*(l\cdot \sigma)$, namely that
$\Lambda_0^2\subseteq L_*\CH^4(F)$.  Consider now $\sigma \in
L_*\CH^4(F)$. Note that $\sigma$ is homologically trivial by statement
\emph{(ii)} of Proposition \ref{prop action of B powers}. By condition
\eqref{assumption hom}, we see that
\[
L_*(l\cdot\sigma) = 25\sigma.
\]
Hence it follows that
\[
((L^2)_*-2)\sigma=-\frac{2}{25}L_*(l\cdot\sigma) + \frac{1}{25}\left(
  \int_F [l]\cup[\sigma] \right) \, l =-2\sigma,
 \]
 which further gives $\sigma\in\Lambda_0^2$. Thus we have the
 inclusion $L_*\CH^4(F) \subseteq \Lambda_0^2.$ The above arguments
 establish the equality
\[
\Lambda_0^2= L_*\CH^4(F).
\]
Actually, since $L_*l^2 = 0$ and $L_*\Lambda_2^4 = 0$, we have
\[
\Lambda_0^2= L_*(\Lambda_0^4)_{\hom}.
\]
The identity \eqref{eq description of W00 hom} then implies that $L_* :
(\Lambda_0^4)_\mathrm{hom}\rightarrow \Lambda_0^2$ is an isomorphism
with inverse $\frac{1}{25}l\cdot : \Lambda_0^2\rightarrow
(\Lambda_0^4)_\mathrm{hom}$.
\vspace{2mm}\\

\indent \textit{Step 4}  : $\CH^1(F) = \Lambda_0^1$.  Indeed,
$\CH^1(F)$ injects into $\HH^2(F,\Q)$ by the cycle class map and $L^2$
acts as zero on $\HH^2(F,\Q)$ by Proposition \ref{prop action of B
  powers}\emph{(ii)}.
\vspace{2mm}\\

\indent \textit{Step 5}  : $\CH^0(F) = \Lambda_0^0$. This is obvious
from Lemma \ref{lem assumption l}.
\end{proof}

\begin{thm} \label{prop main} Let $F$ be a hyperk\"ahler variety of
  $\mathrm{K3}^{[2]}$-type. Assume that there exists $L \in \CH^2(F
  \times F)$ as in Conjecture \ref{conj main} satisfying conditions
  (\ref{assumption pre l}) and (\ref{assumption hom}). For
  $\lambda\in\Q$, denote
  $$
  W^i_{\lambda} :=\{\sigma \in\CH^i(F) :\FF\circ \FF(\sigma)=\lambda
  \sigma\}.$$ Then, with the
  notations of Theorem \ref{thm main splitting} and Theorem
  \ref{prop L2}, we have
  \begin{align*}
    \CH^4(F)_0 & = W^4_{\frac{23\cdot 25}{8}} = \langle l^2 \rangle, \
    \CH^4(F)_2 = W^4_{\frac{25}{2}} = (\Lambda_0^4)_{\hom}, \
    \CH^4(F)_4 = W_1^4 = \Lambda_2^4 ;\\
    \CH^3(F)_0 & = W^3_{\frac{25}{2}} = \Lambda_0^3, \ \CH^3(F)_2 =
    W_1^3 = \Lambda_2^3 ;\\
    W^2_{(\frac{25}{2})^2} & = \Lambda_{25}^2 = \langle l \rangle, \
    W_1^2 = \Lambda_2^2, \ \CH^2(F)_2 = W_{\frac{25}{2}}^2 =
    \Lambda_0^2.
  \end{align*}
  Assume moreover that $L$ satisfies the following condition~:
  \medskip

 \noindent \eqref{assumption 2hom} \hfill
 $(L^2)_*(l\cdot(L^2)_*\sigma) = 0$ for all $\sigma \in \CH^2(F).$
 \hfill {} \medskip

\noindent Then
  \begin{align*}
    \CH^2(F)_0 & = \Lambda_{25}^2 \oplus \Lambda_2^2.
  \end{align*}
\end{thm}
\begin{proof} All is needed is to understand the actions of $(L^s)_*$
  and $(L^r)_*(L^s)_*$ on the pieces $\Lambda^i_\lambda$. Clearly, for
  dimension reasons, $(L^s)_*\Lambda^i_\lambda = 0$ either if $i < 4-
  2s$ or $i>8-2s$.

  First, we note that the action of $L^2$ does not completely
  decompose $\CH^4(F)$.  The correspondence $L^0 = [F \times F]$
  clearly acts as zero on $\CH^4(F)_{\hom}$. In particular $L^0$ acts
  as zero on $\Lambda_2^4$. By Lemma \ref{lem assumption l}, $F$
  satisfies condition (\ref{assumption l}). Thus $L^4 \circ L^0$ acts
  as multiplication by $3\cdot 23\cdot 25$ on $l^2$ and $(L^2)_*l^2 =
  0$.  It follows that the action of $L^4 \circ L^0$ commutes with the
  action of $L^2$ on $\CH^4(F)$. Therefore, $\Lambda_0^4$ further
  splits into
  $$\Lambda_0^4 = \langle l^2 \rangle \oplus (\Lambda_0^4)_{\hom}.$$
  Now by condition \eqref{assumption pre l} we have $L_*l^2 = 0$, and
  by Theorem \ref{prop L2} we have $L_*\Lambda_0^2 = 0$. We claim
  that $L^3 \circ L$ acts by multiplication by $3\cdot 25$ on
  $(\Lambda_0^4)_{\hom}$. Recall that we have equation \eqref{eq L3}
  for the expression of $L^3$.  Thus, having in mind that $L_*\sigma$
  is homologically trivial for $\sigma \in \CH^4(F)$,
  \begin{center}
    $(L^3)_*(L_*\sigma) = 2l\cdot L_*\sigma - \frac{2}{25}\big( l\cdot
    (L^2)_*L_*\sigma + (L^2)_*(l\cdot L_*\sigma) \big) +
    \frac{1}{25}l\cdot L_*(l\cdot L_*\sigma)$, \quad for all $\sigma
    \in \CH^4(F)$.
  \end{center}
  Condition \eqref{assumption hom} gives $l\cdot L_*(l\cdot L_*\sigma) =
  25 l\cdot L_*\sigma$ and Theorem \ref{prop L2}
  gives $L_*\sigma \in \Lambda_0^2$, \textit{i.e.},  $(L^2)_*L_*\sigma
  =0$. Hence
\begin{center}
  $(L^3)_*(L_*\sigma) = 3l\cdot L_*\sigma - \frac{2}{25}(L^2)_*(l\cdot
  L_*\sigma)$, \quad for all $\sigma \in \CH^4(F)$.
\end{center}
The claim then follows from the identity \eqref{eq description of W00
  hom} which stipulates that $l\cdot L_*\sigma = 25\sigma$ for all
$\sigma \in (\Lambda_0^4)_{\hom}$.\medskip

Since $\Lambda_2^3 = \CH^3(F)_{\hom}$, it is clear that
$L_*\Lambda_2^3 = 0$. Consider $\sigma \in \Lambda_0^3$. Then
Theorem \ref{prop L2} gives $L_*\sigma \in \Lambda_0^1$ and $l
\cdot L_*\sigma = 25\sigma$. Equation \eqref{eq L3} applied to
$L_*\sigma$ yields
\begin{equation*}
  (L^3)_*L_*\sigma  = 2 l \cdot L_*\sigma + \frac{1}{25}l
  \cdot L_*(l \cdot L_*\sigma) = 3 \cdot 25 \sigma.
\end{equation*} \medskip

Condition \eqref{assumption pre l}, via Lemma \ref{lem assumption l},
gives $(L^k)_* \Lambda_2^{25}=0$ unless $k=2$, and Proposition
\ref{prop action of B powers} gives $L_*\CH^2(F) =0$. Recall that by
Theorem \ref{prop L2} we have $\Lambda_0^2 =
L_*(\Lambda_0^4)_{\hom}$. Consider $\sigma \in \Lambda_0^2$ and
$\sigma_0 \in (\Lambda_0^4)_{\hom}$ such that $\sigma = L_*\sigma_0$.
Equation \eqref{eq L3}, together with equation \eqref{eq description
  of W00 hom}, then gives $(L^3)_*\sigma = 3\cdot 25\sigma_0,$ so that
$$L_*(L^3)_*\sigma = 3\cdot 25\sigma.$$ Let us now consider an element
$\sigma \in \Lambda_2^2$. According to \eqref{eq action pi6 CH2},
$\sigma \in \Lambda_2^2$ if and only if $l\cdot \sigma \in
\Lambda_2^4$~; and according to Theorem \ref{prop L2}, $l\cdot
\sigma \in \Lambda_2^4$ only if $L_*(l\cdot \sigma) = 0$. Therefore
applying \eqref{eq L3} to $\sigma$ yields
\begin{equation} \label{eq L3 CH2}
(L^3)_*\sigma  = 2l\cdot \sigma + \frac{2}{25} l\cdot (L^2)_*\sigma +
\frac{2}{25}(L^2)_*(l\cdot \sigma) + \frac{1}{25}l\cdot L_*(l\cdot
\sigma)  = \frac{58}{25}l\cdot \sigma.
\end{equation}
It immediately follows that
\begin{align*}
L_*(L^3)_* \sigma =0  , \quad \mbox{for all $\sigma \in \Lambda_2^2$}.
\end{align*}
Note that
\begin{equation}
  \label{eq Lambda22eq} \mbox{Condition \eqref{assumption 2hom}
    is equivalent to} \ l\cdot \Lambda_2^2
  = 0.
\end{equation}
Indeed, consider $\sigma \in \Lambda_2^2$. On the one hand, \eqref{eq
  action pi6 CH2} gives $l \cdot \sigma \in \Lambda_2^4$. On the other
hand, condition \eqref{assumption 2hom} gives $(L^2)_*(l\cdot \sigma)
= 0$, \emph{i.e.}, $l\cdot \sigma \in \Lambda_0^4$. But, by Theorem
\ref{prop L2}, $\CH^4(F) = \Lambda_0^4 \oplus \Lambda_2^4$. Hence $l
\cdot \sigma = 0$.

Thus, assuming \eqref{assumption 2hom}, \eqref{eq L3 CH2} gives
$(L^3)_* \sigma =0$ for all $\sigma \in \Lambda_2^2$. This yields the
identity $\CH^2(F)_0 = \Lambda_{25}^2 \oplus \Lambda_2^2$.\medskip

Finally, both $\CH^1(F)$ and $\CH^0(F)$ inject by the cycle class map
into $\HH^2(F,\Q)$ and $\HH^0(F,\Q)$ respectively, so that the proof
here reduces to Proposition \ref{prop action of B powers}.
 \end{proof}

\vspace{10pt}
\section{The Fourier decomposition is motivic} \label{sec CK}

Let $F$ be a hyperk\"ahler variety of K3$^{[2]}$-type.  In this
section, we show that, provided $F$ is endowed with a cycle $L \in
\CH^2(F \times F)$ representing the Beauville--Bogomolov class
$\mathfrak{B}$ satisfying \eqref{eq rational equation},
\eqref{assumption pre l}, \eqref{assumption hom} and \eqref{assumption
  2hom}, the induced Fourier decomposition of Theorem \ref{thm main
  splitting} arises from a Chow--K\"unneth decomposition of the
diagonal.\medskip

Let $X$ be a smooth projective variety of dimension $d$ defined over a
field $k$. Murre \cite{murre} conjectured the existence of mutually
orthogonal idempotents $\pi^0, \ldots, \pi^{2d}$ in the ring (for the
composition law) of correspondences $\CH^d(X \times X)$ such that
$\Delta_X = \pi^0+ \ldots + \pi^{2d} \in \CH^d(X \times X)$ and such
that the cohomology class of $\pi^j \in \HH^{2d}(X \times X) \subseteq
\End(\HH^*(X))$ is the projector on $\HH^j(X)$. Here $\HH^j(X)$
denotes the $\ell$-adic cohomology group
$\HH^j(X_{\overline{k}},\mathds{Q}_\ell)$. Such a decomposition of the
diagonal is called a \emph{Chow--K\"unneth decomposition}~; it is a
lift of the K\"unneth decomposition of $\mathrm{id} \in
\End(\HH^*(X))$ via the cycle class map $\CH^d(X \times X) \rightarrow
\HH^{2d}(X \times X)$. Furthermore, Murre conjectured that any
Chow--K\"unneth
decomposition satisfies the following properties~: \\
(B) $\pi^i_* \CH^j(X) = 0$ for $i>2j$ and for $i<j$~;\\
(D) $\ker \{\pi^{2j}_*  : \CH^j(X) \rightarrow \CH^j(X)\} =
\CH^j(X)_{\hom} $. \\ If such a Chow--K\"unneth decomposition exists,
we may define a descending filtration $\mathrm{F}^\bullet$ on
$\CH^j(X)$ as follows~: $$\mathrm{F}^r \CH^j(X)  := \ker\{ (\pi^{2j} +
\ldots + \pi^{2j-r+1})_*  : \CH^j(X) \rightarrow \CH^j(X)\}.$$ Thus
$\mathrm{F}^0\CH^j(X) = \CH^j(X) $, $\mathrm{F}^1\CH^j(X)
=\CH^j(X)_{\hom}$ and $\mathrm{F}^{j+1}\CH^j(X) =0$. Murre further
conjectured\\
(C) The filtration $\mathrm{F}^\bullet$ does not depend on the choice of
Chow-K\"unneth decomposition. \medskip

It is a theorem of Jannsen \cite{Jannsen2} that Murre's conjectures
for all smooth projective varieties are equivalent to the conjectures
of Bloch and Beilinson \cite{beilinson, bloch}~; see \cite[Conjecture
2.1]{Jannsen2} for a formulation of these conjectures. Here we show
that a hyperk\"ahler fourfold that satisfies the assumptions of
Theorems \ref{prop L2} and \ref{prop main}, \emph{e.g.} the Hilbert
scheme of length-$2$ subschemes on a K3 surface or the Fano scheme of
lines on a smooth cubic fourfold, has a Chow--K\"unneth decomposition
that satisfies (B) and for which the induced filtration on $\CH^4(F)$
is the one considered in \eqref{eq filtration gal}. Note that the
existence of a Chow--K\"unneth decomposition for the Hilbert scheme of
length-$n$ subschemes on a surface follows directly from the existence
of a Chow--K\"unneth decomposition for surfaces \cite{murre2} and from
the work of de Cataldo and Migliorini \cite{dCM}. We will also give in
Section \ref{sec multCKX2} another way of constructing a
Chow--K\"unneth decomposition for $S^{[2]}$. In the same spirit as
Conjecture \ref{conj sheaf}, all these Chow--K\"unneth decompositions
for $S^{[2]}$ should agree provided that one starts with the
Chow--K\"unneth decomposition for $S$ given by $\pi^0_S =
\mathfrak{o}_S \times S$, $\pi^4_S = S \times \mathfrak{o}_S$ and
$\pi^2_S = \Delta_S - \pi^0_S - \pi^4_S$, where $\mathfrak{o}_S$ is
any point lying on a rational curve on $S$. \medskip

The following lemma relies on a technique initiated by Bloch and
Srinivas \cite{bs}.

\begin{lem} \label{lem BS} Let $X$ be a smooth projective variety over
  a field $k$ and let $\Omega$ be a universal domain containing $k$.
  If $f \in \CH^{\mathrm{dim} \, X} (X \times X)$ is a correspondence
  such that $(f_\Omega)_*\CH_0(X_\Omega) = 0$, then there is a smooth
  projective variety $Y$ of dimension $\mathrm{dim} \, X-1$ and
  correspondences $g \in \CH^{\mathrm{dim} \, X}(Y \times X)$ and $h
  \in \CH_{\mathrm{dim} \, X}(X \times Y)$ such that $f = g \circ h$.
  If ${\mathrm{dim} \, X} = 0$, then $f=0$.
\end{lem}
\begin{proof} The lemma is clear when $d := {\mathrm{dim} \, X} = 0$.
  Let $k(X)$ be the function field of $X$. The assumption that
  $(f_\Omega)_*\CH_0(X_\Omega) = 0$ implies that
  $(f_{k(X)})_*\CH_0(X_{k(X)}) = 0$. Let $\eta$ be the generic point
  of $X$ seen as a $k(X)$-rational point. In particular,
  $(f_{k(X)})_*[\eta] = 0$. The $0$-cycle $(f_{k(X)})_*[\eta] \in
  \CH_0(k(X) \times X)$ coincides with the restriction of $f$ along
  the map $\CH^d(X \times X) \rightarrow \CH_0(k(X) \times X)$
  obtained as the direct limit, indexed by the non-empty open sub-sets
  $U$ of $X$, of the flat pull-back maps $\CH^d(X \times X) \rightarrow
  \CH^d(U \times X)$. Therefore, by the localization exact sequence
  for Chow groups, $f$ is supported on $D \times X$ for some divisor
  $D$ inside $X$. If $Y \rightarrow D$ is an alteration of $D$, then
  $f$ factors through $Y$.
\end{proof}

\begin{prop} \label{prop nilpotent} Let $X$ be a smooth projective
  variety over a field $k$ and let $\Omega$ be a universal domain
  containing $k$. If $f \in \CH^{\mathrm{dim} \, X} (X \times X)$ is a
  correspondence such that $(f_\Omega)_*\CH_*(X_\Omega) = 0$, then $f$
  is nilpotent. Precisely, if $N :=2^{\dim X}-1$, we have $f^{\circ
    N}=0$.
\end{prop}
\begin{proof} We proceed by induction on $\mathrm{dim} \, X$. If
  $\mathrm{dim} \, X= 0$, then by Lemma \ref{lem BS} $f = 0$. Let's
  now assume that $\mathrm{dim} \, X > 0$. By Lemma \ref{lem BS},
  there is a smooth projective variety $Y$ of dimension $\mathrm{dim}
  \, X-1$ and correspondences $g \in \CH^{\mathrm{dim} \, X}(Y \times
  X)$ and $h \in \CH_{\mathrm{dim} \, X}(X \times Y)$ such that $f = g
  \circ h$. The correspondence $(h \circ g \circ h \circ g)_\Omega \in
  \CH_{\mathrm{dim} \, Y}(Y_\Omega \times Y_\Omega)$ then acts as zero
  on $\CH_*(Y_\Omega)$. By induction, $h \circ g \circ h \circ g$ is
  nilpotent of index $2^{\dim Y}-1$. It immediately follows that $f$
  is nilpotent of index $2\cdot (2^{\dim Y}-1) + 1 = 2^{\dim X}-1$.
\end{proof}

\begin{thm} \label{thm CK} Let $F$ be a hyperk\"ahler fourfold of
  $\mathrm{K3}^{[2]}$-type that satisfies the assumptions of Theorems
  \ref{prop L2} and \ref{prop main}.  Then $F$ has a Chow--K\"unneth
  decomposition $\{\pi^0, \pi^2, \pi^4, \pi^6, \pi^8\}$ that satisfies
  (B) and such that
$$\pi^s_* \CH^i(F) = \CH^i(F)_{2i-s}.$$
Moreover, $\{\pi^0, \pi^2, \pi^4, \pi^6, \pi^8\}$ satisfies (D) if and
only if the cycle class map $cl  : \CH^2(F) \rightarrow \HH^4(F,\Q)$
restricted to $\CH^2(F)_0$ is injective.
\end{thm}
\begin{proof}
  In view of Corollary \ref{prop kunneth}, it is tempting to think
  that the identity (\ref{eq rational equation}) already gives a
  Chow--K\"unneth decomposition of the diagonal (and this would be a
  consequence of Conjecture \ref{conj sheaf}~; see Theorem \ref{thm2
    conj repeat}). We cannot prove this but the arguments below
  consist in modifying the correspondences of Proposition \ref{prop
    kunneth} so as to turn them into mutually orthogonal idempotents
  modulo rational equivalence.

  First we define $\pi^0 = \frac{1}{23\cdot 25} l_1^2$ and $\pi^8 =
  \frac{1}{23\cdot 25} l_2^2$. These are clearly orthogonal
  idempotents. Then we define $$p  := \frac{1}{25} L\cdot l_2 \in
  \CH^4(F \times F).$$

  The correspondence $p$ defines an idempotent in $\HH^8(F \times
  F,\Q)$, which is the K\"unneth projector onto $\HH^6(F)$~; see
  Corollary \ref{prop kunneth}. Because $L_*l^2 = 0 \in \CH^*(F)$,
  we see that $p \circ \pi^8 = \pi^8 \circ p = 0$.
  The action of $p_*$ on $\CH^*(F)$ is given by $p_*\sigma = l\cdot
  L_* \sigma$ and the action of $p^*$ is given by $p^*\sigma =
  L_*(l\cdot \sigma)$. The action of $p_*$ on $\CH^i(F)_s$ is zero
  unless $2i-s=6$, in which case $p_*$ acts as the identity, and the
  action of $p^*$ on $\CH^i(F)_s$ is zero unless $2i-s=2$, in which
  case $p^*$ acts as the identity~; see the proof of Proposition
  \ref{prop L2}, especially the identities \eqref{eq description of
    W00 hom}, \eqref{eq action pi6 CH1} and \eqref{eq action pi6 CH2}
  therein.  Finally, we note that $p \circ {}^tp$ factors through $L
  \circ L \in \CH^0(F \times F)$ which is zero because $L_*l^2 =0$.

  The above shows that the correspondence $p\circ p - p \in \CH^4(F
  \times F)$ acts as zero on $\CH^*(F)$. By Proposition \ref{prop
    nilpotent}, $p\circ p - p$ is nilpotent, say of index $N$.  It
  follows that the image, denoted $A$, of the homomorphism of
  $\Q$-algebras $\Q[T] \rightarrow \CH^4(F\times F)$ which sends $T$
  to $p$ is a quotient of $\Q[T]/(T^N(T-1)^N)$, in particular a
  commutative finite-dimensional $\Q$-algebra. Moreover, if we
  consider the reduced algebra, then $\bar{p} \in A/Nil(A)$ is a
  projector. By Wedderburn's section theorem (\emph{cf.} \cite{ak}),
  we can lift $\bar{p}$ to a genuine projector $q$ in $A$, which
  differs from $p$ by a nilpotent element $n \in A$. Since neither $p$
  nor $1-p$ are nilpotent (they both define non-trivial projectors
  modulo homological equivalence), we remark that $n$ must factor
  through $p \circ p - p$.  Thus $p$ and $q$ are homologically
  equivalent correspondences and, because $p\circ p - p$ acts as zero
  on $\CH^*(F)$, it is also apparent that $p_*$ and $q_*$ have the
  same action on $\CH^*(F)$.  Thus $q$ is an idempotent whose
  cohomology class is the projector on $\HH^6(M)$ and whose action
  $q_*$ on $\CH^4(F)$ is the projector with image $\CH^{4}(F)_2$ and
  kernel $\CH^4(F)_0 \oplus \CH^4(F)_4$.

  We noted that $p \circ {}^tp = 0$. We also noted that $q=p+n$, where
  $n$ factors through $p\circ p -p$ and also commutes with
  $p$. Therefore $q \circ {}^tq = 0$.  We then define
  \begin{center}
    $\pi^{2}  := {}^tq \circ (1 - \frac{1}{2}q)$ \quad and \quad
    $\pi^{6}  := (1 - \frac{1}{2}{}^tq) \circ q.$
  \end{center}
  The correspondences $\pi^{2}$ and $\pi^{6}$ are then clearly
  orthogonal idempotents in $\CH^4(F \times F)$. By Poincar\'e
  duality, we see that the classes modulo homological equivalence of
  $\pi^{2}$ and $\pi^{6}$ are $\pi^{2}_{\hom}$ and $\pi^{6}_{\hom}$,
  respectively. It is also apparent that $ \pi^0_* \CH^*(F) =
  \CH^0(F)_0$, $\pi^2_* \CH^*(F) = \CH^1(F)_0 \oplus \CH^2(F)_2$, $
  \pi^6_* \CH^*(F) = \CH^3(F)_0 \oplus \CH^4(F)_2$, and $\pi^8_*
  \CH^*(F) = \CH^4(F)_0$.

  We now define
\begin{center}
  $\pi^{4}  := \Delta_F - (\pi^0 + \pi^2 + \pi^6 + \pi^8)$.
\end{center}
It is then clear that $\{\pi^0, \pi^2, \pi^4, \pi^6, \pi^8\}$ defines
a Chow--K\"unneth decomposition for $F$ which satisfies $\pi^4_*
\CH^*(F) = \CH^2(F)_0 \oplus \CH^3(F)_2 \oplus \CH^4(F)_4$ as well as
Murre's conjecture (B).

Finally, note that the correspondences ${}^tp$ and ${}^tq$ act the
same on $\CH^2(F)_{\hom}$~: they project onto $\CH^2(F)_2$ along
$\CH^2(F)_{0,\hom}$. Moreover, $q$ acts as zero on $\CH^2(F)_{\hom}$.
Thus $\pi^2$ acts like ${}^tp$ on $\CH^2(F)_{\hom}$.  By the above, we
also have a decomposition $\CH^2(F) = \mathrm{im}\{\pi^4_* \} \oplus
\mathrm{im}\{\pi^2_* \} = \mathrm{im}\{\pi^4_* \} \oplus
\mathrm{ker}\{\pi^4_* \}$, where $\pi^4_*$ and $\pi^2_*$ are acting on
$\CH^2(F)$.  Therefore, Murre's conjecture (D) holds if and only if
$\mathrm{ker}\{\pi^4_* \} = \CH^2(F)_{\hom}$ if and only if
$\mathrm{im}\{\pi^2_* \} = \CH^2(F)_{\hom}$ if and only if
$\CH^2(F)_{0,\hom} = 0$.
\end{proof}

Let us end this paragraph with the following proposition which gives
evidence for the uniqueness up to sign of a cycle $L \in \CH^2(F
\times F)$ satisfying the quadratic equation \eqref{eq rational
  equation}.

\begin{prop} \label{prop L unique} Let $F$ be a hyperk\"ahler fourfold
  of $\mathrm{K3}^{[2]}$-type endowed with a cycle $L \in
  \CH^2(F\times F)$ representing the Beauville--Bogomolov class
  $\mathfrak{B}$. Assume that $L$ is symmetric, that is ${}^tL=L$, and
  satisfies \eqref{eq rational equation}, \eqref{assumption pre l},
  \eqref{assumption hom} and \eqref{assumption 2hom}. Assume that the
  Fourier decomposition of Theorem \ref{prop main} satisfies
  $\CH^2(F)_2 \cdot \CH^2(F)_2 \subseteq \CH^4(F)_4$. Assume the
  Bloch--Beilinson conjectures. Then $L$ is the unique symmetric cycle
  in $\CH^2(F \times F)$ representing $\mathfrak{B}$ that satisfies
  the quadratic equation \eqref{eq rational equation}.
\end{prop}
\begin{proof}
  Let $L$ be as in Theorem \ref{thm2 main} and let $L+E$ be another
  cycle representing $\mathfrak{B}$ that satisfies the quadratic
  equation \eqref{eq rational equation}. In particular, $E$ is
  homologically trivial. Writing $l := \iota_\Delta^*L$, $l_i  :=
  p_i^*l$, $\varepsilon  := \iota_\Delta^*E$ and $\varepsilon_i
   :=p_i^*\varepsilon$, where $\iota_\Delta  : F \rightarrow F \times F$
  is the diagonal embedding and $p_i  : F \times F \rightarrow F$ are
  the projections, we have $$(L+E)^2 =2\Delta_F -\frac{2}{25}
  (l_1+\varepsilon_1 + l_2 + \varepsilon_2)\cdot (L + E) -
  \frac{2}{23\cdot 25} \big((l_1+ \varepsilon_1)^2 +
  (l_2+\varepsilon_2)^2\big) +\frac{1}{25} (l_1+
  \varepsilon_1)\cdot(l_2+\varepsilon_2).$$ Subtracting the equation
  satisfied by $L$ and pulling back along the diagonal, we obtain
  $$\varepsilon^2 + 2\, \varepsilon \cdot l = 0\ \in \CH^4(F).$$ Since
  $E$ is homologically trivial, $\varepsilon$ is also homologically
  trivial. Assuming the Bloch--Beilinson conjectures, we get by
  Theorem \ref{thm CK} that $\varepsilon \in \CH^2(F)_2$.  Therefore
  we have on the one hand $\varepsilon^2 \in \CH^4(F)_4$ by
  assumption, and on the other hand $\varepsilon \cdot l \in
  \CH^4(F)_2$ by Theorem \ref{prop main}. It immediately follows that
  $\varepsilon^2 = \varepsilon \cdot l = 0$. By Theorem \ref{prop L2},
  $l \cdot \,  : \CH^2(F)_2 \rightarrow \CH^4(F)$ is injective. We
  conclude that $\varepsilon = 0$.

  Let us now write $\mathfrak{h}(F) = \mathfrak{h}^0(F) \oplus
  \mathfrak{h}^2(F) \oplus \mathfrak{h}^4(F) \oplus \mathfrak{h}^6(F)
  \oplus \mathfrak{h}^8(F)$ for the Chow--K\"unneth decomposition of
  $F$ obtained in Theorem \ref{thm CK}. Here $\mathfrak{h}(F)$ is the
  Chow motive of $F$ and $\mathfrak{h}^{2i}(F)$ is the Chow motive
  $(F,\pi^{2i})$. Then the motive of $F \times F$ has a
  Chow--K\"unneth decomposition $\bigoplus_{i=0}^8 \mathfrak{h}^{2i}(F
  \times F)$, where $\mathfrak{h}^{2i}(F \times F)  :=
  \bigoplus_{j=0}^i \mathfrak{h}^{2j}(F) \otimes
  \mathfrak{h}^{2i-2j}(F).$ The Bloch--Beilinson conjectures imply
  that $\CH^2(F \times F)_{\hom} = \CH^2(\mathfrak{h}^2(F\times F)) =
  \CH^2(\mathfrak{h}^2(F)\otimes \mathfrak{h}^0(F) \oplus
  \mathfrak{h}^0(F) \otimes \mathfrak{h}^2(F)) =
  p_1^*\CH^2(\mathfrak{h}^2(F)) \oplus p_2^*\CH^2(\mathfrak{h}^2(F)) =
  p_1^*\CH^2(F)_{\hom} \oplus p_2^*\CH^2(F)_{\hom}$. Therefore, there
  exist $\mu$ and $\nu \in \CH^2(F)_{\hom}$ such that $$E  := p_1^*\mu
  + p_2^*\nu.$$ That $\varepsilon  := \iota_\Delta^*E = 0$ yields $\mu
  + \nu = 0$. On the other hand, if one assumes that ${}^tL=L$ and
  ${}^t(L+E) = L+E$, then ${}^tE=E$, so that $\mu = \nu$. We conclude
  that $\mu = \nu = 0$ and hence that $E=0$.
\end{proof}

\begin{rmk} \label{rmk assumptions satisfied} As will be shown in
  Sections \ref{sec Fourierdec S2} and \ref{sec Fourier cubic}, if $F$
  is the Hilbert scheme of length-$2$ subschemes on a K3 surface or
  the variety of lines on a cubic fourfold, then there is a symmetric
  cycle $L \in \CH^2(F \times F)$ representing the
  Beauville--Bogomolov class $\mathfrak{B}$ that satisfies \eqref{eq
    rational equation}, \eqref{assumption pre l}, \eqref{assumption
    hom} and \eqref{assumption 2hom}. Moreover, the Fourier
  decomposition induced by such a cycle $L \in \CH^2(F \times F)$
  satisfies the extra assumption $\CH^2(F)_2 \cdot \CH^2(F)_2
  \subseteq \CH^4(F)_4$ of Proposition \ref{prop L unique} (and in
  fact $\CH^2(F)_2 \cdot \CH^2(F)_2 = \CH^4(F)_4$)~; see Proposition
  \ref{prop F4 of S2} and Theorem \ref{thm surjection of intersection
    product}.
\end{rmk}

\vspace{10pt}
\section{First multiplicative results} \label{sec mult div}

\subsection{Intersection with $l$} In this paragraph, we consider a
hyperk\"ahler fourfold $F$ endowed with a cycle $L \in \CH^2(F \times
F)$ representing the Beauville--Bogomolov class $\mathfrak{B}$ that
satisfies hypotheses \eqref{eq rational equation}, \eqref{assumption
  pre l}, \eqref{assumption hom} and \eqref{assumption 2hom}. These
hypotheses which are prerequisites to establishing the Fourier
decomposition for $F$ are related to the understanding of the
intersection of $l$ with $2$-cycles. Here, we basically reformulate
these hypotheses in the context of the Fourier decomposition.\medskip

First, it is straightforward to extract from the statements of
Theorems \ref{prop L2} and \ref{prop main} the following.

\begin{prop} \label{prop l CH2 6} We have $$\CH^4(F)_2 = \langle l
  \rangle \cdot \CH^2(F)_2.$$ \qed
\end{prop}

Second, \eqref{assumption 2hom} can be reformulated as the following.

\begin{prop} \label{prop l CH2 4} We have $$\CH^4(F)_0 = \langle l
  \rangle \cdot \CH^2(F)_0.$$ More precisely, we have $$ W_1^2 = \ker
  \, \{l\cdot  : \CH^2(F) \rightarrow \CH^4(F)\}.$$
\end{prop}
\begin{proof}
  Recall from Theorems \ref{prop L2} \& \ref{prop main} that $\CH^2(F)
  = \CH^2(F)_0 \oplus \CH^2(F)_2$ with $\CH^2(F)_0 = \langle l \rangle
  \oplus \Lambda_2^2$, $\CH^2(F)_2 = \Lambda_0^2$, and that $l\cdot  :
  \CH^2(F)_2 \rightarrow \CH^4(F)$ is injective.  We also have
  $\Lambda_2^2 = W_1^2$. Thus, because $\CH^4(F)_0=\langle l^2
  \rangle$, we are reduced to prove that $l \cdot \Lambda_2^2 = 0$.
But then, this is \eqref{eq Lambda22eq}.
\end{proof}

\subsection{Intersection of divisors}
In this paragraph, we consider a hyperk\"ahler fourfold $F$ which is
either the Hilbert scheme of length-$2$ subschemes on a K3 surface or
the variety of lines on a smooth cubic fourfold. It will be shown in
Parts 2 and 3 that there exists a cycle $L \in \CH^2(F \times F)$
representing the Beauville--Bogomolov class $\mathfrak{B}$ that
satisfies hypotheses \eqref{eq rational equation}, \eqref{assumption
  pre l}, \eqref{assumption hom} and \eqref{assumption 2hom}.
Moreover, it will also be the case that $l := \iota_{\Delta}^*L =
\frac{5}{6}c_2(F)$, where $\iota_\Delta  : F \rightarrow F \times F$ is
the diagonal embedding. The goal here is to prove Theorem \ref{thm
  reformulation voisin}, which sets in the context of the Fourier
decomposition the following result.

\begin{thm} [Voisin \cite{voisin2}] \label{thm voisin} Let $F$ be
  either the Hilbert scheme of length-$2$ subschemes on a K3 surface
  or the variety of lines on a smooth cubic fourfold. Then any
  polynomial cohomological relation $P([c_1(L_j)] ;[c_i(F))]) = 0$ in
  $\HH^{2k}(F,\Q)$, $L_j \in \mathrm{Pic} \, F$, already holds at the
  level of Chow groups  : $P(c_1(L_j) ; c_i(F )) = 0$ in $\CH^k(F )$.
 \end{thm}

\begin{lem}\label{lem intersection of two divisors}
  Let $D_i, D'_i \in \CH^1(F)$ be divisors on $F$. Then
  \begin{enumerate}[(i)]
  \item $D_1 \cdot D_2 \in W_1^2 \oplus \langle l \rangle \ (=
    \CH^2(F)_0)$~;
  \item $\sum_{i=1}^{m} D_i\cdot D'_i \in W_1^2$ if and only if
    $\sum_{i=1}^{m}q_F([D_i],[D'_i])=0$.
  \end{enumerate}
\end{lem}
\begin{proof} By Proposition \ref{prop action of B powers}, $l\cdot
  L_*(D_1\cdot D_2)$ vanishes. Therefore, according to \eqref{eq
    rational equation}, we have $$(L^2)_*(D_1 \cdot D_2) = 2D_1\cdot
  D_2 - \frac{2}{25}L_*(l\cdot D_1 \cdot D_2) +
  \frac{1}{25}\left(\int_F [l]\cup [D_1] \cup [D_2]\right) \, l.$$ By
  Theorem \ref{thm voisin} and the fact that $l=\frac{5}{6}c_2(F)$,
  $l\cdot D_1 \cdot D_2$ is a multiple of $l^2$, so that $L_*(l\cdot
  D_1 \cdot D_2)=0$. It is also a fact that $\int_F [l]\cup [D_1] \cup
  [D_2]=25\, q_F([D_1],[D_2])$. Hence $$(L^2)_*(D_1 \cdot D_2) =
  2D_1\cdot D_2 + q_F([D_1],[D_2]) \, l,$$ which establishes item
  \emph{(ii)}. Since $(L^2)_*l = 25 \, l$ by Theorem \ref{prop L2}, we
  find that $$(L^2-25)_*(L^2-2)_*(D_1\cdot D_2) = 0,$$ that is, $D_1
  \cdot D_2$ belongs to $\Lambda^2_2 \oplus \Lambda^2_{25}$. By
  Theorem \ref{prop main}, $\Lambda^2_2 = W^2_1$ and $\Lambda^2_{25} =
  \langle l \rangle$, and item \emph{(i)} is proved.
\end{proof}

\begin{lem}\label{lem intersection of three divisors}
  Let $D \in \CH^1(F)$ be a divisor. Then
  \begin{enumerate}[(i)]
  \item $L_*D^3 = 3 q_F([D])\, D$~;
\item $(L^2)_*D^3 = 0$~;
\item $D^3 = \frac{3}{25} q_F([D])\, l \cdot D$.
  \end{enumerate}
\end{lem}
\begin{proof}
  \emph{(i)} -- $L_*D^3$ is a divisor whose cohomology class
  is $$[L_*D^3] = \sum_{i=1}^{23}([D^3]\cup e_i) \, e_i =
  \sum_{i=1}^{23} 3 q_F([D])\, q_F([D],e_i) \, e_i = 3q_F([D]) \,
  [D].$$ Thus $L_*D^3 = 3q_F([D])\, D$.

  \emph{(ii)} and \emph{(iii)} -- According to \eqref{eq rational
    equation}, together with the above, we have
  $$(L^2)_*D^3 = 2D^3 - \frac{2}{25}l\cdot L_*D^3 = 2D^3 - \frac{2\cdot
    3}{25}q_F([D]) \, l\cdot D.$$ Since by Theorem \ref{prop L2}
  $\CH^3(F)$ splits as $\Lambda_0^3 \oplus \Lambda_2^3$ under the
  action of $L^2$, with $\Lambda_2^3 = \CH^3(F)_{\hom}$, we get
  $2[D^3] - \frac{2\cdot 3}{25}q_F([D]) \, [l]\cdot [D] =0.$ Theorem
  \ref{thm voisin} yields $2D^3 - \frac{2\cdot 3}{25}q_F([D]) \,
  l\cdot D =0.$
\end{proof}

\begin{thm} \label{thm reformulation voisin} Let $P(D_i,l)$, $D_i
  \in \CH^1(F)$, be a polynomial.  Then
 $$P(D_i,l)\in\bigoplus_i \CH^i(F)_0.$$
\end{thm}
\begin{proof} By Theorem \ref{prop main}, we have to show that
  $P(D_i,l) \in \langle l^2 \rangle \oplus W_{\frac{25}{2}}^3 \oplus
  \langle l \rangle \oplus W_1^2 \oplus \CH^1(F) \oplus \langle l^0
  \rangle.$ Given Theorem \ref{thm voisin} and Lemma \ref{lem
    intersection of two divisors}, we only need to show that
  $P(D_i,l)\in W_{\frac{25}{2}}^3$ for all weighted homogeneous $P$ of
  degree $3$.  First we note that the intersection of three divisors
  can be written as a linear combination of cycles of the form $D^3$.
  Since $l\cdot D$ is proportional to $D^3$, we only need to show that
  $D^3\in W_{\frac{25}{2}}^3 = \Lambda_0^3$. Observe that by Lemma
  \ref{lem intersection of three divisors} we have
\[
l\cdot L_*D^3 = 3q_F([D]) \, l\cdot D = 25 \, D^3,
\]
so that the proof follows from \eqref{eq action pi6 CH1}.
\end{proof}

\begin{rmk} \label{rmk CH3 divisors} By Theorems \ref{prop L2} \&
  \ref{prop main} and Lemma \ref{lem intersection of three divisors},
  we actually have $\CH^3(F)_0 = (\CH^1(F))^{\cdot 3} = \langle l
  \rangle \cdot\CH^1(F)$.
 \end{rmk}

Let us end this section with the following lemma, which will be used
in Part 2 and Part 3.

\begin{lem}\label{lem cohomology to chow}
  Let $\mathrm{V}_F\subset \CH^*(F)$ be the sub-algebra generated by
  all divisors and Chern classes of $F$. Then any cohomological
  polynomial relation $P$ among elements of the sub-algebra of
  $\CH^*(F \times F)$ generated by $p_1^*\mathrm{V}_F$ and
  $p_2^*\mathrm{V}_F$ holds in the Chow ring $\CH^*(F \times F)$.
\end{lem}
\begin{proof}
  This is a consequence of Theorem \ref{thm voisin}. We assume that
  $P$ is a polynomial of bidegree $(e_1,e_2)$. We take a basis
\[
\{\fa_1,\fa_2,\ldots,\fa_m\}
\]
of $\mathrm{V}_F\cap\CH^{e_1}(F)$. Then we can write
\[
P = \sum_{i=1}^{m} p_1^*\fa_i \cdot p_2^*P_i,
\]
where $P_i$ are polynomials in $\mathrm{V}_F$ of degree $e_2$. There
exists a basis $\{\fb_1,\fb_2,\ldots \fb_m\}$ of $\mathrm{V}_F\cap
\CH^{4-e_1}(F)$ such that $\deg(\fa_i\cdot\fb_j)=\delta_{ij}$. If
$P=0$ in cohomology, then $P_i=P_*\fb_i$ is homologically trivial.  It
follows from Proposition \ref{thm voisin} that $P_i=0$.  Consequently,
we get $P=0$ at the Chow group level.
\end{proof}

\section{An application to symplectic automorphisms}

The results of this section will not be used in the rest of this
manuscript. The goal here is to show that the Fourier decomposition
reduces the question of understanding the action $f^*$ of a morphism
$f : F \rightarrow F$ on the Chow group of zero-cycles of a
hyperk\"ahler variety $F$ to the understanding of the action of $f^*$
on certain codimension-two cycles~; see Proposition \ref{prop
  symplectic auto general}.  We then use recent results of Voisin
\cite{voisin symp}, Huybrechts \cite{huybrechts symp} and Fu \cite{fu}
to exemplify this method to the cases where $F$ is the Hilbert scheme
of length-$2$ subschemes on a K3 surface or the variety of lines on a
cubic fourfold~; see Proposition \ref{prop symplectic S2} and Theorem
\ref{prop symplectic cubic}, respectively. Note that, in both cases,
we show that our cycle $L$ representing the Beauville--Bogomolov class
$\mathfrak{B}$ is preserved under the action of any automorphism $f  :
F \rightarrow F$.

\begin{prop}\label{prop symplectic auto general}
  Let $F$ be a hyperk\"ahler variety of dimension $2d$ for which the
  decomposition
$$
\CH^{2d}(F)=\bigoplus_{s=0}^{d}\CH^{2d}(F)_{2s},\qquad \CH^{2d}(F)_{2s} =
l^{d-s}\cdot (L_*\CH^{2d}(F))^{\cdot s}
$$
in Conjecture \ref{conj fourier} holds true. Let $f :F\rightarrow F$ be
a morphism such that $f^*\sigma = a\sigma$ for all $\sigma\in
L_*\CH^{2d}(F)$ and $f^*l=bl$, for some $a,b\in\Q$. Then $f^*$ acts as
multiplication by $a^{s}b^{d-s}$ on $\CH^{2d}(F)_{2s}$. In particular,
if $a=b=1$ then $f^*=\mathrm{Id}$ on $\CH^{2d}(F)$.
\end{prop}
\begin{proof}
  By assumption, $\CH^{2d}(F)_{2s}$ is generated by elements of the
  form $l^{d-s}\sigma_1\cdots\sigma_s$ where $\sigma_i\in
  L_*\CH^{2d}(F)$.  The compatibility of $f^*$ and the intersection
  product shows that $f^*$ acts as multiplication by $a^{s}b^{d-s}$ on
  such elements.
\end{proof}

Note that by Remark \ref{rmk assumptions satisfied}, Theorem \ref{thm
  main splitting} and Proposition \ref{prop l CH2 6}, if $F$ is the
Hilbert scheme of length-$2$ subschemes on a K3 surface or the variety
of lines on a cubic fourfold, then $\CH_0(F)$ has a decomposition as
in Conjecture \ref{conj fourier}. \medskip

An automorphism $f :F\rightarrow F$ of a hyperk\"ahler variety $F$ is
said to be \textit{symplectic} if $f^*=\mathrm{Id}$ on
$\HH^0(F,\Omega_F^2)$.  The conjectures of Bloch--Beilinson predict
that a symplectic automorphism acts trivially on $\CH_0(F)$~; see
\cite{Jannsen2}. This was proved to hold for symplectic involutions on
K3 surfaces by Voisin \cite{voisin symp} and generalized to symplectic
automorphisms of finite order on K3 surfaces by Huybrechts
\cite{huybrechts symp}.\medskip

Combined with Proposition \ref{prop symplectic auto general}, these
results give the following proposition in the case where $F$ is the
Hilbert scheme of length-$2$ subschemes on a K3 surface $S$.

\begin{prop}\label{prop symplectic S2}
  Let $f :S\rightarrow S$ be a symplectic automorphism of finite order
  of $S$ and $\hat{f} :F \rightarrow F$ the induced symplectic
  automorphism of $F=S^{[2]}$, then
$$
\hat{f}^* = \mathrm{Id} : \CH_0(F) \rightarrow \CH_0(F).
$$
\end{prop}
\begin{proof}
  The notations are those of Part 2. Let $x\in S$ be a point and let
  $S_x$ be the surface parameterizing all length-$2$ subschemes
  $W\subset S$ such that $x\in W$. Since $f$ is an automorphism, one
  sees that $x\in W$ if and only if $f^{-1}x\in f^{-1}W$.  Hence at
  the level of Chow groups we have $\hat{f}^*S_x = S_{f^*x}$. As a
  special case, one has $\hat{f}^* S_{\mathfrak{o}}=
  S_{\mathfrak{o}}$, since $\mathfrak{o}$ is the class of any point on
  a rational curve of $S$ and $f$ maps a rational curve to a rational
  curve. Similarly, a length-$2$ subscheme $W\subset S$ is non-reduced
  if and only if $f^{-1}W$ is non-reduced.  This implies
  $\hat{f}^*\delta=\delta$. Since $f$ preserves the incidence of
  subschemes of $S$, we have $(\hat{f}\times \hat{f})^*I=I$. It
  follows that $(\hat{f}\times\hat{f})^*L=L$, where $L$ is as in
  \eqref{eq L S2}. Thus we get
\[
\hat{f}^*l = \hat{f}^* (\iota_{\Delta})^*L =
(\iota_{\Delta})^*(\hat{f}\times \hat{f})^*L =(\iota_{\Delta})^*L =l.
\]
Now since $L_*\CH_0(F)$ is generated by cycles of the form $S_x
-S_{y}$, we see that the condition $\hat{f}^*=\mathrm{Id}$ on
$L_*\CH_0(F)$ is implied by $f^*=\mathrm{Id}$ on
$\CH_0(S)_\mathrm{hom}$, which was established by Voisin \cite{voisin
  symp} and Huybrechts \cite{huybrechts symp}. We can then invoke
Proposition \ref{prop symplectic auto general} with $a=b=1$ to
conclude.
\end{proof}

In the case where $F$ is the variety of lines on a cubic fourfold, let
us mention the following recent result of L.~Fu \cite{fu}.

\begin{thm}[Fu \cite{fu}]\label{prop symplectic cubic}
  Let $f :X\rightarrow X$ be an automorphism of the cubic fourfold $X$
  such that $f^*=\mathrm{Id}$ on $\HH^{3,1}(X)\rightarrow
  \HH^{3,1}(X)$ and $\hat{f} :F\rightarrow F$ the induced symplectic
  automorphism of $F$, then
$$
\hat{f}^*=\mathrm{Id} :\CH_0(F) \rightarrow \CH_0(F).
$$
\end{thm}
\begin{proof} The notations are those of Part 3. We reproduce, in our
  setting, the original argument of Fu.  First we note that the
  polarization $g$ on $F$ is represented by the divisor of all lines
  meeting $H_1\cap H_2$, where $H_1$ and $H_2$ are two general
  hyperplane sections on $X$. Since $f$ maps a line to a line and maps
  a hyperplane section to a hyperplane section, we have
  $\hat{f}^*g=g$. Note that the class $c$ is represented by all lines
  contained in some hyperplane section, and hence we get
  $\hat{f}^*c=c$. For a point $t\in F$, with associated line
  $l_t\subset X$, we define $S_{l_t}$ to be the surface of all lines
  meeting the line $l_t$. Since $f$ is an automorphism of $X$ that
  preserves the degrees of subvarieties, we see that $l_t$ meets
  $l_{t'}$ if and only if $f^{-1}l_t$ meets $f^{-1}l_{t'}$. From this
  we deduce $(\hat{f}\times \hat{f})^*I=I$ and $\hat{f}^*S_{l_{t}} =
  S_{f^*l_{t}}$, for all $t\in F$. It follows that
  $(\hat{f}\times\hat{f})^*L=L$, where $L$ is as in \eqref{eq L
    Cubic}. Thus we get
\[
\hat{f}^*l = \hat{f}^* (\iota_{\Delta})^*L =
(\iota_{\Delta})^*(\hat{f}\times \hat{f})^*L =(\iota_{\Delta})^*L =l.
\]
Now since $L_*\CH_0(F)$ is generated by cycles of the form
$S_{l_{t_1}} -S_{l_{t_2}}$, we see that the condition
$\hat{f}^*=\mathrm{Id}$ on $L_*\CH^4(F)$ is reduced to
\[
 l_{t_1} - l_{t_2} = f^*l_{t_1} - f^* l_{t_2},\qquad \text{in } \CH^3(X).
\]
The above equality is implied by $f^*=\mathrm{Id}$ on
$\CH^3(X)_\mathrm{hom}$, which was established by Fu \cite{fu}. Hence
we can conclude by Proposition \ref{prop symplectic auto general} with
$a=b=1$.
\end{proof}

\vspace{10pt}
\section{On the birational invariance of the Fourier
  decomposition} \label{sec birational} Ulrike Greiner \cite{greiner}
has recently proved the following theorem.
\begin{thm}[Greiner \cite{greiner}] \label{thm greiner} Let $F$ and
  $F'$ be birational hyperk\"ahler varieties of
  $\mathrm{K3}^{[2]}$-type. Then there exists a correspondence $\gamma
  \in \CH^4(F \times F')$ such that $ \gamma_*  : \CH^*(F) \rightarrow
  \CH^*(F')$ is an isomorphism of graded rings.
\end{thm}

The main result of this section is Theorem \ref{thm birational
  invariance}, which roughly states that if $F$ and $F'$ are
birational hyperk\"ahler varieties of $\mathrm{K3}^{[2]}$-type and if
the Chow ring of $F$ has a Fourier decomposition with respect to a
cycle $L \in \CH^2(F \times F)$ representing the Beauville--Bogomolov
class $\mathfrak{B}$, then the Chow ring of $F'$ also has a Fourier
decomposition with respect to a cycle $L' \in \CH^2(F' \times F')$
representing the Beauville--Bogomolov class $\mathfrak{B}' \in
\HH^4(F'\times F',\Q)$. Thus Theorem \ref{thm main mult} also holds
for any hyperk\"ahler variety birational to the Hilbert scheme of
length-$2$ subschemes on a K3 surface, or to the variety of lines on a
very general cubic fourfold. Theorem \ref{thm birational invariance}
is a generalization of Theorem \ref{thm greiner} and the method of
proof follows closely that of Greiner.\medskip

The following result is due to Huybrechts \cite[Theorem 3.4 \& Theorem
4.7]{huybrechts} in the non-algebraic setting. As noted in
\cite{greiner}, the proof adapts to the algebraic setting at the
condition that the two varieties $F$ and $F'$ are connected by Mukai
flops, which is the case for hyperk\"ahler fourfolds of
$\mathrm{K3}^{[2]}$-type by \cite{wierzba}.

\begin{thm}[Huybrechts \cite{huybrechts}] \label{prop huybrechts} Let
  $F$ and $F'$ be birational hyperk\"ahler fourfolds of
  $\mathrm{K3}^{[2]}$-type. Then there exist algebraic varieties
  $\mathscr{F}$ and $\mathscr{F}'$ smooth and projective over a smooth
  quasi-projective curve $T$, and a closed point $0 \in T$ such
  that \begin{enumerate}[(i)]
\item $\mathscr{F}_0 = F$ and $\mathscr{F}'_0 = F'$~;
\item there is an isomorphism $\Psi  : \mathscr{F}_{T\backslash \{0\}}
  \stackrel{\simeq}{\longrightarrow} \mathscr{F}'_{T \backslash
    \{0\}}$ over $T\backslash\{0\}$.  \qed
\end{enumerate}
\end{thm}

In the rest of this section, we are in the situation of Theorem
\ref{prop huybrechts}.  Let $\eta \rightarrow T$ be the generic point
of $T$ and let $\gamma_\eta \in \CH^4(\mathscr{F}_\eta \times_\eta
\mathscr{F}'_\eta)$ be the class of the restriction to the generic
fiber of the graph of $\Psi$. The isomorphism $\Psi$ restricts to an
isomorphism $\psi  : \mathscr{F}_\eta \rightarrow \mathscr{F}'_\eta$,
and $\gamma_\eta$ is nothing but the class of the graph of $\psi$. In
particular, we have $\gamma_\eta \circ {}^t\gamma_\eta =
\Delta_{\mathscr{F}_\eta'}$ and ${}^t\gamma_\eta \circ \gamma_\eta =
\Delta_{\mathscr{F}_\eta}$, where ${}^t\gamma_\eta$ denotes the
transpose of $\gamma_\eta$. Let us write $s_0$ for Fulton's
specialization map on Chow groups of varieties which are smooth over
$T$~; see \cite[\S20.3]{fulton}. If $\mathscr{X} \rightarrow T$ is a
smooth morphism, then the specialization map $s_0  :
\CH^*(\mathscr{X}_\eta) \rightarrow \CH^*(\mathscr{X}_0)$ sends a
cycle $\alpha$ on the generic fiber $\mathscr{X}_\eta$ of $\mathscr{X}
\rightarrow T$ to the restriction to the special fiber $\mathscr{X}_0$
of a closure of $\alpha$ to $\CH^*(\mathscr{X})$. This map commutes
with intersection product, flat pull-backs and proper push-forwards.
In particular, if we are dealing with varieties which are smooth and
projective over $T$, then $s_0$ commutes with composition of
correspondences. We set
$$\gamma :=s_0(\gamma_\eta) \in \CH^4(F \times F').$$ U.~Greiner showed
\cite[Theorem 2.6]{greiner} that $\gamma_*  : \CH^*(F) \rightarrow
\CH^*(F')$ is an isomorphism of graded rings with inverse given by
$\gamma^*$. Let us, for clarity, reproduce Greiner's
proof. Specializing the identities $\gamma_\eta \circ {}^t\gamma_\eta
= \Delta_{\mathscr{F}_\eta'}$ and ${}^t\gamma_\eta \circ \gamma_\eta =
\Delta_{\mathscr{F}_\eta}$ immediately gives $\gamma \circ {}^t\gamma
= \Delta_{{F}'}$ and ${}^t\gamma \circ \gamma = \Delta_{{F}}$, so that
$\gamma_*  : \CH^*(F) \rightarrow \CH^*(F')$ is a group
isomorphism. Given a smooth projective variety $X$, the \emph{small
  diagonal} $\Delta^X_{123} \in \CH^*(X \times X \times X)$ is the
class of the graph of the diagonal embedding $\iota_\Delta  : X
\rightarrow X \times X$ that we see as a correspondence from $X \times
X$ to $X$.  Thus, given cycles $\alpha, \beta \in \CH^*(X)$, we
have $$\alpha \cdot \beta = \iota_\Delta^*(\alpha \times \beta) =
\Delta^X_{123,*}(\alpha \times \beta).$$ In order to prove that
$\gamma_*$ is a ring homomorphism, it is therefore enough to check
that $\gamma \circ \Delta^F_{123} = \Delta^{F'}_{123} \circ (\gamma
\times \gamma)$, or equivalently that $\gamma \circ \Delta^F_{123}
\circ ({}^t\gamma \times {}^t\gamma) = \Delta^{F'}_{123}$. By
specialization, it is enough to prove that $\gamma_\eta \circ
\Delta^{\mathscr{F}_\eta}_{123} \circ ({}^t\gamma_\eta \times
{}^t\gamma_\eta) = \Delta^{\mathscr{F}_\eta'}_{123}$. But then, by
\cite[Lemma 3.3]{vial2}, we have \begin{equation} \label{eq formula
    liebermann} (\beta \times \alpha)_*Z = \alpha \circ Z \circ
  {}^t\beta \ \in \CH^*(X' \times Y'),\end{equation} for all $X$,
$X'$, $Y$ and $Y'$ smooth projective varieties and all correspondences
$Z \in \CH^*(X \times Y)$, $\alpha \in \CH^*(Y \times Y')$ and $\beta
\in \CH^*(X \times X')$. Hence, by \eqref{eq formula liebermann}, we
have $\gamma_\eta \circ \Delta^{\mathscr{F}_\eta}_{123} \circ
({}^t\gamma_\eta \times {}^t\gamma_\eta) = (\gamma_\eta \times
\gamma_\eta \times \gamma_\eta)_*\Delta^{\mathscr{F}_\eta}_{123}$,
which clearly is equal to $\Delta^{\mathscr{F}'_\eta}_{123}$.\medskip

Note that replacing $\Psi$ by $\Psi \times \ldots \times \Psi$, the
above also shows that $(\gamma \times \ldots \times \gamma)_*  :
\CH^*(F\times \ldots \times F) \rightarrow \CH^*(F'\times \ldots
\times F') $ is an isomorphism of graded rings with inverse $(\gamma
\times \ldots \times \gamma)^*$. The compatibility of the Fourier
decomposition with the intersection product can be expressed purely in
terms of the intersection theoretical properties of $L$~:

\begin{lem} \label{lem birational formulas chow}
  Let $L \in \CH^*(F \times F)$ be a correspondence and let $L'  :=
  (\gamma \times \gamma)_*L$. Then, for all integers $p,q,r \geq 0$,
  we have
\begin{align*}
  \gamma \circ L^r \circ \Delta_{123}^F \circ (p_{13}^*L^p \cdot
  p_{24}^*L^q) &= (L')^r \circ \Delta_{123}^{F'}\circ (p_{13}'^*(L')^p
  \cdot p_{24}'^*(L')^q) \circ (\gamma \times \gamma), \\
  \gamma \circ L^r \circ (p_i^*\iota_\Delta^*L^p \cdot L^q) &= (L')^r
  \circ (p_i'^*\iota_\Delta^*(L')^p \cdot (L')^q) \circ \gamma, \\
  \gamma \circ L^r \circ (p_i^*\iota_\Delta^*L^p \cdot
  p_j^*\iota_\Delta^*L^q) &= (L')^r \circ (p_i'^*\iota_\Delta^*(L')^p
  \cdot p_j'^* \iota_\Delta^*(L')^q) \circ \gamma,
\end{align*}
where $p_i  : F \times F \rightarrow F$ is the projection on the
$i^\mathrm{th}$ factor, $p_{ij}  : F \times F \times F \times F
\rightarrow F \times F$ is the projection on the product of the
$i^\mathrm{th}$ and $j^{\mathrm{th}}$ factors, and similarly for $F'$.
\end{lem}
\begin{proof}
  If $Z \in \CH^*(X \times Y)$ and $\alpha \in \CH^*(X)$, then we have
  the following formula $$Z_*\alpha = Z \circ \alpha \in \CH^*(Y),$$
  where on the right-hand side of the equality $\alpha$ is seen as a
  correspondence from a point to $X$. For instance, we have
  $p_i^*\iota_\Delta^*L^p = {}^t\Gamma_{p_i} \circ \Delta^F_{123}
  \circ L^p$, where on the right-hand side $L^p$ is seen as a
  correspondence from a point to $F \times F$. Thus the main point
  consists in checking that $\gamma$ commutes with $L$ and its powers,
  with projections and with diagonals. Precisely, writing $\Gamma_f
  \in \CH_{\dim X}(X\times Y)$ for the class of the graph of a
  morphism $f : X \rightarrow Y$, we claim
  that \begin{equation*} \label{eq claim commute} \gamma \circ L^r =
    L'^r \circ \gamma, \quad \gamma \circ \Delta^{{F}}_{123} =
    \Delta^{{F}'}_{123} \circ (\gamma \times \gamma) \quad \mbox{and}
    \quad \gamma \circ \Gamma_{p_i} = \Gamma_{p'_i} \circ (\gamma
    \times \gamma).
\end{equation*}

The following identities have already been established~: $\gamma \circ
{}^t\gamma = \Delta_{{F}'}$, ${}^t\gamma \circ \gamma = \Delta_{{F}}$ and
$\gamma \circ \Delta^{{F}}_{123} \circ ({}^t\gamma \times {}^t\gamma) =
\Delta^{{F}'}_{123}$. Recall that $(\gamma \times \ldots \times
\gamma)_*$ is multiplicative. Thus $(\gamma \times \gamma)_*(L^r) =
((\gamma \times \gamma)_*L)^r = (L')^r$. By \eqref{eq formula
  liebermann}, we get $\gamma \circ L^r \circ {}^t\gamma = L'^r$ and
hence, by composing with $\gamma$ on the right, $\gamma \circ L^r =
L'^r \circ \gamma$. Therefore it remains only to check that for
$p_{i_1,\ldots,i_s}  : F^n \rightarrow F^s$ the projection on the
$(i_1,\ldots, i_s)$-factor, we have $\gamma^{\times s} \circ
\Gamma_{p_{i_1,\ldots,i_s}} = \Gamma_{p'_{i_1,\ldots,i_s}} \circ
\gamma^{\times n}$, where $\Gamma_{p_{i_1,\ldots,i_s}}$ is the class
of the graph of ${p_{i_1,\ldots,i_s}}$ and where $p'_{i_1,\ldots,i_s}
 : F'^n \rightarrow F'^s$ is the projection on the $(i_1,\ldots,
i_s)$-factor. But then, this follows immediately from $\gamma^{\times
  n} \circ ({}^t\gamma)^{\times n} = \Delta_{F^n}$ and from the identity
$\gamma^{\times s} \circ \Gamma_{p_{i_1,\ldots,i_s}} \circ
({}^t\gamma)^{\times n} = (\gamma^{\times (n+s)})_*
\Gamma_{p_{i_1,\ldots,i_s}} = \Gamma_{p'_{i_1,\ldots,i_s}}$.
 \end{proof}

\begin{lem} \label{lem equivalence mult} Let $F$ be a hyperk\"ahler
  fourfold of $\mathrm{K3}^{[2]}$-type. Assume that there is a cycle
  $L \in \CH^2(F\times F)$ with cohomology class $\mathfrak{B}$ that
  satisfies \eqref{eq rational equation}, \eqref{assumption pre l},
  \eqref{assumption hom} and \eqref{assumption 2hom}.  Let $(i,s)$ and
  $(j,r)$ be pairs of non-negative integers and set $p :=i-s$ and
  $q :=j-r$. The inclusion $\CH^i(F)_{2s}\cdot \CH^j(F)_{2r} \subseteq
  \CH^{i+j}(F)_{2r+2s}$ is equivalent to having
  \begin{equation} \label{eq equivalent mult}
    (L^{t})_*\big((L^{p})_*\sigma \cdot (L^{q})_*\tau \big) = 0, \
    \mbox{for all} \ t \neq 4-p-q, \ \mbox{all} \ \sigma
    \in \CH^{4-i+2s}(F) \ \mbox{and all} \ \tau \in \CH^{4-j+2r}(F).
  \end{equation}
\end{lem}
\begin{proof}
  It follows immediately from Theorem \ref{prop main} that
  \begin{equation} \label{eq image} \CH^i(F)_{2s} =
    (L^{i-s})_*\CH^{4-i+2s}(F) \ \mbox{for all } i,s
  \end{equation} and that  \begin{equation} \label{eq kernel}
    \CH^i(F)_{2s} = \bigcap_{s'\neq 4-i+s} \ker
    \{(L^{s'})_*  : \CH^{i}(F) \rightarrow \CH^{i-4+2s'}(F)\} \
    \mbox{for all } i,s.
    \end{equation}
    This proves the lemma.
\end{proof}

\begin{thm}\label{thm birational invariance}
  Let $F$ and $F'$ be two birational hyperk\"ahler fourfolds of
  $\mathrm{K3}^{[2]}$-type. Assume that there is a cycle $L \in
  \CH^2(F\times F)$ representing the Beauville--Bogomolov class
  $\mathfrak{B}$
   that satisfies \eqref{eq rational equation}, \eqref{assumption pre
    l}, \eqref{assumption hom}, or \eqref{assumption 2hom}. Then the
  cycle $L'  := (\gamma \times \gamma)_*L \in \CH^2(F'\times F')$
  represents the Beauville--Bogomolov class $\mathfrak{B}'$ on $F'$,
  and satisfies \eqref{eq rational equation}, \eqref{assumption pre
    l}, \eqref{assumption hom}, or \eqref{assumption 2hom},
  respectively.

  Assume that $L$ satisfies all four conditions \eqref{eq rational
    equation}, \eqref{assumption pre l}, \eqref{assumption hom} and
  \eqref{assumption 2hom} and that the associated Fourier
  decomposition of Theorem \ref{prop main} satisfies
  $\CH^i(F)_{2s}\cdot \CH^j(F)_{2r} \subseteq \CH^{i+j}(F)_{2r+2s}$
  for some $(i,s)$ and $(j,r)$. Then the Fourier decomposition
  associated to $L'$ satisfies $\CH^i(F')_{2s}\cdot \CH^j(F')_{2r}
  \subseteq \CH^{i+j}(F')_{2r+2s}$
\end{thm}
\begin{proof} The correspondence $\gamma \in \CH^4(F\times F')$
  respects the Beauville--Bogomolov form by Huybrechts \cite[Corollary
  2.7]{huybrechts2}, so that $[L'] = (\gamma \times
  \gamma)_*\mathfrak{B} = \mathfrak{B}'$. If $L$ satisfies \eqref{eq
    rational equation}, then Lemma \ref{lem birational formulas chow}
  immediately shows that $L'=\gamma \circ L \circ {}^t\gamma$ also
  satisfies \eqref{eq rational equation}.  As a consequence of Lemma
  \ref{lem birational formulas chow}, we have for all $\sigma', \tau'
  \in \CH^*(F)$ and all non-negative integers
  $p,q,t$, \begin{equation}\label{eq noname}
    (L'^{t})_*\big((L'^{p})_*\sigma' \cdot (L'^{q})_*\tau' \big) =
    \gamma_* (L^{t})_*\big((L^p)_*\sigma \cdot (L^q)_*\tau \big),
  \end{equation} where we have set $\sigma  := \gamma^*\sigma'$ and
  $\tau  := \gamma^*\tau'$. Thus it is clear that if $L$ satisfies one
  of  \eqref{assumption pre l}, \eqref{assumption hom} or
  \eqref{assumption 2hom}, then so does $L'$.

  Assume now that $L$ satisfies all four conditions \eqref{eq rational
    equation}, \eqref{assumption pre l}, \eqref{assumption hom} and
  \eqref{assumption 2hom}. Then \eqref{eq noname}, together with Lemma
  \ref{lem equivalence mult}, yields that $\CH^i(F)_{2s}\cdot
  \CH^j(F)_{2r} \subseteq \CH^{i+j}(F)_{2r+2s}$ if and only if
  $\CH^i(F')_{2s}\cdot \CH^j(F')_{2r} \subseteq
  \CH^{i+j}(F')_{2r+2s}$.
 \end{proof}

 \begin{rmk} Let $F$ and $F'$ be birational hyperk\"ahler fourfolds of
   $\mathrm{K3}^{[2]}$-type. Assume that $F$ satisfies Conjecture
   \ref{conj sheaf}. Then the same arguments as the ones above
   combined with the fact \cite[Lemma 3.4]{greiner} that
   $\gamma_*c_2(F) = c_2(F')$ show that $F'$ also satisfies Conjecture
   \ref{conj sheaf} with $L' := (\gamma \times \gamma)_*L$.
\end{rmk}

\vspace{10pt}
\section{An alternate approach to the Fourier decomposition on the
  Chow ring of abelian varieties} \label{sec abelian}

We wish to give new insight in the theory of algebraic cycles on
abelian varieties by explaining how Beauville's decomposition theorem
\cite{beauville1} on the Chow ring of abelian varieties can be derived
directly from a recent powerful theorem of Peter O'Sullivan
\cite{o'sullivan}. This approach is in the same spirit as the proof of
Theorem \ref{thm2 conj}, which will be given at the end of Section
\ref{sec multCK}. It is also somewhat closer to our approach for
proving the Fourier decomposition for certain hyperk\"ahler fourfolds
of $\mathrm{K3}^{[2]}$-type.  At the end of the section, we further
explain how O'Sullivan's theorem easily implies finite-dimensionality
\cite{kimura, o'sullivan2} for the Chow motive of an abelian variety.

We insist that the sole purpose of this section is to
illustrate the importance of O'Sullivan's theorem, and to a lesser
extent to make manifest how powerful a similar result for
hyperk\"ahler varieties (\emph{cf.} Conjecture \ref{conj sheaf}) would
be. In particular, this section is logically ill for two
reasons. First, in order to obtain his final result, O'Sullivan
establishes Beauville's decomposition theorem from his own machinery
(independent of Beauville's) as an intermediate result \cite[Theorem
5.2.5]{o'sullivan}. Second,
O'Sullivan uses in a crucial way that the Chow motive of an abelian
variety is finite-dimensional \cite{kimura}.  \medskip

Let $A$ be an abelian variety over a field $k$. O'Sullivan introduces
the following definition. Consider a cycle $\alpha \in \CH^*(A)$. For
each integer $m \geq 0$, denote by $V_m(\alpha)$ the $\Q$-vector
sub-space of $\CH^*(A^m)$ generated by elements of the form
$p_*(p_1^*\alpha^{r_1} \cdot p_2^*\alpha^{r_2} \cdot \ldots \cdot
p_n^*\alpha^{r_n})$, where $n \leq m$, the $r_j$ are integers $\geq
0$, $p_i : A^n \rightarrow A$ are the projections on the
$i^\mathrm{th}$ factor, and $p : A^n \rightarrow A^m$ is a closed
immersion with each component $A^n \rightarrow A$ either a projection
or the composite of a projection with $[-1] : A \rightarrow A$. Then
$\alpha$ is called \emph{symmetrically distinguished} if for every $m$
the restriction of the quotient map $\CH^*(A^m) \rightarrow
\CH^*(A^m)/\sim_{\mathrm{num}}$, where $\sim_{\mathrm{num}}$ is
numerical equivalence of cycles, to $V_m(\alpha)$ is
injective.\medskip

 Here is the main result of \cite{o'sullivan}.

 \begin{thm} [O'Sullivan] \label{thm O'Sullivan} The symmetrically
   distinguished cycles in $\CH^*(A)$ form a $\Q$-sub-algebra which
   maps isomorphically to $\CH^*(A)/{\sim_{\mathrm{num}}}$ under the
   projection $\CH^*(A^m) \rightarrow \CH^*(A^m)/\sim_{\mathrm{num}}$.
   Moreover, symmetrically distinguished cycles are stable under
   push-forward and pull-back of homomorphisms of abelian
   varieties.\qed
\end{thm}

In particular, $[A] \in \CH^0(A)$ is symmetrically distinguished and
consequently the graphs of homomorphisms $f  : A \rightarrow A'$ of
abelian varieties are symmetrically distinguished. Indeed, we may
write $\Gamma_f = (\mathrm{id}_A \times f)_*[A] \in \CH^*(A \times
A')$. According to Theorem \ref{thm O'Sullivan}, another class of
symmetrically distinguished cycles is given by symmetric divisors on
$A$, \emph{i.e.}, divisors $D \in \CH^1(A)$ such that $[-1]^*D = D$. A
direct consequence of Theorem \ref{thm O'Sullivan} is thus the
following analogue of Theorem \ref{thm voisin} in the case of abelian
varieties.

 \begin{cor} \label{cor o'sullivan}
   The sub-algebra of $\CH^*(A)$ generated by symmetric divisors injects
   into cohomology via the cycle class map.\qed
 \end{cor}

 \noindent Note that Corollary \ref{cor o'sullivan} has been recently
 obtained by different methods by Ancona \cite{ancona} and Moonen
 \cite{moonen}. \medskip

 Denote $\hat{A}$ the dual abelian variety of $A$. Let $L$ be the
 Poincar\'e bundle on $A \times \hat{A}$ seen as an element of
 $\CH^1(A\times \hat{A})$. We write $\hat{L}$ for the transpose of
 $L$~; $\hat{L}$ is the Poincar\'e bundle on $\hat{A} \times A$.

 \begin{prop} \label{prop hom poincare bundle} Write $\pi^i_{\hom} \in
   \HH^{2d}(A \times A,\Q)$ for the K\"unneth projector on
   $\HH^i(A,\Q)$. The Poincar\'e bundle $L$ satisfies the following
   cohomological relations~:
\[
[\hat{L}]^j \circ [L]^{2d-i} = \left\{ \begin{array}{ll}
    0 & \mbox{if $j \neq i$}  ;\\
    (-1)^i\, i!\, (2d-i)!\, \pi^i_{\hom} & \mbox{if $j =
      i$}.
  \end{array} \right.
\]
\end{prop}
\begin{proof} The following proof is inspired from \cite[Proposition
  1]{beauvillefourier}. Let us write $H := \HH^1(A,\Q)$ and recall that
  $\HH^k(A,\Q) = \bigwedge^k \HH^1(A,\Q) = \bigwedge^kH$.  Let
  $e_1,\ldots, e_{2d}$ be a basis of $H$ and let $e_1^\vee,\ldots,
  e_{2d}^\vee$ be the dual basis.
  We then have $$[L] = \sum_{k=1}^{2d} e_k \otimes e_k^\vee \in
  \HH^2(A\times \hat{A},\Q).$$ As in the proof of \cite[Proposition
  1]{beauvillefourier}, we write, for $K = \{k_1, \ldots, k_p\}$ a
  sub-set of $\{1, \ldots, 2d\}$ with $k_1<\ldots < k_p$, $e_K  :=
  e_{k_1} \wedge \ldots \wedge e_{k_p}$ and $e_K^\vee  := e^\vee_{k_1}
  \wedge \ldots \wedge e^\vee_{k_p}$. A simple calculation yields
  $$[L]^p =p! \, \sum e_K \otimes e_K^\vee,$$ where the sum runs
  through all sub-sets $K \subseteq \{1,\ldots, 2d\}$ of cardinality
  $|K| = p$.  If $J$ is another sub-set of $ \{1,\ldots, 2d\}$, then
  $(e_K \otimes e^\vee_K)_*e_J = 0$ unless $J = K^c := \{1,\ldots,
  2d\}\backslash K$ in which case $(e_K \otimes e^\vee_K)_*e_{K^c} =
  \varepsilon(K)e_K^\vee$, where $\varepsilon(K) = \pm 1$ is such that
  $e_K\wedge e_{K^c} = \varepsilon(K)\, e_1 \wedge \ldots \wedge
  e_{2d}$. Noting that $\varepsilon(K^c)\varepsilon(K) = (-1)^{|K|}$,
  we get for $I$, $J$ and $K$ sub-sets of $ \{1,\ldots, 2d\}$
  \[
  (e^\vee_I\otimes e_I)_*(e_K\otimes e_K^\vee)_* e_J = \left\{
    \begin{array}{ll}
      (-1)^{|K|}\, e_{K^c} & \mbox{if $I=J=K^c$}  ;\\
      0 & \mbox{otherwise.}
    \end{array} \right.
  \]
  We therefore see that, if $i \neq j$, then $([\hat{L}]^j \circ
  [L]^{2d-i})_*$ acts trivially on $e_K^\vee$ for all sub-sets $K
  \subseteq \{1,\ldots, 2d\}$, \emph{i.e.}, that $[\hat{L}]^j \circ
  [L]^{2d-i} = 0$ for $i \neq j$. On the other hand, if $i=j$ and if
  $J$ is a sub-set of $\{1,\ldots, 2d\}$, then $([\hat{L}]^i \circ
  [L]^{2d-i})_*e_J = 0$ unless if $|J|=i$ in which case $([\hat{L}]^i
  \circ [L]^{2d-i})_*e_J = i!(2d-i)!\, (e^\vee_J\otimes
  e_J)_*(e_{J^c}\otimes e_{J^c}^\vee)_* e_J = (-1)^i i!(2d-i)!\, e_J$.
  Thus $[\hat{L}]^j \circ [L]^{2d-i}$ acts by multiplication by
  $(-1)^i i!(2d-i)!$ on $\HH^i(A,\Q) = \bigwedge^i H$, and as zero on
  $\HH^j(A,\Q) = \bigwedge^j H$ if $j\neq i$.
\end{proof}

By the theorem of the square, the Poincar\'e bundle $L$ is symmetric.
Theorem \ref{thm O'Sullivan} then implies that the relations of
Proposition \ref{prop hom poincare bundle} hold modulo rational
equivalence. More precisely, $\hat{L}^j \circ L^{2d-i} = 0 $ if $j
\neq i$ and if we define
\begin{equation}
  \label{eq CKabelian} \pi^i  := \frac{(-1)^i}{i!\, (2d-i)!} \,
  \hat{L}^i \circ L^{2d-i} \quad \mbox{for all } 0 \leq i \leq 2d,
\end{equation}
then $\{\pi^i  : 0 \leq i \leq 2d\}$ is a Chow--K\"unneth decomposition
of the diagonal $\Delta_A \in \CH^d(A \times A)$ that consists of
symmetrically distinguished cycles. Denote $\FF  := e^L$ the Fourier
transform, it follows that

\begin{equation*}
  \CH^i(A) = \bigoplus_{s=i-d}^{i} \CH^i(A)_s, \ \mbox{with }
  \CH^i(A)_s  := (\pi^{2i-s})^*\CH^i(A) = \{\sigma \in \CH^i(A)  :
  \FF(\sigma) \in \CH^{d-i+s}(\hat{A})\}.
  \end{equation*}
  Let $\Gamma_{[n]} \in \CH^d(A \times A)$ be the class of the graph
  of the multiplication-by-$n$ map $[n]  : A \rightarrow A$. Since
  $[n]^*$ acts by multiplication by $n^j$ on $\HH^{j}(A,\Q)$ and since
  $\Gamma_{[n]}$ is symmetrically distinguished, Theorem \ref{thm
    O'Sullivan} yields
$$  {}^t\Gamma_{[n]} \circ \pi^{j} = n^j \pi^{j},$$ and therefore
$$\CH^i(A)_s = \{\sigma \in \CH^i(A)  : [n]^*\sigma = n^{2i-s}\,
\sigma\}.$$
As a straightforward consequence, we get the multiplicative property
of the Fourier decomposition~: $\CH^i(A)_s \cdot \CH^j(A)_r \subseteq
\CH^{i+j}(A)_{r+s}$. Let us however give a proof that does not use the
multiplication-by-$n$ map. This will be used in the next section to
establish that the Chow--K\"unneth decomposition $\{\pi^i  : 0 \leq i
\leq 2d\}$ is in fact \emph{multiplicative}~; see Example \ref{ex
  abelian}. As in the proof of Lemma \ref{lem equivalence mult}, the
Fourier decomposition may be characterized by $$\CH^i(A)_s =
(\hat{L}^{2i-s})_*\CH^{d-i+s}(\hat{A}) = \bigcap_{s' \neq 2d-2i+s} \,
\ker \{(L^{s'})_*  : \CH^i(A) \rightarrow \CH^{i-d+s'}(\hat{A})\}.$$
Therefore, $\CH^i(A)_s \cdot \CH^j(A)_r \subseteq \CH^{i+j}(A)_{r+s}$
holds for all non-negative integers $i,j,r$ and $s$ if and only if
$(L^{t})_*\big((\hat{L}^p)_*\sigma \cdot (\hat{L}^q)_*\tau \big) = 0$
for all $t + p + q \neq 2d$ and all $ \sigma, \tau \in
\CH^*(\hat{A})$. By Theorem \ref{thm O'Sullivan}, it is enough to
check the following cohomological formula~:
\begin{equation} \label{eq ab coho} [L]^t \circ [\Delta_{123}^A] \circ
  (p_{13}^*[\hat{L}]^p \cdot p_{24}^*[\hat{L}]^q) = 0 \ \mbox{for all}
  \ t+p+q \neq 2d,
\end{equation} 
where $[\Delta_{123}^A]$ is the class of the graph of the diagonal
embedding $\iota_\Delta : A \rightarrow A \times A$ seen as a
correspondence from $A \times A$ to $A$, and where $p_{ij} : A \times
A \times A \times A \rightarrow A \times A$ is the projection on the
$i^\mathrm{th}$ and $j^\mathrm{th}$ factor. But then this is
straightforward as $[L]^p_* : \HH^*(A,\Q) \rightarrow
\HH^*(\hat{A},\Q)$ acts as zero on
$\HH^q({A},\Q)$ for $q+p \neq 2d$.\qed \\

Finally, we highlight the fact that O'Sullivan's Theorem \ref{thm
  O'Sullivan} encompasses Kimura--O'Sullivan's finite-dimensionality
\cite{kimura} for abelian varieties. Indeed, the cycles
\begin{equation*}
  p_{S^n}  := \frac{1}{n!} \sum_{\sigma \in \mathfrak{S}_n}
  \Gamma_{\sigma}
  \quad \mbox{and} \quad p_{\Lambda^n}  := \frac{1}{n!}
  \sum_{\sigma \in \mathfrak{S}_n} \epsilon(\sigma) \Gamma_{\sigma}
  \quad \mbox{in} \ \CH^{dn}(A^n \times A^n)
\end{equation*}
are clearly symmetrically distinguished.  Here $\mathfrak{S}_n$ is the
symmetric group on $n$ elements, $\epsilon$ is the signature
homomorphism and $\Gamma_\sigma$ is the graph of the action of
$\sigma$ on $X^n$ that permutes the factors.  Having in mind the
symmetrically distinguished Chow--K\"unneth projectors $\pi^i$ of
\eqref{eq CKabelian}, we consider the projectors on the even and odd
degree part of the cohomology~:
\begin{equation*}
  \pi_+  := \sum_i \pi^{2i} \quad \mbox{and} \quad \pi_-  := \sum_i
  \pi^{2i+1}
  \quad \mbox{in } \CH^d(A \times A).
\end{equation*}
For a cycle $\alpha \in \CH^d(A \times A)$, we denote $\alpha^{\otimes
  n}$ the cycle $(p_{1,n+1})^*\alpha \cdot (p_{2,n+2})^*\alpha \cdot \ldots
\cdot (p_{n,2n})^* \alpha$, where
$p_{i,j} : A^{2n} \rightarrow A \times A$ are the projections on
the  $i^\mathrm{th}$ and  $j^\mathrm{th}$ factors
for $1 \leq i,j \leq 2n$. Then, with those conventions, the cycles
$p_{S^n} \circ \pi_-^{\otimes n}$ and $p_{\Lambda^n} \circ
\pi_+^{\otimes n}$ become homologically trivial for $n>>0$. Being
symmetrically distinguished, they are also rationally trivial for
$n>>0$ by O'Sullivan's Theorem \ref{thm O'Sullivan}, thereby
establishing Kimura--O'Sullivan's finite-dimensionality for the Chow
motive of abelian varieties.

\vspace{10pt}
\section{Multiplicative Chow--K\"unneth decompositions}\label{sec
  multCK}

Let $F$ be a hyperk\"ahler variety.  A Fourier transform with kernel a
cycle representing the Beauville--Bogomolov class acts as zero on the
orthogonal complement of the image of $\Sym^* \HH^2(F,\Q)$ inside
$\HH^*(F,\Q)$. In general, this orthogonal complement is
non-trivial. Thus, such a Fourier transform cannot induce a bigrading
on the total Chow ring of $F$ in general. However, it seems reasonable
to expect the Chow ring of hyperk\"ahler varieties to actually be
endowed with a bigrading that is induced by a Chow--K\"unneth
decomposition (as in Conjecture \ref{conj2 vanishingc1}). In this
section, we introduce the notion of multiplicative Chow--K\"unneth
decomposition (\emph{cf.} Definition \ref{def multCK}), discuss its
relevance and links with degenerations of \emph{modified diagonals}
(\emph{cf.} Definition \ref{defn small diagonals}), and give first
examples of varieties that admit a multiplicative Chow--K\"unneth
decomposition. We also provide a proof for Theorem \ref{thm2 conj}~;
see Theorem \ref{thm2 conj repeat}. The results of this section will
be used in Section \ref{sec multCKX2} to show that the Hilbert scheme
of length-$2$ subschemes on a K3 surface admits a multiplicative
Chow--K\"unneth decomposition, as well as to prove Theorem \ref{thm2
  CK}.

\subsection{Multiplicative Chow--K\"unneth decompositions : definition
  and first properties}

\begin{defn} \label{def multCK} Let $X$ be a smooth projective variety
  of dimension $d$ endowed with a Chow--K\"unneth decomposition
  $\{\pi^i  : 0 \leq i \leq 2d\}$ as defined in \S \ref{sec CK}. The
\emph{small diagonal} $\Delta_{123}$ is the class in 
$\CH_d(X\times X \times X)$ of the image of the diagonal
embedding $X \hookrightarrow X \times X \times X, x \mapsto (x,x,x)$. The
  Chow--K\"unneth decomposition $\{\pi^i\}$ is said to be
  \begin{itemize}
\item \emph{multiplicative} if, for all $k\neq i+j$, $\pi^k \circ
  \Delta_{123} \circ (\pi^i \otimes \pi^j) = 0 \in \CH_d(X \times X
  \times X)$

  \noindent or equivalently if $\Delta_{123} = \sum_{i+j=k}\pi^k \circ
  \Delta_{123} \circ (\pi^i \otimes \pi^j) \in \CH_d(X \times X \times
  X)$~;

\item \emph{weakly multiplicative} if, for all $k\neq i+j$, $\pi^k
  \circ \Delta_{123} \circ (\pi^i \otimes \pi^j)$ acts as zero on
  decomposable cycles, that is, if $\big(\pi^k \circ \Delta_{123}
  \circ (\pi^i \otimes \pi^j) \big)_*(\alpha\times \beta) =
  \pi^k_*\big(\pi^i_*\alpha \cdot \pi^j_*\beta \big) = 0$ for all
  $\alpha, \beta \in \CH^*(X)$

  \noindent or equivalently if $\Big( \sum_{i+j=k}\pi^k \circ
  \Delta_{123} \circ (\pi^i \otimes \pi^j) \Big)_*(\alpha \times
  \beta) = \alpha \cdot \beta$ for all $\alpha, \beta \in \CH^*(X)$.
\end{itemize}
Here, by definition, we have set $\pi^i \otimes \pi^j := p_{1,3}^*\pi^i \cdot 
p_{2,4}^*\pi^j$, where
$p_{r,s} : A^{4} \rightarrow A \times A$ are the projections on
the  $r^\mathrm{th}$ and  $s^\mathrm{th}$ factors
for $1 \leq r,s \leq 4$.  In both cases, the Chow ring of $X$ inherits a
\emph{bigrading} $$\CH_{\mathrm{CK}}^i(X)_s := \pi^{2i-s}_*\CH^i(X),$$
which means that $\CH_{\mathrm{CK}}^i(X)_s \cdot
\CH_{\mathrm{CK}}^j(X)_r \subseteq \CH_{\mathrm{CK}}^{i+j}(X)_{r+s}.$
\end{defn}

\begin{rmk} Let $\mathfrak{h}(X)$ be the Chow motive of $X$. The
  diagonal embedding morphism $\iota_\Delta : X \hookrightarrow X
  \times X$ induces a multiplication map $\mathfrak{h}(X) \otimes
  \mathfrak{h}(X) \rightarrow \mathfrak{h}(X)$ and a co-multiplication
  map $\mathfrak{h}(X)(d) \rightarrow \mathfrak{h}(X) \otimes
  \mathfrak{h}(X)$. Denote $\mathfrak{h}^i(X) :=
  \pi^i\mathfrak{h}(X)$. Then that the Chow--K\"unneth decomposition
  $\{\pi^i : 0 \leq i \leq 2d\}$ is multiplicative can be restated
  motivically as requiring that the multiplication map restricted to
  $\mathfrak{h}^i(X) \otimes \mathfrak{h}^j(X)$ factors through
  $\mathfrak{h}^{i+j}(X)$ for all $i$ and $j$, or that the
  co-multiplication map restricted to $\mathfrak{h}^k(X)(d)$ factors
  through $\bigoplus_{i+j=k} \mathfrak{h}^i(X) \otimes
  \mathfrak{h}^j(X)$ for all $k$.
\end{rmk}

A first class of varieties that admit a multiplicative Chow--K\"unneth
decomposition is given by abelian varieties~:

\begin{ex}[Abelian varieties] \label{ex abelian} As in the previous
  section, let $A$ be an abelian variety of dimension $d$, $\hat{A}$
  its dual, and $L \in \CH^1(A \times \hat{A})$ the Poincar\'e
  bundle. By Theorem \ref{thm O'Sullivan} and Proposition \ref{prop
    hom poincare bundle}, the cycles $\pi^i := \frac{(-1)^i}{i!\,
    (2d-i)!} \hat{L}^i \circ L^{2d-i}$, $i=0,\ldots, 2d$, are
  symmetrically distinguished and define a Chow--K\"unneth
  decomposition of $A$. This Chow--K\"unneth decomposition is
  multiplicative. Indeed, by Theorem \ref{thm O'Sullivan} again, in
  order to show that $\pi^k \circ \Delta_{123} \circ (\pi^i \otimes
  \pi^j) = 0$ for all $k \neq i+j$, it is enough to show that
  $[L]^{2d-k} \circ [\Delta_{123}^A] \circ (p_{13}^*[\hat{L}]^i \cdot
  p_{24}^*[\hat{L}]^j) = 0$ in cohomology for all $k \neq i+j$. But
  then, this is \eqref{eq ab coho}.
\end{ex}

The following proposition gives a useful criterion for a
Chow--K\"unneth decomposition to be multiplicative.

\begin{prop} \label{rmk delta0} Let $X$ be a smooth projective variety
  of dimension $d$ endowed with a Chow--K\"unneth decomposition
  $\{\pi^i_X : 0 \leq i \leq 2d_X\}$ that is \emph{self-dual}
  (\emph{i.e.}, such that ${}^t\pi_X^i = \pi_X^{2d-i}$ for all
  $i$). Endow $X^3$ with the product Chow--K\"unneth
  decomposition $$\pi^k_{X^3} := \sum_{i_1+i_2+i_3=k} \pi_X^{i_1}
  \otimes \pi_X^{i_2} \otimes \pi_X^{i_3}, \quad 0 \leq k \leq 6d.$$
  Then the Chow--K\"unneth decomposition $\{\pi^i_X\}$ is
  multiplicative if and only if the small diagonal $\Delta_{123}^X$
  lies in the degree-zero graded part of $\CH_d(X^3)$ for the
  aforementioned product Chow--K\"unneth decomposition, that is, if
  and only if $(\pi^{4d}_{X^3})_* \Delta_{123}^X=\Delta_{123}^X$.
\end{prop}
\begin{proof} Here, the cycle $\pi_X^{i_1} \otimes \pi_X^{i_2} \otimes
  \pi_X^{i_3}$ is $p_{1,4}^*\pi_X^{i_1} \cdot p_{2,5}^* \pi_X^{i_2}
  \cdot p_{3,6}^*\pi_X^{i_3}$, where $p_{r,s} : X^6 \rightarrow X^2$
  are the projections on the $r^\text{th}$ and $s^\text{th}$ factors.
  The proposition follows at once from the formula
  \begin{equation}
    \label{eq lieb} (\pi_X^{i_1} \otimes \pi_X^{i_2}
    \otimes \pi_X^{i_3})_*\Delta_{123}^X = \pi_X^{i_1} \circ
    \Delta_{123}^X \circ ({}^t\pi_X^{i_2} \otimes {}^t\pi_X^{i_3}) = 
    \pi_X^{i_1} \circ
    \Delta_{123}^X \circ (\pi_X^{2d-i_2} \otimes \pi_X^{2d-i_3}),
  \end{equation} 
  where the first equality is \eqref{eq formula liebermann}.
  \end{proof}

  \begin{ex}[Abelian varieties, \emph{bis}] \label{ex abelian2} Let
    $A$ be an abelian variety of dimension $d$. Here is another proof
    that $A$ admits a multiplicative Chow--K\"unneth decomposition. By
    Deninger--Murre \cite{dm}, $A$ has a Chow--K\"unneth decomposition
    that induces the same decomposition on $\CH^*(A)$ as the Fourier
    decomposition of Beauville \cite{beauville1}. In order to prove
    that this Chow--K\"unneth decomposition is multiplicative, it is
    enough to check by Proposition \ref{rmk delta0} that
    $[n]^*\Delta_{123}^A = n^{4d}\Delta_{123}^A$, where $[n]: A \times
    A \times A \rightarrow A \times A \times A$ is the
    multiplication-by-$n$ map. But then, this follows simply from the
    fact that $[n]$ restricted to the small diagonal is
    $n^{2d}$-to-$1$.
\end{ex}

The notion of multiplicative Chow--K\"unneth decomposition is stable
under product~:

\begin{thm} \label{thm multCK} Let $X$ and $Y$ be two smooth
  projective varieties of respective dimension $d_X$ and $d_Y$ each
  endowed with a multiplicative Chow--K\"unneth decomposition
  $\{\pi^i_X : 0 \leq i \leq 2d_X\}$ and $\{\pi^j_Y : 0 \leq j \leq
  2d_Y\}$. Then the product Chow--K\"unneth decomposition on $X \times
  Y$, which is defined as
  \begin{equation}
  \label{eq multCK}
  \{\pi^k_{X \times Y}  :=
  \sum_{k=i+j} \pi_X^i \otimes \pi_Y^j  :\, 0 \leq k \leq 2(d_X +d_Y) \},
  \end{equation} is multiplicative.
\end{thm}
\begin{proof}
  Let $p_X : X \times Y \times X \times Y \times X \times Y
  \rightarrow X \times X \times X$ be the projection on the first,
  third and fifth factors, and let $p_Y$ denote the projection on the
  second, fourth and sixth factors. Writing $\Delta_{123}^X$ for the
  small diagonal of $X$ and similarly for $Y$ and $X \times Y$, we
  have the identity $$\Delta_{123}^{X\times Y} = p_X^*\Delta_{123}^X
  \cdot p_Y^*\Delta_{123}^Y.$$ We immediately deduce that
 $$(\pi_{X\times Y}^a \otimes \pi_{X\times Y}^b \otimes \pi_{X\times
   Y}^c)_* \Delta_{123}^{X\times Y} = \sum_{\substack{i+i'=a \\ j+j'=b
     \\ k+k'=c}} p_X^*\left[(\pi_X^i\otimes \pi_X^j \otimes
   \pi_X^k)_*\Delta_{123}^X \right] \cdot
 p_Y^*\left[(\pi_Y^{i'}\otimes \pi_Y^{j'} \otimes
   \pi_Y^{k'})_*\Delta_{123}^Y \right].$$ By Proposition \ref{rmk
   delta0}, the cycles $(\pi_X^i\otimes \pi_X^j \otimes
 \pi_X^k)_*\Delta_{123}^X$ and $(\pi_Y^{i'}\otimes \pi_Y^{j'} \otimes
 \pi_Y^{k'})_*\Delta_{123}^Y$ are both non-zero only if $i+j+k=4d_X$
 and $i'+j'+k'=4d_Y$. Therefore $$(\pi_{X\times Y}^a \otimes
 \pi_{X\times Y}^b \otimes \pi_{X\times Y}^c)_*\Delta_{123}^{X\times
   Y} = 0 \quad \mbox{if } a+b+c \neq 4(d_X+d_Y).$$ Hence, by
 Proposition \ref{rmk delta0} again, the product Chow--K\"unneth
 decomposition on $X \times Y$ is multiplicative.
\end{proof}

Before moving on to modified diagonals, we state a proposition that
will be useful when dealing with multiplicative Chow--K\"unneth
decompositions.

\begin{prop} \label{lem diagonal pullback} Let $X$ and $Y$ be smooth
  projective varieties of dimension $d_X$ and $d_Y$, respectively.
  Assume that both $X$ and $Y$ have Chow--K\"unneth decompositions
  $\{\pi^i_X : 0\leq i\leq 2d_X\}$ and $\{\pi^i_{Y} : 0\leq i\leq
  2d_Y\}$. Endow $X \times Y$ and $X^n$ with the product
  Chow--K\"unneth decomposition. Then the following statements are
  true.
\begin{enumerate}[(i)]
\item If $p_1 : X\times Y\rightarrow X$ and $p_2 : X\times
  Y\rightarrow Y$ are the two projections, then
\[
 p_1^*\alpha\cdot p_2^*\alpha' \in \CH^{p+p'}_{\mathrm{CK}}(X\times Y)_{s+s'}
\]
for all $\alpha\in \CH^p_{\mathrm{CK}}(X)_s$ and $\alpha'\in
\CH^{p'}_{\mathrm{CK}}(Y)_{s'}$.  In particular, we have
\[
p_1^*\CH_{\mathrm{CK}}^p(X)_s \subseteq \CH_{\mathrm{CK}}^p(X\times
Y)_s\quad \text{ and } \quad p_2^*\CH_{\mathrm{CK}}^{p'}(Y)_{s'}
\subseteq \CH_{\mathrm{CK}}^{p'}(X\times Y)_{s'}.
\]
\item If $\set{\pi^i_X}$ is self-dual, then $\Delta_X\in
  \CH^{d_X}_{\mathrm{CK}}(X\times X)_0$.
\item If $\set{\pi^i_X}$ is multiplicative and self-dual, then the
  small diagonal embedding $\iota_{\Delta,n} : X\rightarrow X^n$ is
  compatible with the decompositions, namely
\[
(\iota_{\Delta,n})^* \CH_{\mathrm{CK}}^p(X^n)_s \subseteq
\CH_{\mathrm{CK}}^p(X)_s \qquad \text{and}\qquad
(\iota_{\Delta,n})_*\CH_{\mathrm{CK}}^p(X)_s \subseteq
\CH_{\mathrm{CK}}^{p+d_X(n-1)}(X^n)_s.
\]
In particular, the small diagonal $(\iota_{\Delta,n})_*[X]$ is
contained in $\CH_{\mathrm{CK}}^{d_X(n-1)}(X^n)_0$.
\end{enumerate}
\end{prop}
\begin{proof}
  By definition, $\alpha$ satisfies $(\pi^i_X)_*\alpha = 0$ for all
  $i\neq 2p-s$ and $\alpha'$ satisfies $(\pi^j_{Y})_*\alpha' = 0$ for
  all $j\neq 2p'-s'$.  Note that
\[
\pi^k_{X\times Y}  :=\sum_{i+j=k}\pi^i_X \otimes \pi^j_{Y} =
 \sum_{i+j=k}p_{13}^*\pi^i_X\cdot p_{24}^*\pi^j_{Y}.
\]
We thus have
\begin{align*}
  (\pi^k_{X\times Y})_* (p_1^*\alpha\cdot p_2^*\alpha') &=
  \sum_{i+j=k} (p_{13}^*\pi^i_X\cdot
  p_{24}^*\pi^j_{Y})_*(p_1^*\alpha \cdot p_2^*\alpha')\\
  &= \sum_{i+j=k} p_1^*\big((\pi^i_X)_*\alpha\big) \cdot
  p_2^*\big((\pi^j_{Y})_*\alpha'\big).
\end{align*}
Therefore, since $(\pi^i_X)_*\alpha = 0$ unless $i= 2p-s$ and
$(\pi^j_{Y})_*\alpha' = 0$ unless $j=2p'-s'$, we get that $
(\pi^k_{X\times Y})_* (p_1^*\alpha \cdot p_2^*\alpha')= 0$ for all
$k\neq 2(p+p')-(s+s')$. It follows that $p_1^*\alpha\cdot
p_2^*\alpha'\in \CH_{\mathrm{CK}}^{p+p'}(X\times Y)_{s+s'}$.\medskip

Let us now show that the diagonal $\Delta_X$ belongs to
$\CH_{\mathrm{CK}}^{d_X}(X\times X)_0$, or equivalently that $
(\pi^k_{X\times X})_*\Delta_X = 0$ for all $k\neq 2d_X$.  A direct
computation yields
\begin{align*}
  (\pi^k_{X\times X})_*\Delta_X & = \sum_{i+j=k} (p_{24})_*\Big(
  (p_{13}^*\pi^i_X\cdot p_{24}^*\pi^j_X)\cdot p_{12}^*\Delta_X\Big)\\
  & = \sum_{i+j=k} \pi^j_X\circ{}^t\pi^i_X \\
  & = \sum_{i+j=k} \pi^j_X\circ \pi^{2d_X-i}_X\\
  & = \begin{cases}
    0, & k\neq 2d_X\, ;\\
    \Delta_X, &k=2d_X.
   \end{cases}
\end{align*}

We now prove \emph{(iii)}. By Theorem \ref{thm multCK}, the product
Chow--K\"unneth decomposition on $X^n$ is multiplicative for all
$n\geq 2$. Note that
\[
(\iota_{\Delta,n+1})_*[X] = (p_{1,\ldots,
  n})^*\big((\iota_{\Delta,n})_*[X]\big)\cdot p_{n,n+1}^* \Delta_X \in
\CH_{\mathrm{CK}}^{d_X(n-1)}(X^{n+1})_0\cdot
\CH_{\mathrm{CK}}^{d_X}(X^{n+1})_0.
\]
It follows by induction that $(\iota_{\Delta,n})_*[X]\in
\CH_{\mathrm{CK}}^{d_X(n-1)}(X^n)_0$.

Finally, we show that the decompositions are compatible with small
diagonal embeddings. On the one hand, if $\alpha\in
\CH_{\mathrm{CK}}^p(X)_s$, then
\[
(\iota_{\Delta,n})_*\alpha = p_1^*\alpha\cdot (\iota_{\Delta,n})_*[X]
\in \CH_{\mathrm{CK}}^p(X^n)_s\cdot
\CH_{\mathrm{CK}}^{d_X(n-1)}(X^n)_0 \subseteq
\CH_{\mathrm{CK}}^{p+d_X(n-1)}(X^n)_s.
\]
On the other hand, if $\beta\in \CH_{\mathrm{CK}}^p(X^n)_s$, then
\[
(\iota_{\Delta,n})_* (\iota_{\Delta,n})^* \beta = \beta\cdot
(\iota_{\Delta,n})_*[X] \in \CH_{\mathrm{CK}}^{p}(X^n)_s\cdot
\CH_{\mathrm{CK}}^{d_X(n-1)}(X^n)_0 \subseteq
\CH_{\mathrm{CK}}^{p+d_X(n-1)}(X^n)_s.
\]
Note that $(\iota_{\Delta,n})_*$ is injective because $p_1\circ
\iota_{\Delta,n} = \mathrm{id}_X$. Thus we conclude that
$(\iota_{\Delta,n})^*\beta\in \CH_{\mathrm{CK}}^p(X)_s$.
\end{proof}

\subsection{Modified diagonals and multiplicative Chow--K\"unneth decompositions}
Another class of varieties admitting a multiplicative Chow--K\"unneth
decomposition that we have in mind is given by K3 surfaces. If $X$ is
a K3 surface and $\mathfrak{o}_X$ is the class of a point lying on a
rational curve, then the Chow--K\"unneth decomposition $\pi^0 :=
\mathfrak {o}_X \times X, \pi^{4} := X \times \mathfrak {o}_X$ and
$\pi^2 := \Delta_X - \pi^0 - \pi^{2}$ is weakly multiplicative. This
is due to the fact proved by Beauville--Voisin \cite{bv} that the
intersection of any two divisors on $X$ is a multiple of
$\mathfrak{o}_X$ and was explained in the introduction. We would
however like to explain why this is in fact a \emph{multiplicative}
Chow--K\"unneth decomposition. Proposition \ref{P :red} will show that
this boils down to the main result of \cite{bv} on the vanishing of
some ``modified diagonal''.  Let us recall the definition of the
modified diagonal $\Delta_{\mathrm{tot}}$ as first introduced by Gross
and Schoen \cite{gs}.

\begin{defn}\label{defn small diagonals}
  Recall that the \emph{small diagonal} is defined as
\[
\Delta_{123} :=\mathrm{Im}\{X\rightarrow X\times X\times
X,x\mapsto(x,x,x)\}
\]
viewed as a cycle on $X\times X\times X$.  Let
$\mathfrak{o}_X\in\CH_0(X)$ be a zero-cycle of degree 1. We define for
$\{ i,j,k \} =\{1,2,3\}$ the following cycles in $\CH_d(X\times X\times X)$
\begin{align*}
  \Delta_{ij} &  := p_{ij}^*\Delta_X \cdot p_k^*\mathfrak{o}_X \, ;\\
  \Delta_i & := p_i^*[X]\cdot p_j^*\mathfrak{o}_X\cdot
  p_k^*\mathfrak{o}_X,
\end{align*}
where $p_i : X \times X \times X \rightarrow X$ and $p_{ij} : X \times
X \times X \rightarrow X \times X$ are the projections. The
\emph{modified diagonal} (attached to the degree-$1$ zero-cycle
$\mathfrak{o}_X$) is then the cycle
\[
\Delta_{\mathrm{tot}} := \Delta_{123}-\Delta_{12} -\Delta_{23}
-\Delta_{13} +\Delta_1 +\Delta_2 +\Delta_3 \in \CH_d(X \times X \times
X).
\]
\end{defn}

Here are two important classes of varieties for which
$\Delta_{\mathrm{tot}}$ vanishes modulo rational equivalence.

\begin{thm}[Gross--Schoen \cite{gs}] \label{thm gshyper} If $X$ is a
  hyperelliptic curve and $\mathfrak{o}_X$ is the class of a
  Weierstrass point, then $\Delta_{\mathrm{tot}} = 0$ in $\CH_1(X
  \times X \times X)$.
\end{thm}

\begin{thm}[Beauville--Voisin \cite{bv}] \label{thm bv} If $X$ is a K3
  surface and $\mathfrak{o}_X$ is the class of any point lying on a
  rational curve, then $\Delta_{\mathrm{tot}} = 0$ in $\CH_2(X \times
  X \times X)$.
\end{thm}

Note that the condition that $\Delta_{\mathrm{tot}}$ vanishes is
extremely restrictive. Indeed, if $\dim X > 1$, a straightforward
computation gives $$(\Delta_{\mathrm{tot}})_*(\alpha \times \beta) =
\alpha \cup \beta, \quad \mbox{for all} \ \alpha \in \HH^p(X,\Q), \
\beta \in \HH^q(X,\Q), \ 0<p,q<\dim X.$$ In particular, if $\dim X >
2$, $\Delta_{\mathrm{tot}}$ never vanishes (consider $\alpha =\beta
\in \HH^2(X,\Q)$ the class of an ample divisor) and, if $\dim X =2$,
then $\Delta_{\mathrm{tot}}=0$ imposes that $\HH^1(X,\Q) = 0$.  Even
in the case of curves, the vanishing of $\Delta_{\mathrm{tot}}$ is
very restrictive~; one may consult \cite{gs} and \cite{bv} for further
discussions. More generally, O'Grady \cite{o'grady2} considers a
\emph{higher order} modified diagonal, which is defined as
follows. Consider a smooth projective variety $X$ of dimension $d$
endowed with a closed point $\mathfrak{o}_X$. For $I \subseteq \{1,
\ldots , m\}$, we let $$\Delta^m_I(X; \mathfrak{o}_X) := \{(x_1,
\ldots, x_m) \in X^m : x_i = x_j \mbox{ if } i, j \in I \mbox{ and }
x_i = \mathfrak{o}_X \mbox{ if } i \notin I\}.$$ The
\emph{$m^{\mathrm{th}}$ modified diagonal cycle} associated to
$\mathfrak{o}_X$ is the $d$-cycle on $X^m$ given
by $$\Delta^m_{\mathrm{tot}}(X,\mathfrak{o}_X) := \sum_{\emptyset \neq
  I \subseteq \{1,2,\ldots,m\}} (-1)^{m-|I|} \Delta^m_I(X;
\mathfrak{o}_X).$$ One may also define the $m^\mathrm{th}$ modified
diagonal cycle associated to a zero-cycle of degree $1$ on $X$ by
proceeding as in Definition \ref{defn small diagonals}.  One can then
ask if, given a smooth projective variety $X$, there exists a
zero-cycle $\mathfrak{o}_X \in \CH_0(X)$ of degree $1$ and an integer
$m>0$ such that $\Delta_{\mathrm{tot}}^m(X;\mathfrak{o}_X) =0 \in
\CH_d(X^m)$. In the case of abelian varieties, the following theorem
answers a question raised by O'Grady. This result was obtained
independently by Moonen and Yin \cite{my} by using the Fourier
transform on abelian varieties \cite{beauville1, dm}.

\begin{thm} \label{thm abelianmodified}
  Let $A$ be an abelian variety of dimension $d$. Then
  $\Delta^{2d+1}_{\mathrm{tot}}(A;a) = 0$ in $\CH_d(A^{2d+1})$ for all
  closed points $a \in A$.
\end{thm}
\begin{proof}
  By translation it is enough to show that
  $\Delta_{\mathrm{tot}}^{2d+1}(A,e) = 0$ for $e$ the identity element
  in $A$. The cycle $\Delta_{\mathrm{tot}}^{2d+1}(A,e)$ is apparently
  symmetrically distinguished in the sense of O'Sullivan. Therefore,
  by O'Sullivan's Theorem \ref{thm O'Sullivan}, it is rationally
  trivial if it is homologically trivial. That it is homologically
  trivial is proved by O'Grady in \cite[Proposition 1.3]{o'grady2}.
  
  An alternate proof of this Theorem will be given after the next
  proposition in Remark \ref{rmk moddiagab}.
\end{proof}

The significance of modified diagonals is made apparent by the
following~:

\begin{prop} \label{prop multdiag} Let $X$ be a smooth projective
  variety of dimension $d$ over a field $k$. Assume that $X$ admits a
  multiplicative self-dual Chow--K\"unneth decomposition $\{\pi^i : 0
  \leq i \leq 2d \}$ with $\pi^{2d} = X\times \mathfrak{o}_X$ for some
  degree-$1$ zero-cycle $\mathfrak{o}_X$ on $X$. Then
  $\Delta_{\mathrm{tot}}^{m}(X,\mathfrak{o}_X)=0$ for all $m \geq
  2d+1$.

  \noindent Moreover, if $\pi^1=0$, then
  $\Delta_{\mathrm{tot}}^{m}(X,\mathfrak{o}_X)=0$ for all $m \geq
  d+1$.
\end{prop}

\begin{proof} For any $I=\set{i_1,i_2,\ldots,
    i_r}\subseteq\set{1,2,\ldots,m}$ we define $\Pi_I^m\in
  \CH^{md}(X^m\times X^m)$ as the $m$-fold tensor product of $r$
  copies of $\Delta_X$ and $m-r$ copies of $\pi^{2d}$ with the
  diagonals $\Delta_X$ being placed in the
  $(i_1,i_2,\ldots,i_r)^{\text{th}}$ position. Then one easily checks
  that \[ (\Pi_I^m)_*\Delta^m = \Delta_I^m(X;\mathfrak{o}_X), \quad
  \mbox{where} \ \ \Delta^m : =
  \Delta^m_{\{1,\ldots,m\}}(X,\mathfrak{o}_X) = \{(x,x,\ldots,x) : x
  \in X\}. \] For example, if $I=\set{1,2,\ldots,n}$, $1\leq n\leq m$,
  then $$ \left(\Pi^m_{\set{1,2,\ldots,n}}\right)_*\Delta^m =
  (\underbrace{\Delta_X \otimes \cdots \otimes \Delta_X}_{n\
    \text{times}} \otimes \underbrace{\pi^{2d} \otimes \cdots \otimes
    \pi^{2d}}_{m-n \ \text{times}})_* \Delta^m = \{(\underbrace{x,
    \ldots ,x}_{n\ \text{times}} , \underbrace{\mathfrak{o}_X, \ldots,
    \mathfrak{o}_X}_{m-n \ \text{times}}) : x\in X\}.$$ It follows
  that the $m^\text{th}$ modified diagonal satisfies \[
  \Delta_{\mathrm{tot}}^m(X,\mathfrak{o}_X) = (\Pi^m)_* \Delta^m,\quad
  \text{where} \ \ \Pi^m :=\sum _{\emptyset \neq I\subseteq
    \{1,\ldots,m\}} (-1)^{m-|I|}\Pi_I^m. \] Substituting
  $\Delta_X=\pi^0+\cdots +\pi^{2d}$ into the definition of $\Pi^m_I$,
  we see that $\Pi^m_I$ expands into a sum whose terms are of the form
  $\pi^{j_1}\otimes \pi^{j_2}\otimes \cdots \otimes \pi^{j_m}$ with
  $j_k=2d$ for all $k\notin I$.  \medskip

  First we claim, without any assumptions on the smooth projective
  variety $X$, that $\Pi^m$ can be expressed as a sum of terms of the
  form $\pi^{j_1}\otimes \pi^{j_2}\otimes \cdots \otimes \pi^{j_m}$,
  where none of the indices $j_k$, $1\leq k \leq m$, are equal to
  $2d$.  For that matter, note that the term
\begin{equation}\label{eq special term}
  (\underbrace{\pi^{2d}\otimes \cdots \otimes\pi^{2d}}_{n \text{
      times}}) \otimes (\pi^{l_1}\otimes \cdots \otimes
  \pi^{l_{m-n}}), \qquad 0\leq l_1,\ldots,l_{m-n} <2d \mbox{ and } 1\leq n< m, 
\end{equation} 
appears in $\Pi^m_I$ (with coefficient $1$) if and only if
$\{n+1,n+2,\ldots,m\}\subseteq I$.  There are $\binom{n}{m-r}$ subsets
$\emptyset \neq I\subseteq \{1,\ldots,m\}$ such that
$\{n+1,n+2,\ldots,m\}\subset I$ and $|I|=r$.  It follows that the term
\eqref{eq special term} appears in $\Pi^m$ with coefficient \[
\sum_{r=m-n}^m (-1)^{m-r}\binom{n}{m-r} = 0.
\] 
By symmetry, any term $\pi^{j_1}\otimes \pi^{j_2}\otimes
\cdots \otimes \pi^{j_m}$ with at least one index $j_k$ equal to $2d$
does not appear in $\Pi^m$. The claim is thus settled.\medskip

Now, assume that the Chow--K\"unneth decomposition $\{\pi^i\}$ is
multiplicative.  By Proposition \ref{lem diagonal
  pullback}\emph{(iii)} the small diagonal $\Delta^m$ belongs to
$\CH_{\mathrm{CK}}^{d(m-1)}(X)_0$ for the product Chow--K\"unneth
decomposition on $X^m$ and hence
\[
(\pi^{j_1}\otimes\pi^{j_2}\otimes \cdots \otimes \pi^{j_m})_*\Delta^m
=0, \qquad\text{for all } j_1+j_2+\cdots+j_m\neq 2d(m-1). \] When
$m\geq 2d+1$ (or $m\geq d+1$ if $\pi^1=0$), the condition
$j_1+j_2+\cdots+j_m = 2d(m-1)$ forces at least one index $j_k$ to be
equal to $2d$ for $(\pi^{j_1}\otimes\pi^{j_2}\otimes \cdots \otimes
\pi^{j_m})_*\Delta^m$ to be possibly non-zero. This yields
$\Delta_{\mathrm{tot}}^{m}(X,\mathfrak{o}_X) =
(\Pi^m)_*\Delta^m=0$. \end{proof}

\begin{rmk} \label{rmk moddiagab} Note that by combining Example
  \ref{ex abelian} or Example \ref{ex abelian2} with Proposition
  \ref{prop multdiag}, we obtain another proof of Theorem \ref{thm
    abelianmodified}.
\end{rmk}

In forthcoming work \cite{sv}, we show that
$\Delta_{\mathrm{tot}}^{g+2}(C,c) = 0$ in $\CH_1(C^{g+2})$ for all
curves $C$ of genus $g$ and all closed points $c \in C$. Let us also
point out that, for a K3 surface $S$, O'Grady \cite{o'grady2} shows
that $\Delta_{\mathrm{tot}}^5(S^{[2]}, \mathfrak{o}) =0$, where
$\mathfrak{o}$ is any point on $S^{[2]}$ corresponding to a subscheme
of length $2$ supported at a point lying on a rational curve on
$S$. This result can also be seen, via Proposition \ref{prop
  multdiag}, as a consequence of our Theorem \ref{thm multCK X2} where
we establish the existence of a multiplicative Chow--K\"unneth
decomposition for $S^{[2]}$. \medskip

After this digression on higher order modified diagonals, we go back
to the more down-to-earth $3^\text{rd}$ modified diagonal of
Definition \ref{defn small diagonals}. In both cases covered by
Theorem \ref{thm gshyper} and Theorem \ref{thm bv}, the vanishing of
the $3^\text{rd}$ modified diagonal $\Delta_{\mathrm{tot}}$ is
directly related to the existence of a multiplicative Chow--K\"unneth
decomposition, as shown in the following proposition~; see also
\cite[\S 2.3]{voisin k3}.

\begin{prop} \label{P :red} Let $X$ be either a curve or a surface
  with vanishing irregularity and let $\mathfrak {o}_X \in \CH_0(X)$
  be a zero-cycle of degree $1$. Consider the Chow--K\"unneth
  decomposition of $X$ consisting of the $3$ mutually orthogonal
  projectors $\pi^0 :=\mathfrak {o}_X \times X, \pi^{2d} :=X \times
  \mathfrak {o}_X$ and $\pi^d := \Delta_X - \pi^0 - \pi^{2d}$, $d=1$
  or $2$.  Then this Chow--K\"unneth decomposition is multiplicative
  if and only if $\Delta_{\mathrm{tot}}=0$.
\end{prop}
\begin{proof} The ``if'' part of the proposition was already covered
  by Proposition \ref{prop multdiag}. Let us however give a direct
  proof of the proposition. By Definition \ref{def multCK}, the 
Chow--K\"unneth
  decomposition is multiplicative if and only if $$\Delta_{123} =
  \sum_{i+j=k} \pi^k \circ \Delta_{123} \circ (\pi^i \otimes \pi^j)
  .$$ Substituting $\pi^d = \Delta_X - \pi^0 - \pi^{2d}$ into the
  above expression and expanding yields
  \begin{align} \Delta_{123} = & -\pi^0\circ \Delta_{123}\circ (\pi^0
    \otimes \pi^0) - \Delta_{123} \circ (\pi^{2d} \otimes \pi^0)
    - \Delta_{123} \circ (\pi^0 \otimes \pi^{2d})\label{E :mult}\\
    & + \Delta_{123} \circ (\pi^0 \otimes \Delta_X) + \Delta_{123}
    \circ (\Delta_X \otimes \pi^0) + \pi^{2d} \circ \Delta_{123}.
    \nonumber
\end{align}
In order to conclude it is thus enough to check that the first line of
the right-hand side of \eqref{E :mult} is $-(\Delta_1 +\Delta_2
+\Delta_3)$ and that the second line is $\Delta_{12} +\Delta_{23}
+\Delta_{13}$. This can be checked term-by-term by using the formula
\eqref{eq lieb}.
\end{proof}

\begin{rmk} The notion of multiplicative Chow--K\"unneth decomposition
  is much stronger than the notion of weakly multiplicative
  Chow-K\"unneth decomposition. For instance, it is clear that every
  curve has a weakly multiplicative Chow--K\"unneth decomposition but,
  by Proposition \ref{P :red} and the discussion following Definition
  \ref{defn small diagonals}, the class of curves admitting a
  multiplicative Chow--K\"unneth decomposition is restricted.

  Other classes of varieties admitting a weakly multiplicative
  Chow--K\"unneth decomposition but that should not admit a
  multiplicative Chow--K\"unneth decomposition in general are given by
  smooth hypersurfaces in $\PP^4$. Indeed, let $X$ be a smooth
  hypersurface in $\PP^4$ of degree $d$. Let $\mathfrak {o}_X$ be the
  class of $h^{3}/d$, where $h :=c_1(\mathcal{O}_X(1))$. Then the
  Chow--K\"unneth decomposition $\pi^0 := \mathfrak {o}_X \times X,
  \pi^{2} := \frac{1}{d}h \times h^2, \pi^4 := \frac{1}{d} h^2 \times
  h, \pi^{6} := X \times \mathfrak {o}_X$ and $\pi^3 := \Delta_X -
  \pi^0 - \pi^{2} - \pi^{4} - \pi^{6}$ is weakly
  multiplicative. Indeed, we have $\CH^i_{\mathrm{CK}}(X)_0 = \langle
  h^i \rangle$ for all $i$, and also $\CH^1(X) =
  \CH_{\mathrm{CK}}^1(X)_0$.  We thus only need to check that $h \cdot
  \alpha = 0$ for all $\alpha \in \CH^2_{\mathrm{CK}}(X)_1$. But then
  $d\, h \cdot \alpha = i^*i_*\alpha$, where $i:X \hookrightarrow
  \PP^4$ is the natural embedding. Thus $h\cdot \alpha$ is numerically
  trivial and proportional to $h^3$, and so $h\cdot \alpha = 0$. In
  the case where $X$ is Calabi--Yau, similar arguments as in the proof
  of Proposition \ref{P :red} show that the Chow--K\"unneth
  decomposition above is multiplicative if and only if
  \cite[Conjecture 3.5]{voisin k3} holds.
\end{rmk} 

From Proposition \ref{P :red}, we deduce the following two examples of
varieties having a multiplicative Chow--K\"unneth decomposition.

\begin{ex}[Hyperelliptic curves] \label{ex hyperelliptic} Let $X$ be a
  hyperelliptic curve and let $\mathfrak {o}_X$ be the class of a
  Weierstrass point. Then the Chow--K\"unneth decomposition $\pi^0 :=
  \mathfrak {o}_X \times X, \pi^{2} := X \times \mathfrak {o}_X$ and
  $\pi^1 := \Delta_X - \pi^0 - \pi^{2}$ is multiplicative.
\end{ex}

\begin{ex}[K3 surfaces] \label{ex K3} Let $X$ be a K3 surface and let
  $\mathfrak {o}_X$ be the class of a point lying on a rational
  curve. Then the Chow--K\"unneth decomposition $\pi^0 := \mathfrak
  {o}_X \times X, \pi^{2} := X \times \mathfrak {o}_X$ and $\pi^1 :=
  \Delta_X - \pi^0 - \pi^{2}$ is multiplicative.
\end{ex}

\subsection{On Theorem \ref{thm2 conj}}
Let us finally prove a more precise version of Theorem \ref{thm2
  conj}, which shows that part of Conjecture \ref{conj2 vanishingc1}
for hyperk\"ahler varieties of $\mathrm{K3}^{[2]}$-type reduces to
Conjecture \ref{conj sheaf}.

\begin{thm} \label{thm2 conj repeat} Let $F$ be a hyperk\"ahler
  variety of $\mathrm{K3}^{[2]}$-type endowed with a cycle $L \in
  \CH^2(F\times F)$ representing the Beauville--Bogomolov class
  $\mathfrak{B}$.  Assume that $F$ satisfies the following weaker
  version of Conjecture \ref{conj sheaf}~: the restriction of the
  cycle class map $\bigoplus_n \CH^*(F^n) \rightarrow \bigoplus_n
  \HH^*(F^n,\Q)$ to $V(F ;L)$ is injective. Then $\CH^*(F)$ admits a
  Fourier decomposition as in Theorem \ref{thm main splitting} which
  is compatible with its ring structure. Moreover, $F$ admits a
  multiplicative Chow--K\"unneth decomposition.
  \end{thm}
\begin{proof}
   By Theorem \ref{prop main}, in order to obtain a Fourier
  decomposition for $\CH^*(F)$ as in the statement of Theorem \ref{thm
    main splitting}, one needs only to check the existence of $L$
  satisfying conditions \eqref{eq rational equation},
  \eqref{assumption pre l}, \eqref{assumption hom} and
  \eqref{assumption 2hom}. Note that composition of correspondences
  only involves intersection products, push-forward and pull-back
  along projections.  Therefore, if as is assumed $V(F ;L)$ injects
  into cohomology, then one needs only to check that $\mathfrak{B}$
  satisfies \eqref{eq cohomological equation S2}, $\mathfrak{B}_*
  \mathfrak{b}^2 = 0$, $\mathfrak{B} \circ (\mathfrak{b}_1\cdot
  \mathfrak{B}) = 25 \mathfrak{B}$ and $(\mathfrak{B}^2)\circ
  (\mathfrak{b}_1 \cdot \mathfrak{B}^2) = 0$.  All of these can be
  read off the results of \S \ref{sec coho fourier}.  As
  before, note that the intersection product is induced by pulling
  back along diagonals.
    By Lemma \ref{lem equivalence mult}, in order to check that the
  Fourier decomposition is compatible with the intersection product on
  $\CH^*(F)$, one needs only to check a cohomological statement, namely $
  \mathfrak{B}^t \circ [\Delta_{123}^F] \circ (p_{13}^*\mathfrak{B}^p
  \cdot p_{24}^*\mathfrak{B}^q) = 0$ for all $t+p+q \neq 4$.
   But then, this follows again from the cohomological description of
  the powers of $\mathfrak{B}$ given in \S \ref{sec coho
    fourier}. Finally, under our assumptions, the projectors defined
  in Proposition \ref{prop kunneth} define Chow--K\"unneth projectors
  for $F$ and the fact that they define a multiplicative
  Chow--K\"unneth decomposition follows from the cohomological
  equation above~: $ \mathfrak{B}^t \circ [\Delta_{123}^F] \circ
  (p_{13}^*\mathfrak{B}^p \cdot p_{24}^*\mathfrak{B}^q) = 0$ for all
  $t+p+q \neq 4$.
\end{proof}

\vspace{10pt}
\section{Algebraicity of $\mathfrak{B}$ for hyperk\"ahler varieties of
  $\mathrm{K3}^{[n]}$-type} \label{sec alg B} In this section, we use
a twisted sheaf constructed by Markman \cite{markman} to show that the
Beauville--Bogomolov class $\mathfrak{B}$ lifts to a cycle
$L\in\CH^2(F\times F)$ for all projective hyperk\"ahler manifolds $F$
deformation equivalent to $\mathrm{K3}^{[n]}$, $n\geq 2$. The proof
goes along these lines~: thanks to the work of many authors
(\emph{cf.} Theorem \ref{thm moduli of sheaves hk}), one may associate
to any K3 surface $S$ endowed with a Mukai vector $v$ a variety
$\mathcal{M}$ called the moduli space of stable sheaves with respect
to $v$. The variety $\mathcal{M}$ turns out to be a hyperk\"ahler
variety of K3$^{[n]}$-type. The variety $\mathcal{M} \times S$ comes
equipped with a so-called ``universal twisted sheaf'' $\mathscr{E}$,
from which one naturally produces a twisted sheaf $\mathscr{M}$ on
$\mathcal{M} \times \mathcal{M}$. The idea behind twisted sheaves is
that a sheaf in the usual sense might not deform in a family, while it
could deform once it gets suitably ``twisted''. The main result of
Markman \cite{markman}, which is stated as Theorem \ref{thm markman},
stipulates that given a hyperk\"ahler variety $F$ of K3$^{[n]}$-type,
there exists a K3 surface $S$ and a Mukai vector $v$ such that the
twisted sheaf $\mathscr{M}$ over the self-product $\mathcal{M} \times
\mathcal{M}$ of the associated moduli space $\mathcal{M}$ of stable
sheaves deforms to a twisted sheaf on $F \times F$. Furthermore,
Markman shows the existence of characteristic classes
($\kappa$-classes) for twisted sheaves. Our sole contribution then
consists in computing the cohomology classes of the $\kappa$-classes
of the twisted sheaf on $F \times F$~; see Proposition \ref{prop
  kappa22}, Lemma \ref{lem top kappa of F} and Theorem \ref{thm
  existence of L}.
 
\subsection{Twisted sheaves and $\kappa$-classes}
Let $Y$ be a compact complex manifold. In \cite{markman}, Markman
defined the classes
$\kappa^{\mathrm{top}}_i(\mathscr{E})\in\HH^{2i}(Y,\Q)$ for (twisted)
sheaves $\mathscr{E}$ on $Y$. We will see that when $Y$ is algebraic
the classes $\kappa_i(\mathscr{E})$ can be defined at the level of
Chow groups, lifting the topological classes.

Assume that $Y$ is projective. Let $K^0(Y)$ be the Grothendieck group
of locally free coherent sheaves on $Y$, \textit{i.e.}, the free
abelian group generated by locally free coherent sheaves on $Y$ modulo
the equivalence generated by short exact sequences. Let $y\in K^0(Y)$
be an element with nonzero rank $r$, we define
\begin{equation}\label{eq kappa on K group}
 \kappa(y)  := \mathrm{ch}(y)\cdot \mathrm{exp}(-c_1(y)/r)\in\CH^*(Y).
\end{equation}
We use $\kappa_i(y)$ to denote the component of $\kappa(y)$ in
$\CH^i(Y)$. Note that $\kappa_1(y)=0$ for all $y\in K^0(Y)$ with
nonzero rank. We also easily see that $\kappa$ is multiplicative but
not additive.  We also have $\kappa(\mathscr{L})=1$ for all invertible
sheaves $\mathscr{L}$ on $Y$.

Keep assuming that $Y$ is projective. Let $K_0(Y)$ be the Grothendieck
group of coherent sheaves on $Y$. Since $Y$ is smooth, every coherent
sheaf can be resolved by locally free sheaves and it follows that the
natural homomorphism $K^0(Y)\rightarrow K_0(Y)$ is an isomorphism.
This makes it possible to extend the definition of $\kappa$ to
(bounded complexes of) coherent sheaves as long as the rank is
nonzero. In this case, the multiplicativity of $\kappa$ becomes
$\kappa(\mathscr{E}\overset{\mathrm{L}}{\otimes}\mathscr{F}) =
\kappa(\mathscr{E}) \cdot \kappa(\mathscr{F})$ where $\mathscr{E}$ and
$\mathscr{F}$ are (complexes of) coherent sheaves. In particular, we
have
\begin{equation}\label{eq kappa and line bundle}
  \kappa(\mathscr{E}\otimes \mathscr{L}) = \kappa(\mathscr{E}),
\end{equation}
for all invertible sheaves $\mathscr{L}$.

\begin{defn}
  Let $Y$ be either a compact complex manifold or a smooth projective
  variety over an algebraically closed field. Let $\theta\in
  \mathrm{Br}'(X) := \HH^2(Y,\calO_Y^*)_{\mathrm{tors}}$, where the
  topology is either the analytic topology or the \'etale topology. A
  $\theta$-\textit{twisted sheaf} $\mathscr{E}$ consists of the
  following data.
\begin{enumerate}[(i)]
\item An open covering $\mathcal{U}=\{U_\alpha\}_{\alpha\in I}$ of $Y$
  in either the analytic or the \'etale topology~;

\item A \v{C}ech cocycle $\{\theta_{\alpha \beta \gamma}\}\in
  Z^2(\mathcal{U},\calO_Y^*)$ representing the class $\theta$~;

\item A sheaf $\mathscr{E}_\alpha$ of $\calO_{U_\alpha}$-modules over
  $U_\alpha$ for each $\alpha\in I$ together with isomorphisms
  $g_{\alpha \beta} : \mathscr{E}_{\alpha}|_{U_{\alpha \beta}}
  \rightarrow \mathscr{E}_{\beta}|_{U_{\alpha \beta}}$ satisfying the
  conditions~: (a) $g_{\alpha \alpha} = \mathrm{id}$, (b) $g_{\alpha
    \beta}^{-1} = g_{\beta \alpha}$ and (c) $g_{\alpha
    \beta}g_{\beta\gamma}g_{\gamma\alpha} =
  \theta_{\alpha\beta\gamma}\mathrm{id}$.
\end{enumerate}

  The twisted sheaf $\mathscr{E}$ is \textit{locally free} (resp.
  \textit{coherent}, \textit{torsion free} or \textit{reflexive}) if
  each $\mathscr{E}_\alpha$ is. The \textit{rank} of $\mathscr{E}$ is
  defined to be the rank of $\mathscr{E}_\alpha$. If $\theta=0$, then
  $\mathscr{E}$ is said to be \textit{untwisted}.
\end{defn}

\begin{rmk}\label{rmk untwisted sheaf}
  Assume that $Y$ is a smooth projective variety. If $\theta = 0$, so
  that $\mathscr{E}$ is an untwisted sheaf, then $\theta_{\alpha \beta
    \gamma}$ is a coboundary (after a refinement of the covering if
  necessary) and hence we can find $h_{\alpha \beta}\in
  \calO_{U_{\alpha\beta}}^*$ such that
$$
\theta_{\alpha\beta\gamma}^{-1} = h_{\alpha\beta} h_{\beta\gamma}
h_{\gamma\alpha}.
$$
If we modify $g_{\alpha\beta}$ by $h_{\alpha\beta}$ for all $\alpha$
and all $\beta$ and take $\tilde{g}_{\alpha\beta}  :=
g_{\alpha\beta}h_{\alpha \beta}$, then we may use
$\tilde{g}_{\alpha\beta}$ as transition functions and glue the
$\mathscr{E}_{\alpha}$ to a global coherent sheaf
$\tilde{\mathscr{E}}$.  The ambiguity of the choice of
$h_{\alpha\beta}$ is a transition function of some line bundle on $Y$.
Hence $\tilde{\mathscr{E}}$ is defined up to tensoring with a line
bundle on $Y$. Then equation \eqref{eq kappa and line bundle} implies
that the class $\kappa(\tilde{\mathscr{E}})$ is independent of the
choices. We define $\kappa(\mathscr{E}) :=\kappa(\tilde{\mathscr{E}})$.
By abuse of notation, we will again use $\mathscr{E}$ to denote the
sheaf $\tilde{\mathscr{E}}$.  Hence when $\mathscr{E}$ is untwisted,
we automatically view $\mathscr{E}$ as a usual sheaf of
$\calO_Y$-modules.

Let $\mathscr{E}=(U_\alpha,\mathscr{E}_\alpha,g_{\alpha\beta})$ and
$\mathscr{E}'=(U_\alpha,\mathscr{E}'_\alpha,g'_{\alpha\beta})$ be two
$\theta$-twisted sheaves. We can define a homomorphism $f :\mathscr{E}
\rightarrow \mathscr{E}'$ to be a collection of homomorphisms $f_\alpha :
\mathscr{E}_\alpha \rightarrow \mathscr{E}'_\alpha$ such that
$f_\beta \circ g_{\alpha\beta} = g'_{\alpha\beta}\circ f_\alpha$. Then one
can naturally define exact sequences of $\theta$-twisted sheaves
and the K-group $K_0(Y,\theta)$ of $\theta$-twisted sheaves.
\end{rmk}

With the above remark, it makes sense to ask if $\kappa$ can be
defined for twisted sheaves. The following proposition gives a
positive answer.

\begin{prop}[Markman \cite{markman}]
  Let $Y$ be a smooth projective variety~; the function $\kappa$ can
  be naturally extended to $K_0(Y,\theta)$.
\end{prop}

\begin{proof}[Sketch of proof]
  Assume that $\theta$ has order $r$ and $\mathscr{E}$ is a
  $\theta$-twisted sheaf. Then, with the correct definition, the
  $r^{\text{th}}$ tensor product of $\mathscr{E}$ is untwisted and
  hence $\kappa$ is defined on this tensor product. Then
  $\kappa(\mathscr{E})$ is defined by taking the $r^{\text{th}}$ root.
  See \S2.2 of \cite{markman} for more details.
\end{proof}

\begin{defn}
  Let $\mathcal{X}\rightarrow B$ be a flat family of compact complex
  manifolds (or projective algebraic varieties) and $\theta\in
  \mathrm{Br}'(\mathcal{X})$. A \textit{flat family} of
  $\theta$-twisted sheaves is a $\theta$-twisted coherent sheaf
  $\mathscr{E}$ on $\mathcal{X}$ which is flat over $B$.
\end{defn}

Note that for all $t\in B$ the restriction of $\mathscr{E}$ to the
fiber $\mathcal{X}_t$ gives rise to a $\theta_t$-twisted sheaf
$\mathscr{E}_t$.  Usually we do not specify the class $\theta$ and
call $\mathscr{E}$ a flat family of twisted sheaves. It is a fact that
$\kappa^{\mathrm{top}}(\mathscr{E}_t)$ varies continuously, which we
state as the following lemma.

\begin{lem}\label{lem continuity of top kappa}
  Let $\mathscr{E}$ be a flat family of twisted sheaves on a flat
  family $\pi :\mathcal{X}\rightarrow B$ of compact complex manifolds,
  then
  $\kappa_i^{\mathrm{top}}(\mathscr{E}_t)\in\HH^{2i}(\mathcal{X}_t,\Q)$,
  $t\in B$, are parallel transports of each other in this family.
\end{lem}
\begin{proof}
  We may assume that $B$ is smooth and hence that $\mathcal{X}$ is
  smooth.  By taking some (derived) tensor power of $\mathscr{E}$, we
  get a complex $\mathscr{E}'$ of untwisted locally free sheaves. The
  $\kappa$-class of $\mathscr{E}$ is a $r^{\text{th}}$-root of the
  $\kappa$-class of $\mathscr{E}'$~; see \S2.2 of \cite{markman}. By
  flatness, the above tensor power construction commutes with base
  change. Hence $\kappa^{\mathrm{top}}(\mathscr{E}_t)$ agrees with the
  $r^{\text{th}}$-root of $\kappa^{\mathrm{top}}(\mathscr{E}'_t)$.  As
  a consequence, we have $\kappa^{\mathrm{top}}(\mathscr{E}_t) =
  \kappa^{\mathrm{top}}(\mathscr{E})|_{\mathcal{X}_t}$.  This proves
  the lemma.
\end{proof}

\subsection{Moduli space of stable sheaves on a $\mathrm{K3}$ surface}
The concept of a twisted sheaf naturally appears when one tries to
construct universal sheaves over various moduli spaces of stable
sheaves. We will be mainly interested in the geometry of moduli spaces
of stable sheaves on $\mathrm{K3}$ surfaces~; our main reference is
\cite{hl}.

Let $S$ be an algebraic K3 surface. Let
$$\tilde{\HH}(S,\Z) = \HH^0(S,\Z) \oplus \HH^2(S,\Z) \oplus \HH^4(S,\Z)$$
be the Mukai lattice of $S$ endowed with the Mukai pairing
$$(v,v')  := \int_S -v_0\cup v'_2 +v_1\cup v'_1 - v_2\cup v'_0,$$
for all $v=(v_0,v_1,v_2)$ and $v'=(v'_0,v'_1,v'_2)$ in
$\tilde{\HH}(S,\Z)$. Given a coherent sheaf $\mathscr{E}$ on $S$, its
Mukai vector $v(\mathscr{E})$ is defined to be
$$v(\mathscr{E})  := \mathrm{ch}^{\mathrm{top}}(\mathscr{E})
\sqrt{\mathrm{td}^{\mathrm{top}}(S)}\in\tilde{\HH}(S,\Z). $$ For a
Mukai vector $v=(v_0,v_1,v_2)$, we say that $v_1$ is
\textit{primitive} if it is not divisible as a cohomology class in
$\HH^2(S,\Z)$. Given a Mukai vector $v\in\tilde{\HH}(S,\Z)$, there is
a chamber structure with respect to $v$ on the ample cone of $S$~; see
\cite[Appendix 4.C]{hl}. An ample divisor $H$ on $S$ is said to be
$v$-\textit{generic} if it
is contained in an open chamber with respect to $v$. Given
$v\in\tilde{\HH}(S,\Z)$ and an ample divisor $H$ on $S$, we define
$\mathcal{M}_H(v)$ to be the moduli space of $H$-stable sheaves with
Mukai vector $v$. The following result is due to many authors~; see
\cite[Theorem 6.2.5]{hl} and the discussion thereafter.
\begin{thm}[Mukai, Huybrechts, G\"ottsche, O'Grady and
  Yoshioka]\label{thm moduli of sheaves hk}
  Let $v_1$ be primitive and let $H$ be $v$-generic, then $\mathcal{M}_H(v)$
  is a projective hyperk\"ahler manifold deformation equivalent to
  $S^{[n]}$ for $n=1+\frac{(v,v)}{2}$.
\end{thm}

From now on, we keep the assumptions of the above theorem.
Usually, the moduli space $\mathcal{M}  := \mathcal{M}_H(v)$ is not
fine.  Namely there is no universal sheaf on $\mathcal{M} \times S$.
However, if we allow twisted sheaves, it turns out that there is
always a twisted universal sheaf. We use $\pi_{\mathcal{M}}  :
\mathcal{M}\times S \rightarrow \mathcal{M}$ and
$\pi_S :\mathcal{M}\times S\rightarrow S$ to denote the projections.

\begin{prop}[{\cite[Proposition 3.3.2 and 3.3.4]{caldararu thesis}}]
  \label{prop twisted universal sheaf}
  There exist a unique element $\theta\in \mathrm{Br}'(\mathcal{M})$
  and a $\pi_{\mathcal{M}}^*\theta$-twisted universal sheaf
  $\mathscr{E}$ on $\mathcal{M}\times S$.
\end{prop}

Let $p_{ij}$ be the projections from $\mathcal{M}\times \mathcal{M}
\times S$ to the product of the $i^{\text{th}}$ and $j^{\text{th}}$
factors. Then the relative extension sheaf
$\mathscr{E}xt^i_{p_{12}}(p_{13}^*\mathscr{E}, p_{23}^*\mathscr{E})$,
$i\geq 0$, becomes a $(p_1^*\theta^{-1}p_2^*\theta)$-twisted sheaf on
$\mathcal{M}\times \mathcal{M}$, where $p_1$ and $p_2$ are the
projections of $\mathcal{M}\times \mathcal{M}$ onto the corresponding
factors. For our purpose, the sheaf
$\mathscr{E}xt^1_{p_{12}}(p_{13}^*\mathscr{E}, p_{23}^*\mathscr{E})$
will be the most important. We will restrict ourselves to the case
when $\mathscr{E}$ is untwisted.

\begin{lem}\label{lem support of ext}
  Assume that $\mathscr{E}$ is a universal sheaf on $\mathcal{M}\times
  S$. Then both sheaves $\mathscr{E}xt^0_{p_{12}}(p_{13}^*\mathscr{E},
  p_{23}^*\mathscr{E})$ and
  $\mathscr{E}xt^2_{p_{12}}(p_{13}^*\mathscr{E}, p_{23}^*\mathscr{E})$
  are supported on the diagonal of $\mathcal{M}\times \mathcal{M}$.
  For any point $t\in \mathcal{M}$, let $\mathscr{E}_t$ be the
  corresponding stable sheaf on $S$. Then both
  $\mathscr{E}xt_{\pi_{\mathcal{M}}}^0(\mathscr{E},
  \pi_S^*\mathscr{E}_t)$ and
  $\mathscr{E}xt^2_{\pi_{\mathcal{M}}}(\mathscr{E},{\pi_S}^*\mathscr{E}_t)$
  are supported at the point $t$.
\end{lem}
\begin{proof}
Let $E_1$ and $E_2$ be two stable sheaves on $S$, then
$$
\mathrm{Hom}(E_1,E_2)=\begin{cases}
 0, & E_1\ncong E_2 ;\\ \C, & E_1\cong E_2.
\end{cases}
$$
This implies that $\mathscr{E}xt^0_{p_{12}}(p_{13}^*\mathscr{E},
p_{23}^*\mathscr{E})$ is supported on the diagonal. By
Grothendieck--Serre duality, we have $\mathrm{Ext}^2(E_1,E_2)\cong
\mathrm{Hom}(E_2,E_1)^\vee$. It follows that
$\mathscr{E}xt^2_{p_{12}}(p_{13}^*\mathscr{E}, p_{23}^*\mathscr{E})$
is also supported on the diagonal. The second half of the lemma is
proved similarly.
\end{proof}

Here we recall a result of O'Grady \cite{o'grady}. Let $v\in
\tilde{\HH}(S,\Z)$ with $(v,v)>0$.  Let $$V :=v^\perp\subset
\tilde{\HH}(S,\Z)$$ be the orthogonal complement of $v$ with respect
to the Mukai pairing. Note that the restriction of the Mukai pairing
to $V$ gives a natural bilinear form on $V$. We define
$$
\mu  := \mathrm{ch}^{\mathrm{top}}(\mathscr{E})\cup \pi_S^*(1+
[pt])\in\HH^*(\mathcal{M}\times S,\Q).
$$
Let $f : \tilde{\HH}(S,\Z)\rightarrow {\HH}^*(\mathcal{M},\Q)$ be the
homomorphism sending a class $\alpha\in\tilde{\HH}(S,\Z)$ to
${\pi_{\mathcal{M}}}_*(\mu^{\vee}\cup \pi_S^*\alpha)$.  We recall from
\cite[\S6]{hl} that if $\mu=\sum\mu_i\in \bigoplus
\HH^{2i}(\mathcal{M}\times S,\Q)$, then $\mu^\vee  := \sum
(-1)^i\mu_i$.

\begin{thm}[O'Grady \cite{o'grady}]
  The restriction of $f$ to $V$ followed by the projection to the
  $\HH^2(\mathcal{M},\Z)$-component, usually denoted $\theta_v$, gives
  a canonical isomorphism of integral Hodge structures
$$
\theta_v :V\rightarrow \HH^2(\mathcal{M},\Z)
$$
which respects the bilinear pairing on $V$ and the
Beauville--Bogomolov form on the right-hand side.
\end{thm}

Furthermore, we have the following

\begin{lem}\label{lem components of f} The homomorphism $f$ enjoys the
  following properties.
\begin{enumerate}[(i)]
\item The $\HH^2(\mathcal{M})$-component of $f(v)$ is equal to
  $\tau_1$, where $\tau_1 =
  -\mathrm{ch}_1^{\mathrm{top}}(\mathscr{E}xt^1_{\pi_{\mathcal{M}}}
  (\mathscr{E},\pi_S^*\mathscr{E}_t))$ for some closed point $t\in
  \mathcal{M}$~;

\item The $\HH^0(\mathcal{M})$-component of $f(v)$ is equal to $2-2n$
  and $f(w)$ has no $\HH^0(\mathcal{M})$-component for all $w\in V$.
  \end{enumerate}
\end{lem}
\begin{proof}
  The proof of the lemma goes along the same lines as the proof of
  \cite[Theorem 6.1.14]{hl} so that we omit the details here.
\end{proof}

\begin{prop}\label{prop kappa22}
  Assume that there is a universal sheaf $\mathscr{E}$ on
  $\mathcal{M}\times S$. Let $\mathscr{E}^i
  :=\mathscr{E}xt^i_{p_{12}}(p_{13}^*\mathscr{E},p_{23}^*\mathscr{E})$,
  $i=0,1,2$.  Then the component of
  $\kappa^{\mathrm{top}}_2(\mathscr{E}^1)$ in
  $\HH^2(\mathcal{M},\Q)\otimes \HH^2(\mathcal{M},\Q)$ is equal to
  $-\mathfrak{B}$.
\end{prop}
\begin{proof}
  Take a basis $\{w_1,\ldots,w_{23}\}$ of $V$ and set
  $\hat{w}_i=\theta_v(w_i)$. By the above result of O'Grady, we know
  that $\{\hat{w}_1,\ldots,\hat{w}_{23}\}$ forms a basis of
  $\HH^2(\mathcal{M},\Z)$. Let $\mu = \sum\mu'_{i}$ be the
  decomposition of $\mu$ such that
  $\mu'_{i}\in\HH^{2i}(\mathcal{M})\otimes \HH^*(S)$.  Let
  $A=(a_{ij})_{23\times 23}$ be the intersection matrix of the
  bilinear form on $V$ with respect to the chosen basis,
  \textit{i.e.}, $a_{ij}=(w_i,w_j)$. Note that $A$ is also the matrix
  representing the Beauville--Bogomolov form on
  $\HH^2(\mathcal{M},\Z)$ with respect to the basis $\hat{w}_i$. Let
  $B=(b_{ij})=A^{-1}$. The result of O'Grady together with the first
  part of Lemma \ref{lem components of f} implies that
$$
(\mu'_1)^\vee = -\sum_{1\leq i,j\leq 23}b_{ij}\hat{w}_i\otimes
w_j^\vee -\frac{1}{2n-2}\tau_1\otimes v^\vee, \qquad \mu'_1 =
\sum_{1\leq i,j\leq 23}b_{ij}\hat{w}_i\otimes w_j +
\frac{1}{2n-2}\tau_1\otimes v.
$$
The second part of Lemma \ref{lem components of f} implies that
$\mu'_0=[\mathcal{M}]\otimes v$. Then we can compute
\begin{align*}
  \nu &  :=
  \mathrm{ch}^{\mathrm{top}}(p_{12_!}(p_{13}^![\mathscr{E}]^\vee
  \cdot p_{23}^![\mathscr{E}])\\
  & = p_{12_*}(p_{13}^*\mathrm{ch}^{\mathrm{top}}(\mathscr{E})^\vee
  \cup p_{23}^*\mathrm{ch}^{\mathrm{top}}(\mathscr{E})
  \cup p_3^*\mathrm{td}_S) \\
  & = p_{12_*} (p_{13}^*\mu^\vee\cup p_{23}^*\mu),
\end{align*}
where the second equation uses the Grothendieck--Riemann--Roch
theorem.  Denoting $[\nu]_{i,j}$ the $\HH^{i}(\mathcal{M})\otimes
\HH^{j}(\mathcal{M})$-component of $[\nu]$, we readily get
\begin{align*}
  [\nu]_{2,2}& = p_{12_*}(p_{12}^*(\mu'_1)^\vee \cup p_{23}^*\mu'_1) \\
  &= \sum_{i,j=1}^{23} \sum_{i',j'=1}^{23} b_{ij}b_{i'j'} a_{jj'}
  \hat{w}_i \otimes \hat{w}_{i'} +\frac{1}{2n-2}\tau_1\otimes \tau_1\\
  &= \sum_{i,i'=1}^{23} b_{ii'}\hat{w}_i\otimes \hat{w}_{i'}
  +\frac{1}{2n-2} \tau_1\otimes \tau_1\\
  &= \mathfrak{B} +\frac{1}{2n-2}\tau_1\otimes \tau_1.
\end{align*}
Similarly, we find $[\nu]_{0,2}=-[\mathcal{M}]\otimes\tau_1$ and
$[\nu]_{2,0}=\tau_1\otimes [\mathcal{M}]$ from the second part of
Lemma \ref{lem components of f}. Note that
$$
p_{12,!}(p_{13}^![\mathscr{E}]^\vee\cdot p_{23}^![\mathscr{E}]) = \sum
(-1)^i[\mathscr{E}^i] \ \in K_0(\mathcal{M}\times \mathcal{M})$$ and
recall that by Lemma \ref{lem support of ext}, $\mathscr{E}^0$ and
$\mathscr{E}^2$ do not contribute to $\mathrm{ch}_1$ or
$\mathrm{ch}_2$. Hence we get
\begin{align*}
  &[\mathrm{ch}^{\mathrm{top}}_2(\mathscr{E}^1)]_{2,2} = -[\nu]_{2,2}
  =
  -\mathfrak{B} - \frac{1}{2n-2}p_1^*\tau_1\cup p_2^*\tau_1,\\
  &\mathrm{ch}^{\mathrm{top}}_1(\mathscr{E}^1) = -[\nu]_{0,2}
  -[\nu]_{2,0} =-p_1^*\tau_1 +p_2^*\tau_1.
\end{align*}
Then the $\kappa$-class of $\mathscr{E}^1$ can be computed as follows
\begin{align*}
  \kappa^{\mathrm{top}}(\mathscr{E}^1) & =
  \mathrm{ch}^{\mathrm{top}}(\mathscr{E}^1)\cdot
  \mathrm{exp}\Big(-\frac{c_1^{\mathrm{top}}(\mathscr{E}^1)}{2n-2}\Big)\\
  & = \Big(2n-2 + \mathrm{ch}^{\mathrm{top}}_1(\mathscr{E}^1)
  +\mathrm{ch}^{\mathrm{top}}_2(\mathscr{E}^1) +\cdots\Big) \Big(1
  -\frac{\mathrm{ch}_1^{\mathrm{top}}(\mathscr{E}^1)}{2n-2} +
  \frac{(\mathrm{ch}_1^{\mathrm{top}}(\mathscr{E}^1))^2}{2(2n-2)^2}
  +\cdots\Big).
\end{align*}
Therefore
$$
\kappa_2^{\mathrm{top}}(\mathscr{E}^1) =
\mathrm{ch}_2^{\mathrm{top}}(\mathscr{E}^1)
-\frac{1}{4n-4}(\mathrm{ch}_1^{\mathrm{top}}(\mathscr{E}^1))^2.
$$
The Lemma follows by taking the
$\HH^2(\mathcal{M})\otimes\HH^2(\mathcal{M})$-component.
\end{proof}

\subsection{Markman's twisted sheaves}
Let $F$ be a hyperk\"ahler manifold of $\mathrm{K3}^{[n]}$-type,
\textit{i.e.}, deformation equivalent to the Hilbert scheme of
length-$n$ subschemes of a $\mathrm{K3}$ surface, where $n\geq 2$ is
an integer. In \cite{markman}, Markman constructed a twisted sheaf
$\mathscr{M}$ on $F\times F$ which is the deformation of some sheaf
that is well understood.

\begin{thm}[Markman \cite{markman}]\label{thm markman}
  Let $F$ be a hyperk\"ahler manifold of $\mathrm{K3}^{[n]}$-type.
  There exists a $\mathrm{K3}$ surface $S$ together with a Mukai
  vector $v\in\tilde{\HH}(S,\Z)$ with primitive $v_1$ and a
  $v$-generic ample divisor $H$ such that there is a proper flat
  family $\pi : \mathcal{X}\rightarrow C$ of compact hyperk\"ahler
  manifolds satisfying the following conditions~:
\begin{enumerate}[(i)]
\item The curve $C$, which is connected but possibly reducible, has
  arithmetic genus 0.

\item There exist $t_1,t_2\in C$ such that $\mathcal{X}_{t_1}=F$ and
  $\mathcal{X}_{t_2} = \mathcal{M}_H(v)$.

 \item A universal sheaf $\mathscr{E}$ exists on
  $\mathcal{M}_H(v)\times S$.

  \item There is a torsion-free reflexive coherent twisted sheaf
  $\mathscr{G}$ on $\mathcal{X}\times_C \mathcal{X}$, flat over $C$,
  such that $\mathscr{G}_{t_2} :=\mathscr{G}|_{\mathcal{X}_{t_2}}$ is
  isomorphic to $\mathscr{E}xt^1_{p_{12}}(p_{13}^*\mathscr{E},
  p_{23}^*\mathscr{E})$.
  \end{enumerate}
\end{thm}

\begin{defn} \label{def Markman sheaf} The twisted sheaf
  $\mathscr{G}_{t_1}$ on $F\times F$ obtained in Theorem \ref{thm
    markman} will be denoted $\mathscr{M}$ and, although it might
  depend on choices, it will be referred to as \emph{Markman's sheaf}.
\end{defn}

However, the low-degree cohomological $\kappa$-classes do not depend
on choices~:
\begin{lem}\label{lem top kappa of F}
  Let $\kappa^{\mathrm{top}}_2(\mathscr{M})=\kappa_2^{0,4} +
  \kappa_2^{2,2} + \kappa_2^{4,0}$ be the decomposition of
  $\kappa^{\mathrm{top}}_2(\mathscr{M})$ into the K\"unneth components
  $\kappa_2^{i,j}\in\HH^i(F)\otimes\HH^j(F)$. Then we have

\begin{enumerate}[(i)]
\item $\kappa_2^{2,2} = -\mathfrak{B}$ in $\HH^4(F\times F,\Q)$~;

\item $\kappa_2^{4,0}=\frac{1}{2}p_1^*(\mathfrak{b} -
  c_2^{\mathrm{top}}(F))$ and
  $\kappa_2^{0,4}=\frac{1}{2}p_2^*(\mathfrak{b} -
  c_2^{\mathrm{top}}(F))$ in $\HH^4(F\times F,\Q)$~;

\item $(\iota_{\Delta})^*\kappa_2^{\mathrm{top}}(\mathscr{M}) = -
  c_2^{\mathrm{top}}(F)$ in $\HH^4(F,\Q)$, where
  $\iota_\Delta :F\rightarrow F\times F$ is the diagonal embedding.
  \end{enumerate}
\end{lem}
\begin{proof}
  Since $\mathscr{M}$ deforms to the sheaf
  $\mathscr{E}^1=\mathscr{E}xt^1_{p_{12}}(p_{13}^*\mathscr{E},
  p_{23}^*\mathscr{E})$ (by Theorem \ref{thm markman}), we only need
  to check \emph{(i)} for $\mathscr{E}^1$ which is Proposition
  \ref{prop kappa22}. Statement \emph{(ii)} follows from \cite[Lemma
  1.4]{markman}, and statement \emph{(iii)} follows from \emph{(i)}
  and \emph{(ii)}.
\end{proof}

\begin{thm}\label{thm existence of L}
  Let $F$ be a projective hyperk\"ahler manifold of
  $\mathrm{K3}^{[n]}$-type, $n\geq 2$. Then there exists
  $L\in\CH^2(F\times F)$ lifting the Beauville--Bogomolov class
  $\mathfrak{B}\in\HH^4(F\times F,\Q)$.
\end{thm}
\begin{proof}
  Let $e(F)=\deg(c_{2n}(F))$ be the Euler number of $F$. We first
  define
$$
l :=\frac{2}{e(F)}p_{2,*}(\kappa_2(\mathscr{M})\cdot p_1^*c_{2n}(F)) +
c_2(F) \ \in\CH^2(F),
$$
which, by \emph{(ii)} of Lemma \ref{lem top kappa of F}, lifts the
class $\mathfrak{b}$. Then we set
\begin{equation*}
  L  := -\kappa_2(\mathscr{M}) +\frac{1}{2}p_1^*(l - c_2(F)) +\frac{1}{2}p_2^*(l
  - c_2(F)) \ \in\CH^2(F\times F).
\end{equation*}
By \emph{(i)} and \emph{(ii)} of Lemma \ref{lem top kappa of F}, we
see that $L$ lifts $\mathfrak{B}$.
\end{proof}

\begin{rmk}
  With definitions as in the proof of Theorem \ref{thm existence of
    L}, the equality $\iota_{\Delta}^*L=l$ holds if and only if item
  \emph{(iii)} of Lemma \ref{lem top kappa of F} holds at the level of Chow
  groups.
\end{rmk}

\begin{rmk}
  Let $S$ be a K3 surface and let $F=S^{[n]}$ be the Hilbert scheme of
  length-$n$ subschemes of $S$.  Let $Z\subset F\times S$ be the
  universal family of length-$n$ subschemes of $S$ and let
  $\mathcal{I}$ be the ideal sheaf defining $Z$ as a subscheme of
  $F\times S$.  In this way, $F$ is realized as the moduli space of
  stable sheaves with Mukai vector $v=(1,0,1-n)$, and $\mathcal{I}$ is
  the universal sheaf. In this situation, there is no need to deform
  and Markman's sheaf is simply given by $\mathscr{M} :=
  \mathscr{E}xt^1_{p_{12}}(p_{13}^*\mathcal{I},p_{23}^*\mathcal{I})$. We
  can then define, for $F=S^{[n]}$,
  \begin{equation} \label{def L Markman sheaf} L  :=
    -\kappa_2(\mathscr{M}) +\frac{1}{2}p_1^*(l - c_2(F))
    +\frac{1}{2}p_2^*(l - c_2(F)) \ \in\CH^2(F\times F),
  \end{equation}
  where $l  :=\frac{2}{e(F)}p_{2,*}(\kappa_2(\mathscr{M}) \cdot
  p_1^*c_{2n}(F)) + c_2(F)$. This cycle $L$ represents the
  Beauville--Bogomolov class $\mathfrak{B}$ by virtue of Theorem
  \ref{thm existence of L}. In the case where $F=S^{[2]}$, this cycle
  will be shown to coincide with the cycle constructed in \eqref{eq L
    S2}~; see Proposition \ref{prop L agree}.
\end{rmk}


\newpage

\vspace{1pt}
\begin{large}
\part{The Hilbert scheme $S^{[2]}$}
\end{large}

\vspace{10pt}
\section{Basics on the Hilbert scheme of length-$2$ subschemes on a
  variety $X$}\label{sec basics X2}

Let $X$ be a smooth projective variety of dimension $d$. Let
$F=X^{[2]}$ be the Hilbert scheme of length-$2$ subschemes on $X$ and
let
$$
\xymatrix{
  Z\ar[d]_p\ar[r]^q & X \\
  F & }
$$
be the universal family seen as a correspondence between $F$ and
$X$. For any $w\in F$, we use $Z_w := q(p^{-1}w)$ to denote the
corresponding subscheme of $X$. Note that $Z$ is naturally isomorphic
to the blow-up of $X\times X$ along the diagonal.  Let $\rho :
Z\rightarrow X\times X$ be the blow-up morphism with $E\subset Z$ the
exceptional divisor. Then $E$ is naturally identified with the
geometric projectivization $\PP(\mathscr{T}_X) := \mathrm{Proj}
\left(\mathrm{Sym}^\bullet\, \Omega_X^1 \right)$ of the tangent bundle
$\mathscr{T}_X$ of $X$. Let $\pi :E\rightarrow X$ be the natural
projection. Then $q$ can be chosen to be $p_1 \circ\rho$, where $p_1
:X\times X\rightarrow X$ is the projection onto the first factor. Let
$\Delta\subset F$ be the points corresponding to nonreduced subschemes
of $X$, or equivalently $\Delta=p(E)$. As varieties, we have
$\Delta\cong E$.  We will identify $\Delta$ and $E$ with
$\PP(\mathscr{T}_X)$ and accordingly we have closed immersions $j
:\PP(\mathscr{T}_X)\hookrightarrow F$ and $j'
:\PP(\mathscr{T}_X)\hookrightarrow Z$. Hence we have the following
diagram
\begin{equation}\label{eq basic blow-up diagram}
 \xymatrix{
 X &E\ar[r]^{j'}\ar@{=}[d]\ar[l]_{\pi} & Z\ar[r]^{\rho\quad}\ar[d]^p &X
\times X\\
 &\PP(\mathscr{T}_X)\ar[r]^{j} &F &
}
\end{equation}
Let $\delta$ be the cycle class of $\frac{1}{2}\Delta$. By abuse of
notation, we also use $\delta$ to denote its cohomological class in
$\HH^2(F,\Z)$. If we further assume that $\HH^1(X,\Z)=0$, then we
see that
\begin{equation}\label{decomposition H2 general}
\HH^2(F,\Z) = p_*q^*\HH^2(X,\Z) \oplus \Z\delta.
\end{equation}
We also note that the double cover $Z\rightarrow F$ naturally
induces an involution $\tau :Z\rightarrow Z$.

\begin{rmk}
  If $X=S$ is a K3 surface, then $F$ is a hyperk\"ahler manifold~; see
  \cite{beauville0}. The homomorphism $[Z]^* :\HH^2(S,\Z)\rightarrow
  \HH^2(F,\Z)$ induces the following orthogonal direct sum
  decomposition with respect to the Beauville--Bogomolov form $q_F$  :
\begin{equation}\label{eq decomposition H2}
  \HH^2(F,\Z)=\HH^2(S,\Z) \ \hat{\oplus} \ \Z\delta.
\end{equation}
Furthermore, $q_F(\delta,\delta)=-2$ and $q_F$ restricted to
$\HH^2(S,\Z)$ is the intersection form on $S$.
\end{rmk}

For any given point $x\in X$, we can define a smooth variety
$X_x=p(q^{-1}x)\subset F$. The variety $X_x$ is isomorphic to the
blow-up of $X$ at the point $x$ and it represents the cycle $Z^*x$.
For two distinct points $x,y\in X$, we use $[x,y]\in F$ to denote the
point of $F$ that corresponds to the subscheme $\{x,y\}\subset X$.
When $x=y$, we use $[x,x]$ to denote the element in $\CH_0(F)$
represented by any point corresponding to a nonreduced subscheme of
length $2$ of $X$ supported at $x$.  As cycles, we have
\begin{equation}
\label{eq XxXy} X_x\cdot X_y = [x,y].
\end{equation}

Intersecting with $\delta$ is unveiled in the following lemma.

\begin{lem}\label{lem intersection identity X2}
  Let $\pi :\PP(\mathscr{T}_X)\rightarrow X$ be the natural morphism
  and $\xi \in \CH^1\big(\PP(\mathscr{T}_X)\big)$ be the first Chern
  class of the relative $\calO(1)$ bundle.  Then we have the following
  identities in $\CH^*(F)$,
\begin{align*}
  &\delta^k =\frac{(-1)^{k-1}}{2}j_*\xi^{k-1},\quad k=1,2,\ldots,2d \, ;\\
  &\delta^k\cdot X_x =(-1)^{k-1}(j_x)_*(\xi_x^{k-1}), \quad
  k=1,2,\ldots,d,
\end{align*}
where $j_x :\pi^{-1}x=\PP(\mathscr{T}_{X,x})\rightarrow F$ is the
natural closed immersion and $\xi_x$ is the class of a hyperplane on
$\PP(\mathscr{T}_{X,x})$. In particular, we have 
\begin{equation}
  \label{eq delta Xx} \delta^d\cdot X_x=(-1)^{d-1}[x,x] 
  \quad \mbox{in} \ \CH_0(F).
\end{equation}
\end{lem}
\begin{proof}
  First recall that $p : Z \rightarrow F$ has degree $2$ and that
  $\delta=\frac{1}{2}p_*E$. By the projection formula, we
  get $$p_*(E^k) = p_*p^*(\delta^k) = \delta^k\cdot p_*(Z) =
  2\delta^k.$$ Note that
  $E^k=j'_*(-\xi)^{k-1}=(-1)^{k-1}j'_*\xi^{k-1}$. It follows that
  $\delta^k=\frac{(-1)^{k-1}}{2}j_*\xi^{k-1}$, $k=1,\ldots,2d-1$. Now,
\begin{align*}
  \delta^k\cdot X_x =\frac{1}{2}p_*E^k\cdot X_x
  =\frac{1}{2}p_*(E^k\cdot p^*X_x) = \frac{1}{2}p_*(E^k\cdot
  (\widetilde{X}_x +\tau^*\widetilde{X}_x)) =p_*(E^k\cdot
  \widetilde{X}_x),
\end{align*}
where $\widetilde{X}_x=q^{-1}(x)$. Note that $E\cdot \widetilde{X}_x =
(j'_x)_*\PP(\mathscr{T}_{X,x})$, where $j'_x =j'|_{\pi^{-1}x}$ and
that $(j'_x)^*E = -\xi_x$. We get
\[
E^k\cdot \widetilde{X}_x =E^{k-1}\cdot (j'_x)_*\PP(\mathscr{T}_{X,x})
=(j'_x)_*(j'_x)^*(E^{k-1}) =(-1)^{k-1}(j'_x)_*(\xi_x^{k-1}).
\]
Apply $p_*$ to the above equation and the second equality of the lemma
follows since $j_x=p\circ j'_x$.
\end{proof}


\vspace{10pt}
\section{The incidence correspondence $I$} \label{sec I X2} Let $X$ be
a smooth projective variety of dimension $d$ and let $F :=X^{[2]}$ be
the Hilbert scheme of length-$2$ subschemes on $X$. We introduce the
incidence correspondence $I \in \CH^d(F \times F)$ and establish a
quadratic equation satisfied by $I$~; see Proposition \ref{prop I
  square on S2}.  We also compute, in Lemma \ref{lem diagonal
  pull-back of I}, the pull-back of $I$ along the diagonal embedding
$\iota_\Delta : F \hookrightarrow F \times F$. The first result will
already make it possible to obtain in that generality a splitting of
the Chow group of zero-cycles on $X^{[2]}$~; see Section \ref{sec mult
  X2}. When $X$ is a K3 surface, the incidence cycle $I$ will be
closely related to the crucial lifting of the Beauville--Bogomolov
class $\mathfrak{B}$, and Proposition \ref{prop I square on S2} and
Proposition \ref{lem diagonal pull-back of I} will constitute the
first step towards establishing the quadratic equation \eqref{eq
  rational equation} of Conjecture \ref{conj main}~; see \S \ref{sec
  Fourier S2}.  \medskip

The notations are those of the previous section. Consider the diagram
$$
\xymatrix{
  Z\times Z\ar[r]^{q\times q}\ar[d]_{p\times p} &X\times X\\
  F\times F & }
$$
\begin{defn}\label{defn incidence correspondence S2}
The \textit{incidence subscheme} $I\subset F\times F$ is defined by
\[
I  : = \{ (w_1,w_2)\in F\times F : \text{there exists a closed point
}x\in X \text{ such that } x\in Z_{w_i}, i=1,2\},
\]
where the right hand side is equipped with the reduced closed
subscheme structure. The \emph{incidence correspondence} is the cycle
class of $I$ in $\CH^2(F \times F)$ and, by abuse of notations, is
also denoted $I$.
\end{defn}

The following lemma justifies this abuse of notations.

\begin{lem}\label{lem cycle class of I on S2} The incidence subscheme
  $I$ enjoys the following properties.
\begin{enumerate}[(i)]
\item The sub-variety $I$ is irreducible and its cycle class is given
  by
\[
I=(p\times p)_*(q\times q)^{*}\Delta_X \quad \mbox{in} \ \CH^d(F
\times F).
\]
Equivalently, if one sees the universal family $Z$ of length-$2$
subschemes on the variety $X$ as a correspondence between $F =
X^{[2]}$ and $X$, then $$ I = {}^tZ \circ Z ,\quad \text{in
}\CH^d(F\times F).$$
\item The diagonal $\Delta_F$ is contained in $I$ and
  $I\backslash\Delta_F$ is smooth.
\item The action of $I$ on points is given by
 \begin{equation}
 I_*[x,y] = X_x+X_y.
 \end{equation}
\end{enumerate}
\end{lem}

\begin{proof}
  Recall that $Z$ is the blow-up of $X\times X$ along the diagonal and
  that $q :Z\rightarrow X$ is the composition of the blow-up morphism
  $Z \rightarrow X \times X$ with the first projection $p_1 :X\times
  X\rightarrow X$. As a consequence, $q$ is a smooth morphism and
  $q^{-1}x\cong \mathrm{Bl}_x(X)$ is the blow-up of $X$ at the closed
  point $x\in X$. It follows that $q\times q : Z\times Z\rightarrow
  X\times X$ is smooth. Hence $Z_{\Delta} :=(q\times
  q)^{-1}\Delta_X\subset Z\times Z$ is a smooth sub-variety whose
  cycle class is $(q\times q)^*\Delta_X$. Note that the image of
  $p\times p$ restricted to $Z_{\Delta}$ is $I$. In particular, $I$ is
  irreducible.  To prove \emph{(i)}, it suffices to show that the
  morphism $(p\times p)|_{Z_{\Delta}}$ is generically one-to-one. A
  general point of $I$ has the form $([x,y],[y,z])$, where $x,y,z\in
  X$ are distinct points and the inverse image of $([x,y],[y,z])$ in
  $Z\times Z$ is
\[
 \{(\rho^{-1}(x,y),\rho^{-1}(y,z)), (\rho^{-1}(y,x),\rho^{-1}(y,z)),
(\rho^{-1}(x,y),\rho^{-1}(z,y)),
(\rho^{-1}(y,x),\rho^{-1}(z,y))\}.
\]
Among these four points, the only point contained in $Z_\Delta$ is
$(\rho^{-1}(y,x),\rho^{-1}(y,z))$. Since the degree of $p\times p$ is
4, this proves that $(p\times p)|_{Z_\Delta}$ has degree 1 and hence
statement \emph{(i)} follows.

To prove \emph{(ii)}, we note that the fiber of the natural morphism
$I\rightarrow F$ (projection to the first factor) over a point
$[x,y]\in F$ is $[x,y]\times (X_x\cup X_y)$, which is smooth away from
the point $([x,y],[x,y])$ if $x\neq y$.  Hence the singular locus of
$I$ is contained in the union of $\Delta_F$ and $E\times_X E$.  Note
that a length-$2$ subscheme supported at a point $x$ is given by a
1-dimensional sub-space $\C v\subset \mathscr{T}_{X,x}$~; we use
$(x,\C v)$ to denote the corresponding point of $F$.  Let $u=((x,\C
v_1),(x,\C v_2))\in I$ be a point supported on $E\times_X E$.  Let
$\tilde{u}\in Z\times Z$ be the unique lift of $u$, and let
\[
t\mapsto \varphi_i (t)\in X, \quad t\in\C,\, |t|<\epsilon,
\]
be a germ of a smooth analytic curve with $\varphi'_i(0)\in \C v_i$,
$i=1,2$. We define
\[
 \tilde{\varphi}_i : \{t\in \C : |t|< \epsilon\} \longrightarrow Z
\]
to be the strict transform of the curve
\[
 t\mapsto (\varphi_i(t), \varphi_i(-t))\in X\times X.
\]
Let $S\subset Z\times Z$ be the surface parameterized by
\[
 (t_1,t_2)\mapsto (\tilde{\varphi}_1(t_1),\tilde{\varphi}_2(t_2))\in Z\times Z.
\]
Then $\mathscr{T}_{S,\tilde{u}}\subset \mathscr{T}_{Z\times
  Z,\tilde{u}}$ is the kernel of $d(p\times p) : \mathscr{T}_{Z\times
  Z,\tilde{u}} \rightarrow \mathscr{T}_{F\times F, u}$. The
scheme-theoretic intersection of $Z_\Delta$ and $S$ is given by the
equation $\varphi_1(t_1)=\varphi_2(t_2)$. It follows that $S$
intersects $I$ at the point $\tilde{u}$ with multiplicity 1 if $\C
v_1\neq \C v_2$. This implies that the map
$\mathscr{T}_{Z_\Delta,\tilde{u}}\rightarrow \mathscr{T}_{I,u}$ is an
isomorphism. In particular, $I$ is smooth at $u$ if $\C v_1\neq \C
v_2$ or, equivalently, if $u\notin \Delta_F$. Hence we conclude that
$I\backslash \Delta_F$ is smooth. Statement \emph{(iii)} follows directly
from the definition of $I$.
\end{proof}

\begin{defn}
Let $\sigma\in\CH^r(X)$ ~; we define an element $\Gamma_\sigma
\in\CH^{d+r}(F \times F)$ as follows.
\[
 \Gamma_\sigma  := (p\times p)_* (q\times q)^*
 (\iota_{\Delta_X})_*\sigma,
\]
where $\iota_{\Delta_X} :X\rightarrow X\times X$ is the diagonal
embedding.
\end{defn}

\noindent Note that, if $\sigma$ is represented by an irreducible
closed sub-variety $Y\subset X$, then $\Gamma_\sigma$ is the cycle
class of
\begin{equation}\label{eq Gamma_Y}
  \Gamma_Y :=\{(w_1,w_2)\in F\times F : \exists y\in Y,\, y\in Z_{w_1}\text{
    and } y\in Z_{w_2}\}
\end{equation}
In particular, if we take $Y=X$, then we have $\Gamma_X = I$.\medskip

The following proposition is the analogue of an identity of Voisin,
stated in \eqref{identity of voisin2}, which was originally proved in
the case of varieties of lines on cubic fourfolds.

\begin{prop}\label{prop I square on S2}
  Let $\delta_i$, $i=1,2$, be the pull-back of $\delta$ via the
  projection of $F\times F$ to the $i^\text{th}$ factor. We write
\[
y_k  := (-1)^{k-1}(\delta_1^k -\sum_{i=1}^{k-1} \delta_1^i\delta_2^{k-i}
+\delta_2^k),\quad k=1,2,\ldots, 4d,\,\,d=\dim X,
\]
and $y_0 :=[F \times F]$. Then the cycle $I$ satisfies the following
equation in $\CH^{2d}(F\times F)$,
\begin{equation}\label{eq I equation on S2}
I^2 = 2\Delta_F + y_d\cdot I +\sum_{k=0}^{d-1} y_k\cdot
\Gamma_{c_{d-k}(X)}.
\end{equation}
\end{prop}
\begin{proof}
  To prove this result, we first give a different description of the
  cycle $I$. Consider the following diagram
$$
\xymatrix{
 Z\times F\ar[r]^{q\times 1}\ar[d]_{p\times 1} &X\times F\\
 F\times F &
}
$$
The variety $Z$ is naturally a sub-variety of $X\times F$. Then
$I=(p\times 1)_*(q\times 1)^*Z$. Let $\tilde{I}=(q\times
1)^{-1}Z\subset Z\times F$. Note that $\tilde{I}$ is smooth since
$q\times 1$ is a smooth morphism. Then we see that $\tilde{I}$ is a
local complete intersection and
$$
\mathscr{N}_{\tilde{I}/Z\times F}\cong (q\times
1|_{\tilde{I}})^*\mathscr{N}_{Z/ X\times F}.
$$
The morphism $p\times 1$ induces a homomorphism
$$
\vartheta=d(p\times 1) :\mathscr{N}_{\tilde{I}/Z\times F}
\longrightarrow (p\times 1|_{\tilde{I}})^*\mathscr{N}_{I/F\times F}.
$$
The homomorphism $\vartheta$ fits into the following commutative
diagram with short exact rows and columns
$$
\xymatrix{ 0\ar[r] &\mathscr{T}_{\tilde{I}}\ar[rr]\ar[d] &&
  \mathscr{T}_{Z\times F}|_{\tilde{I}}\ar[rr]\ar[d]
  && \mathscr{N}_{\tilde{I}/Z\times F}\ar[d]^\vartheta\ar[r] &0\\
  0\ar[r] &(p\times 1|_{\tilde{I}})^*\mathscr{T}_{I}\ar[rr]\ar[d] &&
  (p\times 1|_{\tilde{I}})^*\mathscr{T}_{F\times F}|_I\ar[rr]\ar[d] &&
  (p\times 1|_{\tilde{I}})^*\mathscr{N}_{I/F\times F}\ar[r]\ar[d] &0\\
  0\ar[r] &\mathfrak{Q}'\ar[rr]&&p_1^*\calO_E(2E)|_{\tilde{I}} \ar[rr]
  &&\mathfrak{Q}\ar[r] &0 }
$$
where $p_i$, $i=1,2$, are the projections of $Z\times F$ onto the two
factors and all the sheaves are viewed as their restrictions to
$\tilde{I}\backslash (p\times 1)^{-1}\Delta_F$. In the above diagram,
the middle column follows from the following fundamental short exact
sequence
$$
\xymatrix{ 0\ar[r] &\mathscr{T}_Z\ar[rr] & & p^*\mathscr{T}_F\ar[rr]
&&\calO_E(2E)\ar[r] &0. }
$$
Note that the morphism $p\times 1|_{\tilde{I}} : \tilde{I}\rightarrow
I$ is an isomorphism away from $\Delta_F\subset I$ ~; it follows that
$\mathfrak{Q}'=0$ away from $(p\times 1)^{-1}\Delta_F$. Hence a Chern
class computation on $\tilde{I}\backslash(p\times 1)^{-1}\Delta_F$
gives
\begin{align}
  c((p\times 1|_{\tilde{I}})^*\mathscr{N}_{I/F\times F}) & =
  c(p_1^*\calO_E(2E))|_{\tilde{I}}\cdot
  c(\mathscr{N}_{\tilde{I}/Z\times F})
  \nonumber \\
  & = c(p_1^*\calO_E(2E))|_{\tilde{I}} \cdot c((q\times
  1|_{\tilde{I}})^*\mathscr{N}_{Z/ S\times F}). \label{eq chernItilde}
\end{align}
We have the following short exact sequence
$$
\xymatrix{ 0\ar[r] &\calO_Z(E)\ar[rr] &&\calO_Z(2E)\ar[rr]
  &&\calO_E(2E)\ar[r] &0. }
$$
From this exact sequence we get
\begin{equation} \label{eq chern2E}
c(\calO_E(2E))=\frac{1+2E}{1+E}=1 + E -E^2 +E^3 -E^4+\cdots
-E^{2d},\quad \text{in }\CH^*(Z), d=\dim X.
\end{equation}
In order to compute the Chern classes of $\mathscr{N}_{Z/ S\times F}$,
we consider the following diagram
$$
\xymatrix{
  && q^*\mathscr{T}_X\ar@{=}[r]\ar[d] &q^* \mathscr{T}_X\ar[d] &\\
  0\ar[r] &\mathscr{T}_Z\ar[r]\ar@{=}[d] &\mathscr{T}_{X\times
    F}|_{Z}\ar[r]\ar[d] &
  \mathscr{N}_{Z/ X\times F}\ar[d]\ar[r] & 0 \\
  0\ar[r] &\mathscr{T}_Z\ar[r] &p^*\mathscr{T}_F\ar[r]
  &\calO_E(2E)\ar[r] &0 }
$$
where all the rows and columns are short exact. As a result we have
\begin{equation}\label{eq Chern class of normal bundle}
  c(\mathscr{N}_{Z/ X\times F}) = c(q^*\mathscr{T}_X)\cdot
  c(\calO_E(2E)) = q^*c(X) \cdot (1+E-E^2+\cdots
  -E^{2d}),
\end{equation}
where $c(X) = 1+ c_1(X)+\cdots +c_d(X)$. Substituting \eqref{eq
  chernItilde} and \eqref{eq chern2E} into \eqref{eq Chern class of
  normal bundle} yields
\begin{align}
  c((p\times 1|_{\tilde{I}})^* \mathscr{N}_{I/F\times F}) &=
  p_1^*(1+E-E^2+E^3-E^4+\cdots -E^{2d})|_{\tilde{I}}\cdot (q\times
  1|_{\tilde{I}})^* q^*
  c(X) \label{eq bigeq}\\
  &\quad \cdot
  (q\times 1|_{\tilde{I}})^*(1+E-E^2+E^3-E^4+\cdots -E^{2d}) \nonumber\\
  &= p_1^*p^*(1+\delta-\delta^2+\delta^3-\delta^4+\cdots
  -\delta^{2d})|_{\tilde{I}}\cdot p_1^* q^*c(X)|_{\tilde{I}} \nonumber
  \\
  &\quad \cdot (q\times 1|_{\tilde{I}})^*p^*(1+\delta-\delta^2 +
  \delta^3-\delta^4+\cdots -\delta^{2d}) \nonumber\\
  &= p_1^*p^*(1+\delta-\delta^2+\delta^3-\delta^4+\cdots
  -\delta^{2d})|_{\tilde{I}}\cdot p_1^*q^*c(X)|_{\tilde{I}} \nonumber
  \\
  &\quad\cdot p_2^*(1+\delta-\delta^2+\delta^3-\delta^4+\cdots
  -\delta^{2d})|_{\tilde{I}}\nonumber \\
  &=(p\times 1)^*\{(1+\delta_1-\delta_1^2+\delta_1^3-\delta_1^4+\cdots
  -\delta^{2d})
  \nonumber\\
  &\quad \cdot(1+\delta_2-\delta_2^2 + \delta_2^3-\delta_2^4+\cdots
  -\delta^{2d})\}|_{\tilde{I}}\cdot p_1^*q^*c(X)|_{\tilde{I}},
  \nonumber
\end{align}
when restricted to $\tilde{I} \backslash (p\times 1)^{-1}\Delta_F$. The
second equation follows from the fact that $p^*\delta=E$ and that the
composition $q\circ(q\times 1)|_{\tilde{I}}$ is equal to $(q\circ
p_1)|_{\tilde{I}}$. The third equation holds since the composition of
$q\times 1|_{\tilde{I}} :\tilde{I}\rightarrow Z$ and $p :Z\rightarrow F$
is simply the second projection $p_2 :Z\times F\rightarrow F$
restricted to $\tilde{I}$. Now we need the following lemma.

\begin{lem}\label{lem Gamma alpha dot I}
Let $\sigma\in \CH^r(X)$ be a cycle. Then
\[
 (p\times 1)_*(p_1^*q^*\sigma \cdot \tilde{I} )=\Gamma_\sigma,
\]
where $p_1 :Z \times F \rightarrow Z$ is the projection to the first
factor and $p\times 1 : Z\times F\rightarrow F\times F$ is the
natural morphism.
\end{lem}
\begin{proof}
  We may assume that $\sigma$ is represented by an irreducible closed
  sub-variety $Y\subset X$. Then by definition, the class
  $p_1^*q^*\sigma \cdot \tilde{I}$ is represented by the cycle
\[
 \{(z,w)\in Z\times F : q(z)\in Y\text{ and } q(z)\in Z_w \},
\]
whose push-forward to $F \times F$ is easily seen to be $\Gamma_Y$~;
see equation \eqref{eq Gamma_Y}.
\end{proof}

Then the above equation \eqref{eq bigeq} for $c((p\times
1|_{\tilde{I}})^*\mathscr{N}_{I/F\times F})$ implies
\[
c_d((p\times 1|_{\tilde{I}})^* \mathscr{N}_{I/F\times F}) =
p_1^*q^*c_n(X)|_{\tilde{I}} + \sum _{k=1}^{d}(p\times 1)^*y_k\cdot
p_1^*q^*c_{d-k}(X)|_{\tilde{I}}.
\]
We apply $(p\times 1|_{\tilde{I}})_*$ to the above identity followed
by $(i_0)_*$, where $i_0 :I_0=I\backslash\Delta_F \rightarrow F\times
F\backslash \Delta_F$ is the natural inclusion. Then we get
\[
(i_0)_*c_d(\mathscr{N}_{I/F\times F}) = \sum_{k=0}^{d} y_k\cdot
(p\times 1)_* (p_1^*q^*c_{d-k}(X)\cdot \tilde{I}), \quad \text{when
  restricted to }F\times F\backslash \Delta_F.
\]
Recall that the self-intersection of $I$ outside of the diagonal is
given by the top Chern class of its normal bundle. We then apply the
localization sequence for Chow groups and Lemma \ref{lem Gamma alpha
  dot I} to obtain
\begin{equation}\label{eq with alpha}
 I^2 = \alpha\Delta_F + y_d\cdot I +\sum_{k=0}^{d-1} y_k\cdot
\Gamma_{c_{d-k}(X)}
\end{equation}
for some $\alpha\in\Z$. In order to determine the integer $\alpha$, we
will compute the degree of $(I^2)_*[x,y]$ for a general point
$[x,y]\in F$. First we have the following lemma.

\begin{lem}\label{lem action of simple terms X2}
  Let $\tau\in\CH_0(F)$, then $\Gamma_*\tau=0$ for all
  $\Gamma=y_k\cdot \Gamma_{c_{d-k}(X)}$, $k=0,1,\ldots,d-1$.
\end{lem}
\begin{proof}
Since $y_k$ is a linear combination of $\delta_1^i\delta_2^{k-i}$,
$i=0,1,\ldots,k$, we only need to show the following
\[
(\delta_1^i\delta_2^{k-i}\cdot \Gamma_{c_{d-k}(X)})_*\tau =
\delta^{k-i}\cdot (\Gamma_{c_{d-k}(X)})_*(\delta^i\cdot \tau) = 0.
\]
Note that $\delta^i\cdot \tau = 0$ unless $i=0$. It reduces to
showing that $(\Gamma_\sigma)_*\tau = 0$ for all cycles $\sigma$ on
$X$ of dimension less than $d$. Using the definition of
$\Gamma_\sigma$, we get
\[
p_1^*\tau\cdot\Gamma_\sigma = (p\times p)_*(p^{*}\tau\times
Z\cdot (q\times q)^*(\iota_{\Delta_X})_*\sigma) = 0,
\]
where the second equality holds because $q_*p^{*}\tau$ is
zero-dimensional and because $\sigma$ has dimension less than $d=\dim
X$. Thus $(\Gamma_\sigma)_*\tau = 0$.
\end{proof}

Equation \eqref{eq with alpha} above together with Lemma \ref{lem
  action of simple terms X2} shows that
\[
(I^2)_*[x,y]=\alpha[x,y] + (-1)^{d-1}\delta^d\cdot I_*[x,y].
\]
Recall from Lemma \ref{lem intersection identity X2} that
$I_*[x,y]=X_x + X_y$, and $\delta^d\cdot X_x=(-1)^{d-1}[x,x]$.  It
follows that
\[
 (I^2)_*[x,y] =\alpha\,[x,y] +[x,x] +[y,y].
\]
Meanwhile, we can calculate cohomologically the degree of $
(I^2)_*[x,y]$ as follows. Note that under the K\"unneth decomposition,
the $\HH^0(X)\otimes \HH^{2d}(X)$-component of $[\Delta_X]\in
\HH^{2d}(X)$ is given by
\[
 [\Delta_X]_{0,2d} = [X]\otimes [pt]
\]
By Lemma \ref{lem cycle class of I on S2}, we see that the
$\HH^0(F)\otimes \HH^{2d}(F)$-component of $[I]\in \HH^{2d}(F\times
F)$ is given by
\[
 [I]_{0,2d} =(p\times p)_*(q\times q)^*([X]\otimes [pt]) = 2[F]\otimes [X_x],
\]
where $x\in X$ is any point. It follows that the $\HH^0(F)\otimes
\HH^{4d}(F)$-component of $I^2$ is given by
\[
 [I^2]_{0,4d} = 4[F]\otimes [X_x]^2 = 4[F]\otimes [pt].
\]
Therefore
\[
\deg((I^2)_*[x,y])=4.
\]
This implies that $\alpha+2=4$. The proof of Proposition \ref{prop I
  square on S2} is complete.
\end{proof}

Let us conclude this section with the following proposition.

\begin{prop}\label{lem diagonal pull-back of I} Let $I \in \CH^{d}(F
  \times F)$ be the incidence correspondence. Then
  $$\iota_\Delta^*I = p_*q^*c_d(X) +\sum_{i=1}^d (-1)^{i-1}
  p_*q^*c_{d-i}(X)\cdot\delta^i,$$ where $\iota_\Delta :F\rightarrow
  F\times F$ is the diagonal embedding.
\end{prop}
\begin{proof}
  We take on the notations from the proof of Proposition \ref{prop I
    square on S2}.  The image $\Gamma_p$ of the morphism
  $(1,p) :Z\hookrightarrow Z\times F$ coincides with $(p\times
  1)^{-1}\Delta_F$. From the commutative diagram
\[
 \xymatrix{
  Z\ar[r]^{(1,p)\quad}\ar[d]_p  &Z\times F\ar[d]^{p\times 1}\\
  F\ar[r]^{\iota_\Delta\quad} &F\times F
 }
\]
we get
\begin{align*}
  p^*\iota_\Delta^*I & = (1,p)^*(p\times 1)^*I\\
  &= (1,p)^*(\tilde{I} +\tilde{I}')\\
  & = c_d((1,p)^*\mathscr{N}_{\tilde{I}/Z\times F})
  + c_d((1,p)^*\mathscr{N}_{\tilde{I}'/Z\times F})\\
  &= p^*p_*c_d((1,p)^*\mathscr{N}_{\tilde{I}/Z\times F}).
\end{align*}
Note that $\mathscr{N}_{\tilde{I}/Z\times F}\cong (q\times
1|_{\tilde{I}})^*\mathscr{N}_{Z/X\times F}$ and that $(q\times
1|_{\tilde{I}})\circ (1,p)$ is the identity morphism. Thus
$(1,p)^*\mathscr{N}_{\tilde{I}/Z\times F} \cong \mathscr{N}_{Z/X\times
  F}$. It follows that
\[
 \iota_\Delta^*I=p_*c_d(\mathscr{N}_{Z/X\times F}).
\]
By \eqref{eq Chern class of normal bundle}, we have
\[
c_d(\mathscr{N}_{Z/X\times F}) = q^*c_d(X) +
\sum_{i=1}^{d}(-1)^{i-1}q^*c_{d-i}(X)\cdot E^i.
\]
Note that $E^i=p^*\delta^i$. We apply $p_*$ to the above equation and
get
\[
\iota_\Delta^*I = p_*q^*c_d(X) +\sum_{i=1}^d (-1)^{i-1}
p_*q^*c_{d-i}(X)\cdot\delta^i.
\]
This proves the lemma.
\end{proof}

\vspace{10pt}
\section{Decomposition results on the Chow groups of
  $X^{[2]}$} \label{sec mult X2}

Let $X$ be a smooth projective variety of dimension $d$ and
$F=X^{[2]}$ as in the previous two sections. In that generality, it
follows easily from the results in the previous section that the
action of $I^2 = I \cdot I \in \CH^{2d}(F\times F)$ on $\CH_0(F)$
diagonalizes with eigenvalues $2$ and $4$~; see Proposition \ref{prop
  first dec result}. When $X$ is a K3 surface or a smooth Calabi--Yau
complete intersection, the cycle $\Delta_{\mathrm{tot}}$ introduced in
Definition \ref{defn small diagonals} is supported on $D\times X
\times X$ for some divisor $D\subset X$ (see Remark \ref{rmk support}), 
and as a consequence we give other characterizations of
the eigenspaces for the action of $(I^2)_*$ on $\CH_0(F)$; see
Propositions \ref{prop second dec result} and \ref{prop divisors Ahom
  X2} for the eigenvalue $4$, and Proposition \ref{prop F4 of S2} for
the eigenvalue $2$. \medskip

Along the proof of Proposition \ref{prop I square on S2}, we have
proved the following.
\begin{prop} \label{prop I2 X2}
The action of $(I^2)_*$ on $\CH_0(F)$ can be described as
\[
 (I^2)_*\sigma = 2\sigma + (-1)^{d-1} \delta^d\cdot I_*\sigma,\quad
 \forall \sigma \in \CH_0(F).
\]
In concrete terms, we have
\begin{equation}\label{eq I2 X2}
 (I^2)_*[x,y]= 2[x,y] + [x,x] +[y,y].
\end{equation}
\end{prop}

We may then deduce the following eigenspace decomposition.
\begin{prop} \label{prop first dec result}
  The correspondence $(I^2-4\Delta_F) \circ (I^2-2\Delta_F)$ acts as
  zero on $\CH_0(F)$. Moreover,
  \begin{align*}
    \ker\,\{(I^2-2\Delta_F)_*  : \CH_0(F) \rightarrow \CH_0(F)\} &=
    \ker \, \{I_*  : \CH_0(F) \rightarrow \CH_d(F)\} \\
    \ker\,\{(I^2-4\Delta_F)_*  : \CH_0(F) \rightarrow \CH_0(F)\} & =
    \mathrm{im} \, \{\CH_0(\Delta) \rightarrow \CH_0(F)\} \\
    & = \mathrm{im} \, \{\delta^d \cdot  : \CH_d(F) \rightarrow
    \CH_0(F)\}.
  \end{align*}
  Here, $\Delta$ denotes the sub-variety of $X^{[2]}$ parameterizing
  non-reduced subschemes of length $2$ on $X$.
\end{prop}
\begin{proof}
  That $(I^2-4\Delta_F) \circ (I^2-2\Delta_F)$ acts as zero on
  $\CH_0(F)$ follows immediately from Proposition \ref{prop I2 X2} and
  from the two formulas~: $I_*[x,y] = X_x + X_y$ and $\delta^d\cdot
  X_x = (-1)^{d-1}[x,x]$ (Lemma \ref{lem intersection identity X2}).
  It is also immediate from Proposition \ref{prop I2 X2} that $\ker \,
  \{I_*\} \subseteq \ker\,\{(I^2-2\Delta_F)_*\}$. The inclusion $\ker
  \, \{I_*\} \supseteq \ker\,\{(I^2-2\Delta_F)_*\}$ follows from the
  following computation $$I_*(\delta^d \cdot I_*[x,y]) = I_*(\delta^d
  \cdot (X_x+X_y)) = (-1)^{d-1}I_*([x,x]+[y,y]) = (-1)^{d-1}\cdot 2
  (X_x+X_y) = (-1)^{d-1}\cdot 2 I_*[x,y].$$

  The image of $\CH_0(\Delta) \rightarrow \CH_0(F)$ is spanned by the
  classes of the points $[x,x]$. By \eqref{eq I2 X2}, it is clear that
  $(I^2)_*[x,x] = 4[x,x]$. Conversely, if $\sigma \in \CH_0(F)$ is
  such that $(I^2)_*\sigma = 4\sigma$, then by Proposition \ref{prop
    I2 X2} we have $2\sigma = (-1)^{d-1}\delta^d \cdot I_*\sigma$.
  Therefore $2\sigma \in \mathrm{im} \, \{\delta^d \cdot  : \CH_d(F)
  \rightarrow \CH_0(F)\}$, which concludes the proof since we are
  working with rational coefficients.
\end{proof}

We now wish to turn our focus on K3 surfaces and give other
characterizations of the eigenspace decomposition above. As it turns
out, the property of K3 surfaces that is going to be used is also
satisfied by smooth Calabi--Yau complete intersections. This property
can be stated as follows.

\begin{assm}\label{assm small diagonal}
 There exists a special cycle
  $\mathfrak{o}_X\in\CH_0(X)$ of degree $1$ such that
 $$
 p_1^*(x-\mathfrak{o}_X)\cdot p_2^*(x -\mathfrak{o}_X)= 0
 $$
 in $\CH_0(X\times X)$ for all $x\in X$, where $p_i :X\times
 X\rightarrow X$ are the two projections.
\end{assm}

\begin{rmk} \label{rmk support} One observes that the cycle
  $p_1^*(x-\mathfrak{o}_X)\cdot p_2^*(x -\mathfrak{o}_X)$ can be
  obtained by restricting the cycle $\Delta_{\mathrm{tot}}$ of
  Definition \ref{defn small diagonals} to $\{x\} \times X\times
  X$. Therefore, by Bloch--Srinivas \cite{bs}, Assumption \ref{assm
    small diagonal} is equivalent to asking for the existence of a
  special cycle $\mathfrak{o}_X\in\CH_0(X)$ of degree 1 such that
  $\Delta_{\mathrm{tot}}$ is supported on $D \times X\times X$ for
  some divisor $D\subset X$. This is proved for smooth Calabi--Yau hypersufaces
by Voisin \cite{voisin k3} and for smooth Calabi--Yau complete intersections
by Fu \cite{fu2}. In the case of K3 surfaces, $\Delta_{\mathrm{tot}} = 0$
by Theorem \ref{thm bv} of Beauville--Voisin. Thus Assumption 
\ref{assm small diagonal} is satisfied by K3 surfaces and smooth
Calabi--Yau complete intersections.
\end{rmk}

\begin{defn}
  Given $\mathfrak{o}_X\in\CH_0(X)$, we define
  \begin{equation}\label{eq defn of special 0-cycle}
    \mathfrak{o}_F  := p_*\rho^*(p_1^*\mathfrak{o}_X\cdot
    p_2^*\mathfrak{o}_X)\in\CH_0(F),
  \end{equation}
  where $p_i :X\times X\rightarrow X$, $i=1,2$, are the projections.
  We also define
  \begin{equation}\label{eq defn of special d-cycle}
  X_{\mathfrak{o}}  := p_*q^*\mathfrak{o}_X\in\CH_d(F).
\end{equation}
\end{defn}

We say that $\mathfrak{o}_X$ is \emph{effective} if it is the class of
a single point. Assume that $\mathfrak{o}_X$ is effective.  By abuse
of notation, we will use $\mathfrak{o}_X$ to denote both the cycle
class and an actual point representing this canonical class. We also
use $X_{\mathfrak{o}}$ to denote the variety
$p(q^{-1}\mathfrak{o}_X)$.  Furthermore, $\mathfrak{o}_F$ is also
effective since it is represented by the point
$[\mathfrak{o}_X,\mathfrak{o}_X]\in F$. In this situation where
$\mathfrak{o}_X$ is supposed to be effective, Assumption \ref{assm
  small diagonal} can be restated as requiring that, for all points $x
\in X$, we have \begin{equation}\label{eq Beauville-Voisin equation}
  [x,x] - 2[\mathfrak{o}_X,x] +[\mathfrak{o}_X,\mathfrak{o}_X]=0 \
  \mbox{in} \ \CH_0(F).
 \end{equation}

Let us then state the following proposition.

\begin{prop} \label{prop bvf} Let $X$ be a K3 surface or a smooth
  Calabi--Yau complete intersection. Then there exists a point
  $\mathfrak{o}_X \in X$ such that for all points $x \in X$ we have
   \begin{equation*} [x,x] -
     2[\mathfrak{o}_X,x] +[\mathfrak{o}_X,\mathfrak{o}_X]=0 \ \mbox{in}
     \ \CH_0(F).
   \end{equation*} Moreover, the intersection of any two
   positive-dimensional
   cycles of complementary codimension inside $X$ is rationally
   equivalent to a multiple of $\mathfrak{o}_X$.
\end{prop}
\begin{proof}
  The first part of the proposition is a combination of the discussion
  above and Remark \ref{rmk support}. The second part of the proposition 
  concerned with
  intersection of cycles is proved by Beauville--Voisin \cite{bv} in
  the case of K3 surfaces, by Voisin \cite{voisin k3} in the case of
  smooth Calabi--Yau hypersurfaces, and by Fu \cite{fu2} in the case
  of smooth Calabi--Yau complete intersections. Note that in the case
  of K3 surfaces, Beauville and Voisin first establish that the
  intersection of any two divisors on a K3 surface is a multiple of
  $\mathfrak{o}_X$ and then, after some intricate arguments, deduce
  that $\Delta_{\mathrm{tot}}=0$, whereas in the case of smooth
  Calabi--Yau complete intersections the logic is reversed~: it is
  first proved that $\Delta_{\mathrm{tot}}$ is supported on a nice
  divisor and then the intersection property is easily deduced.
\end{proof}

We then have the following complement to Proposition \ref{prop first
  dec result}.

\begin{prop} \label{prop second dec result} If $X$ satisfies
  Assumption \ref{assm small diagonal} with $\mathfrak{o}_X$
  effective, then
  \begin{align*}
    \ker\,\{(I^2-4\Delta_F)_*  : \CH_0(F) \rightarrow \CH_0(F)\} & =
    \mathrm{im} \, \{\CH_0(X_{\mathfrak{o}}) \rightarrow \CH_0(F)\} \\
    &= \mathrm{im} \, \{X_{\mathfrak{o}} \cdot  : \CH_d(F) \rightarrow
    \CH_0(F)\}.
  \end{align*}
\end{prop}
\begin{proof} The image of $\CH_0(X_{\mathfrak{o}}) \rightarrow
  \CH_0(F)$ is spanned by the cycle classes of points
  $[\mathfrak{o}_X,x]$. A direct computation using \eqref{eq I2 X2}
  and \eqref{eq Beauville-Voisin equation} shows that
  $(I^2)_*[\mathfrak{o}_X,x] = 4[\mathfrak{o}_X,x].$ Conversely,
  Proposition \ref{prop first dec result} shows that
  $\ker\,\{(I^2-4\Delta_F)_*  : \CH_0(F) \rightarrow \CH_0(F)\}$ is
  spanned by cycle classes of points $[x,x]$. But then, \eqref{eq
    Beauville-Voisin equation} says that $[x,x] = 2[\mathfrak{o}_X,x]
  - [\mathfrak{o}_X,\mathfrak{o}_X] = X_{\mathfrak{o}}\cdot (2X_x-
  X_{\mathfrak{o}})$, which implies that $\ker\,\{(I^2-4\Delta_F)_*  :
  \CH_0(F) \rightarrow \CH_0(F)\} \subseteq \mathrm{im} \,
  \{X_{\mathfrak{o}} \cdot  : \CH_d(F) \rightarrow \CH_0(F)\}$.
\end{proof}

\begin{prop} \label{prop divisors Ahom X2} Let $X$ be a smooth
  projective variety with $\HH^1(X,\Z)=0$. Assume that Assumption
  \ref{assm small diagonal} holds on $X$ and that the intersection of
  any $d$ divisors on $X$ is a multiple of $\mathfrak{o}_X$
  (\emph{e.g.}, $X$ a K3 surface or a smooth Calabi--Yau
  complete intersection~; \emph{cf.} Proposition \ref{prop bvf}),
  then
$$
\CH^1(F)^{\cdot d} \cdot I_*\CH_0(F) = \ker \, \{(I^2-4\Delta_F)_*  :
\CH_0(F) \rightarrow \CH_0(F)\}.
$$
\end{prop}
\begin{proof} By Proposition \ref{prop first dec result}, $\ker \,
  \{(I^2-4\Delta_F)_*\}$ is spanned by the classes of the points
  $[x,x]$ for all $x \in X$. Since $[x,x] = (-1)^{d-1}\delta^d\cdot
  I_*x$, we get the inclusion ``$\supseteq$''.

  As for the inclusion ``$\subseteq$'', it is enough to show, thanks
  to Propositions \ref{prop first dec result} and \ref{prop second dec
    result}, that ${D}^d\cdot X_x$ is a linear combination of $[x,x]$
  and $[x,\mathfrak{o}]$ for all divisors ${D}$ and all points $x \in
  X$. The decomposition \eqref{decomposition H2 general} implies that
  any divisor $D$ on $F$ is of the form
  \begin{center}
    ${D} = \hat{\mathfrak{D}} + a\delta,$ where
    $\hat{\mathfrak{D}}=p_*q^* {\mathfrak{D}}$, $\mathfrak{D} \in
    \CH^1(X)$, and $a \in \Q$.
  \end{center}
Given $\sigma \in \CH_1(F)$, $\delta \cdot \sigma$ belongs to the
image of $\CH_0(\Delta) \rightarrow \CH_0(F)$, so that by Proposition
\ref{prop first dec result} it suffices to verify~:
\begin{center}
  $\hat{\mathfrak{D}}^d \cdot X_x$ is a linear combination of $[x,x]$
  and $[x,\mathfrak{o}]$.
\end{center}
Recall that we have the following diagram
\[
\xymatrix{ Z \ar[r]^{\rho \ \ }\ar[d]_{p} & X\times
  X\ar[r]^{\ \ p_1} & X\\
  F & & }
\] with $q=p_1 \circ \rho.$ Let us compute $\hat{\mathfrak{D}}^d$ for
any divisor $\hat{\mathfrak{D}}$ of the form $p_*q^*{\mathfrak{D}}$,
${\mathfrak{D}}\in \CH^1(X)$.
\begin{align*}
  \hat{\mathfrak{D}}^d &= p_*q^*\mathfrak{D} \cdot
  (p_*q^*\mathfrak{D})^{d-1}\\
  & =  p_*(q^*\mathfrak{D} \cdot (p^*p_*q^*\mathfrak{D})^{d-1})  \\
  & = p_*(\rho^*p_1^*\mathfrak{D} \cdot (\rho^*p_1^*\mathfrak{D} +
  \rho^*p_2^*\mathfrak{D})^{d-1})  \\
  & = p_*\rho^*p_1^*(\mathfrak{D}^d) + \sum_{i=1}^{d-1} c_i \,
  p_*\rho^*(p_1^*\mathfrak{D}^{d-i} \cdot p_2^*\mathfrak{D}^i), \quad
  c_i \in \Q  \\
  & = \deg(\mathfrak{D}^d) \, X_{\mathfrak{o}} + \sum_{i=1}^{d-1} c_i
  \, p_*\rho^*(\mathfrak{D}^{d-i}\times \mathfrak{D}^i).
\end{align*}
where the last equality holds since $\mathfrak{D}^d$ is a multiple
of $\mathfrak{o}_X$. Thus
$$\hat{\mathfrak{D}}^d \cdot X_x = \deg(\mathfrak{D}^d) \,
[x,\mathfrak{o}_X].$$ Actually, the above can be made more precise. We
have
\begin{equation}
  \label{eq divisor Sx X2}
  \delta \cdot \hat{\mathfrak{D}} \cdot X_x = 0, \quad \mbox{for all }
  x \in X.
\end{equation}
Indeed, $\delta \cdot \hat{\mathfrak{D}}$ can be represented by a
divisor $E$ on $\Delta$, so that for general $x \in X$, $S_x$ does not
meet $E$.  Thus we have proved $$D^d\cdot X_x = (-1)^{d-1} a^d \,
[x,x] + \deg(\mathfrak{D}^d)\, [x,\mathfrak{o}_X].$$
\end{proof}

Since $X_x \cdot X_y =
[x,y]$, it is clear that the intersection product $$\CH_d(F) \otimes
\CH_d(F) \rightarrow \CH_0(F)$$ is surjective.
We define
  \begin{equation}
  \label{eq A} \mathcal{A}  :=
  I_*\CH_0(F) \subseteq \CH_d(F).
  \end{equation}
  The following proposition essentially shows, as will become apparent
  once $L$ is defined, the equality $\CH^2(F)_2 \cdot \CH^2(F)_2 =
  \CH^4(F)_4$ of Theorem \ref{thm main mult} in the case when $F$ is
  the Hilbert scheme of length-$2$ subschemes on a K3 surface.

\begin{prop}\label{prop F4 of S2}
 Let $X$ be a smooth projective variety. We have the following inclusions
$$\mathcal{A}_\mathrm{hom}\cdot \mathcal{A}_\mathrm{hom} \subseteq
\ker\{Z_* :\CH_0(F)\rightarrow\CH_0(X)\} \subseteq
\ker\{I_* :\CH_0(F)\rightarrow\CH_d(F)\}~;
$$
and $\mathcal{A}_\mathrm{hom}\cdot \mathcal{A}_\mathrm{hom}$ is
spanned by cycles of the form $[x,y]+[z,w]-[x,z]-[y,w]$.

\noindent Moreover, if $X$ satisfies Assumption \ref{assm small
  diagonal}, then the above inclusions are all equalities.
\end{prop}
\begin{proof}
  Two distinct points $x,y\in X$ determine a point $[x,y]\in F$ and a
  general point of $F$ is always of this form.  By definition we have
$$
I_*[x,y]=X_x + X_y=Z^*Z_*[x,y].
$$
If $\sigma\in\CH_0(F)$ is such that $Z_*\sigma=0$, then by the above
equality we have $I_*\sigma=0$. Thus $\ker(Z_*)\subseteq\ker(I_*)$.
Recall that $X_x\cdot X_y=[x,y]$ so that, for a general point $p \in
X$, we have
\begin{align*}
[x,y]+[z,w]-[x,z]-[y,w] & = (X_x-X_w)\cdot (X_y-X_z)\\
& = I_*([x,p] - [w,p]) \cdot
I_*([y,p] - [z,p]),
\end{align*}
which by bilinearity shows that
$\mathcal{A}_\mathrm{hom}\cdot\mathcal{A}_\mathrm{hom}$ is spanned by
cycles of the form $[x,y]+[z,w]-[x,z]-[y,w]$. It is then immediate to
see that $\mathcal{A}_\mathrm{hom}\cdot\mathcal{A}_\mathrm{hom}
\subseteq\ker\{Z_*\}$.

Let us now assume that $X$ satisfies Assumption \ref{assm small
  diagonal} and let us show that $\ker\{I_*\} \subseteq
\mathcal{A}_\mathrm{hom}\cdot\mathcal{A}_\mathrm{hom}$.  We further
assume that $\mathfrak{o}_X$ is effective. Otherwise, in what follows,
we replace $[x,\mathfrak{o}_X]$ by $p_*\rho^*(p_1^*x\cdot
p_2^*\mathfrak{o}_X)$ and $[\mathfrak{o}_X,\mathfrak{o}_X]$ by
$p_*\rho^*(p_1^*\mathfrak{o}_X\cdot p_2^*\mathfrak{o}_X)$. Then the
same proof carries through.  By Assumption \ref{assm small diagonal},
the key identity \eqref{eq Beauville-Voisin equation}
\begin{equation}
[x,x]-2[x,\mathfrak{o}_X] +\mathfrak{o}_F=0,\quad\text{in }\CH_0(F)
\end{equation} is satisfied.
Now let $\sigma=\sum_{i=1}^{r}[x_i,y_i]-\sum_{i=1}^{r}[z_i,w_i] \in
\CH_0(F)_\mathrm{hom}$ be such that $I_*\sigma=0$, \textit{i.e.},
$$
\sum(X_{x_i}+X_{y_i}) - \sum(X_{z_i}+X_{w_i})=0,\quad \text{in
}\CH_d(F).
$$
By intersecting the above sum with $X_{\mathfrak{o}_X}$, we get
$$
\sum([x_i,\mathfrak{o}_X] + [y_i,\mathfrak{o}_X]) -
\sum([z_i,\mathfrak{o}_X] + [w_i,\mathfrak{o}_X])=0,\quad \text{in
}\CH_0(F).
$$
Combined with equation \eqref{eq Beauville-Voisin equation}, this
yields $\sum ([x_i,x_i]+[y_i,y_i]-[z_i,z_i]-[w_i,w_i])=0$. Hence it
follows that
\begin{align*}
2\sigma &= 2\left(\sum[x_i,y_i] - \sum [z_i,w_i]\right) \\
 &= \sum(X_{x_i}-X_{y_i})^2
 -\sum(X_{z_i}-X_{w_i})^2 -\sum([x_i,x_i]
 +[y_i,y_i]-[z_i,z_i]-[w_i,w_i])\\
 &=\sum(X_{x_i}-X_{y_i})^2
 -\sum(X_{z_i}-X_{w_i})^2\\
 & =\sum (I_*([x_i,\mathfrak{o}_X]-[y_i,\mathfrak{o}_X]))^2 -\sum
 (I_*([z_i,\mathfrak{o}_X]-[w_i,\mathfrak{o}_X]))^2.
\end{align*}
This shows that $\ker\{I_*\}\subseteq
\mathcal{A}_\mathrm{hom}\cdot\mathcal{A}_\mathrm{hom}$.
\end{proof}


\vspace{10pt}
\section{Multiplicative Chow--K\"unneth decomposition for
  $X^{[2]}$} \label{sec multCKX2}

In this section, we consider a smooth projective variety $X$ of
dimension $d$ endowed with a multiplicative Chow--K\"unneth
decomposition in the sense of Definition \ref{def multCK}. Recall
that, denoting $\Delta_{123}^X$ the class of $\{(x,x,x) : x \in X\}$
inside $X \times X \times X$ seen as a correspondence of degree $d$
from $X \times X$ to $X$, this means that there is a Chow--K\"unneth
decomposition of the diagonal
 \[
 \Delta_X = \sum_{i=0}^{2d}\pi_X^i, \quad \mbox{such that} \ \
 \Delta_{123}^X = \sum_{i+j=k} \pi_X^k \circ \Delta_{123}^X \circ
 (\pi_X^i\otimes \pi_X^j).
\]
In particular, such a multiplicative decomposition induces a bigrading
of the Chow ring
\[
\CH^*(X) = \bigoplus_{p,s} \CH_{\mathrm{CK}}^p(X)_s,\qquad
\text{with}\quad \CH_{\mathrm{CK}}^p(X)_s :=
(\pi_X^{2p-s})_*\CH^p(X).\] Under the condition that the Chern classes
of the tangent bundle of $X$ belong to the degree-zero graded part of
the above decomposition, and viewing $X^{[2]}$ as a quotient of the
blow-up of $X \times X$ along the diagonal for the action of the
symmetric group $\mathfrak{S}_2$, we show in Theorem \ref{thm multCK
  X2} that $X^{[2]}$ is naturally endowed with a multiplicative
Chow--K\"unneth decomposition. We first show that the projective
normal bundle of $X$ can be endowed with a multiplicative
Chow--K\"unneth decomposition. Then we show that the blow-up of $X$
along the diagonal has a multiplicative $\mathfrak{S}_2$-equivariant
Chow--K\"unneth decomposition. Finally we deduce that $X^{[2]}$ has a
multiplicative Chow--K\"unneth decomposition.
The point of Section \ref{section mult S2} will be to check that the
decomposition induced by this multiplicative Chow--K\"unneth
decomposition on the Chow groups of $S^{[2]}$ coincides with the
decomposition induced by the Fourier transform. Finally, in \S
\ref{sec proof thm6}, we prove Theorem \ref{thm2 CK}.

\subsection{Multiplicative Chow--K\"unneth decompositions for
  projective bundles}
Let $\mathscr{E}$ be a locally free coherent sheaf of rank $r+1$ and
let $\pi : \PP(\mathscr{E})\rightarrow X$ be the associated projective
bundle. Let $\xi\in \CH^1(\PP(\mathscr{E}))$ be the class of the relative
$\calO(1)$-bundle.  Then the Chow ring of $\PP(\mathscr{E})$ is given by
\begin{equation}\label{eq chow proj bundle}
  \CH^*(\PP(\mathscr{E})) = \CH^*(X)[\xi],\qquad \text{where }\quad
  \xi^{r+1} + \pi^*c_1(\mathscr{E})\xi^r + \cdots
  +\pi^*c_r(\mathscr{E})\xi + \pi^*c_{r+1}(\mathscr{E}) = 0.
\end{equation}
The same formula holds for the cohomology ring of
$\PP(\mathscr{E})$. We may then write, for each $i\geq 0$,
\begin{equation}\label{eq xi power}
\xi^{i} = \sum_{j=0}^r \pi^*a_{i,j}\xi^j,
\end{equation}
for some unique $a_{i,j}\in\CH^{i-j}(X)$, where $a_{i,j}$ is a
polynomial of the Chern classes of $\mathscr{E}$. It follows in
particular that $\pi_*\xi^i = a_{i,r}$ is a polynomial of the Chern
classes of $\mathscr{E}$.

\begin{prop}\label{prop mult proj bundle}
  Suppose $X$ has a multiplicative Chow-K\"unneth decomposition and
  $c_i(\mathscr{E})\in \CH_{\mathrm{CK}}^i(X)_0$ for all $i\geq
  0$. Then $\PP(\mathscr{E})$ has a multiplicative Chow--K\"unneth
  decomposition such that the associated decomposition of the Chow
  ring of $\PP(\mathscr{E})$ satisfies
  \begin{equation} \label{eq projbundle}
    \CH_{\mathrm{CK}}^p(\PP(\mathscr{E}))_s = \bigoplus_{l=0}^{r}
    \xi^l \cdot \pi^* \CH_{\mathrm{CK}}^{p-l}(X)_s.
\end{equation}
\end{prop}
\begin{proof} The projective bundle formula for Chow groups gives an
  isomorphism $$\sum_{l=0}^r \xi^l \cdot \pi^* : \bigoplus_{l=0}^{r}
  \CH^{p-l}(X) \stackrel{\simeq}{\longrightarrow}
  \CH^p(\PP(\mathscr{E})).$$ This isomorphism is in fact an
  isomorphism of motives ; see Manin \cite{manin}. Therefore, the
  Chow--K\"unneth decomposition of $X$ induces a Chow--K\"unneth
  decomposition of $\Delta_{\PP(\mathscr{E})}$ whose associated pieces
  $\CH_{\mathrm{CK}}^p(\PP(\mathscr{E}))_s$ of the Chow groups
  $\CH^p(X)$ are described in the formula \eqref{eq projbundle}.

  The self-product $\PP(\mathscr{E})\times \PP(\mathscr{E})\times
  \PP(\mathscr{E})$ can be obtained from $X\times X\times X$ by
  successively taking the projectivizations of the vector bundles
  $p_i^*\mathscr{E}$, $i=1,2,3$. The subtle point here is that there
  are thus two ways of obtaining a Chow--K\"unneth decomposition for
  $\PP(\mathscr{E})\times \PP(\mathscr{E})\times \PP(\mathscr{E})$~:
  either one considers the product Chow--K\"unneth decomposition on
  $\PP(\mathscr{E})\times \PP(\mathscr{E})\times \PP(\mathscr{E})$
  induced by that of $\PP(\mathscr{E})$ (and this is the one that
  matters in order to establish the multiplicativity property), or one
  considers the product Chow--K\"unneth decomposition on $X\times
  X\times X$ induced by that of $X$
 and then uses the projective bundle isomorphism together
  with Manin's identity principle to produce a Chow--K\"unneth
  decomposition on the successive projectivizations of the vector
  bundles $p_i^*\mathscr{E}$, $i=1,2,3$ (and this is the one used to
  compute the induced decomposition on the Chow groups of
  $\PP(\mathscr{E})\times \PP(\mathscr{E})\times \PP(\mathscr{E})$ in
  terms of the Chow groups of $X\times X \times X$ as in \eqref{eq
    chow P^3 graded pieces} below). Either way, these Chow--K\"unneth
  decompositions agree. In particular, the decompositions induced on
  the Chow groups of $\PP(\mathscr{E})\times \PP(\mathscr{E})\times
  \PP(\mathscr{E})$ agree and satisfy~:
  \begin{equation}\label{eq chow P^3 graded pieces}
    \CH^p_{\mathrm{CK}}(\PP(\mathscr{E})\times \PP(\mathscr{E})\times 
    \PP(\mathscr{E}))_s = 
    \bigoplus_{0\leq i,j,k\leq r} (\pi^{\times 3})^* 
    \CH^{p-i-j-k}_{\mathrm{CK}}(X\times X\times X)_s\cdot \xi_1^i\xi_2^j\xi_3^k,
  \end{equation}
  where $\pi^{\times 3}: \PP(\mathscr{E})\times \PP(\mathscr{E})\times
  \PP(\mathscr{E}) \rightarrow X\times X\times X$ is the the 3-fold
  product of $\pi$, and where $\xi_i\in \CH^1(\PP(\mathscr{E})\times
  \PP(\mathscr{E})\times \PP(\mathscr{E}))$ is the pull-back of
  $\xi\in \CH^1(\PP(\mathscr{E}))$ from the $i^{\text{th}}$ factor,
  $i=1,2,3$.
  
  By the criterion for multiplicativity of a Chow--K\"unneth
  decomposition given in Proposition \ref{rmk delta0}, we only need to
  show that
  \[
   \Delta_{123}^{\PP(\mathscr{E})} \in \CH_{\mathrm{CK}}^{2d+2r}(\PP(\mathscr{E})\times 
\PP(\mathscr{E})\times \PP(\mathscr{E}))_0,
  \]
  where $\Delta_{123}^{\PP(\mathscr{E})}$ is the small diagonal of
  $\PP(\mathscr{E})$ (\emph{cf.} Definition \ref{defn small
    diagonals}). Consider the following diagram
  \[
  \xymatrix{ \PP \ar[r]^{\iota_{\PP/X}\qquad}\ar[rd]_{\pi} &
    \PP\times_X\PP\times_X\PP \ar[r]^{\iota'}\ar[d]^{\pi^{\times
        3}_{/X}} &
    \PP\times \PP\times \PP\ar[d]^{\pi^{\times 3}} \\
    &X\ar[r]^{\iota_{\Delta,3}\quad} & X\times X\times X, }
  \]
  where in order to lighten the notations we have set
  $\PP:=\PP(\mathscr{E})$, and where $\iota_{\PP/X}$ is the small
  diagonal embedding of $\PP$ relative to $X$.  By definition, we have
  \[
  \Delta_{123}^{\PP(\mathscr{E})} = (\iota')_*(\iota_{\PP/X})_*[\PP].
  \]
  Note that the Chow groups of $\PP\times_X \PP\times_X \PP$ are given
  by
  \begin{equation}\label{eq Chow rel 3-fold projbundle}
    \CH^p(\PP\times_X\PP\times_X\PP) = 
    \bigoplus_{0\leq i,j,k\leq r}(\pi^{\times 3}_{/X})^*\CH^{p-i-j-k}(X)\cdot 
    (\xi'_1)^i(\xi'_2)^j(\xi'_3)^k,
  \end{equation}
  where $\xi'_i = (\iota')^*\xi_i$. Hence we have
  \[
  (\iota_{\PP/X})_*[\PP] = \sum_{0\leq i,j,k\leq r} (\pi^{\times
    3}_{/X})^*\alpha_{ijk} \cdot (\xi'_1)^i(\xi'_2)^j(\xi'_3)^k
  \]
  for some $\alpha_{ijk}\in \CH^*(X)$. For each triple $(l_1,l_2,l_3)$
  of non-negative integers, we intersect the above equation with
  $(\xi'_1)^{l_1}(\xi'_2)^{l_2}(\xi'_3)^{l_3}$ and then push forward
   to $X$ along $\pi^{\times 3}_{/X}$. Noting that
\[
(\pi^{\times 3}_{/X})_* \Big((\iota_{\PP/X})_*[\PP]\cdot
(\xi'_1)^{l_1}(\xi'_2)^{l_2}(\xi'_3)^{l_3}\Big) = \pi_*
\xi^{l_1+l_2+l_3} = \text{polynomial in }\set{c_i(\mathscr{E})},
\]
we obtain a bunch of linear equations involving the $\alpha_{ijk}$'s
whose coefficients are all polynomials in terms of the Chern classes
of $\mathscr{E}$. One inductively deduces from these equations that
the $\alpha_{ijk}$ are all polynomials of the Chern classes of
$\mathscr{E}$. Then, by assumption, we have $\alpha_{ijk}\in
\CH_{\mathrm{CK}}^*(X)_0$. Note that
\begin{align*}
  \Delta_{123}^{\PP(\mathscr{E})} & = \iota'_*(\iota_{\PP/X})_*[\PP]\\
  & = \sum_{0\leq i,j,k\leq r} \iota'_* \Big((\pi^{\times
    3}_{/X})^*\alpha_{ijk}
  \cdot (\xi'_1)^i(\xi'_2)^j(\xi'_3)^k\Big)\\
  & = \sum_{0\leq i,j,k\leq r} \iota'_* \Big((\pi^{\times
    3}_{/X})^*\alpha_{ijk}
  \cdot (\iota')^*(\xi_1^i\xi_2^j\xi_3^k)\Big)\\
  & = \sum_{0\leq i,j,k\leq r} \big(\iota'_* (\pi^{\times
    3}_{/X})^*\alpha_{ijk}\big)
  \cdot \xi_1^i\xi_2^j\xi_3^k\\
  & = \sum_{0\leq i,j,k\leq r} \big( (\pi^{\times
    3})^*(\iota_{\Delta,3})_*\alpha_{ijk}\big) \cdot
  \xi_1^i\xi_2^j\xi_3^k.
\end{align*}
Since $(\iota_{\Delta,3})_*\alpha_{ijk}\in \CH_{\mathrm{CK}}^*(X^3)_0$
by Proposition \ref{lem diagonal pullback}, we conclude from \eqref{eq
  chow P^3 graded pieces} that $\Delta_{123}^{\PP}\in
\CH^*_{\mathrm{CK}}(\PP\times\PP\times\PP)_0$.
\end{proof}

\subsection{Multiplicative Chow--K\"unneth decompositions for smooth
  blow-ups}

In this paragraph, we investigate the stability of the
multiplicativity property for Chow--K\"unneth decompositions under
smooth blow-up.  Let $X$ be a smooth projective variety and let $i
:Y\hookrightarrow X$ be a smooth closed sub-variety of codimension
$r+1$. Let $\rho :\widetilde{X}\rightarrow X$ be the blow-up of $X$
along $Y$. Consider the following diagram
\[
\xymatrix{
  E\ar[r]^j\ar[d]_\pi &\widetilde{X}\ar[d]^\rho\\
  Y\ar[r]^i &X }
\]
where $E$ is the exceptional divisor on $\widetilde{X}$ ; it is
isomorphic to the geometric projectivization $ \PP(\mathscr{N}_{Y/X})$
of the normal bundle of $Y$ inside $X$. Let $\xi\in\CH^1(E)$ be the
class of the relative $\calO(1)$-bundle of the projective bundle
$E\rightarrow X$. The smooth blow-up formula for Chow groups
\cite[Proposition 6.7 (e)]{fulton} gives an isomorphism
\begin{equation}\label{eq chow blow up}
  \CH^p(X) \oplus
  \bigoplus_{l = 0}^{r-1}  \CH^{p-l-1}(Y)
  \stackrel{\simeq}{\longrightarrow}
  \CH^p(\widetilde{X}), \quad
  (\alpha, \beta_0, \ldots, \beta_{r-1})  \mapsto \rho^*\alpha
  +   \sum_{l = 0}^{r-1} j_*(\xi^l\cdot \pi^*\beta_l).
\end{equation}

\begin{prop}\label{prop mult blow up}
  Assume the following conditions  :
\begin{enumerate}[(i)]
\item both $X$ and $Y$ have a multiplicative Chow--K\"unneth
  decomposition ;

\item $i^*\CH_{\mathrm{CK}}^p(X)_s \subseteq \CH_{\mathrm{CK}}^p(Y)_s$
  and $i_*\CH_{\mathrm{CK}}^{p-1}(Y)_s\subseteq
  \CH_{\mathrm{CK}}^{p+r}(X)_s$ for all $p$ and $s$;

\item the homomorphism $i_*:\CH_*(Y)\rightarrow \CH_*(X)$ is
  injective;

\item $c_p(\mathscr{N}_{Y/X})\in \CH_{\mathrm{CK}}^p(Y)_0$ for all
  $p\geq 0$.
\end{enumerate}
Then $\widetilde{X}$ has a multiplicative Chow--K\"unneth
decomposition such that
\begin{equation}\label{eq blowupdec}
  \CH_{\mathrm{CK}}^p(\widetilde{X})_s = \rho^*\CH_{\mathrm{CK}}^p(X)_s \oplus \Big(
  \bigoplus_{l=0}^{r-1} j_*(\xi^l\cdot\pi^*\CH_{\mathrm{CK}}^{p-l-1}(Y)_s) \Big).
\end{equation}
Furthermore, if $\CH_{\mathrm{CK}}^p(E)_s$ is the decomposition
obtained in Proposition \ref{prop mult proj bundle}, then we have
\[
j^*\CH_{\mathrm{CK}}^p(\widetilde{X})_s\subseteq
\CH_{\mathrm{CK}}^p(E)_s \qquad\text{and }\quad
j_*\CH_{\mathrm{CK}}^{p-1}(E)_s \subseteq
\CH_{\mathrm{CK}}^p(\widetilde{X})_s.
\]
\end{prop}
\begin{proof} By Manin \cite{manin}, the isomorphism \eqref{eq chow
    blow up} is in fact an isomorphism of motives so that the
  Chow--K\"unneth decompositions of $X$ and $Y$ induce a
  Chow--K\"unneth decomposition of $\widetilde{X}$ such that the
  associated pieces of the decomposition of the Chow groups are as
  given in \eqref{eq blowupdec}.  Before proving that the above
  decomposition is multiplicative, we first establish the
  compatibility results for $j^*$ and $j_*$. We first note that if
  $\alpha = \rho^*\alpha'$, for some $\alpha'\in
  \CH_{\mathrm{CK}}^p(X)_s$, then $i^*\alpha'\in
  \CH_{\mathrm{CK}}^p(Y)_s$ and hence
\[
 j^*\alpha = \pi^*i^*\alpha' \in \CH_{\mathrm{CK}}^p(E)_s.
\]
Similarly, if $\alpha = j_*(\xi^l\cdot\pi^*\alpha')$ for some
$\alpha'\in \CH_{\mathrm{CK}}^{p-l-1}(Y)_s$, then
\[
j^*\alpha = j^*j_*(\xi^l\cdot \pi^*\alpha')=
E|_E\cdot\xi^l\cdot\pi^*\alpha' = -\xi^{l+1}\cdot\pi^*\alpha'\in
\CH_{\mathrm{CK}}^p(E)_s,
\]
where we used the fact that $j^*E = -\xi$. Hence we have
\[
j^*\CH_{\mathrm{CK}}^p(\widetilde{X})_s \subseteq
\CH_{\mathrm{CK}}^p(E)_s.
\]
To prove the compatibility for $j_*$, consider $\alpha =
\xi^l\cdot\pi^*\alpha'$ for some $\alpha'\in
\CH_{\mathrm{CK}}^{p-l}(Y)_s$.  If $l<r$, then that $j_*\alpha$
belongs to $\CH_{\mathrm{CK}}^{p+1}(\widetilde{X})_s$ is
automatic. Assume that $l=r$. Then the key formula of
\cite[Proposition 6.7 (a)]{fulton} (which is recalled later on in
Lemma \ref{lem blow-up formula Fulton}) implies
\begin{align*}
  \rho^*i_*\alpha' & = j_*
  \Big(c_{r}(\pi^*\mathscr{N}_{Y/X}/\calO(-1)) \cdot \pi^*\alpha' \Big)\\
  & = j_*\Big( (\pi^*c_{r}(\mathscr{N}_{Y/X}) + \xi\cdot \pi^*
  c_{r-1}(\mathscr{N}_{Y/X}) + \cdots +
  \xi^{r-1}\pi^*c_1(\mathscr{N}_{Y/X}) + \xi^r)\cdot \pi^*\alpha'
  \Big).
\end{align*}
 Note that
\[
j_*\Big(\xi^{l'}\cdot \pi^*(c_{r-l}(\mathscr{N}_{Y/X})\cdot
\alpha')\Big) \in \CH_{\mathrm{CK}}^{p+1}(\widetilde{X})_s,
\]
for all $l'<r$. Meanwhile, using the assumption \emph{(ii)}, we also know
that $\rho^*i_*\alpha'\in
\CH_{\mathrm{CK}}^{p+1}(\widetilde{X})_s$. Then one easily sees that
$j_*(\xi^r\pi^*\alpha')\in
\CH_{\mathrm{CK}}^{p+1}(\widetilde{X})_s$. In all cases, we have
$j_*\CH_{\mathrm{CK}}^p(E)_s\subseteq
\CH_{\mathrm{CK}}^{p+1}(\widetilde{X})_s$.\medskip

It remains to show that the Chow--K\"unneth decomposition of
$\widetilde{X}$ is multiplicative. As in the proof of Proposition
\ref{prop mult proj bundle}, we would like to use the criterion of
Proposition \ref{rmk delta0} to show that the small diagonal
$\Delta_{123}^{\widetilde{X}}$ (\emph{cf.} Definition \ref{defn small
  diagonals}) belongs to $\CH^{2d}_{\mathrm{CK}}(\widetilde{X} \times
\widetilde{X}\times \widetilde{X})_0$, where the Chow--K\"unneth
decomposition used is the $3$-fold product Chow--K\"unneth
decomposition on $\widetilde{X}$.  For that purpose, we need to
understand at least how the Chow groups of $X\times X \times X$ fit
into the Chow groups of $\widetilde{X} \times \widetilde{X}\times
\widetilde{X}$ (see \eqref{eq diagonal diff}). Note that the $3$-fold
product $\widetilde{X}^3$ of $\widetilde{X}$ can also be viewed as
blowing up $X \times X \times X$ successively three times into
$\widetilde{X} \times X \times X$, $\widetilde{X} \times \widetilde{X}
\times X$ and then into $\widetilde{X} \times \widetilde{X} \times
\widetilde{X}$. The crucial point is then to observe that the product
Chow--K\"unneth decomposition on $\widetilde{X}^3$ agrees with the one
obtained by starting with the product Chow--K\"unneth decomposition on
$X^3$ and then by using the blow-up formula \eqref{eq chow blow up}
and Manin's identity principle after each blow-up necessary to obtain
$\widetilde{X}^3$.  As such, these Chow--K\"unneth decompositions for
$\widetilde{X}^3$ induce the same decomposition of the Chow groups of
$\widetilde{X}^3$ ; in view of \eqref{eq blowupdec} applied to
$X\times X \times X$, $\widetilde{X} \times X \times X$ and
$\widetilde{X} \times \widetilde{X} \times X$, successively, one finds
that the induced decomposition on the Chow groups of $\widetilde{X}
\times \widetilde{X} \times \widetilde{X}$ satisfies
\begin{equation}\label{eq 3-fold blowup}
  (\rho^{\times 3})^*\CH_{\mathrm{CK}}^p({X}^3)_s + 
  (j^{\times 3})_* \left(\bigoplus_{i_1,i_2,i_3=0}^{r}
    \xi_1^{i_1}\xi_2^{i_2}\xi_3^{i_3}\cdot (\pi^{\times 3})^*
    \CH_{\mathrm{CK}}^{p-3-i_1-i_2-i_3}(Y^3)_s \right)
  \subseteq \CH_{\mathrm{CK}}^p(\widetilde{X}^3)_s.
\end{equation}
As already mentioned, in order to establish the multiplicativity
property of the Chow--K\"unneth decomposition on $\widetilde{X}$, it
suffices by the criterion of Proposition \ref{rmk delta0} to show that
the small diagonal $\Delta_{123}^{\widetilde{X}}$ satisfies
\[
\Delta_{123}^{\widetilde{X}} \in
\CH_{\mathrm{CK}}^{2d}(\widetilde{X}^3)_0.
\]
Consider the following fiber square
\[
\xymatrix{
  Z \ar[r]^{\iota'} \ar[d]_{\rho'} & \widetilde{X}^3\ar[d]^{\rho^{\times 3}}\\
  X \ar[r]^{\iota_{\Delta,3}} & X^3, }
\]
where $Z$ contains two components : $Z_1 \cong \widetilde{X}$ which is
such that the restriction of $\iota'$ to $Z_1$ is simply the small
diagonal embedding $\tilde{\iota}_{\Delta,3}: \widetilde{X}\rightarrow
\widetilde{X}^3$, and $Z_2 = E\times_Y E \times_Y E$. It follows that
\begin{equation}\label{eq diagonal diff}
  \Delta_{123}^{\widetilde{X}} = (\rho^{\times 3})^* \Delta_{123}^{X} 
  + \iota'_*\alpha,
\end{equation}
for some $\alpha\in \CH^*(E\times_Y E\times_Y E)$. By equation
\eqref{eq Chow rel 3-fold projbundle}, we get
\begin{equation}\label{eq alphablowup}
  \alpha = \sum_{i_1,i_2,i_3=0}^{r} (\pi^{\times 3}_{/Y})^*\alpha_{i_1 i_2 i_3}\cdot 
  (\xi'_1)^{i_1}(\xi'_2)^{i_2}(\xi'_3)^{i_3},\quad\text{for some } 
  \alpha_{i_1i_2i_3}\in \CH^*(Y),
\end{equation}
where $\xi'_{k}\in\CH^1(E\times_X E\times_X E)$ is the pull-back of
$\xi\in \CH^1(E)$ from the $k^{\text{th}}$ factor. For any triple
$(l_1,l_2,l_3)$ of non-negative integers, we intersect both sides of
\eqref{eq diagonal diff} with $(-E_1)^{l_1}(-E_2)^{l_2}(-E_3)^{l_3}$ and
then apply $(\rho^{\times 3})_*$ to get
\begin{align*}
  (\rho^{\times 3})_*\Big( \Delta_{123}^{\widetilde{X}}\cdot
  (-E_1)^{l_1}(-E_2)^{l_2}(-E_3)^{l_3}\Big) &= (\rho^{\times
    3})_*\Big((-E_1)^{l_1}(-E_2)^{l_2}(-E_3)^{l_3}\Big)\cdot
  \Delta_{123}^{X}\\
  &\qquad + (\rho^{\times 3})_*\Big(
  (-E_1)^{l_1}(-E_2)^{l_2}(-E_3)^{l_3}\cdot\iota'_*\alpha\Big).
\end{align*}
We note that
\[
(\rho^{\times 3})_*\Big( \Delta_{123}^{\widetilde{X}}\cdot
(-E_1)^{l_1}(-E_2)^{l_2}(-E_3)^{l_3}\Big) =
(\iota_{\Delta,3})_*\rho_*(-E)^{l_1+l_2+l_3}
\]
and we also observe that $\rho_*(E^l)\in \CH_{\mathrm{CK}}^*(X)_0$
since it is either the class of $X$ (when $l=0$) or of the form
\[
 i_*(\text{polynomial of }c(\mathscr{N}_{Y/X})).
\]
It follows that
\[
(\rho^{\times 3})_*\Big( \Delta_{123}^{\widetilde{X}}\cdot
(-E_1)^{l_1}(-E_2)^{l_2}(-E_3)^{l_3}\Big) \in
\CH_{\mathrm{CK}}^*(X\times X\times X)_0.
\]
Similarly, using Proposition \ref{lem diagonal pullback}, one also
shows that
\[
(\rho^{\times 3})_*\Big((-E_1)^{l_1}(-E_2)^{l_2}(-E_3)^{l_3}\Big)\cdot
\Delta_{123}^{X} \in \CH_{\mathrm{CK}}^*(X\times X\times X)_0,
\]
Hence we conclude that
\[
(\rho^{\times 3})_*\Big(
(-E_1)^{l_1}(-E_2)^{l_2}(-E_3)^{l_3}\cdot\iota'_*\alpha\Big) \in
\CH_{\mathrm{CK}}^*(X\times X\times X)_0.
\]
Substituting the expression \eqref{eq alphablowup} into the above
equation, we find
\begin{align*}
  (\rho^{\times 3})_*\Big(
  (-E_1)^{l_1}(-E_2)^{l_2}(-E_3)^{l_3}\cdot\iota'_*\alpha\Big) & =
  (\rho^{\times 3})_*\iota'_*\Big(
  (\xi'_1)^{l_1}(\xi'_2)^{l_2}(\xi'_3)^{l_3} \cdot\alpha\Big)\\
  & = (\iota_{\Delta,3})_*i_*(\pi^{\times 3}_{/Y})_*\Big(
  (\xi'_1)^{l_1}(\xi'_2)^{l_2} (\xi'_3)^{l_3}\cdot\alpha\Big)\\
  & = (\iota_{\Delta,3})_*i_*L_{l_1 l_2 l_3}(\alpha_{i_1 i_2 i_3}) \
  \in \CH_{\mathrm{CK}}^*(X\times X\times X)_0,
\end{align*}
where $L_{l_1l_2l_3}$ is a linear form of $\alpha_{i_1 i_2 i_3}$ whose
coefficients are polynomials of the Chern classes of
$\mathscr{N}_{Y/X}$. Since both $i_*$ and $(\iota_{\Delta,3})_*$ are
injective and respect the decompositions of the Chow groups, we get
\[
 L_{l_1l_2l_2}(\alpha_{i_1i_2i_3})\in \CH^*(Y)_0.
\]
By induction, it follows from these linear equations that
$\alpha_{i_1i_2i_3}\in \CH^*(Y)_0$. Consider the following fiber
square
\[
\xymatrix{
 E\times_Y E\times_Y E\ar[r]^{\qquad\varphi}\ar[d]_{\pi^{\times 3}_{/Y}} 
& E^3\ar[r]^{j^{\times 3}}\ar[d]^{\pi^{\times 3}} &\widetilde{X}^3\\
 Y\ar[r]^{\iota'_{\Delta,3}} & Y^3 & 
}
\]
where $\iota'_{\Delta,3}$ is the small diagonal embedding of $Y$. Note
that
\begin{align*}
  \iota'_*\alpha & = (j^{\times 3})_*\varphi_*\alpha\\
  & = (j^{\times 3})_*\left(\sum_{i_1,i_2,i_3=0}^{r} \varphi_*
    \big((\pi^{\times 3}_{/Y})^*\alpha_{i_1i_2i_3}\big)\cdot
    \xi_1^{i_1} \xi_2^{i_2}\xi_3^{i_3}\right)\\
  & = (j^{\times 3})_*\left(\sum_{i_1,i_2,i_3=0}^{r} (\pi^{\times
      3})^* \big((\iota'_{\Delta,3})_*\alpha_{i_1i_2i_3}\big)\cdot
    \xi_1^{i_1}\xi_2^{i_2}\xi_3^{i_3}\right).
\end{align*}
Combining this with \eqref{eq 3-fold blowup}, \eqref{eq diagonal diff}
and the fact that $\alpha_{i_1i_2i_3}\in \CH^*(Y)_0$ yields that
$\Delta_{123}^{\widetilde{X}}$ lies in the piece 
$\CH^{2d}(\widetilde{X}^3)_0$, which
finishes the proof.
\end{proof}

\subsection{Multiplicative Chow--K\"unneth decomposition for
  $X^{[2]}$} If $X$ is a smooth projective variety, we consider
$X^{[2]}$ as the quotient under the action of $\mathfrak{S}_2$ that
permutes the factors on the blow-up of $X\times X$ along the
diagonal. First we note that Proposition \ref{lem diagonal pullback}
implies that condition \emph{(ii)} of Proposition \ref{prop mult blow
  up} is automatically satisfied for the blow-up of $X\times X$ along
the diagonal if $X$ is endowed with a multiplicative Chow--K\"unneth
decomposition that is \emph{self-dual}~:

\begin{prop}
  \label{prop multCK blowup diag} Let $X$ be a smooth projective
  variety endowed with a multiplicative self-dual Chow--K\"unneth
  decomposition. Assume that $c_p({X})$ belongs to
  $\CH_{\mathrm{CK}}^p(X)_0$ for all $p\geq 0$. Then the induced
  Chow--K\"unneth decomposition (as described in Proposition \ref{prop
    mult blow up}) on the blow-up of $X\times X$ along the diagonal is
  multiplicative.
\end{prop}
\begin{proof}
  Indeed, by Proposition \ref{lem diagonal pullback}, the assumptions
  of Proposition \ref{prop mult blow up} are met (and note also that,
 since $p_1 \circ \iota_\Delta = \mathrm{id}_X : X \rightarrow X$, 
$(\iota_\Delta)_*
: \CH^*(X) \rightarrow \CH^{*+d}(X \times X)$ is injective).
\end{proof}

\begin{thm} \label{thm multCK X2} Let $X$ be a smooth projective
  variety endowed with a multiplicative self-dual Chow--K\"unneth
  decomposition. Assume that $c_p({X})$ belongs to
  $\CH_{\mathrm{CK}}^p(X)_0$ for all $p\geq 0$. For instance, $X$
  could be an abelian variety (the positive Chern classes are trivial)
  or a K3 surface endowed with the Chow--K\"unneth decomposition
  \eqref{eq CKK3} ($c_2(X) = 24 \mathfrak{o}_X$ by \cite{bv}). Then
  $X^{[2]}$ has a multiplicative Chow--K\"unneth decomposition.
\end{thm}
\begin{proof}
  Let $\tau : X \times X \rightarrow X \times X$ be the morphism that
  permutes the factors. This gives an action of the symmetric group
  $\mathfrak{S}_2 = \{1,\tau\}$ on $X \times X$, which by
  functoriality of blow-ups extends to an action of $\mathfrak{S}_2$
  on the blow-up $\rho : Z = (\widetilde{X \times X})_\Delta
  \rightarrow X \times X$ of $X \times X$ along the diagonal.  The
  morphism $p :Z \rightarrow X^{[2]}$ is then the quotient morphism
  corresponding to that action.

  The action of $\mathfrak{S}_2 = \{1,\tau\}$ on $X \times X$ induces
  an action of $\mathfrak{S}_2$ on the ring of self-correspondences of
  $X \times X$~: if we view $1$ and $\tau$ as correspondences from $X
  \times X$ to $X \times X$, then we have
\begin{center}
  $\tau \bullet \gamma := \tau\circ \gamma \circ \tau$, \quad for all
  $\gamma \in \CH^{2d}\big((X \times X) \times (X \times X)\big).$
\end{center}
It is then clear that the product Chow--K\"unneth decomposition of $X
\times X$ as defined in \eqref{eq multCK} is
$\mathfrak{S}_2$-equivariant for that action~; it is also
multiplicative by virtue of Theorem \ref{thm multCK}.  Likewise, the
induced action of $\mathfrak{S}_2$ on the blow-up $Z$ induces an
action of $\mathfrak{S}_2$ on the ring of self-correspondences of $Z$,
and again we view $1$ and $\tau$ as correspondences from $Z$ to
$Z$. With respect to that action, the Chow--K\"unneth decomposition
$\{\pi^i_Z : 0 \leq i \leq 4d\}$ on $Z$ induced by that on $X \times
X$ and $X$ via the blow-up formula \eqref{eq chow blow up} is
$\mathfrak{S}_2$-equivariant~; it is also multiplicative by
Proposition \ref{prop multCK blowup diag}.  Moreover, as
correspondences, we have
\begin{center}
 $\Gamma_p \circ {}^t\Gamma_p = 2 \Delta_{X^{[2]}}$ \quad and \quad
${}^t\Gamma_p \circ \Gamma_p = 1 +\tau$,
\end{center}
where $\Gamma_p \in \CH^{2d}\big(Z \times X^{[2]}\big)$ denotes the
graph of $p  : Z \rightarrow X^{[2]}$ and ${}^t\Gamma_p$ denotes its
transpose.

We now claim that the correspondences $$\pi^i  := \frac{1}{2} \,
\Gamma_p \circ \pi^i_Z \circ {}^t\Gamma_p \in \CH^{2d}(X^{[2]} \times
X^{[2]})$$ define a multiplicative Chow--K\"unneth
decomposition on $X^{[2]}$. First it is clear that the action of
$\pi^i$ on $\HH^*(X^{[2]},\Q)$ is the projector on
$\HH^i(X^{[2]},\Q)$. Next we have
\begin{align*}
  4\, \pi^i \circ \pi^j &= \Gamma_p \circ \pi^i_Z \circ {}^t\Gamma_p
  \circ \Gamma_p \circ \pi^j_Z \circ {}^t\Gamma_p = \Gamma_p \circ
  \pi^i_Z \circ
  (1+\tau) \circ \pi^j_Z \circ {}^t\Gamma_p \\
  &= \Gamma_p \circ \pi^i_Z \circ \pi^j_Z \circ {}^t\Gamma_p +
  \Gamma_p \circ \pi^i_Z \circ \tau \circ \pi^j_Z \circ {}^t\Gamma_p.
\end{align*}
Since $\pi^i_Z$ is $\mathfrak{S}_2$-equivariant, we have $\tau\circ
\pi^i_Z \circ \tau = \pi^i_Z$, so that $\tau \circ \pi^i_Z =
\pi^i_Z\circ \tau$. Thus $\Gamma_p \circ \pi^i_Z \circ \tau \circ
\pi^j_Z \circ {}^t\Gamma_p = \Gamma_p \circ \tau \circ \pi^i_Z \circ
\pi^j_Z \circ {}^t\Gamma_p$, and because $\Gamma_p \circ \tau =
\Gamma_p$, we get $\Gamma_p \circ \pi^i_Z \circ \tau \circ \pi^j_Z
\circ {}^t\Gamma_p = \Gamma_p \circ \pi^i_Z \circ \pi^j_Z \circ
{}^t\Gamma_p$. This yields $$4\, \pi^i \circ \pi^j = 2 \, \Gamma_p
\circ \pi^i_Z \circ \pi^j_Z \circ {}^t\Gamma_p,$$ and hence that the
correspondences $\pi^i$ define a Chow--K\"unneth decomposition on
$X^{[2]}$. It remains to check that this decomposition is
multiplicative. By the criterion given in Proposition \ref{rmk
  delta0}, it suffices to check that the small diagonal
$\Delta_{123}^{X^{[2]}}$ belongs to
$\CH_{\mathrm{CK}}^{4d}(X^{[2]})_0$, \emph{i.e.}, that
\begin{center}
  $(\pi^i \otimes \pi^j \otimes \pi^k)_*\Delta_{123}^{X^{[2]}} =0$,
  \quad whenever $i+j+k \neq 4d$.
\end{center}
We note that $ \Gamma_p \circ \pi^i_Z \circ \tau = \Gamma_p \circ
\tau\circ \pi^i_Z =\Gamma_p \circ \pi^i_Z$ and therefore
that $$\Gamma_p\circ \pi^i_Z \circ {}^t\Gamma_p \circ \Gamma_p =
\Gamma_p\circ \pi^i_Z \circ (1+\tau) = 2 \, \Gamma_p \circ \pi_Z^i
\quad \mbox{in } \CH^{2d}(Z\times X^{[2]}).$$ We also note
that $$(\Gamma_p \otimes \Gamma_p \otimes \Gamma_p)_*\Delta_{123}^Z =
4\,\Delta_{123}^{X^{[2]}}.$$
It follows that 
\begin{align*} 
  (\pi^i \otimes \pi^j \otimes \pi^k)_*\Delta_{123}^{X^{[2]}} &=
  \frac{1}{32}\big((\Gamma_p\circ \pi_Z^i \circ {}^t\Gamma_p) \otimes
  (\Gamma_p \circ \pi_Z^j \circ {}^t\Gamma_p) \otimes (\Gamma_p\circ
  \pi_Z^k
  \circ {}^t\Gamma_p)\big)_*(p \times p \times p)_*\Delta_{123}^{Z} \\
  &= \frac{1}{32} \big((\Gamma_p\circ \pi_Z^i \circ {}^t\Gamma_p \circ
  \Gamma_p) \otimes (\Gamma_p \circ \pi_Z^j \circ {}^t\Gamma_p \circ
  \Gamma_p) \otimes (\Gamma_p\circ \pi_Z^k \circ {}^t\Gamma_p\circ
  \Gamma_p)\big)_*
  \Delta_{123}^{Z}\\
  &= \frac{1}{4}(p\times p \times p)_*(\pi_Z^i \otimes \pi_Z^j \otimes
  \pi_Z^k)_*
  \Delta_{123}^{Z}\\
  &= 0 \qquad \mbox{if} \ i+j+k\neq 4d,
\end{align*} where the last equality follows from multiplicativity for
the Chow--K\"unneth decomposition $\{\pi^i_Z\}$ of $Z$.
This finishes the proof of the theorem.
\end{proof}

\begin{rmk} \label{rmk compatibility blowup hilbert} Define as usual
  $\CH_{\mathrm{CK}}^p(X^{[2]})_s := \pi^{2p-s}_*\CH^p(X^{[2]})$. Then
  note, for further reference, that with the notations of Theorem
  \ref{thm multCK X2} and its proof the quotient map $p : Z
  \rightarrow X^{[2]}$ is compatible with the Chow--K\"unneth
  decompositions on $Z$ and $X^{[2]}$, in the sense that for all $p$
  and all $s$ we have
\begin{center}
  $p_*\CH_{\mathrm{CK}}^p(Z)_s =
  \CH_{\mathrm{CK}}^p\big(X^{[2]}\big)_s$ \quad and \quad
  $p^*\CH_{\mathrm{CK}}^p\big(X^{[2]}\big)_s \subseteq
  \CH_{\mathrm{CK}}^p(Z)_s$.
\end{center}
\end{rmk}

\begin{rmk} \label{rmk hyperelliptic} Let $C$ be a smooth projective
  curve. The Hilbert scheme of length-$n$ subschemes on $C$ is smooth
  and is nothing but the $n^\mathrm{th}$ symmetric product of $C$.
  Assume that $C$ admits a multiplicative Chow--K\"unneth
  decomposition, \emph{e.g.}, $C$ is a hyperelliptic curve~; see
  Example \ref{ex hyperelliptic}. Then the product Chow--K\"unneth
  decomposition on $C^n$ is multiplicative by Theorem \ref{thm
    multCK}, and it is clearly $\mathfrak{S}_n$-equivariant. Here,
  $\mathfrak{S}_n$ is the symmetric group on $n$ elements acting on
  $C^n$ by permuting the factors. An easy adaptation of the proof of
  Theorem \ref{thm multCK X2} yields that $C^{[n]}$ has a
  multiplicative Chow--K\"unneth decomposition.
\end{rmk}

\subsection{Proof of Theorem \ref{thm2 CK}} \label{sec proof thm6}
According to Theorem \ref{thm multCK}, Proposition \ref{prop mult proj
  bundle}, Proposition \ref{prop multCK blowup diag} and Theorem
\ref{thm multCK X2}, it suffices to check that
\begin{enumerate}[(a)]
\item the varieties listed in Theorem \ref{thm2 CK} admit a
  multiplicative Chow--K\"unneth decomposition such that their Chern
  classes belong to the degree-zero graded part of the induced
  decomposition on the Chow ring~;
\item the property that the Chern classes belong to the degree-zero
  graded part of the Chow ring is stable under the operations
  \emph{(i)}--\emph{(iv)} listed in Theorem \ref{thm2 CK}.
\end{enumerate}

\noindent Item (b) is taken care of by the following lemma~:

\begin{lem}
  Let $X$ and $Y$ be two smooth projective varieties endowed with
  multiplicative Chow--K\"unneth decompositions such that
  $c_i(X)\in\CH_{\mathrm{CK}}^i(X)_0$ for all $0\leq i\leq\dim X$ and
  $c_i(Y)\in\CH_{\mathrm{CK}}^i(Y)_0$ for all $0\leq i\leq\dim
  Y$. Then the following statements hold.
\begin{enumerate}[(i)]
\item\label{enum chern class prod} $c_i(X\times Y) \in
  \CH_{\mathrm{CK}}^i(X\times Y)_0$ for all $0\leq i\leq \dim X+\dim
  Y$.
\item\label{enum chern class proj} $c_i(\PP(\mathscr{T}_X)) \in
  \CH_{\mathrm{CK}}^i(\PP(\mathscr{T}_X))_0$ for all $0\leq i\leq
  2\dim X-1$.
\item\label{enum chern class blow-up} $c_i(Z)\in
  \CH_{\mathrm{CK}}^i(Z)_0$ for all $0\leq i\leq 2\dim X$, where $Z$
  is the blow-up of $X\times X$ along the diagonal.
\item\label{enum chern class hilbert} $c_i(X^{[2]}) \in
  \CH_{\mathrm{CK}}^i(X^{[2]})_0$ for all $0\leq i\leq 2\dim X$.
\end{enumerate} \end{lem} 
\begin{proof}
  Statement \emph{(i)} follows directly from Theorem \ref{thm multCK}
  and Proposition \ref{lem diagonal pullback}\emph{(i)}, and from the
  formula $c(X\times Y)=p_1^*c(X)\cdot p_2^*c(Y)$, where $p_1:X\times
  Y \rightarrow X$ and $p_2:X\times Y\rightarrow Y$ are the two
  projections. \medskip
 
  Let $E:=\PP(\mathscr{T}_X)$ be the geometric projectivization of the
  tangent bundle of $X$ and let $\pi:E\rightarrow X$ be the projection
  morphism. The relative Euler exact sequence in this setting
  becomes \[ \xymatrix{ 0\ar[r] &\calO_E\ar[r]
    &\pi^*\mathscr{T}_X\otimes \calO(1) \ar[r] &
    \mathscr{T}_{E/X}\ar[r] &0, } \] where $\mathscr{T}_{E/X}$ is the
  relative (or vertical) tangent bundle. It follows that \[
  \mathrm{ch}(\mathscr{T}_{E/X}) =
  \pi^*\mathrm{ch}(\mathscr{T}_X)\cdot e^\xi - 1, \] where $\xi$ is
  the first Chern class of the relative $\calO(1)$-bundle. Note that
  the multiplicative Chow--K\"unneth decomposition on $E$ given by
  Proposition \ref{prop mult proj bundle} satisfies $\xi
  \in\CH_{\mathrm{CK}}^1(E)_0$. Hence we conclude that the Chern
  character $\mathrm{ch}(\mathscr{T}_{E/X})$ lies in the piece
  $\CH_{\mathrm{CK}}^*(E)_0$, so that \[ \mathrm{ch}(\mathscr{T}_E) =
  \pi^*\mathrm{ch}(\mathscr{T}_X) + \mathrm{ch}(\mathscr{T}_{E/X})\in
  \CH_{\mathrm{CK}}^*(E)_0. \] Statement \emph{(ii)} follows
  immediately.\medskip

  Let $\rho:Z\rightarrow X\times X$ be the blow-up morphism and let
  $j':E\rightarrow Z$ be the exceptional divisor. Note that $E$ is
  isomorphic to $\PP(\mathscr{T}_X)$. Consider the short exact
  sequence \[ \xymatrix{ 0\ar[r] &\rho^*\Omega^1_{X\times X} \ar[r]
    &\Omega_Z^1 \ar[r] &j'_*\Omega_{E/X}^1\ar[r] &0. } \] By taking
  Chern characters, one finds \[ \mathrm{ch}(\Omega_Z^1) =
  \rho^*\mathrm{ch}(\Omega_{X\times X}^1) +
  \mathrm{ch}(j'_*\Omega_{E/X}^1). \] By \emph{(i)} and Proposition
  \ref{prop multCK blowup diag}, we see that \[
  \rho^*\mathrm{ch}(\Omega_{X\times X}^1) \in
  \CH_{\mathrm{CK}}^*(Z)_0.
\]
Hence it suffices to show that
\[
\mathrm{ch}(j'_*\Omega_{E/X}^1)\in \CH_{\mathrm{CK}}^*(Z)_0.
\]
This can be obtained from the Grothendieck--Riemann--Roch
Theorem. Indeed, we have
\begin{equation*}
  \mathrm{ch}(j'_*\Omega_{E/X}^1)  =
  \frac{j'_*(\mathrm{ch}(\Omega_{E/X}^1)\cdot
    \mathrm{td}(E))}{\mathrm{td}(Z)} =
  j'_*\left(\frac{\mathrm{ch}(\Omega_{E/X}^1)}
    {\mathrm{td}(\mathscr{N}_{E/Z})}\right). 
\end{equation*} 
Now note that $\mathscr{N}_{E/Z}=\calO_E(-1)$, so that
$\mathrm{td}(\mathscr{N}_{E/Z})\in
\CH_{\mathrm{CK}}^*(E)_0$. Statement \emph{(ii)} implies that
$\mathrm{ch}(\Omega_{E/X}^1)\in \CH_{\mathrm{CK}}^*(E)_0$. The proof
of \emph{(iii)} then follows from the compatibility of $j'_*$ with the
gradings induced by the Chow--K\"unneth decompositions as in
Proposition \ref{prop mult blow up}.\medskip

Let $p: Z\rightarrow X^{[2]}$ be the natural double cover ramified
along the divisor $E$. Then we have the following short exact
sequence \[ \xymatrix{ 0\ar[r] &\mathscr{T}_Z \ar[r]
  &p^*\mathscr{T}_{X^{[2]}} \ar[r] &\calO_E(2E) \ar[r] &0, } \] from
which one obtains \[ p^*\mathrm{ch}(\mathscr{T}_{X^{[2]}}) =
\mathrm{ch}(\mathscr{T}_Z) + \mathrm{ch}(\calO_E(2E)). \] Note that
$\calO_E(2E)$ fits into the short exact sequences \[
\xymatrix{ 0\ar[r] &\calO_Z(E)\ar[r] &\calO_Z(2E)\ar[r] &\calO_E(2E)
  \ar[r] &0. } \] It follows that $\mathrm{ch}(\calO_E(2E))\in
\CH^*_{\mathrm{CK}}(Z)_0$. Statement \emph{(iii)} implies that
$\mathrm{ch}(\mathscr{T}_Z)\in\CH^*_{\mathrm{CK}}(Z)_0$. Hence we
get \[
p^*\mathrm{ch}(\mathscr{T}_{X^{[2]}})\in\CH^*_{\mathrm{CK}}(Z)_0, \]
which proves \emph{(iv)}~; see Remark \ref{rmk compatibility blowup
  hilbert}.
\end{proof}

\noindent As for item (a), first note that if $X$ has a
Chow--K\"unneth decomposition, we have $\CH^0(X) =
\CH^0_{\mathrm{CK}}(X)_0$ so that obviously $c_0(X)$ lies in
$\CH_{\mathrm{CK}}^0(X)_0$. We then proceed case-by-case~:\medskip

\noindent \emph{Case 1 :} $X$ is a smooth projective variety with Chow
groups of finite rank (as $\Q$-vector spaces). By \cite[Theorem
5]{vial3}, the Chow motive of $X$ is isomorphic to a direct sum of
Lefschetz motives (see \cite{vial3} for the definitions) so that $X$
has a Chow--K\"unneth decomposition such that the induced grading on
$\CH^*(X)$ is concentrated in degree zero, that is, $\CH^*(X) =
\CH^*_{\mathrm{CK}}(X)_0$. The Chow motive of $X \times X \times X$,
which is the $3$-fold self tensor product of the Chow motive of $X$,
is also isomorphic to a direct sum of Lefschetz motives and the
product Chow--K\"unneth decomposition is such that $\CH^*(X\times X
\times X)$ is concentrated in degree $0$. It is then clear that the
Chern classes of $X$ lie in the degree-zero graded part and in view of
the criterion of Proposition \ref{rmk delta0} that this
Chow--K\"unneth decomposition is multiplicative.  \medskip
  
\noindent \emph{Case 2 :} $X$ is the $n^{\mathrm{th}}$-symmetric
product of an hyperelliptic curve $C$. By Remark \ref{rmk
  hyperelliptic}, $X = C^{[n]}$ admits a multiplicative
Chow--K\"unneth decomposition. The Chern classes $c_i({C^{[n]}})$
belong to $ \CH^i_{\mathrm{CK}}(C^{[n]})_0$ for the following reason.
Let $Z\subset C^{[n]}\times C$ be the universal family. Given a line
bundle $\mathscr{L}$ on $C$, the associated tautological bundle
$\mathscr{L}^{[n]}$ on $C^{[n]}$ can be defined by \[
\mathscr{L}^{[n]} := p_{*}(\calO_{Z}\otimes q^*\mathscr{L}), \] where
$p:C^{[n]}\times C\rightarrow C^{[n]}$ and $q: C^{[n]}\times
C\rightarrow C$ are the two projections. If we take
$\mathscr{L}=\Omega_C^1$, it turns out that the associated
tautological bundle is $\Omega_{C^{[n]}}^1$. The
Grothendieck--Riemann--Roch theorem applied to this situation
becomes \[ \mathrm{ch}(\Omega_{C^{[n]}}^1) =
p_*\Big(\mathrm{ch}(\calO_Z)\cdot q^*\big(\mathrm{ch}(\Omega_C^1)\cdot
\mathrm{td}(C)\big)\Big). \] Hence it suffices to see that
$\mathrm{ch}(\calO_Z)\in \CH^*_{\mathrm{CK}}(C^{[n]}\times C)_0$. Let
$Z_{i}:=p_{i,n+1}^{-1}\Delta_C\subset C^n\times C$, $1\leq i\leq n$,
and let $\pi:C^{n}\rightarrow C^{[n]}$ be the symmetrization
morphism. Then we have \[ (\pi\times \mathrm{Id}_C)^*\calO_Z =
\sum_{i=1}^n\calO_{Z_{i}}, \quad \text{in }\mathrm{K}_0(C^{n}\times
C). \] By taking the Chern character, we see that $(\pi\times
\mathrm{Id}_C)^*\mathrm{ch}(\calO_Z)$ lies in the degree-zero graded
part of the Chow ring of $C^n\times C$ and hence, after applying
$(\pi\times \mathrm{Id}_C)_*$, that $\mathrm{ch}(\calO_Z)\in
\CH^*_{\mathrm{CK}}(C^{[n]}\times C)_0$.  \medskip
 
\noindent \emph{Case 3 :} $X$ is an abelian variety endowed with the
multiplicative self-dual Chow--K\"unneth decomposition of Example
\ref{ex abelian}. In that case, the tangent bundle of $X$ is trivial
so that its Chern classes of positive degree vanish and hence
obviously lie in the degree-zero graded part of $\CH^*(X)$.\medskip
  
\noindent \emph{Case 4 :} $X$ is a K3 surface endowed with the
multiplicative self-dual Chow--K\"unneth decomposition of Example
\ref{ex K3}. In that case, $c_1(X) = 0$, and $c_2(X)=24\,
\mathfrak{o}_X$ by Beauville--Voisin \cite{bv}.\medskip
  
The proof of Theorem \ref{thm2 CK} is now complete. \qed

\vspace{10pt}
\section{The Fourier decomposition for $S^{[2]}$}
\label{sec Fourierdec S2}

In this section, we construct a cycle $L \in \CH^2(S^{[2]}\times
S^{[2]})$ that represents the Beauville--Bogomolov form $\mathfrak{B}$
and show that $S^{[2]}$ endowed with that cycle $L$ satisfies the
conclusions of Theorem \ref{thm2 main} and Theorem \ref{thm main
  splitting}. This is achieved in \S \ref{sec Fourier S2}. A
preliminary crucial ingredient for that purpose is an equality, given
in Lemma \ref{lem lifting hilbert}, in the ring of correspondences of
$S^{[2]}$.

\subsection{A key lemma}\label{sec small diag} Let $X$ be a smooth
projective variety of dimension $d$. The set-up is that of Section
\ref{sec basics X2}. We consider the diagonals (introduced in of
Definition \ref{defn small diagonals}) $\Delta_{123}, \Delta_{ij}$ and
$\Delta_i$ with respect to the choice of a zero-cycle $\mathfrak{o}_X$
of degree $1$ on $X$. For simplicity of notations, from now on we
assume that $\mathfrak{o}_X$ is effective. Otherwise, one may replace
$X_\mathfrak{o}$ by $p_*q^*\mathfrak{o}_X$ and Lemma \ref{lem terms of
  small diagonal} below still holds true.\medskip

The following blow-up formula, which was already used in the proof of
Proposition \ref{prop mult blow up}, will be very useful in explicit
computations. We include it here for the convenience of the reader.

\begin{lem}[\cite{fulton}, Proposition 6.7]\label{lem blow-up formula Fulton}
  Let $Y$ be a smooth projective variety, and $Y_1\subset Y$ a smooth
  closed sub-variety of codimension $e$. Let $\rho_1
  :\widetilde{Y}\rightarrow Y$ be the blow-up of $Y$ along $Y_1$ and
  $E_1=\PP(\mathscr{N}_{Y_1/Y})\subset\widetilde{Y}$ the exceptional
  divisor. We have the following diagram attached to this situation
\[
\xymatrix{
 E_1\ar[d]_{\pi_1}\ar[r]^{j_1} &\widetilde{Y}\ar[d]^{\rho_1}\\
 Y_1\ar[r]^{j'_1} & Y.
}
\]
Then for any $\sigma\in\CH^*(Y_1)$ we have
\[
\rho_1^* j'_{1,*}\sigma = j_{1,*}(c_{e-1}(\mathscr{F})\cdot
\pi_1^*\sigma),
\]
where $\mathscr{F}=(\pi_1^*\mathscr{N}_{Y_1/Y})/
\mathscr{N}_{E_1/\widetilde{Y}}$.\qed
\end{lem}

We fix some notations that will be used frequently in the remaining
part of this section. Let $f : Y\rightarrow Y'$ be a morphism, then
$\mathscr{N}_f$ denotes the cokernel of the sheaf homomorphism
$\mathscr{T}_{Y}\rightarrow f^*\mathscr{T}_{Y'}$.  This quotient sheaf
$\mathscr{N}_f$ will be called the \textit{normal sheaf} (or
\textit{normal bundle} if it is locally free) of the morphism $f$. If
two varieties $Y_1$ and $Y_2$ are defined over some base $B$, then
$p_{Y_1\times_B Y_2,i}$, $i=1,2$, denote the projections $Y_1\times_B
Y_2\rightarrow Y_i$.\medskip

Consider the following diagram (\emph{cf.} \S \ref{sec basics X2} for
the notations)~:
\[
\xymatrix{ Z\times Z\ar[r]^{1\times q}\ar[d]_{p\times p} &Z\times
  X\ar[r]^{\rho\times 1 \quad} &X\times X\times X.\\
  F\times F & & }
\]

\begin{lem}\label{lem terms of small diagonal}
  The following equations hold in $\CH^{2d}(F\times F)$.
\begin{align*}
  (p\times p)_*(\rho\times q)^*\Delta_{123} & =
  \sum_{i=0}^{d-1}\frac{(-1)^{d-1-i}}{2}
  p_1^*(j_*\pi^*c_i(X))\cdot\delta_1^{d-1-i}\cdot I ;\\
  (p\times p)_*(\rho\times q)^*\Delta_{12} & = \sum_{i=0}^{d-1}
  (-1)^{d-1-i} p_1^*(j_*\pi^*c_i(X))\cdot\delta_1^{d-1-i}
  \cdot p_2^* X_{\mathfrak{o}} ;\\
  (p \times p)_* (\rho \times q)^* \Delta_{23} & =
  p_1^*X_{\mathfrak{o}}\cdot I - \Gamma_{\mathfrak{o}_X} ;\\
  (p \times p)_* (\rho \times q)^* \Delta_{13} &=
  p_2^*X_{\mathfrak{o}}\cdot I - \Gamma_{\mathfrak{o}_X} ;\\
  (p \times p)_* (\rho \times q)^* \Delta_i &= p_1^*X_{\mathfrak{o}}
  \cdot p_2^*X_{\mathfrak{o}},\quad
  i=1,2 ;\\
  (p \times p)_* (\rho \times q)^* \Delta_3 &= 2p_1^*\mathfrak{o}_F.
\end{align*}
\end{lem}

\begin{proof}
  Let $\iota_{\Delta_X}\times \mathrm{Id}_X : X\times X\rightarrow
  X\times X\times X$ be the morphism $(x,y)\mapsto (x,x,y)$. Then we
  have the following commutative diagram
\begin{equation}\label{eq big commutative diagram 1}
  \xymatrix{
    E\times Z\ar[r]^{\mathrm{Id}_E\times q}\ar[d]_{j'\times
      \mathrm{Id}_Z} & E\times X\ar[r]^{\pi\times
      \mathrm{Id}_X}\ar[d]_{j'\times\mathrm{Id}_X} &
    X\times X\ar[d]_{\iota_{\Delta_X} \times   \mathrm{Id}_X}\\
    Z\times Z\ar[r]^{\mathrm{Id}_Z\times q} &
    Z\times X\ar[r]^{\rho\times \mathrm{Id}_X} &X\times X\times X,
  }
\end{equation}
where all the squares are fiber products. Note that
$(\iota_{\Delta_X}\times \mathrm{Id}_X)_*\Delta_X =\Delta_{123}$.
\begin{equation}\label{eq Delta123 first}
 (\rho\times \mathrm{Id}_X)^*\Delta_{123} = (\rho\times
\mathrm{Id}_X)^*(\iota_{\Delta_X}\times \mathrm{Id}_X)_*\Delta_X.
\end{equation}
We also see that the morphism $\rho\times \mathrm{Id}_X$ is naturally
the blow-up of $X\times X\times X$ along the image of
$\iota_{\Delta_X}\times \mathrm{Id}_X$. By the blow-up formula (Lemma
\ref{lem blow-up formula Fulton}), we have
\begin{equation}\label{eq Delta123 second}
  (\rho\times \mathrm{Id}_X)^*(\iota_{\Delta_X}\times \mathrm{Id}_X)_*
  \Delta_X =
  (j'\times \mathrm{Id}_X)_* (c_{d-1}(\mathscr{E}) \cdot (\pi\times
  \mathrm{Id}_X)^*\Delta_X),
\end{equation}
where $\mathscr{E}$ is the locally free sheaf defined by
\[
\mathscr{E}  :=\mathrm{coker}\{\mathscr{N}_{j'\times
  \mathrm{Id}_X}\longrightarrow (\pi\times \mathrm{Id}_X )^*
\mathscr{N}_{\iota_{\Delta_X} \times \mathrm{Id}_X}\}.
\]
Note that $\mathscr{N}_{j'\times \mathrm{Id}_X} = (p_{E\times
  X,1})^*\mathscr{N}_{E/Z}$ and $\mathscr{N}_{\iota_{\Delta_X} \times
  \mathrm{Id}_X} = (p_{X\times X,1})^*\mathscr{T}_{X}$. On
$E=\PP(\mathscr{T}_X)$ we have the following short exact sequence
\begin{equation}\label{eq Euler sequence on E}
  \xymatrix{
    0\ar[r] &\calO(-1) \ar[rr] &&\pi^*\mathscr{T}_X \ar[rr] &&\mathscr{E}_0
    \ar[r] &0
  }
\end{equation}
Then the vector bundle $\mathscr{E}$ can be identified with
$(p_{E\times X,1})^*\mathscr{E}_0$. The above short exact sequence
implies
\[
 c(\mathscr{E}_0) = \pi^*(1 + c_1(X) +c_2(X) + \cdots +c_d(X))\cdot (1 +
\xi +\xi^2 +\cdots +\xi^{2d-1}),
\]
where $\xi=-E|_E$ is the first Chern class of the relative
$\calO(1)$-bundle. It follows that
\begin{equation}\label{eq top chern class of E}
 c_{d-1}(\mathscr{E}) = \sum_{i=0}^{d-1} (p_{E\times
X,1})^*\big(\pi^*c_i(X)\cdot \xi^{d-1-i}\big).
\end{equation}
Substituting \eqref{eq top chern class of E} into \eqref{eq Delta123
  second} gives
\begin{align*}
  (\rho\times \mathrm{Id}_X)^*(\iota_{\Delta_X}\times
  \mathrm{Id}_X)_*\Delta_X &= \sum_{i=0}^{d-1} (j'\times
  \mathrm{Id}_X)_* \{(p_{E\times X,1})^*\big(\pi^*c_i(X)\cdot
  \xi^{d-1-i}\big)\cdot (\pi\times
  \mathrm{Id}_X)^*\Delta_X \}\\
  &= \sum_{i=0}^{d-1} (j'\times \mathrm{Id}_X)_* \{(p_{E\times
    X,1})^*\big(\pi^*c_i(X)\cdot \xi^{d-1-i}\big)\cdot (\mathrm{Id}_E,\pi)_*E \}\\
  &= \sum_{i=0}^{d-1} (j'\times \mathrm{Id}_X)_*(\mathrm{Id}_E,\pi)_*
  \{(\mathrm{Id}_E,\pi)^*(p_{E\times X,1})^*\big(\pi^*c_i(X)\cdot
  \xi^{d-1-i}\big) \}\\
  & = \sum_{i=0}^{d-1} (j',\pi)_* \big(\pi^*c_i(X)\cdot
  \xi^{d-1-i}\big).
\end{align*}
Here the second equality uses the following fiber product square
\[
 \xymatrix{
  E\ar[r]^{\pi}\ar[d]_{(\mathrm{Id}_E,\pi)} &  X\ar[d]^{\iota_{\Delta_X}}\\
  E\times X\ar[r]^{\pi\times \mathrm{Id}_X} &X\times X.
 }
\]
The third equality uses the projection formula. The last equality
follows from the facts that $(j'\times \mathrm{Id}_X) \circ
(\mathrm{Id}_E,\pi) = (j',\pi)$ and that $(p_{E\times X,1})\circ
(\mathrm{Id}_E,\pi)=\mathrm{Id}_E$. Substituting the above equality
into equation \eqref{eq Delta123 first} yields
\begin{equation}\label{eq Delta123 third}
  (\rho\times \mathrm{Id}_X)^*\Delta_{123} = \sum_{i=0}^{d-1} (j',\pi)_*
  \big(\pi^*c_i(X)\cdot \xi^{d-1-i}\big).
\end{equation}
Note that we have the following fiber square
\[
\xymatrix{
  E\times_X Z\ar[r]^{p_1}\ar[d]_{j'\times_{_X} q} & E\ar[d]^{(j',\pi)}\\
  Z\times Z \ar[r]^{\mathrm{Id}_Z\times q} & Z\times X.  }
\]
From this and the above equation \eqref{eq Delta123 third}, we get
\begin{equation}\label{eq Delta123 fourth}
  (\mathrm{Id}_Z\times q)^* (\rho\times \mathrm{Id}_X)^*\Delta_{123} =
  \sum_{i=0}^{d-1} (j'\times_{_X} q)_*(p_1)^*\big(\pi^*c_i(X)\cdot
  \xi^{d-1-i}\big).
\end{equation}
We now observe that, for any cycle
class $\fa$ on $E$, we have
\begin{equation}\label{eq key observation 1}
 (j'\times_{_X} q)_* (p_1)^*\fa = (p_{Z\times Z,1})^*(j'_*\fa) \cdot
(q\times q)^*\Delta_X.
\end{equation}
Indeed, consider the following commutative diagram
\[
\xymatrix{
  X\ar[r]^{\iota_{\Delta_X}} & X\times X &\\
  E\times_X Z\ar[r]^\varphi\ar[u] &E\times Z \ar[r]^{j'\times
    \mathrm{Id}_Z}\ar[u]^{\pi\times q} & Z\times Z\ar[ul]_{q\times q}
}
\]
where the square is a fiber product. Then we have
\begin{align*}
  (j'\times_X q)_*p_1^*\fa  &= (j'\times \mathrm{Id}_Z)_*\varphi_*p_1^*\fa \\
  & = (j'\times \mathrm{Id}_Z)_*\Big((p_{E\times Z,1})^*\fa
  \cdot(\pi\times q)^*
  \Delta_X\Big)\\
  & =(j'\times \mathrm{Id}_Z)_*\Big((p_{E\times Z,1})^*\fa
  \cdot(j'\times
  \mathrm{Id}_Z)^*(q\times q)^*\Delta_X\Big)\\
  &= (j'\times \mathrm{Id}_Z)_*(p_{E\times Z,1})^*\fa \cdot (q\times
  q)^*\Delta_X.
\end{align*}
Then equation \eqref{eq key observation 1} follows easily by noting
\[
 (j'\times \mathrm{Id}_Z)_*(p_{E\times Z,1})^*\fa = (p_{Z\times Z,1})^*j'_*\fa.
\]
Combining equations \eqref{eq Delta123 fourth} and \eqref{eq key
observation 1} gives
\begin{align*}
  (\rho\times q)^* \Delta_{123}
  &= \sum_{i=0}^{d-1} (p_{Z\times Z,1})^* (j'_*(\pi^*c_i(X)\cdot
  \xi^{d-1-i}))\cdot (q\times q)^*\Delta_X\\
  & = \sum_{i=0}^{d-1} (p_{Z\times Z,1})^* (j'_*(\pi^*c_i(X)\cdot
  j'^*(-E)^{d-1-i})) \cdot (q\times q)^*\Delta_X\\
  &= \sum_{i=0}^{d-1} (-1)^{d-1-i} (p_{Z\times Z,1})^*
  (j'_*\pi^*c_i(X)\cdot
  E^{d-1-i}) \cdot (q\times q)^*\Delta_X\\
  &= \sum_{i=0}^{d-1} \frac{(-1)^{d-1-i}}{2}(p\times
  p)^*(p_1^*j_*\pi^*c_i(X)\cdot
  \delta_1^{d-1-i}) \cdot (q\times q)^*\Delta_X.
\end{align*}
Now we apply $(p\times p)_*$ to the above equation and find
\[
(p\times p)_*(\rho\times q)^*\Delta_{123}= \sum_{i=0}^{d-1}
\frac{(-1)^{d-1-i}}{2}p_1^*(j_*\pi^*c_i(X))\cdot
  \delta_1^{d-1-i} \cdot (p\times p)_*(q\times q)^*\Delta_X.
\]
This proves the first identity.\medskip

In order to prove the second equality, we note that $\Delta_{12} =
(\iota_{\Delta_X}\times \mathrm{Id}_X)_*(p_{X\times
  X,2})^*\mathfrak{o}_X$. Hence we have
\begin{align*}
  (\rho\times \mathrm{Id}_X)^*\Delta_{12} &= (\rho\times
  \mathrm{Id}_X)^* (\iota_{\Delta_X}\times \mathrm{Id}_X)_*(p_{X\times
    X,2})^*\mathfrak{o}_X\\
  &= (j'\times \mathrm{Id}_X)_* (c_{d-1}(\mathscr{E})\cdot (\pi\times
  Id_X)^*(p_{X\times X,2})^*\mathfrak{o}_X)\\
  &=\sum_{i=0}^{d-1}(-1)^{d-1-i} (j'\times \mathrm{Id}_X)_*
  \{(p_{E\times X,1})^*(\pi^*c_i(X)\cdot E^{d-1-i})
  \cdot (p_{E\times X,2})^*\mathfrak{o}_X\}\\
  &=\sum_{i=0}^{d-1}(-1)^{d-1-i} (p_{Z\times
    X,1})^*(j'_*\pi^*c_i(X)\cdot E^{d-1-i})\cdot (p_{Z\times
    X,2})^*\mathfrak{o}_X.
\end{align*}
Applying $(\mathrm{Id}_Z\times q)^*$ to the above equation, we get
\[
(\rho\times q)^*\Delta_{12} = \sum_{i=0}^{d-1}(-1)^{d-1-i} (p_{Z\times
  Z,1})^*(j'_*\pi^*c_i(X)\cdot E^{d-1-i})\cdot (p_{Z\times
  Z,2})^*(q^*{\mathfrak{o}_X}).
\]
It follows from this equation that
\[
(p\times p)_*(\rho\times q)^*\Delta_{12} = \sum_{i=0}^{d-1}
(-1)^{d-1-i} p_1^*(j_*\pi^*c_i(X))\cdot\delta_1^{d-1-i} \cdot
p_2^*(p_*q^*\mathfrak{o}_X),
\]
which proves the second equality.
\medskip

For the third equality, we proceed to the following direct
computation.  Recall that there is an involution $\tau$ on $Z$. We
will denote $\tau_1 :=\tau\times\mathrm{Id}_Z$ the involution of
$Z\times Z$ induced by the action of $\tau$ on the first
factor. Similarly, we have $\tau_2 :=\mathrm{Id}_Z \times \tau$. Then
we have
\begin{align*}
  (p \times p)_* (\rho \times q)^* \Delta_{23} & = (p \times p)_*
  (\rho
  \times q)^* \{(p_{X^3,1})^*\mathfrak{o}_X \cdot (p_{X^3,23})^*\Delta_X \}\\
  & = (p \times p)_* \{ (p_{Z\times Z, 1})^*q^*\mathfrak{o}_X \cdot
  \tau_1^*(q\times q)^*\Delta_X \}\\
  & = (p \times p)_* \{ (p\times p)^*p_1^*p_*q^*\mathfrak{o}_X \cdot
  \tau_1^*(q\times q)^*\Delta_X \} \\
  & \qquad - (p \times p)_* \{ \tau_1^*(p_{Z\times Z,
    1})^*q^*\mathfrak{o}_X
  \cdot \tau_1^*(q\times q)^*\Delta_X \}\\
  & = p_1^*p_*q^*\mathfrak{o}_X \cdot (p \times p)_* \tau_1^*(q\times
  q)^*\Delta_X \\
  & \qquad - (p \times p)_* \{ (p_{Z\times Z, 1})^*q^*\mathfrak{o}_X
  \cdot
  (q\times q)^*\Delta_X \}\\
  & = p_1^*(p_*q^*\mathfrak{o}_X) \cdot I - \Gamma_{\mathfrak{o}_X}.
\end{align*}
Here the first equality follows from the definition of $\Delta_{23}$.
The second equality follows from the fact that $(p_{X^3,1}) \circ
(\rho \times q) = q \circ (p_{Z \times Z,1})$ and that $(p_{X^3,1})
\circ (\rho \times q) =(q\times q)\circ\tau_1$.  The third equality is
obtained by using
\[
(p\times p)^*p_1^*p_*\fa = (p_{Z\times Z,1})^*\fa + \tau_1^*
(p_{Z\times Z,1})^*\fa,\quad \forall \fa\in \CH^*(Z\times Z).
\]

The fourth equality uses the projection formula and the fact that
$(p\times p)_*\tau_1^*\fa=(p \times p)_*\fa$ for all cycle classes
$\fa$ on $Z\times Z$. The last equality follows from the definition of
$\Gamma_\sigma$.  \medskip

The equation involving $\Delta_{13}$ follows from that involving
$\Delta_{23}$ by symmetry. The remaining equalities are proved easily.
\end{proof}

Note that a direct computation shows that, if $X$ satisfies Assumption
\ref{assm small diagonal}, then for all $\sigma \in \CH_0(F)$ we have
\[ \left(p_1^*(2X_\mathfrak{o} +(-1)^{d}\delta^d)\cdot I
\right)^*\sigma = (2X_\mathfrak{o} +(-1)^{d}\delta^d)\cdot I_*\sigma =
2\deg(\sigma)\mathfrak{o}_F.
\]
The reason for carrying out the computations involved in the proof of
Lemma \ref{lem terms of small diagonal} is to show that the formula
above can actually be made much more precise in the case when $X=S$ is
a K3 surface. The following lemma gives an equality of correspondences
(rather than merely an equality of actions of correspondences on
zero-cycles) and is essential to establishing Conjecture \ref{conj
  main} for $F= S^{[2]}$.

\begin{lem}\label{lem lifting hilbert}
  Let $S$ be a K3 surface and $F=S^{[2]}$ the Hilbert scheme of
  length-$2$ subschemes of $S$. Then the following relation holds in
  $\CH^4(F \times F)$~:
$$
p_1^*(2S_{\mathfrak{o}} + \delta^2) \cdot I = 2\delta_1^2\cdot
p_2^*S_{\mathfrak{o}} + 4p_1^*S_{\mathfrak{o}}\cdot
p_2^*S_{\mathfrak{o}} +2p_1^*\mathfrak{o}_F.
$$
\end{lem}
\begin{proof}
  We know from \cite{bv} that $\Delta_{\mathrm{tot}}=0 \in
  \CH_2(S\times S\times S)$ (see Theorem \ref{thm bv}), and also that
  $c_1(S)=0$ and $c_2(S)=24\mathfrak{o}_S$. Then the lemma follows
  from spelling out the equation $(p\times p)_*(\rho\times
  q)^*\Delta_{\mathrm{tot}}=0$ term by term using Lemma \ref{lem terms
    of small diagonal}.
\end{proof}


\subsection{The Fourier transform for $S^{[2]}$}
\label{sec Fourier S2}

In the remaining part of this section, we assume that $X=S$ is a K3
surface so that $F=S^{[2]}$ is a hyperk\"ahler fourfold. In that case,
it is possible to refine the results of \S \ref{sec mult X2} and
obtain a decomposition of the Chow groups of $F$, similar in every way
to the Fourier decomposition we will eventually establish, by using
the powers of the action of the incidence correspondence $I$. However,
the incidence correspondence $I$ does not deform in moduli and the
whole point of modifying $I$ into a cycle $L$ with cohomology class
the Beauville--Bogomolov class $\mathfrak{B}$ is to have at our
disposal a cycle that would deform in moduli (and in fact our cycle
$L$ defined in \eqref{eq L S2} does deform as will be shown in \S
\ref{sec L agree}) and whose deformations would induce a Fourier
decomposition for all hyperk\"ahler varieties of
$\mathrm{K3}^{[2]}$-type.

In concrete terms, the goal here is to prove Theorem \ref{thm main
  splitting} for $F$. For that purpose, we first construct a cycle $L
\in \CH^2(F \times F)$ that represents the Beauville--Bogomolov class
$\mathfrak{B}$, and we show (Theorem \ref{prop cubic L conjecture})
that $L$ satisfies Conjecture \ref{conj main}, \textit{i.e.}, that
$L^2=2 \, \Delta_F -\frac{2}{25}(l_1+l_2)\cdot L -\frac{1}{23\cdot
  25}(2l_1^2 - 23l_1l_2 + 2l_2^2) \in \CH^4(F \times F)$. Then we
check that Properties \eqref{assumption pre l} and \eqref{assumption
  hom} are satisfied by $L$ (Proposition \ref{prop cubic assumption
  pre l hom}). We can thus apply Theorem \ref{prop L2}. Finally, we
show that $L$ satisfies Property \eqref{assumption 2hom} (Proposition
\ref{lem condition 3}) so that we can apply Theorem \ref{prop main}
and get the Fourier decomposition of Theorem \ref{thm main splitting}
for $F$.\medskip

It is easy to see that $\Delta_S -p_1^*\mathfrak{o}_S
-p_2^*\mathfrak{o}_S$ represents the intersection form on $S$. From
the definition of $I$, a simple computation gives
\begin{align*}
  (p\times p)_*(q\times q)^*(\Delta_S -p_1^*\mathfrak{o}_S
  -p_2^*\mathfrak{o}_S) & = I - (p\times
  p)_*(p_1^*q^*\mathfrak{o}_S\cdot p_2^*p^*[F] +
  p_2^*q^*\mathfrak{o}_S\cdot p_1^*p^*[F])\\
  & = I- 2p_1^*S_{\mathfrak{o}} -2p_2^*S_{\mathfrak{o}}.
\end{align*}
Here the coefficient $2$ appears in the second equality because
$p :Z\rightarrow F$ is of degree $2$. As a result, the cycle
$$
L_\delta :=I-2p_1^*S_{\mathfrak{o}}-2p_2^*S_{\mathfrak{o}}
$$
represents the $\big(\HH^2(S,\Q) \otimes \HH^2(S,\Q)\big)$-component
of the Beauville--Bogomolov class $\mathfrak{B} \in \HH^2(F,\Q)
\otimes \HH^2(F,\Q) \subset \HH^4(F\times F,\Q)$ with respect to the
$q_F$-orthogonal decomposition \eqref{eq decomposition H2}.  Hence we
see that
\begin{equation}
  \label{eq L S2} L :=I - 2p_1^*S_{\mathfrak{o}} - 2p_2^*S_{\mathfrak{o}} -
  \frac{1}{2}\delta_1\delta_2
\end{equation}
represents the Beauville--Bogomolov class $\mathfrak{B}$, where
$\delta_i=p_i^*\delta$, $i=1,2$. Proposition \ref{lem diagonal
  pull-back of I}, together with the identity
$c_2(S)=24\mathfrak{o}_S$ of Beauville--Voisin \cite{bv}, implies
\[
 (\iota_\Delta)^*I = 24 S_{\mathfrak{o}} - 2\delta^2.
\]
Then we have
\begin{equation}\label{eq l on S2}
  l  := (\iota_{\Delta})^* L = 20 S_{\mathfrak{o}} - \frac{5}{2}\delta^2
  =\frac{5}{6}c_2(F),
\end{equation}
where the last equality is the following lemma.

\begin{lem}\label{lem chern class S2}
  If $S$ is a K3 surface, then $c_2(S^{[2]})=p_*q^*c_2(S) -3\delta^2
  =24S_{\mathfrak{o}} -3\delta^2$.
\end{lem}
\begin{proof} We take up the notations of diagram \eqref{eq basic
    blow-up diagram} with $X=S$.  The blow-up morphism $\rho :
  Z\rightarrow S\times S$ gives a short exact sequence
\[
\xymatrix@C=0.5cm{ 0 \ar[r] & \rho^*\Omega^1_{S\times S} \ar[rr] &&
  \Omega_Z^1 \ar[rr] && j'_*\Omega_{E/S}^1 \ar[r] & 0. }
\]
Note that $j'\Omega_{E/S}^1\cong \calO_E(2E)$, and hence
\[
c(j'_*\Omega_{E/S}^1)=\frac{1+2E}{1+E} = 1 + E -E^2 +E^3 -E^4.
\]
By taking Chern classes of the sheaves in the above short exact
sequence, we get
\[
c(\Omega_{Z}^1) = \rho^*c(\Omega_{S\times S}^1)\cdot
c(j'_*\Omega_{E/S}^1) = \rho^*(p_1^*c(\Omega_S^1)\cdot
p_2^*c(\Omega_S^1))\cdot \frac{1+2E}{1+E}.
\]
Meanwhile the double cover $p :Z\rightarrow F$ gives the following
short exact sequence
\[
\xymatrix@C=0.5cm{ 0 \ar[r] & p^*\Omega_F^1 \ar[rr] && \Omega_Z^1
  \ar[rr] && \calO_E(-E) \ar[r] & 0. }
\]
By taking Chern classes, we get
\[
p^*c(\Omega_F^1) = c(\Omega_Z^1)/c(\calO_E(-E)) = c(\Omega_Z^1)
\cdot \frac{1-2E}{1-E}
\]
Combining the above two computations yields
\[
p^*c(\Omega_F^1) = \rho^*(p_1^*c(\Omega_S^1)\cdot
p_2^*c(\Omega_S^1))\cdot \frac{1-4E^2}{1-E^2}
\]
This gives $p^*c_2(F) = \rho^*(p_1^*c_2(S) + p_2^*c_2(S)) -3E^2$. We
apply $p_*$ to this equation and note that $p_*E^2=2\delta^2$ and
$p_*\rho^* p_i^*c_2(S)=p_*q^*c_2(S)$, $i=1,2$. Then we get
\[
2c_2(F) = 2 p_*q^*c_2(S) -6\delta^2,
\]
which, in light of Beauville--Voisin's formula
$c_2(S)=24\mathfrak{o}_S$, proves the lemma.
\end{proof}

\begin{thm} \label{prop cubic L conjecture}
Conjecture \ref{conj main} holds for $F$ with $L$ as in \eqref{eq L S2}.
\end{thm}
\begin{proof}
  First note that Proposition \ref{prop I square on S2} gives the
  equation
\begin{equation}\label{eq I square equation on S2}
  I^2 = 2\Delta_F -(\delta_1^2 -\delta_1\delta_2 +\delta_2^2)\cdot I +
  24 p_1^*S_{\mathfrak{o}}\cdot p_2^*S_{\mathfrak{o}}.
\end{equation}
If we substitute $I= L +2 p_1^* S_{\mathfrak{o}} +2p_2^*
S_{\mathfrak{o}} +\frac{1}{2}\delta_1\delta_2$ into the above
equation, we get
\begin{equation}\label{eq first L equation on S2}
  L^2 =2\Delta_F - (\delta_1^2 + \delta_2^2 + 4p_1^* S_{\mathfrak{o}}
  + 4p_2^* S_{\mathfrak{o}})\cdot L
  +P_1(\delta_1,\delta_2,p_1^*S_{\mathfrak{o}},
  p_2^*S_{\mathfrak{o}}),
\end{equation}
where $P_1$ is a weighted homogeneous polynomial of degree $4$.  Lemma
\ref{lem lifting hilbert} implies that $(\delta_1^2 +
2p_1^*S_{\mathfrak{o}}) \cdot L$ is a weighted homogeneous polynomial
of degree $4$ in $(\delta_1,\delta_2,p_1^*S_{\mathfrak{o}},
p_2^*S_{\mathfrak{o}})$. By symmetry, so is $(\delta_2^2 + 2
p_2^*S_{\mathfrak{o}}) \cdot L$. From the equation \eqref{eq l on S2},
we get
$$
\frac{2}{25}l - (4S_{\mathfrak{o}} + \delta^2) =
-\frac{6}{5}(\delta^2 + 2S_{\mathfrak{o}}).
$$
Hence $p_i^*(\frac{2}{25}l - (4S_{\mathfrak{o}} + \delta^2))\cdot
L$, $i=1,2$, is a weighted homogeneous polynomial of degree 4 in
$(\delta_1,\delta_2,p_1^*S_{\mathfrak{o}}, p_2^*S_{\mathfrak{o}})$.
From equation \eqref{eq first L equation on S2}, we get
\begin{equation}\label{eq second L equation on S2}
L^2 = 2\Delta_F -\frac{2}{25}(l_1 + l_2)\cdot L +
P_2(\delta_1,\delta_2,p_1^*S_{\mathfrak{o}}, p_2^*S_{\mathfrak{o}}),
\end{equation}
where $P_2$ is a weighted homogeneous polynomial of degree 4 which is
homologically equivalent to $\frac{-1}{23\cdot 25}(2l_1^2 - 23 l_1 l_2
+2 l_2^2)$. Then Lemma \ref{lem cohomology to chow} shows that
$P_2=\frac{-1}{23\cdot 25}(2l_1^2 - 23 l_1 l_2 +2 l_2^2)$.
\end{proof}

We now check that $L$ satisfies conditions \eqref{assumption pre l} and
\eqref{assumption hom} of Theorem \ref{prop L2}.

\begin{prop} \label{prop cubic assumption pre l hom} We have $ L_*l^2
  = 0$ and $L_*(l\cdot L_*\sigma) = 25\; L_*\sigma$ for all $\sigma
  \in \CH_0(F)$.
\end{prop}
\begin{proof} A direct computation using \eqref{eq L S2} yields
  \begin{equation} \label{eq Lxy} L_*[x,y] = I_*[x,y] -
    2S_{\mathfrak{o}} = S_{x} + S_y - 2S_{\mathfrak{o}}.
  \end{equation}
  By \eqref{eq l on S2}, $l$ is proportional to $c_2(F)$ and it
  follows, thanks to \cite{voisin2}, that $l^2$ is proportional to
  $[\mathfrak{o}_S,\mathfrak{o}_S]$. Thus \eqref{eq Lxy} gives $L_*l^2
  = 0$.
 
  Let now $a$ and $b$ be rational numbers. An easy computation using
  \eqref{eq Lxy}, $S_x \cdot S_y = [x,y]$ and $\delta^2 \cdot S_x =
  -[x,x]$ yields $$L_*\big((aS_{\mathfrak{o}} + b \delta^2)\cdot
  L_*[x,y]\big) = (a-2b)L_*[x,y].$$ The identity \eqref{eq l on S2}
  $l=20S_{\mathfrak{o}} - \frac{5}{2}\delta^2$ then gives $L_*(l\cdot
  L_*\sigma) = 25\;  L_*\sigma$ for all $\sigma \in \CH_0(F)$.
\end{proof}

We then check that $L$ satisfies condition \eqref{assumption 2hom} of
Theorem \ref{prop main}. Let us first compare the decompositions
of $\CH_0(F)$ induced by $I$ and $L$.

\begin{lem}\label{lem comparison I L} We have
\begin{align*}
  \ker \{L_*  : \CH_0(F)_{\hom} \rightarrow \CH_2(F)\} &= \ker \{I_*  :
  \CH_0(F) \rightarrow \CH_2(F)\} ;\\
  \ker\,\{(L^2)_*  : \CH_0(F) \rightarrow \CH_0(F)\} &=
  \ker\,\{(I^2-4\Delta_F)_*  : \CH_0(F) \rightarrow \CH_0(F)\}.
\end{align*}
\end{lem}
\begin{proof}
  The first equality follows from \eqref{eq Lxy} and from
  $I_*[\mathfrak{o}_S,\mathfrak{o}_S] = 2S_{\mathfrak{o}} \neq 0 \in
  \CH_2(F)$. A straightforward computation involving Proposition
  \ref{prop bvf}, \eqref{eq l on S2}, \eqref{eq rational equation} and
  \eqref{eq Lxy} gives $$(L^2)_*[x,y] =
  2[x,y]-2[\mathfrak{o}_S,x]-2[\mathfrak{o}_S,y] +
  2[\mathfrak{o}_S,\mathfrak{o}_S].$$ Thus $L^2$ acts as zero on
  cycles of the form $[x,x]$. Proposition \ref{prop first dec result}
  implies that $\ker\,\{(I^2-4\Delta_F)_*\} \subseteq
  \ker\,\{(L^2)_*\}$. Consider now $\sigma \in \CH_0(F)$ such that
  $(L^2)_*\sigma = 0$. Then by \eqref{eq rational equation} we see
  that $\sigma$ is a linear combination of $l\cdot L_*\sigma$ and of
  $l^2$. Since $l$ is a linear combination of $\delta^2$ and
  $S_{\mathfrak{o}}$, the cycle $\sigma$ is a linear combination of
  cycles of the form $[\mathfrak{o},x]$. By Proposition \ref{prop
    second dec result}, we conclude that $(I^2)_*\sigma = 4\sigma$.
\end{proof}

\begin{prop}\label{lem condition 3}
  $(L^2)_*(l \cdot (L^2)_*\sigma)=0$ for all $\sigma \in \CH^2(F)$.
\end{prop}
\begin{proof}
  In view of Theorem \ref{prop cubic L conjecture} and Proposition
  \ref{prop cubic assumption pre l hom}, Theorem \ref{prop L2} gives
  eigenspace decompositions $\CH^4(F) = \Lambda_0^4 \oplus
  \Lambda_2^4$ and $\CH^2(F) = \Lambda_{25}^2 \oplus \Lambda_2^2
  \oplus \Lambda_0^2$ for the action of $L^2$. Moreover,
  $\Lambda_{25}^2 = \langle l \rangle$ and Lemma \ref{lem assumption
    l} gives $(L^2)_*l^2 = 0$. Therefore, it is sufficient to prove
  that $l \cdot \sigma = 0$ for all $\sigma \in \Lambda_2^2$.
  According to \eqref{eq action pi6 CH2}, $\sigma \in \Lambda_2^2$ if
  and only if $l\cdot \sigma \in \Lambda_2^4$. But, since \eqref{eq l
    on S2} $l= 20S_{\mathfrak{o}}- \frac{5}{2}\delta^2$, Propositions
  \ref{prop first dec result} and \ref{prop second dec result},
  together with Lemma \ref{lem comparison I L}, show that $l \cdot
  \sigma \in \Lambda_0^4$ for all $\sigma \in \CH^2(F)$.  Hence if
  $\sigma \in \Lambda_2^2$, then $l \cdot \sigma \in \Lambda_0^4\cap
  \Lambda_2^4 = \{0\}$.
\end{proof}

By the work carried out in Part I, the Fourier decomposition as in
Theorem \ref{thm main splitting} is proved for the Hilbert scheme of
length-$2$ subschemes of K3 surfaces.\qed


\vspace{10pt}
\section{The Fourier decomposition for $S^{[2]}$ is
  multiplicative} \label{section mult S2}

Let $F=S^{[2]}$ be the Hilbert scheme of length-$2$ subschemes on a
$K3$ surface $S$. To avoid confusion, the Fourier decomposition of the
Chow groups of $S^{[2]}$ obtained in Section \ref{sec Fourier S2} will
be denoted $\CH^p_{\FF}(-)_s$. We continue to denote
$\CH_{\mathrm{CK}}^p(-)_s$ the decomposition induced by the
multiplicative Chow--K\"unneth decomposition of Theorem \ref{thm
  multCK X2} obtained by considering the multiplicative
Chow--K\"unneth decomposition on $S$ of Example \ref{ex K3}. The goal
of this section is to show that these two decompositions coincide,
thereby proving Theorem \ref{thm main mult} in the case $F=S^{[2]}$.

\subsection{The Chow--K\"unneth decomposition of $\CH^*(S^{[2]})$}
Let $S$ be a K3 surface and let $\mathfrak{o}_S$ be the class of a
point on $S$ lying on a rational curve. Recall that the following
idempotents in $\CH^2(S\times S)$
\begin{equation} \label{eq K3multCK} \pi^0_S =\mathfrak{o}_S\times S,
  \qquad \pi^4_S = S\times \mathfrak{o}_S , \qquad \pi^2_S = \Delta_S
  -\pi^0_S-\pi_S^4
\end{equation}
define a multiplicative Chow--K\"unneth decomposition of $S$ in the
sense of Definition \ref{def multCK}. The Chow ring of $S$ inherits a
bigrading $$\CH^i(S)_s := (\pi_S^{2i-s})_* \CH^i(S).$$ As explained in
the introduction, this bigrading coincides with the one induced by the
Fourier transform $\FF (-) := (p_2)_*(\pi_S^2\cdot p_1^* -)$. We will
thus omit the subscript ``CK'' or ``$\FF$'' when mentioning the Chow
ring of $S$.

The variety $S\times S$ is naturally endowed with the product
Chow--K\"unneth decomposition given by
\begin{equation}
  \label{eq productCK} \pi^k_{S\times S} = \sum_{i+j=k}\pi^i_S\otimes
  \pi^j_S,
  \quad 0 \leq k \leq 8.
\end{equation}
According to Theorem \ref{thm multCK}, this Chow--K\"unneth
decomposition is multiplicative and induces a
bigrading $$\CH^i_{\mathrm{CK}}(S\times S)_s := (\pi^{2i-s}_{S\times
  S})_*\CH^i(S\times S)$$ of the Chow ring of $S \times S$.\medskip

Let $p_i :S\times S\rightarrow S$, $i=1,2$, be the two projections. We
write $(\mathfrak{o}, \mathfrak{o})  := p_1^*\mathfrak{o}_S\cdot
p_2^*\mathfrak{o}_S\in\CH_{\mathrm{CK}}^4(S\times S)$. We define
\begin{align*}
 \phi_1 : S\rightarrow S\times S, \qquad x\mapsto (x,\mathfrak{o}_S) \, ;\\
 \phi_2 : S\rightarrow S\times S, \qquad x\mapsto (\mathfrak{o}_S,x).
\end{align*}

\begin{lem}\label{lem phi compatible}
  The above morphisms $\phi_i$, $i=1,2$, satisfy
\[
\phi_i^*\CH_{\mathrm{CK}}^p(S\times S)_s \subseteq \CH^p(S)_s \qquad
\text{and}\qquad \phi_{i_*}\CH^p(S)_s \subseteq
\CH_{\mathrm{CK}}^{p+2}(S\times S)_s.
\]
\end{lem}
\begin{proof}
  By Proposition \ref{lem diagonal pullback}\emph{(i)}, we see that
  $\phi_{1_*} [S] = p_2^*\mathfrak{o}_S$ belongs to
  $\CH_{\mathrm{CK}}^2(S\times S)_0$.  Let $\alpha\in\CH^p(S)_s$~;
  then
\[
\phi_{1_*}\alpha = p_1^*\alpha \cdot \phi_{1_*}S \in
\CH_{\mathrm{CK}}^p(S\times S)_s \cdot \CH_{\mathrm{CK}}^2(S\times
S)_0\subseteq \CH_{\mathrm{CK}}^{p+2}(S\times S)_s.
\]
Let $\beta\in\CH_{\mathrm{CK}}^p(S\times S)_s$~; then
\[
\phi_{1_*}\phi_1^*\beta = \beta\cdot \phi_{1_*}S \in
\CH_{\mathrm{CK}}^{p+2}(S\times S)_2.
\]
Given that $\phi_{1_*}$ is injective and respects the Chow--K\"unneth
decomposition, the above equation implies $\phi_{1}^*\beta \in
\CH^p(S)_s$.
\end{proof}

\begin{prop}\label{prop CK pieces S2}
  The Chow--K\"unneth decomposition of the Chow groups of $S\times S$
  can be described as
\begin{align*}
  \CH_{\mathrm{CK}}^2(S\times S)_0 & = \{\alpha\in \CH^2(S\times S) :
  \alpha \cdot p_1^*\mathfrak{o}_S \in \langle (\mathfrak{o},
  \mathfrak{o})\rangle,\, \alpha \cdot p_2^*\mathfrak{o}_S \in
  \langle(\mathfrak{o},\mathfrak{o})\rangle\} \, ;\\
  \CH_{\mathrm{CK}}^2(S\times S)_2 & = p_1^*\CH^2(S)_2 \oplus p_2^* \CH^2(S)_2\, ;\\
  \CH_{\mathrm{CK}}^3(S\times S)_0 & = p_1^*\mathfrak{o}_S\cdot p_2^*
  \CH^1(S) \oplus p_1^*\CH^1(S)\cdot p_2^*\mathfrak{o}_S\, ;\\
  \CH_{\mathrm{CK}}^3(S\times S)_2 & = \CH^3(S\times S)_{\mathrm{hom}}\, ;\\
  \CH_{\mathrm{CK}}^4(S\times S)_0 & = \Q (\mathfrak{o},\mathfrak{o})\, ;\\
  \CH_{\mathrm{CK}}^4(S\times S)_2 & = p_1^*\mathfrak{o}_S\cdot
  p_2^*\CH^2(S)_2 \oplus p_1^*\CH^2(S)_2\cdot p_2^*\mathfrak{o}_S\, ;\\
  \CH_{\mathrm{CK}}^4(S\times S)_4 & = p_1^*\CH^2(S)_2\cdot
  p_2^*\CH^2(S)_2.
\end{align*}
\end{prop}
\begin{proof}
  In this proof, we use $p_{S^4,i}$ and $p_{S^4,ij}$ to denote the
  projections from $S^4$ to the corresponding factor(s) and we use
  $p_{i}$ to denote the projection from $S^2$ to the corresponding
  factor.

  Let $\alpha\in \CH^2(S\times S)$. Using the expression \eqref{eq
    productCK} for $\pi^k_{S\times S}$, we have
\begin{align*}
  (\pi^0_{S\times S})_* \alpha & = \big((\mathfrak{o},\mathfrak{o})
  \times S\times S\big)_*\alpha =0.\\
  (\pi^2_{S\times S})_*\alpha & = (\pi^0_S\otimes \pi^2_S)_*\alpha +
  (\pi^2_S\otimes \pi^0_S)_*\alpha\\
  & = (p_{S^4,34})_*\Big( p_{S^4,1}^*\mathfrak{o}_S\cdot
  p_{S^4,24}^*\pi^2_S
  \cdot p_{S^4,12}^*\alpha \Big)\\
  &\quad +( p_{S^4,34})_*\Big( p_{S^4,2}^*\mathfrak{o}_S\cdot
  p_{S^4,13}^*\pi^2_S \cdot
  p_{S^4,12}^*\alpha \Big)\\
  & = p_{2}^* (\pi^2_S)_*\phi_2^*\alpha  + p_{1}^* (\pi^2_S)_*\phi_1^*\alpha.\\
  (\pi^6_{S\times S})_* \alpha & = 0.\\
  (\pi^8_{S\times S})_*\alpha &=0.
\end{align*}
Note that $\alpha \cdot p_{i}^*\mathfrak{o}_S \in \langle
(\mathfrak{o,o})\rangle$ if and only if $\phi_j^*\alpha \in \langle
\mathfrak{o}_S \rangle$ for $(i,j)=(1,2)$ and $(i,j)=(2,1)$. If
follows that if $\alpha\cdot
p_{i}^*\mathfrak{o}_S\in\langle(\mathfrak{o,o}) \rangle$, then
$(\pi^k_{S\times S})_*\alpha =0$ for all $k\neq 4$. Therefore $\alpha$
belongs to $\CH_{\mathrm{CK}}^2(S\times S)_0$. Conversely, if
$\alpha\in \CH_{\mathrm{CK}}^2(S\times S)_0$, then $\alpha\cdot
p_i^*\mathfrak{o}_S\in \CH_{\mathrm{CK}}^4(S\times S)_0=\langle
(\mathfrak{o,o}) \rangle$.  This proves the first equation. The
remaining equations are easier to establish and their proofs are thus
omitted.
\end{proof}

Let $\rho : Z\rightarrow S\times S$ be the blow-up of $S\times S$ along
the diagonal. We have the following commutative diagram
\[
 \xymatrix{
  E\ar[r]^{j'}\ar[d]_\pi & Z\ar[d]^\rho \\
  S\ar[r]^{\iota_\Delta\quad} &S\times S.
}
\]
The K3 surface $S$ is endowed with its multiplicative self-dual
Chow--K\"unneth decomposition \eqref{eq K3multCK}. We also have
$c_1(S)=0$, and $c_2(S) = 24 \mathfrak{o}_S$ which is in
$\CH^2(S)_0$. Therefore Proposition \ref{prop multCK blowup diag}
yields the existence of a  multiplicative Chow--K\"unneth
decomposition for $Z$ inducing a bigrading
\begin{equation} \label{eq blowupS} \CH^i(Z)_s = \rho^*\CH^i(S\times
  S)_s \, \oplus \, j'_*\pi^*\CH^{i-1}(S)_s
\end{equation}
of the Chow ring of $Z$.
\begin{prop}\label{prop CK pieces Z}
  The Chow--K\"unneth decomposition \eqref{eq blowupS} of the Chow
  groups of $Z$ have the following description
\begin{align*}
  \CH_{\mathrm{CK}}^2(Z)_0 &= \{\alpha \in \CH^2(Z) : \alpha\cdot
  \rho^*p_i^*\mathfrak{o}_S \in \langle \mathfrak{o}_Z\rangle,\,
  \forall i =1,2\}\, ;\\
  \CH_{\mathrm{CK}}^2(Z)_2 & = \rho^*\CH_{\mathrm{CK}}^2(S\times S)_2\, ;\\
  \CH_{\mathrm{CK}}^3(Z)_0 & = \rho^*\CH^3_{\mathrm{CK}}(S\times S)_0
  \oplus \Q\cdot j'_*\pi^*\mathfrak{o}_S\, ;\\
  \CH_{\mathrm{CK}}^3(Z)_2 & = \CH^3(Z)_{\mathrm{hom}}\, ;\\
  \CH_{\mathrm{CK}}^4(Z)_0 & = \langle \mathfrak{o}_Z\rangle,\\
  \CH_{\mathrm{CK}}^4(Z)_2 & = \rho^*\CH^4_{\mathrm{CK}}(S\times S)_2\, ;\\
  \CH_{\mathrm{CK}}^4(Z)_4 & = \rho^*\CH^4_{\mathrm{CK}}(S\times
  S)_4\, ;
\end{align*}
where $\mathfrak{o}_Z := \rho^*(\mathfrak{o,o})$.
\end{prop}
\begin{proof}
  This follows from formula \eqref{eq blowupS} and Proposition
  \ref{prop CK pieces S2}. For example, consider $\alpha\in
  \CH^2(Z)$. We can write
\[
 \alpha = \rho^*\alpha' + j'_*\pi^*\beta,
\]
for some $\alpha'\in \CH^2(S\times S)$ and some $\beta\in
\CH^1(S)$. Then
\begin{align*}
  \alpha \cdot \rho^*p_1^*\mathfrak{o}_S & = \rho^*(\alpha'\cdot
  p_1^*\mathfrak{o}_S) +
  j'_*(\pi^*\beta \cdot j'^* \rho^*p_1^*\mathfrak{o}_S)\\
  & = \rho^*(\alpha'\cdot p_1^*\mathfrak{o}_S) +
  j'_*(\pi^*\beta \cdot \pi^*\iota_\Delta^* p_1^*\mathfrak{o}_S)\\
  & = \rho^*(\alpha'\cdot p_1^*\mathfrak{o}_S) +
  j'_*(\pi^*\beta \cdot \pi^*\mathfrak{o}_S)\\
  & = \rho^*(\alpha'\cdot p_1^*\mathfrak{o}_S).
\end{align*}
Hence the condition $\alpha\cdot \rho^*p_1^*\mathfrak{o}_S\in \langle
\mathfrak{o}_Z\rangle$ is equivalent to $\alpha'\cdot
p_1^*\mathfrak{o}_S\in \langle (\mathfrak{o,o})\rangle$, which is
further equivalent to $\alpha'\in \CH^2_{\mathrm{CK}}(S\times
S)_0$. By symmetry, the same holds if we replace $p_1$ by $p_2$. This
yields the first equality. All the other equalities can be verified
similarly.
\end{proof}

\subsection{The Fourier decomposition of $\CH^*(S^{[2]})$}
We now provide an explicit description of the Fourier decomposition on
the Chow groups of $F$. A satisfying description consists in
understanding how $S_{\mathfrak{o}}$ and $\delta$ intersect with
$2$-cycles. For that matter, Lemma \ref{lem intersecting 2-cycle and
  S_o} is crucial~; it relies on the key Lemma \ref{lem lifting
  hilbert}.\medskip

Let $p :Z\rightarrow F=S^{[2]}$ be the natural double cover and
$j=p\circ j' :E\rightarrow F$. Recall the following definition
\eqref{eq A} from Section \ref{sec mult X2}~:
\[
 \mathcal{A}  := I_*\CH^4(F)\subset \CH^2(F).
\]
Note that $\mathcal{A}$ is generated by cycles of the form $S_x$, $x\in S$.

\begin{lem}\label{lem CH^2_2 equals Ahom and more}
The Fourier decomposition enjoys the following properties.
\begin{enumerate}[(i)]
\item $\CH_{\FF}^2(F)_2 = \mathcal{A}_{\mathrm{hom}}$~;

\item $\CH^4_{\FF}(F)_2 = \langle [\mathfrak{o},x]-[\mathfrak{o}, y]
 \rangle$, $x,y\in S$~;

\item $I_* : \CH^4_{\FF}(F)_2 \rightarrow \CH^2_{\FF}(F)_2$ is an
 isomorphism and its inverse is given by intersecting with
 $S_\mathfrak{o}$~;

\item $\delta\cdot \CH^3(F) \subseteq \CH^4(F)_2\oplus \CH^4(F)_0$.
\end{enumerate}
\end{lem}
\begin{proof}
  By Theorems \ref{prop L2} and \ref{prop main} and equation \eqref{eq
    Lxy}, we have
\[
 \CH^2_{\FF}(F)_2 = \Lambda_0^2 = L_*\CH^4(F) = \mathcal{A}_{\mathrm{hom}}.
\]
This proves \emph{(i)}. For \emph{(ii)}, we argue as follows
\begin{align*}
  \CH^4_{\FF}(F)_2 & = (\Lambda_0^4)_{\mathrm{hom}} \qquad
  \text{(Theorem \ref{prop main})}\\
  & = \Big(\ker\{L^2 : \CH^4(F)\rightarrow \CH^4(F)\}\Big)_\mathrm{hom}
  \qquad
  \text{(by definition of $\Lambda^4_0$)}\\
  & = \Big(\ker\{I^2-4\Delta_F : \CH^4(F)\rightarrow
  \CH^4(F)\}\Big)_\mathrm{hom} \qquad
  \text{(Lemma \ref{lem comparison I L})}\\
  & = \mathrm{im}\{\CH_0(S_\mathfrak{o})_\mathrm{hom}\rightarrow
  \CH_0(F)\} \qquad \text{(Proposition \ref{prop second dec result})}.
\end{align*}
Statement \emph{(iii)} follows from \emph{(i)} and \emph{(ii)}. Note
that Proposition \ref{prop bvf} shows that a 0-cycle supported on $E$
is also supported on $S_\mathfrak{o}$. Hence any 0-cycle supported on
$E$ belongs to $\CH^4_{\FF}(F)_0 \oplus \CH^4_{\FF}(F)_2$. Statement
\emph{(iv)} follows since $2\delta = j_*E$.
\end{proof}

We now come to the crucial lemma.

\begin{lem}\label{lem intersecting 2-cycle and S_o}
 Let $\alpha\in \CH^2_{\FF}(F)_0$, then
\[
 \alpha \cdot S_{\mathfrak{o}} \in \CH^4_{\FF}(F)_0 \qquad\text{and }\quad
 \alpha \cdot \delta^2 \in \CH^4_{\FF}(F)_0.
\]
\end{lem}
\begin{proof}
  On the one hand, by Proposition \ref{prop l CH2 4} and equation
  \eqref{eq l on S2}, we know that
\[
\alpha\cdot l = \alpha \cdot \big(20 S_{\mathfrak{o}} + \frac{5}{2}
\delta^2 \big) = 20 \alpha \cdot S_{\mathfrak{o}} +
\frac{5}{2}\alpha\cdot \delta^2\in \CH_{\FF}^4(F)_0.
\]
On the other hand, by the key Lemma \ref{lem lifting hilbert}, we see
that
\[
I_*(2\alpha\cdot S_{\mathfrak{o}} +\alpha\cdot\delta^2 )=\Big( p_1^*(2
S_\mathfrak{o}+ \delta^2)\cdot I\Big)_*\alpha \quad \text{is a
  multiple of }S_{\mathfrak{o}} \ \text{in} \ \CH^2(F).
\]
Since $\alpha\cdot S_{\mathfrak{o}}$ is supported on $S_\mathfrak{o}$
and since $\alpha\cdot\delta^2$ is supported on the boundary divisor,
it follows from Lemma \ref{lem CH^2_2 equals Ahom and more} that
\[
\beta  := 2\alpha\cdot S_{\mathfrak{o}} +\alpha\cdot\delta^2 \in
\CH^4_{\FF}(F)_0 \oplus \CH_{\FF}^0(F)_2.
\]
Note that $S_\mathfrak{o}$ is a linear combination of $l$ and
$\delta^2$. We know that $l\in \CH^2_{\FF}(F)_0$ (Theorem \ref{prop
  main}) and $\delta^2\in \CH^2_{\FF}(F)_0$ (Theorem \ref{thm
  reformulation voisin}).  It follows that $S_\mathfrak{o}\in
\CH^2_{\FF}(F)_0$.  By Lemma \ref{lem CH^2_2 equals Ahom and more},
$I_* :\CH^4_\FF(F)_2\rightarrow \CH^2_\FF(F)_2$ is an isomorphism.
Hence the fact that $I_*\beta$ is a multiple of
$S_\mathfrak{o}\in\CH^2_{\FF}(F)_0$ implies that $\beta\in
\CH^4_\FF(F)_0$. Thus we have two different linear combinations of
$S_{\mathfrak{o}}$ and $\delta^2$ whose intersections with $\alpha$
are multiples of $\mathfrak{o}_F$. The lemma follows immediately.
\end{proof}

In order to get a description of the Fourier decomposition on $F$, we
introduce some new notations. For a divisor $D\in \CH^1(S)$, we write
\begin{align*}
  \hat D &  := p_*\rho^* p_1^*D \in \CH^1(F) \, ;\\
  D_\mathfrak{o} &  := S_{\mathfrak{o}} \cdot \hat{D}.
\end{align*}

\begin{prop}\label{prop Fourier pieces F}
 The Fourier decomposition enjoys the following properties.
\begin{align*}
  \CH_{\FF}^2(F)_0 & = \{\alpha\in \CH^2(F) : \alpha\cdot
  S_{\mathfrak{o}}\in
  \langle \mathfrak{o}_F\rangle \}\\
  & = \{ \alpha\in \CH^2(F) : \alpha\cdot \delta^2 \in \langle
  \mathfrak{o}_F\rangle \} \, ;\\
  \CH^2_{\FF}(F)_2 & = \langle S_x -S_y\rangle, \quad x,y\in S \, ;\\
  \CH^3_{\FF}(F)_0 & = j_*\pi^*\mathfrak{o}_S\oplus
  \{D_{\mathfrak{o}} :
  D\in \CH^1(S)\} \, ;\\
  \CH^3_{\FF}(F)_2 & = \CH^3(F)_{\mathrm{hom}} \, ;\\
  \CH^4_{\FF}(F)_0 & = \langle \mathfrak{o}_F \rangle \, ;\\
  \CH^4_{\FF}(F)_2 & = \langle[\mathfrak{o},x] - [\mathfrak{o}, y]
  \rangle,
  \quad x,y\in S \, ; \\
  \CH^4_{\FF}(F)_4 & = \CH^2_{\FF}(F)_2\cdot \CH^2_{\FF}(F)_2.
\end{align*}
\end{prop}
\begin{proof}
  If $\alpha\in\CH^2_{\FF}(F)_0$, then by Lemma \ref{lem intersecting
    2-cycle and S_o} we know that both of $\alpha\cdot
  S_{\mathfrak{o}}$ and $\alpha\cdot \delta^2$ are multiples of
  $\mathfrak{o}_F$. Conversely, let $\alpha\in \CH^2(F)$ be an
  arbitrary 2-cycle on $F$. Then we can write
 \[
 \alpha = \alpha_0 +\alpha_2,\qquad \alpha_0\in \CH_\FF^2(F)_0,\quad
 \alpha_2\in \CH^2_\FF(F)_2.
 \]
 Assume that $\alpha\cdot S_{\mathfrak{o}}\in \CH_\FF^4(F)_0$. Note
 that Lemma \ref{lem intersecting 2-cycle and S_o} implies that
 $\alpha_0\cdot S_{\mathfrak{o}}\in \CH^4_{\FF}(F)_0$. Hence we get
 $\alpha_2\cdot S_{\mathfrak{o}}\in \CH^4_\FF(F)_0$. Since
 intersecting with $S_{\mathfrak{o}}$ defines an isomorphism between
 $\CH^2_\FF(F)_2$ and $\CH^4_\FF(F)_2$ by Lemma \ref{lem CH^2_2 equals
   Ahom and more}\emph{(iii)}, it follows that $\alpha_2\cdot
 S_{\mathfrak{o}}=0$.  Therefore
 \[
 \alpha_2 = I_*(\alpha_2\cdot S_{\mathfrak{o}}) =0.
 \]
 Thus $\alpha=\alpha_0\in \CH_\FF^2(F)_0$. This establishes the first
 equality. The second equality is established similarly. The equations
 involving $\CH^2_{\FF}(F)_2$ and $\CH^4_{\FF}(F)_2$ are direct
 consequences of Lemma \ref{lem CH^2_2 equals Ahom and
   more}. Combining Theorems \ref{prop L2} and \ref{prop main}, we
 have
\[
 \CH^3_{\FF}(F)_2 = \Lambda^3_2 =\CH^3(F)_{\mathrm{hom}}.
\]

The equation involving $\CH^3_{\FF}(F)_0$ follows from the identity
$\CH^3_{\FF}(F)_0=\CH^1(F)^{\cdot 3}$ of Remark \ref{rmk CH3
  divisors}. Indeed, we have $\CH^1(F)= \langle \delta\rangle \oplus
\{\hat{D} :D\in \CH^1(S)\}$. A direct calculation gives
\begin{align*}
  \delta^3 & = -12 j_*\pi^*\mathfrak{o}_S \, ;\\
  \delta^2 \cdot \hat D &= -j_*(\xi\cdot \pi^*D) \, ;\\
  \delta\cdot \hat{D}_1 \cdot \hat{D}_2 & = 2j_*\pi^*(D_1\cdot D_2) \, ;\\
  \hat{D}_1\cdot \hat{D}_2\cdot \hat{D}_3 &= a_1 (D_1)_\mathfrak{o} +
  a_2 (D_2)_\mathfrak{o} +a_3 (D_3)_\mathfrak{o} \, ;
\end{align*}
where $\xi\in \CH^1(E)$ is the class of the relative $\calO(1)$
bundles and the $a_i$'s are rational numbers. Note that
\begin{align*}
  p^* j_* (\xi\cdot \pi^*D) = 2 j'_*(\xi\cdot \pi^*D)\in
  \CH^3_{\mathrm{CK}}(Z)_0.
\end{align*}
By Propositions \ref{prop CK pieces S2} and \ref{prop CK pieces Z}, we
get
\[
p^*j_* (\xi\cdot \pi^*D) = \rho^*(D'_1 \times \mathfrak{o}_S +
\mathfrak{o}_S\times D'_2) + b j'_*\pi^*\mathfrak{o}_S,\qquad b\in\Q.
\]
Applying $p_*$ to the above we find
\[
j_*(\xi\cdot \pi^*D) = \frac{1}{2} ((D'_1)_\mathfrak{o}+
(D'_2)_\mathfrak{o}) + \frac{1}{2} j_*\pi^*\mathfrak{o}_S.
\]
Combining all the above equations, we get the expression for
$\CH_{\FF}^3(F)_0$.

It remains to show the equation involving $\CH_{\FF}^4(F)_4$. This can
be done as follows
\begin{align*}
  \CH_{\FF}^4(F)_4 & = \Lambda_2^4\qquad\qquad \text{(Theorem \ref{prop main})}\\
  &= \ker\{L_* : \CH^4(F)_{\mathrm{hom}}\rightarrow \CH^2(F)\}
  \qquad \text{(Theorem \ref{prop L2})}\\
  & =\ker\{I_* : \CH^4(F)\rightarrow \CH^2(F)\}\qquad
  \text{(Lemma \ref{lem comparison I L})}\\
  & = \mathcal{A}_{\mathrm{hom}}\cdot \mathcal{A}_{\mathrm{hom}}
  \qquad \text{(Proposition \ref{prop F4 of S2})}\\
  & = \CH^2_{\FF}(F)_2\cdot \CH^2_{\FF}(F)_2,
\end{align*}
where the last equality follows from Lemma \ref{lem CH^2_2 equals Ahom
  and more}.
\end{proof}

\subsection{The Chow--K\"unneth decomposition and the Fourier
  decomposition agree} Recall that $p : Z \rightarrow F$ denotes the
quotient map from the blow-up of $X\times X$ along the diagonal to the
Hilbert scheme $F=X^{[2]}$. The following proposition shows that the
Chow--K\"unneth decomposition \eqref{eq blowupS} on $\CH^*(Z)$ is
compatible with the Fourier decomposition of Theorem \ref{thm main
  splitting} on $\CH^*(F)$, via the morphism $p$.
\begin{prop}
  The following holds
 \[
 \CH^i_{\FF}(F)_s = \{\alpha\in \CH^i(F) : p^*\alpha\in
 \CH^i_{\mathrm{CK}}(Z)_s\}.
 \]
\end{prop}
\begin{proof}
  This is a direct consequence of Propositions \ref{prop CK pieces
    S2}, \ref{prop CK pieces Z} and \ref{prop Fourier pieces F}. For
  example, consider a $2$-cycle $\alpha\in \CH^2(F)$. If $p^*\alpha\in
  \CH^2_{\mathrm{CK}}(F)_0$, then
\[
p^*\alpha \cdot \rho^*p_1^*\mathfrak{o}_S = a \mathfrak{o}_Z, \quad
\mbox{for some} \ a\in \Q.
\]
We apply $p_*$ to the above equation and find
\[
\alpha\cdot S_\mathfrak{o} = p_*(p^*\alpha \cdot
\rho^*p_1^*\mathfrak{o}_S) = a p_*\mathfrak{o}_Z = a \mathfrak{o}_F.
\]
This yields that $\alpha\in \CH^2_{\FF}(F)_0$. Conversely, if
$\alpha\in \CH^2_{\FF}(F)_0$, then we write
\[
 p^*\alpha = \beta_0 +\beta_2, \qquad \beta_s\in \CH^2_{\mathrm{CK}}(Z)_s.
\]
The condition $\alpha \cdot S_\mathfrak{o} = b\mathfrak{o}_F$ for some
$b\in \Q$ implies that
\[
p_*(p^*\alpha \cdot \rho^*p_i^*\mathfrak{o}_S) =\alpha\cdot
S_{\mathfrak{o}} = b\mathfrak{o}_F.
\]
It follows that $p_*(\beta_2\cdot \rho^*p_i^*\mathfrak{o}_S) =
b'\mathfrak{o}_F$. Since $\beta_2$ is homologically trivial, we
actually have
 \[
 p_*(\beta_2\cdot \rho^*p_i^*\mathfrak{o}_S) = 0.
 \]
 Since $\beta_2$ is symmetric, we can write
\[
\beta_2 = \sum_i\rho^*\big(p_1^*(x_i-y_i) + p_2^*(x_i -y_i)\big),
\quad \mbox{for some} \ x_i, y_i \in S.
\]
Therefore we have
\[
0=p_*(\beta_2\cdot \rho^*p_1^*\mathfrak{o}_S) = p_* \rho^*
\left(\sum_i (\mathfrak{o},x_i)-(\mathfrak{o},y_i)\right) =
\sum_i([\mathfrak{o},x_i]-[\mathfrak{o},y_i]).
\]
It follows that $\sum_i (x_i-y_i)=0$ in $\CH^2(S)$. Thus we get
$\beta_2=0$, namely $p^*\alpha\in \CH^2_{\mathrm{CK}}(Z)_0$. All the
other cases can be easily proved and the details are thus omitted.
\end{proof}

Given Remark \ref{rmk compatibility blowup hilbert}, we readily get
the following
\begin{thm} \label{thm multS2complete} The Fourier decomposition
  agrees with the Chow--K\"unneth decomposition, namely
\[
\CH^p_{\FF}(S^{[2]})_s = \CH^p_{\mathrm{CK}}(S^{[2]})_s.
\]
Consequently, the Fourier decomposition is compatible with
intersection product, \emph{i.e}, the conclusion of Theorem \ref{thm
  main mult} holds for $F=S^{[2]}$.  \qed
\end{thm}

\begin{rmk}
  One can check that the cycle $L$ representing the
  Beauville--Bogomolov class defined in \eqref{eq L S2} belongs to the
  graded piece $\CH_{\mathrm{CK}}^2(S^{[2]}\times S^{[2]})_0$. Here,
  the Chow--K\"unneth decomposition of $S^{[2]}\times S^{[2]}$ is
  understood to be the product Chow--K\"unneth decomposition of the
  multiplicative Chow--K\"unneth decomposition of $S^{[2]}$ given by
  Theorem \ref{thm2 CK}.
\end{rmk}

\section{The cycle $L$ of $S^{[2]}$ via moduli of stable sheaves}
\label{sec L agree}

Let $F=S^{[2]}$ for some K3 surface $S$. So far, we have defined two
cycles $L \in \CH^2(F \times F)$ lifting the Beauville--Bogomolov
class $\mathfrak{B}$. The first one was defined in \eqref{def L
  Markman sheaf} using Markman's twisted sheaf $\mathscr{M}$ of
Definition \ref{def Markman sheaf} on $F \times F$ with $F$ seen as
the moduli space of stable sheaves on $S$ with Mukai vector
$v=(1,0,-1)$, while the second one was defined in \eqref{eq L S2}
using the incidence correspondence $I$. The goal of this section is to
prove that these two cycles agree.

\begin{prop}\label{prop L agree}
  The cycle $L$ defined in \eqref{eq L S2} agrees with the cycle $L$
  defined in \eqref{def L Markman sheaf}.
\end{prop}

Let $Z\subset F\times S$ be the universal family of length-$2$
subschemes of $S$. Let $\mathcal{I}$ be the ideal sheaf defining $Z$
as a subscheme of $F\times S$. In this way, we realize $F$ as the
moduli space of stable sheaves with Mukai vector $v=(1,0,-1)$ and
$\mathcal{I}$ is the universal sheaf. As before, we use
$p :Z\rightarrow F$ and $q :Z\rightarrow S$ to denote the natural
projections. These are the restriction to $Z$ of the natural
projections $\pi_F$ and $\pi_S$ from $F\times S$ to the corresponding
factors.

\begin{lem}\label{lem chern character of ideal} Let $Y$ be a smooth
projective variety and $W\subset Y$ a smooth closed sub-variety. If
$\mathcal{I}_{W}$ is the ideal sheaf defining the sub-variety $W$, then
$$ \mathrm{ch}(\mathcal{I}_W) = 1 -
i_{W_*}\left(\frac{1}{\mathrm{td}(\mathscr{N}_{W/Y})}\right),$$ where
$i_W : W\hookrightarrow Y$ is the natural closed embedding.
\end{lem}
\begin{proof}
  There is a short exact sequence $$ \xymatrix{ 0\ar[r] &\mathcal{I}_W
    \ar[r] &\calO_Y \ar[r] &i_{W_*}\calO_W \ar[r] &0 }
$$ on $Y$ and hence we have $\mathrm{ch}(\mathcal{I}_W) = 1 -
\mathrm{ch}(i_{W_*}\calO_W)$. If we apply the
Grothendieck--Riemann--Roch theorem to the morphism $i_W$ and the
structure sheaf $\calO_W$, we get
\begin{align*}
\mathrm{ch}(i_{W_*}\calO_W) & =
\frac{i_{W*}\left(\mathrm{ch}(\calO_W)\mathrm{td}_W \right)}{\mathrm{td}_Y}\\
 & =\frac{i_{W_*}\mathrm{td}_W}{\mathrm{td}_Y}\\
 & = i_{W_*}\left(\frac{1}{\mathrm{td}(\mathscr{N}_{W/Y})}\right),
\end{align*}
 which proves the lemma.
\end{proof}

\begin{proof}[Proof of Proposition \ref{prop L agree}]
  Considering the universal family $i_Z : Z\hookrightarrow F\times S$,
  we easily deduce from \eqref{eq Chern class of normal bundle} that
 $$
 c_1(\mathscr{N}_{Z/F\times S}) = E,\quad c_2(\mathscr{N}_{Z/F\times
   S})= 24 q^*\mathfrak{o}_S - E^2, \quad c_3(\mathscr{N}_{Z/F\times
   S})=0,\quad c_4(\mathscr{N}_{Z/F\times S})=0.
 $$
Then the Todd classes of $\mathscr{N}_{Z/F\times S}$ can be written
 explicitly as
$$
\mathrm{td}_1(\mathscr{N}_{Z/F\times
   S})=\frac{1}{2}E, \quad \mathrm{td}_2(\mathscr{N}_{Z/F\times S})= 2
 q^*\mathfrak{o}_S, \quad \mathrm{td}_3(\mathscr{N}_{Z/F\times S})= 2
 q^*\mathfrak{o}_S\cdot E , \quad \mathrm{td}_4(\mathscr{N}_{Z/F\times
   S})=0,
$$
where we use the fact that  $E^3 = -24
 j_{Z*}\eta^*\mathfrak{o}_S$. By Lemma \ref{lem chern character of
   ideal}, we have
\begin{equation}\label{eq chern character of I}
\mathrm{ch}(\mathcal{I}) = 1 -Z +\frac{1}{2}\pi_F^*\delta\cdot Z
+(2\pi_S^*\mathfrak{o}_S\cdot Z -\frac{1}{4}\pi_F^*\delta^2\cdot Z)
+\frac{1}{8}\pi_F^*\delta^3\cdot Z +\frac{1}{2}\pi_F^*\mathfrak{o}_F\cdot
Z.
\end{equation}
Let $\mathscr{E}=\mathrm{R}p_{12_*}(p_{13}^*\mathcal{I}^\vee \otimes
p_{23}^*\mathcal{I})$, where $p_{ij}$ and $p_k$ are the natural
projections from $F\times F\times S$ and all the push-forward,
pull-back and tensor product are in the derived sense. Then, by the
Grothendieck--Riemann--Roch formula, we have
$$
\mathrm{ch}(\mathscr{E}) =
p_{12_*}(p_{13}^*\mathrm{ch}^\vee(\mathcal{I}) \cdot
p_{23}^*\mathrm{ch}(\mathcal{I}) \cdot p_3^*\mathrm{td}_S).
$$
We put
equation \eqref{eq chern character of I} into the above formula and
eventually get
\begin{equation}
\mathrm{ch}(\mathscr{E}) = -2 +(-\delta_1 +\delta_2)
  +(I - \frac{1}{2}\delta_1^2 -\frac{1}{2}\delta_2^2) + \cdots,
\end{equation}
where $I$ is the incidence correspondence of Definition \ref{defn
  incidence correspondence S2}. Let $\mathscr{E}^i :=
\mathcal{E}xt^i_{p_{12}}(p_{13}^*\mathcal{I},p_{23}^*\mathcal{I})$,
$i=0,1,2$. Note that in the current situation, Markman's sheaf $\mathscr{M}$
is simply $\mathscr{E}^1$. In the $K$-group, we have
$$
 [\mathscr{E}] = [\mathscr{E}^0]-[\mathscr{E}^1]+[\mathscr{E}^2]
$$
By Lemma \ref{lem support of ext}, $\mathscr{E}^0$ and $\mathscr{E}^2$
do not contribute to lower degree terms in the total Chern
character. Hence $\mathrm{ch}(\mathscr{E}^1)$ agrees with
$-\mathrm{ch}(\mathscr{E})$ in low degree terms. To be more precise,
we have
$$
\mathrm{ch}(\mathscr{E}^1) = 2 + (\delta_1
- \delta_2) +(-I +\frac{1}{2}\delta_1^2 +\frac{1}{2}\delta_2^2)
+\cdots.
$$
It follows that
$$
\kappa(\mathscr{E}^1) =
\mathrm{ch}(\mathscr{E}^1)\mathrm{exp}(-c_1(\mathscr{E}^1)/r) = 2 -I
+\frac{1}{4}\delta_1^2 +\frac{1}{4}\delta_2^2
+\frac{1}{2}\delta_1\delta_2+\cdots,
$$
where $r=\rk(\mathscr{E}^1)=2$. Thus we have
\begin{align*}
  \kappa_2(\mathscr{E}^1) &= -I +\frac{1}{4}\delta_1^2
  +\frac{1}{4}\delta_2^2
  +\frac{1}{2}\delta_1\delta_2\\
  &= -L -2p_1^*S_{\mathfrak{o}} -2p_2^*S_{\mathfrak{o}}
  ++\frac{1}{4}\delta_1^2 +\frac{1}{4}\delta_2^2\\
  &= -L -\frac{1}{12}p_1^*c_2(F) -\frac{1}{12}p_2^*c_2(F),
\end{align*} where $p_i :F\times F\rightarrow F$ are the projections.
Here, $L$ denotes the cycle defined in \eqref{eq L S2} and the
third equality uses Lemma \ref{lem chern class S2}. One consequence of
the above is
\begin{equation}\label{eq kappa c2}
\iota_{\Delta}^*\kappa_2(\mathscr{E}^1)= -l -\frac{1}{6}c_2(F) =-c_2(F).
\end{equation}
Comparing the above computation with \eqref{def L Markman sheaf}
yields the proposition.
\end{proof}


\newpage
\begin{large}
\part{The variety of lines on a cubic fourfold}
\end{large}

Fix a $6$-dimensional vector space $V$ over $\C$. Let $X\subset\PP(V)$
be a smooth cubic fourfold and $F=F(X)$ be the variety of lines on
$X$. It is known that $F$ is smooth of dimension $4$~; see \cite{ak}.
Let $h\in\Pic(X)$ be the class of a hyperplane section on $X$. The
variety $F$ naturally embeds into the Grassmannian $G(2,V)$ and hence
inherits the tautological rank $2$ sub-bundle $\mathscr{E}_2\subset
V$. Let $g=-c_1(\mathscr{E}_2)\in\Pic(F)$ be the Pl\"ucker
polarization on $F$ and $c=c_2(\mathscr{E}_2)$. Set
$g_i=p_i^*g\in\CH^1(F\times F)$ and $c_i=p_i^*c\in\CH^2(F\times F)$,
$i=1,2$, where $p_i :F\times F\rightarrow F$ are the projections. For
all points $x\in X$, let $C_x$ be the sub-variety of $F$
parameterizing all lines passing through $x$. If $x \in X$ is general,
then $C_x$ is a curve (\emph{cf.} \cite[Lemma 2.1]{cs}). Voisin
\cite{voisin2} showed that $F$ carries a canonical 0-cycle,
$\mathfrak{o}_F$, of degree 1 such that the intersection of four
divisors is always a multiple of $\mathfrak{o}_F$. We always use
$$
\xymatrix{
 P\ar[r]^q\ar[d]_p &X\\
 F &
}
$$
to denote the universal family of lines on $X$.

\vspace{10pt}
\section{The incidence correspondence $I$}\label{sec I cubic}

We introduce the incidence correspondence $I \in \CH^2(F\times F)$ on
the variety of lines on a cubic fourfold. This cycle is the starting
point to define the Fourier transform on the Chow ring of $F$. We
establish three core identities involving $I$, namely we compute $I^2
= I \cdot I$, the pull-back of $I$ along the diagonal embedding
$\iota_\Delta : F \hookrightarrow F \times F$ and $c_1 \cdot
I$.\medskip

If $l\subset X$ is a line, then we use $[l]\in F$ to denote the
corresponding point on $F$~; If $t\in F$ is a closed point, then we
use $l_t\subset X$ to denote the corresponding line on $X$.

\begin{defn}\label{defn of I}
  The \emph{incidence subscheme} $I\subset F\times F$ is defined to be
$$
I :=\{(s,t)\in F\times F : l_s\cap l_t\neq\emptyset\},
$$
which is endowed with the reduced closed subscheme structure. Its
cycle class in $\CH^2(F\times F)$, also denoted $I$ by abuse of
notations, is called the \emph{incidence correspondence}.  With the
projection onto the first factor $F$, we get a morphism $I\rightarrow
F$ whose fiber over a point $[l]\in F$ is denoted $S_l$.
\end{defn}
It is known that $S_l$ is always a surface and smooth if $l$ is
general~; see \cite{voisin}. Consider the following diagram
\begin{equation}\label{eq double cylingder}
\xymatrix{
 P\times P\ar[rr]^{q\times q}\ar[d]_{p\times p} &&X\times X\\
 F\times F &&
}
\end{equation}

\begin{lem}\label{lem cycle class of I on F(X)}
The cycle class of $I$ is given by
\[
 I = (p\times p)_*(q\times q)^*\Delta_X,\quad \text{in }\CH^2(F\times F).
\]
Equivalently, if one sees the universal family of lines $P$ on the
cubic fourfold $X$ as a correspondence between $F$ and $X$, then $$ I
= {}^tP \circ P ,\quad \text{in }\CH^2(F\times F).$$
\end{lem}
\begin{proof}
  For $x\in X$, the fiber $q^{-1}(x)$ is identified with $C_x$, the
  space of lines passing through $x$.  In \cite[Lemma 2.1]{cs}, it is
  shown that $\dim C_x=1$ for a general point $x\in X$. By
  \cite[Corollary 2.2]{cs}, there are at most finitely many points
  $x_i\in X$, called \emph{Eckardt points} (see \cite[Definition
  2.3]{cs}), such that $C_{x_i}$ is a surface.  As a consequence, $J
  :=(q\times q)^{-1}\Delta_X\subset P\times P$ is an irreducible
  variety of dimension $6$.  Furthermore, the image of $J$ under the
  morphism $p\times p$ is equal to $I$. To prove the lemma, it
  suffices to show that $(p\times p)|_{J} : J\rightarrow I$ is of
  degree $1$. This is clear since a general pair of intersecting lines
  on $X$ meet transversally in a single point.
\end{proof}

For future reference, let us state and prove the following basic
general lemma.
\begin{lem} \label{lem smooth general fiber} Let $Y$ and $Y'$ be
  smooth projective varieties and $\Gamma\subset Y\times Y'$ be a
  closed sub-variety of codimension $r$. Assume that
  $\Gamma\rightarrow Y$ is dominant with smooth general fibers. Then
$$
(\Gamma^2)_*y = (\Gamma_*y)^2
$$
holds in $\CH^*(Y)$ for a general point $y\in Y$. Here, $\Gamma^2$
denotes the self-intersection $\Gamma \cdot \Gamma \in \CH^*(Y \times
Y')$.
\end{lem}
\begin{proof}
This follows from the following computation.
\[
p_{2,*}(\Gamma^2\cdot (y\times Y')) =
i_*c_r(\mathscr{N}_{\Gamma/Y\times Y'}|_{\Gamma_y}) = i_*
c_r(\mathscr{N}_{\Gamma_y/Y'}) = (\Gamma_y)^2,
\]
where $\Gamma_y=\Gamma\cap(y\times Y')$ is viewed as a sub-variety of
$Y'$ and $i :\Gamma_y\hookrightarrow Y'$ is the natural embedding.
\end{proof}

In \cite{voisin2}, C.~Voisin established the following
identity in $\CH^4(F\times F)$,
\begin{equation}\label{identity of voisin2}
  I^2 = \alpha\Delta_F + I\cdot(g_1^2+g_1g_2+g_2^2) +
  \Gamma_2(g_1,g_2,c_1,c_2) \quad \mbox{for some integer } \alpha \neq 0,
\end{equation}
where $\Gamma_2(g_1,g_2,c_1,c_2)$ is a weighted homogeneous polynomial
of degree 4. Voisin's method actually gives
$\Gamma_2=6\Gamma_{h^2}-3(g_1+g_2)\Gamma_h$~; see equations \eqref{eq
  Gamma_h} and \eqref{eq Gamma_h^2} for the notations. Given
Proposition \ref{prop class of Gamma_h}, we readily obtain
$$
\Gamma_2=-g_1^4-g_2^4-g_1^2c_2-g_2^2c_1+2g_1^2c_1+2g_2^2c_2
-g_1g_2(c_1+c_2)+ 2c_1c_2.
$$
This shows that $(\Gamma_2)_*[l]=-g^4+2g^2c$, which is of degree
$-18$. One also sees that
$$
(I\cdot(g_1^2+g_1g_2+g_2^2))_*[l]=g^2\cdot (I_*[l]),
$$
which is of degree $21$ since the cohomology class of $I_*[l]$ is
equal to that of $\frac{1}{3}(g^2-c)$. Meanwhile, by Lemma \ref{lem
  smooth general fiber} we have $(I^2)_*[l]=(I_*[l])^2$, which is of
degree $5$. Letting both sides of the identity \eqref{identity of
  voisin} act on a point $[l]$ and then counting the degrees, we see
that $\alpha=2$.  We have thus established the following refinement of
Voisin's identity (\ref{identity of voisin2}).

\begin{prop}  \label{rmk value of alpha}
The following identity holds true in $\CH^4(F\times F)$,
\begin{equation}\label{identity of voisin}
I^2=2\Delta_F+I\cdot(g_1^2+g_1g_2+g_2^2)+\Gamma_2(g_1,g_2,c_1,c_2),
\end{equation}
where $\Gamma_2(g_1,g_2,c_1,c_2)$ is a weighted homogeneous polynomial
of degree 4. \qed
\end{prop}

The pull-back of $I$ along the diagonal embedding can be expressed in
terms of $c$ and $g^2$~:

\begin{prop}\label{lem I dot diagonal}
  The equation $\iota_\Delta^* I=6c-3g^2$ holds in $\CH^2(F)$, where
  $\iota_\Delta :F\rightarrow F\times F$ is the diagonal embedding.
\end{prop}
\begin{proof}
Consider the fiber product square
\[
\xymatrix{
 P\times_F P\ar[d]^f \ar[r]^i &P\times P\ar[d]^{p\times p}\\
 F\ar[r]^{\iota_{\Delta}} &F\times F.
}
\]
By Lemma \ref{lem cycle class of I on F(X)} we have $I=(p\times
p)_*(q\times q)^*\Delta_X$ and hence we get
$$
\iota_{\Delta}^* I= \iota_{\Delta}^* (p\times p)_*(q\times
q)^*\Delta_X = f_*i^*(q\times q)^* \Delta_X.
$$
Let $\Delta_{P/F} :P\rightarrow P\times_F P$ be the diagonal
morphism. Then one sees that $i^*(q\times
q)^*{\Delta_X}=(\Delta_{P/F})_*P$. Consider the following square
$$
\xymatrix{
  P\ar[r]^{\Delta_{P/F}}\ar[d]_\alpha\ar[rd]^{\iota_{\Delta_P}} &
  P\times_F P\ar[d]^f\\
  (q\times q)^{-1}\Delta_X\ar[r] &P\times P }
$$
where $\iota_{\Delta_P} :P\rightarrow P\times P$ is the diagonal
embedding. Let $\mathscr{N}_1$ be the normal bundle of $(q\times
q)^{-1}\Delta_X$ in $P\times P$. Then $\alpha^*\mathscr{N}_1\cong
q^*\mathscr{T}_X$. Let $\mathscr{N}_2$ be the normal bundle of the
diagonal $\Delta_{P/F}$ in $P\times_F P$. Then $\mathscr{N}_2$ is
naturally isomorphic to $\mathscr{T}_{P/F}$. The quotient
$\alpha^*\mathscr{N}_1/\mathscr{N}_2$ is simply the total normal
bundle $\mathscr{N}_{P/X}$. The refined intersection formula of
Chapter 6 of Fulton \cite{fulton} shows that $ i^* (q\times
q)^*\Delta_X = (\Delta_{P/F})_* c_3(\mathscr{N}_{P/X})$. By applying
$f_*$ we get
$$
\iota_{\Delta}^* I = f_*(\Delta_{P/F})_*c_3(\mathscr{N}_{P/X})
=p_*c_3(\mathscr{N}_{P/X}).
$$
A routine Chern class computation gives
\begin{align*}
  c(\mathscr{N}_{P/X}) & =
  \frac{c(q^*\mathscr{T}_X)}{c(\mathscr{T}_{P/F})} =
  \frac{q^*c(\mathscr{T}_{\PP^5})}{q^*c(\calO_X(3))\cdot
    c(\mathscr{T}_{P/F})} \\
  & = \frac{q^*(1+6h+15h^2+20h^3+15h^4)}{(1+3q^*h)(1+2q^*h-p^*g)}.
\end{align*}
It follows that $c_3(\mathscr{N}_{P/X})=-6q^*h^3+6p^*g\cdot
q^*h^2-3p^*g^2\cdot q^*h+p^*g^3$. Hence we have $
p_*c_3(\mathscr{N}_{P/X})=-6(g^2-c) +6g^2-3g^2=6c-3g^2$. This
concludes the proof.
\end{proof}

The following lifting lemma is essential to getting statements that do
not depend on the choice of a polarization. Specifically, Lemma
\ref{lifting cubic} is to the variety of lines on a cubic fourfold
what Lemma \ref{lem lifting hilbert} is to the Hilbert scheme of
length-$2$ subschemes on a K3 surface~: it is essential to
establishing Conjecture \ref{conj main} for $F$.

\begin{lem}\label{lifting cubic}
  The following relation holds in $\CH^4(F \times F)$~:
$$
c_1 \cdot I = P(g_1,g_2,c_1,c_2)
$$
for some weighted homogeneous polynomial $P$ of degree
$4$.
\end{lem}
\begin{proof}
  Let $H_1\subset X$ be a general hyperplane section of $X$. Then the
  class $c$ can be represented by the surface of all lines contained
  in $H_1$. Then $c_1\cdot I$ is represented by the following cycle
\[
Z_1  := \{([l_1],[l_2])\in F\times F : l_1\subset H_1,\;  l_1\cap l_2
\neq \emptyset\}.
\]
Let $H_2\subset X$ be another general hyperplane section of $X$ and let
\[
Y_1  := \{[l]\in F : l\cap H_1\cap H_2\neq\emptyset\}.
\]
Note that the cycle class of $Y_1$ is simply $g$. The following two
cycles are defined in Appendix A, equations \eqref{eq Gamma_h} and
\eqref{eq Gamma_h^2} respectively.
\[
\Gamma_h= \{([l_1],[l_2])\in F\times F : \exists x\in H_1,\; x\in
l_1\cap l_2\}  ;
\]
\[
\Gamma_{h^2}= \{([l_1],[l_2])\in F\times F : \exists x\in H_1\cap H_2,\;
x\in l_1\cap l_2\}.
\]
We
have the following scheme-theoretic intersection
\[
(Y_1\times F)\cap \Gamma_h = Z_1\, \cup\, \Gamma_{h^2}.
\]
It follows that $g_1\cdot \Gamma_h$ is a linear combination of
$Z_1$ and $\Gamma_{h^2}$. Or equivalently $c_1\cdot I$ can be
written as a linear combination of $g_1\cdot \Gamma_h$ and $\Gamma_{h^2}$.
It follows from Proposition \ref{prop class of Gamma_h} that
$c_1\cdot I$ is a polynomial in $(g_1,g_2,c_1,c_2)$.
\end{proof}

\vspace{10pt}
\section{The rational self-map $\varphi  : F \dashrightarrow F$}
\label{sec varphi intro}

The goal of this section is to understand the action of a rational map
$\varphi  : F \dashrightarrow F$, first constructed by Voisin
\cite{voisin3}, on zero-cycles on $F$. We state and prove two lemmas
that will be used in Section \ref{sec Fourier cubic} to show Theorem
\ref{thm main splitting} for $F$.
\medskip

We refer to Appendix A for definitions and notations.  In
\cite{voisin3}, Voisin defines a rational self-map
$\varphi :F\dashrightarrow F$ as follows.  If a line $l\subset X$ is of
first type, then there is a unique plane $\Pi_l$ which contains $l$
and is tangent to $X$ along $l$. If $\Pi_l$ is not contained in $X$,
then we have
$$
\Pi_l\cdot X=2l+l'
$$
for some line $l'$. Let $\Sigma_1\subset F$ be the sub-variety of lines
contained in some linear plane $\PP^2\subset X$. When $X$ does not contain
any plane, then $\Sigma_1=\emptyset$. If $X$ contains at least a
plane, then $\Sigma_1$ is a disjoint union of $\PP^2$'s. Let
$\Sigma_2\subset F$ be the surface of lines of second type.

\begin{defn}[\cite{voisin3}]\label{defn varphi}
Let $\varphi :F\backslash(\Sigma_1\cup \Sigma_2)\rightarrow F$ be the
morphism defined by $\varphi([l]) = [l']$.
\end{defn}

Lemma \ref{lem intersection of g^2 and S_l} below will be used in the
proof of Proposition \ref{prop cubic assumption l hom} in order to
establish conditions \eqref{assumption pre l} and \eqref{assumption
  hom} in the case where $F$ is the variety of lines on a cubic
fourfold. Note that Lemma \ref{lem intersection of g^2 and S_l} can be
seen as a consequence of Proposition \ref{prop varphi action on
  homologically trivial cycles}. We have however decided for clarity
to postpone the more detailed study of $\varphi$ until after the
Fourier decomposition for $F$ is established.

\begin{lem}\label{lem intersection of g^2 and S_l}
  For any point $[l]\in F$, we have
\begin{equation*}
  g^2\cdot I_*[l] =\varphi_*[l]-4[l]+24\mathfrak{o}_F \ \in
  \CH_0(F).
\end{equation*}
\end{lem}
\begin{proof}
  We may assume that the point $[l]$ is general. We first recall some
  results on the geometry of the surface $S_l\subset F$ of all lines
  meeting $l\subset X$. Our main reference for this is \cite{voisin}.
  It is known that for general $l$ the surface $S_l$ is smooth. On
  $S_l$ we have an involution $\iota$ which is defined as follows. Let
  $[l_1]\in S_l$ which is not $[l]$ then $\iota([l_1])$ is the residue
  line of $l\cup l_1$. If $l_1=l$, then $\iota([l])=\varphi([l])$.
  There are two natural divisors on $S_l$. The first one is $C_x$
  which is the curve parameterizing all the lines passing through a
  point $x\in l$.  The second one is $C_x^\iota=\iota(C_x)$. Then we
  have
$$
C_x^2=[l],\qquad (C_x^\iota)^2=\iota([l]),\qquad
g|_{S_l}=2C_x+C_x^\iota.
$$
Let $j :S_l\rightarrow F$ be the natural inclusion, then
$$
g^2\cdot I_*[l]=g^2\cdot S_l=j_*(j^*g^2)=j_*(2C_x+C_x^\iota)^2 =
j_*(4[l]+4C_x\cdot C_x^\iota+\iota([l]))
$$
One notes that $j_*(2[l]+C_xC_x^\iota)=j_*(C_x)\cdot
g=6\mathfrak{o}_F$. The lemma follows easily.
\end{proof}

We now study the action of $\varphi$ on points $s\in\Sigma_2$, that is,
on points $s \in F$ whose corresponding line $l = l_s$ is of second type~;
see Appendix A. The following lemma is crucial to establishing
condition \eqref{assumption 2hom} for $F$~; see Proposition \ref{prop
  cubic assumption 2hom}.

\begin{lem}\label{lem varphi on Sigma2}
Let $s\in\Sigma_2$, then we have
\[
\varphi_*s = -2s + 3\mathfrak{o}_F \ \in
\CH_0(F).
\]
\end{lem}
\begin{proof}
  We write $l=l_s\subset X$. Let $\PP^3_{\langle l\rangle}$ be the
  linear $\PP^3$ tangent to $X$ along $l$. Let $S=\PP^3_{\langle
    l\rangle}\cap X$. Then there is a canonical curve $l'\subset S$
  such that for any point $s'\in \mathcal{E}_{[l]}$ the corresponding
  line $l_{s'}$ meets both $l$ and $l'$. Let $x\in l$ be a general
  point, then there are two points $s'_1,s'_2\in\mathcal{E}_{[l]}$
  with corresponding lines $l_1$ and $l_2$ such that both of them pass
  through the point $x\in l$~; see Proposition \ref{prop lines of
    second type}. In Lemma \ref{lem specialization of secant lines},
  we established the following fact~: if $L_1$ and $L_2$ are two
  intersecting lines and $L_t$ a family of lines that specialize to
  the line $L_2$, then four secant lines of $(L_1,L_t)$ specialize to
  four lines $E_1,\ldots, E_4$ passing through $x=L_1\cap L_2$ with
\[
[L_1]+ [L_2] + [E_1] +[E_2] +[E_3] +[E_4]= 6\mathfrak{o}_F,\quad
\text{in }\CH_0(F)
\]
and the fifth secant line specializes to the residue line of $L_1$ and
$L_2$. Now we apply this specialization argument by taking $L_1=l_1$,
$L_2=l_2$ and $L_t=l_t$, $t\in\mathcal{E}_{[l]}\backslash\{s_1,s_2\}$
with $t\to s_2$. By Proposition \ref{prop secant line mult}, the
secant lines of $(l_1,l_t)$ are constantly given by $4l + l'$. In
particular the limit is again $4l+l'$. Note that $l'$ is the residue
line of $l_1$ and $l_2$. Hence we have
\[
 4[l] + [l_1] +[l_2] = 6\mathfrak{o}_F.
\]
The following claim shows that $[l_1]=[l_2]=\varphi_*[l]$ in
$\CH_0(F)$.

\begin{claim}
  Let $s=[l]\in \Sigma_2$, then $\varphi_*s = s'$ in $\CH_0(F)$, where
  $s'\in \mathcal{E}_{[l]}$ is an arbitrary point.
\end{claim}
\begin{proof}[Proof of Claim.]
  Consider a general curve $C\subset F$ such that $C$ meets the
  indeterminacy loci of $\varphi$ in the single point $s\in\Sigma_2$.
  Then $\varphi|_{C\backslash\{s\}}$ extends to a morphism
  $\varphi' :C\rightarrow F$ such that $s'=\varphi'(s)\in
  \mathcal{E}_{[l]}$. Note that $s$ is rationally equivalent to a
  cycle $\gamma$ supported on $C\backslash\{s\}$. Then by definition,
  we have
$$
\varphi_*s = \varphi_*\gamma = \varphi'_*\gamma = \varphi'_*s = s'.
$$
Since $\mathcal{E}_{[l]}$ is a rational curve, any point
$s'\in\mathcal{E}_{[l]}$ represents $\varphi_*s$.
\end{proof}
The lemma is now proved.
\end{proof}

\vspace{10pt}
\section{The Fourier decomposition for $F$}\label{sec Fourier cubic}

In this section, we prove Theorem \ref{thm main splitting} for the
variety of lines $F$ on a cubic fourfold. First we construct a cycle
$L$ in $\CH^2(F \times F)$ that represents the Beauville--Bogomolov
form $\mathfrak{B}$, and we show (Theorem \ref{thm cubic L
  conjecture}) that $L$ satisfies Conjecture \ref{conj main},
\textit{i.e.}, that $L^2=2 \, \Delta_F -\frac{2}{25}(l_1+l_2)\cdot L
-\frac{1}{23\cdot 25}(2l_1^2 - 23l_1l_2 + 2l_2^2)$ in $\CH^4(F \times
F)$. Then we check that Properties \eqref{assumption pre l} and
\eqref{assumption hom} are satisfied by $L$ (Proposition \ref{prop
  cubic assumption l hom}). We can thus apply Theorem \ref{prop
  L2}. Finally, we show that $L$ satisfies Property \eqref{assumption
  2hom} (Proposition \ref{prop cubic assumption 2hom}) so that we can
apply Theorem \ref{prop main} and get the Fourier decomposition of
Theorem \ref{thm main splitting} for $F$.\medskip

Let $\{\fa_1,\fa_2,\ldots,\fa_{23}\}$ be a basis of $\HH^4(X,\Z)$.  We
use $b_0 :\HH^4(X,\Z)\times \HH^4(X,\Z)\rightarrow \Z, (\mathfrak{a},
\mathfrak{a}') \mapsto \int_X \mathfrak{a} \cup \mathfrak{a}' $ to
denote the intersection pairing on $\HH^4(X,\Z)$. Let
$A=\big(b_0(\fa_i,\fa_j)\big)_{1\leq i,j \leq 23}$ be the intersection
matrix. For any $\fa\in\HH^4(X,\Q)$ we define
$\hat{\fa}=p_*q^*\fa\in\HH^2(F,\Q)$.  Then by a result of Beauville
and Donagi \cite{bd}, the set $\{\hat{\fa}_1,\ldots,\hat{\fa}_{23}\}$
forms a basis of $\HH^2(F,\Z)$. Consider now $\Lambda :=\HH^2(F,\Z)$
endowed with the Beauville--Bogomolov bilinear form $q_F$. Then $q_F$
can be described as follows
$$
q_F(\hat{\fa},\hat{\fa'})=b_0(\fa, h^2) b_0(\fa',
h^2) - b_0(\fa, \fa').
$$
Hence as in equation \eqref{eq defn b inverse} we obtain a
well-defined element
$$
q_F^{-1}\in\Sym^2(\Lambda)[\frac{1}{2}]\subset\Lambda \otimes
\Lambda[\frac{1}{2}]\subset \HH^4(F\times F,\Q).
$$

\begin{prop}\label{prop cohomology class of I}
  The cohomological class $[I]\in\HH^4(F\times F,\Z)$ is given by
$$
[I]=\frac{1}{3}(g_1^2+\frac{3}{2}g_1g_2+g_2^2-c_1-c_2)-q_F^{-1}.
$$
\end{prop}
\begin{proof}
  Consider the diagram \eqref{eq double cylingder}. By Lemma \ref{lem
    cycle class of I on F(X)} we have $I=(p\times p)_*(q\times
  q)^*\Delta_X$. The cohomology class of the diagonal $\Delta_X\subset
  X\times X$ is given by
$$
[\Delta_X]=pt\times [X] +[X]\times pt + \frac{1}{3}(h\otimes h^3
+h^3\otimes h)+ b_0^{-1},
$$
where $b_0 :\HH^4(X,\Z)\times \HH^4(X,\Z)\rightarrow \Z, (\mathfrak{a},
\mathfrak{a}') \mapsto \int_X \mathfrak{a} \cup \mathfrak{a}' $ is the
intersection pairing. Then the cohomology class of $I$ can be computed
as
\begin{align*}
[I] &= (p\times p)_*(q\times q)^*[\Delta_X]\\
 &= \frac{1}{3}(p_1^*(p_*q^*h^3) + p_2^*(p_*q^*h^3)) -(p\times
 p)_*(q\times q)^* b_0^{-1}\\
 &= \frac{1}{3}(g_1^2+g_2^2 -c_1-c_2) -(p\times
 p)_*(q\times q)^* b_0^{-1}
\end{align*}
Here, we used the fact that $p_*q^*h^3=g^2-c$~; see Lemma \ref{lem
  identities self intersection of g}. We take a basis
$\{\fa'_1=h^2,\fa'_2,\ldots,\fa'_{23}\}$ of $\HH^4(X,\Q)$ such that
$b_0(h^2, \fa'_i)  := \int_X h^2 \cup \fa'_i=0$, for all $2\leq i\leq
23$. Let $A'=\big(b_0(\fa'_i, \fa'_j)\big)_{2 \leq i,j \leq 23}$ and
$B'=A'^{-1}$. Then we have
$$
b_0^{-1}=\frac{1}{3}h^2\otimes h^2 +\sum_{2\leq i,j\leq 23}
b'_{ij}\fa'_{i}\otimes\fa'_j.
$$
On $F$, the Beauville--Bogomolov bilinear form $q_F$ is given by
$q_F(g,g)=6$ and $q_F(\hat{\fa}'_i,\hat{\fa}'_j)=-a'_{ij}$, for all
$2\leq i,j\leq 23$. It follows that
$$
q_F^{-1}=\frac{1}{6}g\otimes g\;  -\sum_{2\leq i,j\leq 23}
b'_{ij}\hat{\fa}'_i\otimes\hat{\fa}'_j.
$$
Hence we have
\begin{align*}
  (p\times p)_*(q\times q)^* b_0^{-1} & = \frac{1}{3}g_1g_2 +
  \sum_{2\leq i,j\leq 23} b'_{ij}p_1^*\hat{\fa}'_i p_2^*\hat{\fa}'_j =
  \frac{1}{2}g_1g_2 -q_F^{-1}.
\end{align*}
Combining this with the expression of $[I]$ obtained above gives
$$
[I]=\frac{1}{3}(g_1^2+\frac{3}{2}g_1g_2 +g_2^2-c_1-c_2)-q_F^{-1},
$$
which completes the proof.
\end{proof}

Thus it follows from Proposition \ref{prop cohomology class of I} and
its proof that the cycle
\begin{equation}
  \label{eq L Cubic} L  := \frac{1}{3}(g_1^2+\frac{3}{2}g_1g_2+g_2^2-c_1-c_2) -
  I
\end{equation}
represents the Beauville--Bogomolov class $\mathfrak{B}$.
Also, we see that the cohomology class of the  cycle
\begin{equation*}
  L_g  :=
\frac{1}{3}(g_1^2+g_1g_2+g_2^2-c_1-c_2) - I
\end{equation*}
is $(q_F|_{\HH^2(F,\Z)_\mathrm{prim}})^{-1}$, where
$\HH^2(F,\Z)_\mathrm{prim}=g^\perp =p_*q^*\HH^4(X,\Z)_\mathrm{prim}$.\medskip

Note that Proposition \ref{lem I dot diagonal} gives
\begin{equation}\label{eq l cubic fourfold}
l   := (\iota_\Delta)^*L = \frac{7}{6}g^2-\frac{2}{3}c - \iota_{\Delta}^*I =
\frac{25}{6}g^2 - \frac{20}{3}c = \frac{5}{6}c_2(F).
\end{equation}
The last step follows from Lemma \ref{lem second chern class cubic}.
One should compare this with equation \eqref{eq l on S2}.

\begin{thm} \label{thm cubic L conjecture} The cycle $L$ of \eqref{eq
    L Cubic} satisfies Conjecture \ref{conj main}.
\end{thm}
\begin{proof}
  Using Voisin's identity \eqref{identity of voisin} and the
  definition \eqref{eq L Cubic} of $L$, we easily get
\[
L^2 = 2\Delta_F - \frac{1}{3}(g_1^2 +g_2^2 +2c_1 +2c_2)\cdot L +
\Gamma_4,
\]
where $\Gamma_4$ is a weighted homogeneous polynomial of degree 4 in
$(g_1,g_2,c_1,c_3)$. Note that Lemma \ref{lifting cubic} implies that
$c_i\cdot L$ is a weighted homogeneous polynomial in
$(g_1,g_2,c_1,c_2)$. Hence we can modify the term in the middle by
$c_i\cdot L$ and get
\[
L^2 = 2\Delta_F -\frac{1}{3}(g_1^2 +g_2^2 - \frac{8}{5}c_1 -
\frac{8}{5}c_2)\cdot L +\Gamma_5,
\]
where $\Gamma_5$ is a degree 4 polynomial in $(g_1,g_2,c_1,c_2)$.
Equation \eqref{eq l cubic fourfold} gives
\[
\frac{1}{3}(g_1^2 +g_2^2 - \frac{8}{5}c_1 - \frac{8}{5}c_2)
=\frac{2}{25}(l_1 + l_2),
\]
where $l_i=p_i^* l$ as before. Thus we have the following equation
\[
L^2 = 2\Delta_F -\frac{2}{25}(l_1+l_2)\cdot L +\Gamma_5.
\]
Comparing this with equation \eqref{eq cohomological equation}, we see
that $\Gamma_5$ is cohomologically equivalent to $\frac{1}{23\cdot
  25}(2l_1^2 +23l_1l_2 +2l_2^2)$.  Lemma \ref{lem cohomology to chow}
yields $\Gamma_5 = \frac{1}{23\cdot 25}(2l_1^2 +23l_1l_2 +2l_2^2)$.
\end{proof}

\begin{lem}\label{lem L vanishes on triangle} The cycle $L$ acts as
  zero on $\mathfrak{o}_F$ and on triangles
  $([l],[l],\varphi([l]))$. Precisely,
\begin{enumerate}[(i)]
\item $L_*\mathfrak{o}_F=0$~;
\item $L_*(\varphi_* +2)=0$ on $\CH_0(F)$, where $\varphi$ is as in
Definition \ref{defn varphi}.
\end{enumerate}
\end{lem}
\begin{proof}
Using the definition of $L$, we have
  \[
L_*\mathfrak{o}_F = \frac{1}{3}(g^2 - c) - I_*\mathfrak{o}_F =0,
  \]
  where the last equality uses Lemma \ref{lem I [o]}. This proves
  statement \emph{(i)}. Let $t\in F$, then we have
\[
I_*(\varphi_*t+2t) = \Phi(\Psi(\varphi_*t + 2t)) = \Phi(h^3) = g^2-c
\]
where the last equality uses Lemma \ref{lem identities self
  intersection of g}. By the definition of $L$, we get
\[
L_*(\varphi_*t +2t) = (g^2 - c)- I_*(\varphi_*t + 2t) = 0.
\]
This proves statement \emph{(ii)}.
\end{proof}

We now check that the cycle $L$ satisfies conditions \eqref{assumption
  pre l} and \eqref{assumption hom} of Theorem \ref{prop L2}~:

\begin{prop} \label{prop cubic assumption l hom} We have $L_*l^2 = 0$ and
  $L_*(l\cdot L_*\sigma) = 25 L_*\sigma$ for all $\sigma \in
  \CH_0(F)$.
\end{prop}
\begin{proof}
  By equation \eqref{eq l cubic fourfold}, we see that $l^2$ is
  proportional to $c_2(F)^2$. By \cite{voisin2}, $l^2$ is a multiple
  of $\mathfrak{o}_F$. Hence we only need to prove that
  $L_*\mathfrak{o}_F=0$, which is statement \emph{(i)} of Lemma \ref{lem L
    vanishes on triangle}. To prove the second equality, we take an
  arbitrary closed point $t\in F$. Then we have
  \begin{align*}
    l \cdot L_*t &= \big(\frac{25}{6}g^2 - \frac{20}{3}c\big)\cdot
    \big(\frac{1}{3}(g^2 -c)- I_*t\big) \\
    & = a\mathfrak{o}_F -\frac{25}{6}g^2\cdot I_*t,\qquad a\in\Z\\
    & = a\mathfrak{o}_F -\frac{25}{6}(\varphi_*t -4t
    +24\mathfrak{o}_F)\\
    & = a'\mathfrak{o}_F - \frac{25}{6}(\varphi_*t +2t) + 25t,\qquad
    a' = a -100.
  \end{align*}
  Here the first equality uses the definition of $L$ and the explicit
  expression of $l$ as given in equation \eqref{eq l cubic fourfold}.
  The second equality used the fact that any degree 4 polynomial in
  $(g,c)$ is a multiple of $\mathfrak{o}_F$~; see \cite{voisin2}. The
  third equality uses Lemma \ref{lem intersection of g^2 and S_l}. We
  apply $L_*$ to the above equality and use Lemma \ref{lem L vanishes
    on triangle}, we get
\[
L_*(l\cdot L_*t) = 0 + 0+ 25 L_*t = 25L_*t.
\]
This establishes the second equality.
\end{proof}

We now check that the cycle $L$ satisfies condition \eqref{assumption
  2hom} of Theorem \ref{prop main}.  By Theorem \ref{thm cubic L
  conjecture} and Proposition \ref{prop cubic assumption l hom}, we
see that $\CH^4(F)\, (= \CH_0(F))$ has a Fourier decomposition as in
Theorem \ref{prop main}. First, we have the following.

\begin{prop} \label{prop support} We have
  \begin{equation}
    \CH^4(F)_0 \oplus \CH^4(F)_2  =   \mathrm{im} \, \{\CH_0(\Sigma_2)
      \rightarrow  \CH_0(F) \}. \nonumber
   \end{equation}
\end{prop}
\begin{proof}
    Recall that \eqref{eq l cubic fourfold} $l =
    \frac{25}{6}g^2-\frac{20}{3}c$, that $c \cdot \sigma$ is a
    multiple of $\mathfrak{o}_F$ for all $\sigma \in \CH^2(F)$ (Lemma
    \ref{lem special zero cycle}) and that $\Sigma_2 = 5(g^2-c)$
    (Lemma \ref{lem class of Sigma2}). Therefore $l\cdot \sigma$ is
    proportional to $\Sigma_2 \cdot \sigma$ for all $\sigma \in
    \CH^2(F)$.  Now that Theorem \ref{thm cubic L conjecture} and
    Proposition \ref{prop cubic assumption l hom} have been proved, we
    know from Theorems \ref{prop L2} \& \ref{prop main} that
    $\CH^4(F)_0 = \langle l^2 \rangle$ and that $\CH^4(F)_2 = l \cdot
    \CH^2(F)_2$. Hence $\CH^4(F)_0 \oplus \CH^4(F)_2 \subseteq
    \mathrm{im} \, \{\Sigma_2\cdot  : \CH^2(F) \rightarrow \CH^4(F)
    \}$.  Consider now a zero-cycle $\tau \in \CH^4(F)$ which is
    supported on $\Sigma_2$.  The key point is then Lemma \ref{lem
      varphi on Sigma2} which gives $\varphi_*\tau = -2\, \tau +
    3\deg(\tau)\, \mathfrak{o}_F \in \CH^4(F)$. We also have, by Lemma
    \ref{lem intersection of g^2 and S_l}, $g^2\cdot I_*(\tau) =
    \varphi_*\tau - 4\, \tau +24\deg(\tau) \, \mathfrak{o}_F$. It
    follows that $\tau$ is a linear combination of $l\cdot L_*(\tau)$
    and $\mathfrak{o}_F$. By Theorems \ref{prop L2} \& \ref{prop
      main}, $l\cdot L_*(\tau)$ and $\mathfrak{o}_F$ belong
    respectively to $\CH^4(F)_2$ and $\CH^4(F)_0$ and we are done.
  \end{proof}

\begin{prop} \label{prop cubic assumption 2hom}
  $(L^2)_*(l\cdot (L^2)_*\sigma) = 0$ for all $\sigma \in \CH^2(F)$.
\end{prop}
\begin{proof}
  By Theorem \ref{thm cubic L conjecture} and Proposition \ref{prop
    cubic assumption l hom}, we see that the decomposition in Theorem
  \ref{prop L2} holds on $F$. Thus we have eigenspace decompositions
  $\CH^4(F) = \Lambda_0^4 \oplus \Lambda_2^4$ and $\CH^2(F) =
  \Lambda_{25}^2 \oplus \Lambda_2^2 \oplus \Lambda_0^2$ for the action
  of $L^2$. Moreover, $\Lambda_{25}^2 = \langle l \rangle$ and Lemma
  \ref{lem assumption l} gives $(L^2)_*l^2 = 0$.  Therefore, it is
  sufficient to prove that $l \cdot \sigma = 0$ for all $\sigma \in
  \Lambda_2^2$.  According to \eqref{eq action pi6 CH2}, $\sigma \in
  \Lambda_2^2$ if and only if $l\cdot \sigma \in \Lambda_2^4$. In
  order to conclude it is enough to show that $l\cdot \sigma \in
  \Lambda_0^4$, where $\Lambda_0^4 = \CH^4(F)_0 \oplus \CH^4(F)_2$ by
  Theorem \ref{prop main}. But as in the proof of Proposition
  \ref{prop support}, we see that $l\cdot \sigma$ is proportional to
  $\Sigma_2\cdot\sigma$ so that the statement of Proposition \ref{prop
    support} yields the result.
\end{proof}

By Theorems \ref{prop L2} and \ref{prop main}, the Fourier
decomposition as in Theorem \ref{thm main splitting} is now
established for the variety of lines on a cubic fourfold.\qed

\vspace{10pt}
\section{A first multiplicative result} \label{sec mult1 cubic}

Let us start with the following definition.
\begin{defn}\label{defn the filtration}
We set
\begin{align*}\mathrm{F}^2\CH_{\Z}^4(F)&=\mathrm{F}^1\CH_{\Z}^4(F) =
  \ker\{cl :\CH_{\Z}^4(F)\rightarrow \HH^{8}(F,\Z)\},\\
  \mathrm{F}^4\CH^4_{\Z}(F)&
  =\mathrm{F}^3\CH^4_{\Z}(F)=\ker\{I_* :\CH_{\Z}^4(F)\rightarrow
  \CH_{\Z}^2(F)\},\\
  \mathcal{A} &=I_*\CH_{\Z}^4(F)\subset\CH^2_{\Z}(F).
\end{align*}
\end{defn}
\noindent After tensoring with $\Q$, this filtration coincides with
the one given in \eqref{eq filtration gal} in the introduction~; see
the proof of Proposition \ref{prop cubic CH0426}.\medskip

The main result of this section is the following theorem which shows
that the multiplicative structure on the cohomology of $F$ reflects on
its Chow groups.

\begin{thm}\label{thm surjection of intersection product} The
  intersection of $2$-cycles on $F$ enjoys the following properties.
\begin{enumerate}[(i)]
\item The natural homomorphism
$$
\CH_{\Z}^2(F)\otimes\CH^2_{\Z}(F)\rightarrow \CH^4_{\Z}(F)
$$
is surjective~;
\item The natural homomorphism
$$
\CH^2_{\Z}(F)\otimes\CH^2_{\Z}(F)_{\mathrm{hom}}\rightarrow
\CH^4_{\Z}(F)_{\mathrm{hom}}
$$
is also surjective~;
\item The image of the natural homomorphism
$$
\mathcal{A}_{\mathrm{hom}}\otimes\mathcal{A}_{\mathrm{hom}}\rightarrow
\CH^4_{\Z}(F)_{\mathrm{hom}}
$$
is equal to $\mathrm{F}^4\CH^4_{\Z}(F)$.\end{enumerate}
Here, the subscript ``\,$\mathrm{hom}$'' denotes those cycles that are
homologically trivial.

\end{thm}

As an immediate corollary of Theorem \ref{thm surjection of
  intersection product}\emph{(iii)}, we obtain the following.
\begin{prop} \label{prop cubic CH0426}
We have $$\CH^4(F)_4 = \CH^2(F)_2 \cdot \CH^2(F)_2.$$
\end{prop}
\begin{proof} In view of Theorem \ref{thm main splitting} and Theorem
  \ref{thm surjection of intersection product}, it suffices to check
  that $\mathcal{A}_{\hom} \otimes \Q = L_*\CH^4(F)$ and that
  $\ker\{I_* :\CH^4(F)\rightarrow \CH^2(F)\} =
  \ker\{L_* :\CH^4(F)_{\hom} \rightarrow \CH^2(F)\}$.  Both equalities
  follow from the fact proved in Lemma \ref{lem I [o]} that
  $I_*\mathfrak{o}_F = \frac{1}{3}(g^2-c)$ which is not zero in
  $\HH^4(F,\Q)$, and from \eqref{eq L Cubic} which gives $L_*\sigma =
  \frac{1}{3}\deg(\sigma)(g^2-c)-I_\sigma$ for all $\sigma \in
  \CH^4(F)$.
\end{proof}

A key step to proving Theorem \ref{thm surjection of intersection
  product}\emph{(iii)} is embodied by Theorem \ref{prop F3 triangle}~; it
consists in giving a different description of the filtration
$\F^\bullet$ of Definition \ref{defn the filtration}.  On the one
hand, we define
\begin{center}
  $\Phi=p_*q^* :\CH^i(X)\rightarrow\CH^{i-1}(F)$ and
  $\Psi=q_*p^* :\CH^i(F)\rightarrow \CH^{i-1}(X)$
\end{center}
to be the induced homomorphisms of Chow groups~; see Definition
\ref{defn cylinder and aj}. On the other hand, we introduce the
following crucial definition.

\begin{defn}\label{dfn triangle}
  Three lines $l_1,l_2,l_3\subset X$ form a \textit{triangle} if there
  is a linear $\Pi=\PP^2\subset\PP^5$ such that $\Pi\cdot
  X=l_1+l_2+l_3$. Each of the lines in a triangle will be called an
  \textit{edge} of the triangle. A line $l\subset X$ is called a
  \textit{triple line} if $(l,l,l)$ is a triangle on $X$. Let
  $\mathcal{R}\subset\CH_0^{\Z}(F) = \CH^4_{\Z}(F)$ be the sub-group
  generated by elements of the form $s_1+s_2+s_3$ where
  $(l_{s_1},l_{s_2},l_{s_3})$ is a triangle.
\end{defn}

\begin{thm}\label{prop F3 triangle}
  Let $\mathcal{R}_{\mathrm{hom}}\subset \mathcal{R}$ be the sub-group of
  all homologically trivial elements, then
$$
\mathrm{F}^4\CH^4(F) = \ker\{\Psi  : \CH^4(F) \rightarrow \CH^3(X)\}
=\mathcal{R}_{\mathrm{hom}}.
$$
\end{thm}
\begin{proof}
  To establish the first equality, we note that the composition
  $\Phi\circ\Psi$ is equal to $I_*$. Hence we only need to see that
  $\Phi :\CH^3(X)_{\mathrm{hom}}\rightarrow\CH^2(F)_{\mathrm{hom}}$ is
  injective. This is an immediate consequence of \cite[Theorem
  4.7]{surface} where it is shown that the composition of $\Phi$ and
  the restriction
  $\CH^2(F)_{\mathrm{hom}}\rightarrow\CH^2(S_l)_{\mathrm{hom}}$ is
  injective for $l$ a general line on $X$.

  As for the second equality, the inclusion
  $\mathcal{R}_{\mathrm{hom}} \subseteq \ker\{\Psi\}$ is obvious and
  the inclusion $\ker\{\Psi\} \subseteq \mathcal{R}_{\mathrm{hom}}$ is
  Lemma \ref{lem rationally equivalent cycles} below.
\end{proof}

\begin{lem}\label{lem rationally equivalent cycles} Let $l_i$ and $l'_i$,
  $i=1,\ldots,n$, be two collections of lines on $X$ such that $\sum
  l_i=\sum l'_i$ in $\CH^3(X)$. Then we have $$
  \sum_{i=1}^{n}[l_i]-\sum_{i=1}^{n}[l'_i] \in
  \mathcal{R}_\mathrm{hom}. $$
\end{lem}
\begin{proof} We recall the concept, which was defined in
  \cite{relation}, of a secant line of a pair of (disjoint) curves
  $C_1$ and $C_2$ on $X$. Namely $l$ is a secant line of $(C_1,C_2)$
  if it meets both curves. This concept is naturally generalized to
  the concept of secant lines of two 1-dimensional cycles on $X$. A
  pair of curves are well-positioned if they have finitely many
  (counted with multiplicities) secant lines. To prove the lemma,
  we first prove the following.\medskip

  \noindent \textbf{Claim}  : Let $(l_1,l_2)$ be a pair of
  well-positioned lines on $X$ and $E_i$, $i=1,\ldots,5$, the secant
  lines of $(l_1,l_2)$.  Then we have
$$ 2[l_1]+2[l_2]+\sum_{i=1}^5[E_i]\in\mathcal{R}.$$
\begin{proof}[Proof of Claim] we may assume that the pair $(l_1,l_2)$
  is general.  This is because the special case follows from the
  generic case by a limit argument. Let $\Pi\cong\PP^3$ be the linear
  span of $l_1$ and $l_2$. Since $(l_1,l_2)$ is general, the
  intersection $\Sigma=\Pi\cap X\subset\Pi\cong\PP^3$ is a smooth
  cubic surface.  Hence
  $\Sigma\cong\mathrm{Bl}_{\{P_1,\ldots,P_6\}}(\PP^2)$ is the blow-up
  of $\PP^2$ at 6 points. Let $R_i\subset\Sigma$, $1\leq i\leq 6$, be
  the the 6 exceptional curves. Let $L_{ij}\subset\Sigma$ be the
  strict transform of the line on $\PP^2$ connecting $P_i$ and $P_j$,
  where $1\leq i<j\leq 6$. Let $C_i\subset\Sigma$, $1\leq i\leq 6$, be
  the strict transform of the conic on $\PP^2$ passing through all
  $P_j$ with $j\neq i$. The set $\{R_i,L_{ij},C_i\}$ gives all the 27
  lines on $\Sigma$. Without loss of generality, we may assume that
  $R_1=l_1$ and $R_2=l_2$. Then the set of all secant lines
  $\{E_i\}_{i=1}^6$ is explicitly given by
  $\{L_{12},C_3,C_4,C_5,C_6\}$.  The triangles on $\Sigma$ are always
  of the form $(R_i,L_{ij},C_j)$ (here we allow $i>j$ and set
  $L_{ij}=L_{ji}$) or $(L_{i_1 j_1}, L_{i_2 j_2},L_{i_3 j_3})$ where
  $\{i_1,i_2,i_3,j_1,j_2,j_3\}=\{1,\ldots,6\}$. Then we easily check
  that
\begin{align*}
2l_1+2l_2+\sum E_i = &2R_1+2R_2+ L_{12}+C_3+C_4+C_5+C_6\\
 =&(R_1+L_{13}+C_3) + (R_1+L_{14}+C_4) +(R_2+L_{25}+C_5)\\
  &+(R_2+L_{26}+C_6)+(L_{12}+L_{46}+L_{35})\\
  &-(L_{13}+L_{25}+L_{46})-(L_{14}+L_{26}+L_{35}).
\end{align*}
Hence $2[l_1]+2[l_2]+\sum[E_i]\in\mathcal{R}$.
\end{proof}
Back to the proof of the lemma. We pick a general line $l$ such that
$(l,l_i)$ and $(l,l'_i)$ are all well-positioned. Let $E_{i,j}$,
$j=1,\ldots,5$, be the secant lines of $l_i$ and $l$~; similarly let
$E'_{i,j}$, $j=1,\ldots,5$, be the secant lines of $l'_i$ and $l$.
Then by definition, we have $$
\sum_{i=1}^{n}\sum_{j=1}^{5}[E_{i,j}]=\Phi(\gamma)\cdot S_l, \qquad
\sum_{i=1}^{n}\sum_{j=1}^{5}[E'_{i,j}]=\Phi(\gamma')\cdot
S_l,\quad\text{ in }\CH_0(F), $$ where $\gamma=\sum l_i\in\CH_1(X)$
and $\gamma'=\sum l'_i\in\CH_1(X)$. By assumption $\gamma=\gamma'$ and
hence $$ \sum_{i=1}^{n}\sum_{j=1}^{5}[E_{i,j}] =
\sum_{i=1}^{n}\sum_{j=1}^{5}[E'_{i,j}]. $$ Thus we get
\begin{align*} 2(\sum_{i=1}^{n}[l_i]-\sum_{i=1}^{5} [l'_i]) &
=(2\sum_{i=1}^n [l_i]+2n[l]+ \sum_{i=1}^{n}\sum_{j=1}^{5}[E_{i,j}]) -
(2\sum_{i=1}^n [l'_i]+2n[l]+ \sum_{i=1}^{n}\sum_{j=1}^{5}[E'_{i,j}])\\ & =
\sum_{i=1}^{n}(2[l] + 2[l_i] +\sum_{j=1}^{5}[E_{i,j}]) - \sum_{i=1}^{n}(2[l] +
2[l'_i] +\sum_{j=1}^{5}[E'_{i,j}]) \ \  \in\mathcal{R}_{\mathrm{hom}}.
\end{align*} Here the last step uses the claim. If we can prove that
$\mathcal{R}_{\mathrm{hom}}$ is divisible (and hence uniquely
divisible), then we get
$\sum[l_i]-\sum[l'_i]\in\mathcal{R}_{\mathrm{hom}}$. Let
$\widetilde{R}$ be the (desingularized and compactified) moduli space
of triangles on $X$, then we have a surjection
$\CH_0(\widetilde{R})_{\mathrm{hom}} \twoheadrightarrow
\mathcal{R}_{\mathrm{hom}}$.  It is a standard fact that
$\CH_0(\widetilde{R})_{\mathrm{hom}}$ is divisible.  Hence so is
$\mathcal{R}_{\mathrm{hom}}$. This finishes the proof.
\end{proof}

Proposition \ref{prop intersection identities} below will be used to
prove Theorem \ref{thm surjection of intersection product}.

\begin{prop}\label{prop intersection identities}
  Let $(l_1,l_2,l_3)$ be a triangle, then the following are true.
  \begin{enumerate}[(i)]
  \item The identity $S_{l_1}\cdot S_{l_2}=
  6\mathfrak{o}_F+[l_3]-[l_1]-[l_2]$ holds true in $\CH_0(F)$~;
\item If $(l'_1,l'_2,l'_3)$ is another triangle, then
$$
6(\sum_{i=1}^{i=3}[l_i]-\sum_{i=1}^{3}[l'_i])=\sum_{1\leq i<j\leq
  3}(S_{l_i}-S_{l_j})^2 -\sum_{1\leq i<j\leq 3}(S_{l'_i}-S_{l'_j})^2
$$
holds true in $\CH_0(F)$.
\end{enumerate}
\end{prop}

\begin{proof}[Proof of Proposition \ref{prop intersection identities}]
  We prove \emph{(i)} for the case where all the edges of the triangle are
  distinct. Then the general case follows by a simple limit argument.
  Take a 1-dimensional family of lines $\{l_t :t\in T\}$ such that
  $l_{t_1}=l_1$. We may assume that $(l_{t},l_2)$ is well-positioned
  when $t\neq t_1$. Let $\{E_{t,i} :i=1,\ldots,5\}$ be the secant lines
  of the pair $(l_t,l_2)$ for $t\neq t_1$. The tangent space
  $\mathscr{T}_{t_1}(T)$ determines a section of
  $\HH^0(l_1,\mathscr{N}_{l_1/X})$ which in turn gives a normal
  direction $v$ in $\mathscr{N}_{l_1/X,x}$, where $x=l_1\cap l_2$.
  When $t\to t_1$, the secant lines $E_{t,i}$ specializes to
  $\{L_1,\ldots,L_4, l_3\}$ where $L_1,\ldots,L_4$ together with
  $\{l_1,l_2\}$ form the six lines through $x$ that are contained in
  the linear $\PP^3$ spanned by $(l_1,l_2,v)$. This means
  $[L_1]+\cdots+[L_4]+[l_1]+[l_2]=6\mathfrak{o}_F$~; see Lemma \ref{lem
    specialization of secant lines}. By construction, we have
$$
S_{l_t}\cdot S_{l_2}=\sum_{i=1}^5[E_{t,i}].
$$
Let $t\to t_1$ and take the limit, we have
$$
S_{l_1}\cdot S_{l_2}=\sum_{i=1}^4[L_i]+[l_3].
$$
We combine this with
$[L_1]+\cdots+[L_4]+[l_1]+[l_2]=6\mathfrak{o}_F$ and deduce \emph{(i)}.

Let $(l_1,l_2,l_3)$ be a triangle, then we have
\begin{align*}
\sum_{1\leq i<j\leq 3} (S_{l_i}-S_{l_j})^2 &=2\big( \sum_{1\leq i\leq
3}S_{l_i}\big)^2-6\sum_{1\leq i<j\leq 3}S_{l_i}\cdot S_{l_j}\\
&=2(\Phi(h^3))^2+6\sum_{1\leq i\leq 3}[l_i]-108\mathfrak{o}_F.
\end{align*}
Then \emph{(ii)} follows easily from this computation.
\end{proof}

\begin{proof}[Proof of Theorem \ref{thm surjection of intersection
    product}] To prove \emph{(ii)}, it is enough to show that $[l_1]-[l_2]$ is
  in the image of the map. Let $l'$ be a line meeting both $l_1$ and
  $l_2$, then $[l_1]-[l_2]=([l_1]-[l'])+([l']-[l_2])$. Hence we may
  assume that $l_1$ meets $l_2$. We may further assume that
  $(l_1,l_2)$ is general. Since $\CH^4(F)_{\mathrm{hom}} =
  \CH_0(F)_{\mathrm{hom}}$ is uniquely divisible, we only need to show
  that $2([l_1]-[l_2])$ is in the image. Now let $l_3$ be the residue
  line of $l_1\cup l_2$ so that $(l_1,l_2,l_3)$ forms a triangle. Then
$$
2([l_1]-[l_2])=(S_{l_2}-S_{l_1})\cdot S_{l_3},
$$
is in the image of
$\CH^2(F)\otimes\CH^2(F)_{\mathrm{hom}}\rightarrow\CH^4(F)_{\mathrm{hom}}$.

To prove \emph{(i)}, it suffices to show that the image of
$\CH^2_{\Z}(F)\otimes\CH^2_{\Z}(F)\rightarrow\CH^4_{\Z}(F)$ contains
a cycle of degree 1. By Lemma \ref{lem special zero cycle} and the
fact that a general pair of lines has $5$ secant lines (see
\cite{relation}), we have
$$
\deg(g^2\cdot g^2)=108,\quad \deg(S_{l_1}\cdot S_{l_2})=5.
$$
Hence the image of the intersection of 2-cycles hits a 0-cycle of
degree 1.

\emph{(iii)} As an easy consequence of Proposition \ref{prop
  intersection identities}\emph{(ii)} we have $\mathcal{R}_{\hom} =
\mathcal{A}_{\hom} \cdot \mathcal{A}_{\hom}$. But then, Theorem
\ref{prop F3 triangle} gives $\F^4\CH^4(F) = \mathcal{A}_{\hom} \cdot
\mathcal{A}_{\hom}$
\end{proof}

\begin{rmk} An alternate proof of statements \emph{(i)} and
  \emph{(ii)} of Theorem \ref{thm surjection of intersection product}
  can be obtained as a combination of Voisin's identity (\ref{identity
    of voisin}) and of the basic Lemma \ref{lem smooth general fiber}.
  However, (\ref{identity of voisin}) does not seem to imply statement
  \emph{(iii)}.  In addition, our proofs also work for the case of
  cubic threefolds as follows. If $X$ is a cubic threefold, then its
  variety of lines, $S$, is a smooth surface. Then
  $\mathrm{F}^2\CH^2(S)$, the Albanese kernel, is the same as
  $\ker\{\Psi :\CH^2(S)\rightarrow\CH^3(X)\}$.  As in Theorem
  \ref{prop F3 triangle}, $\mathrm{F}^2\CH^2(S)$ is identified with
  $\mathcal{R}_{\mathrm{hom}}$. Statements \emph{(i)} and (ii) of
  Proposition \ref{prop intersection identities} are true in the
  following sense.  In this case $S_l$, the space of all lines meeting
  a given line $l$, is a curve. The constant class $6\mathfrak{o}_F$
  in statement \emph{(i)} should be replaced by the class of the sum
  of all lines passing through a general point of $X$. Then the
  statement (iii) of Theorem \ref{thm surjection of intersection
    product} reads as $\Pic^0(S)\otimes\Pic^0(S)\rightarrow
  \mathrm{F}^2\CH^2(S)$ being surjective. This result concerning the
  Fano scheme of lines on a smooth cubic threefold already appears in
  \cite{bloch}.
\end{rmk}

\vspace{10pt}
\section{The rational self-map $\varphi :F\dashrightarrow F$ and the
  Fourier decomposition}
\label{sec varphi}

In this section we determine completely the class in $\CH^4(F\times
F)$ of the closure of the graph $\Gamma_{\varphi}$ of the rational map
$\varphi$ of Definition \ref{defn varphi}~; see Proposition \ref{prop
  key identity of varphi}. We obtain thus in Proposition \ref{prop
  cubic hom phi} a complete description of the cohomology class of
$\Gamma_{\varphi}$, thereby complementing the main result of Amerik
\cite{amerik}. After some work, we also obtain in Theorem \ref{thm
  eigenspace decomposition} an eigenspace decomposition for the action
of $\varphi$ on the Chow groups of $F$. It turns out that the action
of $\varphi$ is compatible with the Fourier decomposition on the Chow
groups of $F$~; see \S \ref{sec varphi fourier compatible}. This
interplay between $\varphi$ and the Fourier decomposition will be
crucial to showing in Section \ref{sec mult} that the Fourier
decomposition is compatible with intersection product.  \medskip

Let $\omega$ be a global $2$ form on
$F$. The action of $\varphi^*$ on $\omega$ is known~:

\begin{prop}[\cite{av}]\label{prop varphi action on cohomology}
  The rational map $\varphi :F\dashrightarrow F$ has degree 16 and we
  have
$$
\varphi^*\omega=-2\omega,\quad \varphi^*\omega^2=4\omega^2.
$$
\end{prop}

By the very definition of $\varphi$, the triple $(l,l,\varphi(l))$ is
a triangle, in the sense of Definition \ref{dfn triangle}, whenever
$l$ is a line of first type. Thus $I_*(\varphi_*\sigma + 2\sigma) = 0$
for all $\sigma \in \CH_0(F)_{\hom}$~; see Proposition \ref{prop F3
  triangle}. Other links between the rational map $\varphi$ and the
incidence correspondence $I$ have already been mentioned~: Lemma
\ref{lem intersection of g^2 and S_l} shows that $g^2\cdot I_*[l]
=\varphi_*([l])-4[l]+24\mathfrak{o}_F \in \CH_0(F)$, and Proposition
\ref{prop intersection identities} implies, for instance, that
$I_*[l]\cdot I_*\varphi_*[l] = 6\mathfrak{o}_F- \varphi_*[l]$ and
$(I_*[l])^2 = 6\mathfrak{o}_F+\varphi_*[l]-2[l]$.  Here we wish to
relate $I$ and $\varphi$ as correspondences in $F \times F$ and not
only the actions of $I$ and $\varphi$ on the Chow groups of $F$. As
such, the main result of this section is Proposition \ref{prop key
  identity of varphi}. As a consequence, we determine the action of
$\varphi$ on $\CH^*(F)$ and study its compatibility with the Fourier
transform.

\subsection{The correspondence $\Gamma_\varphi$}
Let $\varphi :F\dasharrow F$ be the rational map introduced in
Definition \ref{defn varphi} and let $\Gamma_\varphi\subset F\times F$ be
the closure of the graph of $\varphi$. We denote $\Gamma_{\varphi,0}\subset
F\times F$ the set of points $([l_1],[l_2])\in F\times F$ such that
there is a linear plane $\PP^2\subset \PP^5$, not contained in $X$, with
$\PP^2\cdot X=2l_1+l_2$. Then $\Gamma_{\varphi,0}$ is irreducible and
$\Gamma_\varphi$ is the closure of $\Gamma_{\varphi,0}$. Furthermore,
if $X$ does not contain any plane, then $\Gamma_\varphi =
\Gamma_{\varphi,0}$.  Our next goal is to study the correspondence
$\Gamma_\varphi$ in more details.

Let us now introduce some other natural cycles on $F\times F$~; see
Appendix A for more details. Let $\Gamma_h\subset F\times F$ be the
5-dimensional cycle of all points $([l_1],[l_2])$ such that $x\in
l_1\cap l_2$ for some point $x$ on a given hyperplane $H\subset X$.
Let $\Gamma_{h^2}\subset F\times F$ be the 4-dimensional cycle of all
points $([l_1],[l_2])$ such that $x\in l_1\cap l_2$ for some $x$ on an
intersection $H_1\cap H_2\subset X$ of two general hyperplane
sections. Let $\Pi_j\subset X$, $j=1,\ldots,r$, be all the planes
contained in $X$ and $\Pi^\ast_j \subset F$ the corresponding dual
planes. Then we define $I_1=\sum_{j=1}^{r} \Pi_j^\ast \times
\Pi_j^\ast \in \CH^4(F\times F)$. Note that a general cubic fourfold
does not contain any plane and hence $I_1=0$. Now we can state the
main result of this section.

\begin{prop}\label{prop key identity of varphi}
  The following equation holds true
  in $\CH^4(F\times F)$,
\begin{equation} \label{eq simplified varphi}
  \Gamma_\varphi+I_1  = 4\Delta_F+(2g_1^2+3g_1g_2+
  g_2^2)\cdot I -(5g_1+4g_2) \cdot \Gamma_h+3\Gamma_{h^2}.
\end{equation}
Furthermore, $-(5g_1+4g_2)\cdot \Gamma_h+3\Gamma_{h^2}$ can be
expressed as a weighted homogeneous polynomial
$\Gamma'_2(g_1,g_2,c_1,c_2)$ of degree $4$~; see Proposition \ref{prop
  class of Gamma_h}.
\end{prop}

\begin{proof}
Let $I_0 = I \backslash \Delta_F$. Then we have a natural morphism
\[
q_0  : I_0 \rightarrow X,\quad ([l_1],[l_2])\mapsto x=l_1\cap l_2.
\]
Let $\pi_i=p_i|_{I_0}  : I_0\rightarrow F$ be the two projections. We
use $\mathscr{E}_2\subset V$ to denote the natural rank two sub-bundle
of $V$ on $F$ and $\mathscr{E}_1=\calO_X(-1)\subset V$ to denote the
restriction of the tautological line bundle. We first note that on
$I_0$, we have natural inclusions $q_0^* \mathscr{E}_1\hookrightarrow
\pi_i^*\mathscr{E}_2$, $i=1,2$, which correspond to the geometric fact
that $x\in l_i$. In this way we get the following short exact
sequences
\begin{equation}\label{eq Q_i}
  \xymatrix{
    0\ar[r] & q_0^*\mathscr{E}_1 \ar[r] &\pi_i^*\mathscr{E}_2\ar[r]
    &\mathcal{Q}_i\ar[r] &0,
  }
\end{equation}
for $i=1,2$, where $\mathcal{Q}_i$ is an invertible sheaf on $I_0$.
Let $\mathcal{V}_3\subset V$ be the rank 3 sub-bundle spanned by
$\pi_1^*\mathscr{E}_2$ and $\pi_2^*\mathscr{E}_2$, then we have the
following short exact sequence
\begin{equation*}
  \xymatrix{
    0\ar[r] &q_0^*\mathscr{E}_1\ar[rr]^{(+,-)\quad} &&
    \pi_1^*\mathscr{E}_2\oplus \pi_2^*\mathscr{E}_2\ar[rr]^{\quad +}
    &&\mathcal{V}_3\ar[r] &0.
  }
\end{equation*}
At a point $([l_1],[l_2])\in I_0$, the fiber of $\mathcal{V}_3$ is the
3-dimensional sub-space of $V$ that corresponds to the linear $\PP^2$
spanned by $l_1$ and $l_2$. We have a natural inclusion
$\pi_i^*\mathscr{E}_2\subset \mathcal{V}_3$ whose quotient is
isomorphic to $\mathcal{Q}_j$ for $\{i,j\}=\{1,2\}$. This gives a
natural surjection $\mathcal{V}_3\rightarrow \mathcal{Q}_1\oplus
\mathcal{Q}_2$ which fits into the following short exact sequence
\begin{equation}\label{eq V3 sequence}
  \xymatrix{
    0\ar[r] &q_0^*\mathscr{E}_1\ar[r] &\mathcal{V}_3\ar[r]
    &\mathcal{Q}_1\oplus \mathcal{Q}_2\ar[r] &0.
  }
\end{equation}
We write $\mathscr{F}_i=\pi_i^*\mathscr{E}_2$, $i=1,2$, and note that
we have natural inclusions $\mathscr{F}_i\subset\mathcal{V}_3$. By
taking the third symmetric power, we get
\begin{equation}\label{eq sym3 V3}
  \xymatrix{
    0\ar[r] &\mathscr{F}_0\ar[r] &\Sym^3\mathcal{V}_3\ar[r]
    &\mathscr{G}\ar[r] &0,
  }
\end{equation}
where $\mathscr{F}_0=\frac{\Sym^3\mathscr{F}_1 \oplus
  \Sym^3\mathscr{F}_2}{q_0^*\mathscr{E}_1^{\otimes 3}}$ and
$\mathscr{G}$ is locally free of rank 3 on $I_0$. Let $\phi=0$, for
some $\phi\in\Sym^3(V)^\vee$, be the defining equation of
$X\subset\PP(V)$. Then by restricting to the sub-bundle
$\Sym^3\mathcal{V}_3$ of $\Sym^3V$, we get a homomorphism
$\phi_1 :\Sym^3\mathcal{V}_3\rightarrow \calO_{I_0}$. Note that at a
closed point $([l_1],[l_2])\in I_0$, $\mathscr{F}_i\subset V$
corresponds to the line $l_i$, $i=1,2$. The fact that $l_i$ is
contained in $X$ implies that $\phi_1$ vanishes on
$\Sym^3\mathscr{F}_i$. Hence by the short exact sequence \eqref{eq
  sym3 V3}, we see that $\phi_1$ induces a homomorphism
\[
 \phi_2 :\mathscr{G}\longrightarrow \calO_{I_0}.
\]
Consider the following diagram
\begin{equation*}
  \xymatrix{
    0\ar[r] &\mathscr{F}_0 \ar[rr]\ar[d] &&
    \Sym^3\mathcal{V}_3 \ar[rr]\ar[d] &&\mathscr{G}\ar[r]\ar[d]^\rho &0\\
    0\ar[r] &\Sym^3\mathcal{Q}_1\oplus \Sym^3\mathcal{Q}_2 \ar[rr]
    &&\Sym^3(\mathcal{Q}_1\oplus \mathcal{Q}_2)\ar[rr]
    &&\mathcal{Q}_1^2\otimes\mathcal{Q}_2 \oplus
    \mathcal{Q}_1\otimes\mathcal{Q}_2^2 \ar[r] &0
  }
\end{equation*}
where all the vertical arrows are surjections. The homomorphism $\rho$
fits into the following short exact sequence
\begin{equation}\label{eq G sequence}
  \xymatrix{
    0\ar[r] & q_0^*\mathscr{E}_1\otimes
    \mathcal{Q}_1\otimes\mathcal{Q}_2 \ar[rr] &&\mathscr{G}\ar[rr]
    &&\mathcal{Q}_1^2\otimes\mathcal{Q}_2 \oplus \mathcal{Q}_1
    \otimes\mathcal{Q}_2^2 \ar[r] &0.
  }
\end{equation}
Hence $\phi_2$ induces an homomorphism
$$
\phi_3 :q_0^*\mathscr{E}_1\otimes\mathcal{Q}_1\otimes\mathcal{Q}_2
\longrightarrow \calO_{I_0}.
$$
Let $Z\subset I_0$ be the locus defined by $\phi_3=0$. Hence on $Z$
the homomorphism $\phi_3$ factors through the third term in the
sequence \eqref{eq G sequence} and gives
$$
\phi_4 : \mathcal{Q}_1^2\otimes\mathcal{Q}_2 \oplus
\mathcal{Q}_1\otimes\mathcal{Q}_2^2 \longrightarrow \calO_{I_0}.
$$
This homomorphism further splits as $\phi_4=\phi_{4,1} +\phi_{4,2}$,
where
\begin{align*}
  \phi_{4,1} : & \mathcal{Q}_1^2\otimes\mathcal{Q}_2 \longrightarrow
  \calO_{I_0},\\
  \phi_{4,2} : & \mathcal{Q}_1\otimes\mathcal{Q}_2^2 \longrightarrow
  \calO_{I_0}.
\end{align*}
Let $Z'\subset Z$ be the locus defined by $\phi_{4,1}=0$. The
cycle class of $Z'$ on $I_0$ can then be computed as
\begin{align*}
  Z' & = c_1(q_0^*\mathscr{E}_1^{-1}\otimes \mathcal{Q}_1^{-1} \otimes
  \mathcal{Q}_2^{-1})\cdot c_1(\mathcal{Q}_1^{-2}\otimes \mathcal{Q}_2^{-1})\\
  & = (q_0^*h +g_1|_{I_0} -q_0^*h +g_2|_{I_0} -q_0^*h)(2g_1|_{I_0} -
  2q_0^*h +g_2|_{I_0} -q_0^*h)\\
  & = (2g_1^2 +3g_1g_2 +g_2^2)|_{I_0} -(5g_1+4g_2)\cdot q_0^*h
  +3q_0^*h^2.
\end{align*}
Here the second equality uses the fact that
$-c_1(\mathcal{Q}_i)=g_i|_{I_0} -q_0^*h$, $i=1,2$, which is a simple
consequence of the sequence \eqref{eq Q_i}. Let $i_0 :I_0\rightarrow
F\times F\backslash\Delta_F$ be the inclusion, then we have
\begin{equation}\label{eq push forward of Z'}
  (i_0)_*Z' = (2g_1^2 +3g_1g_2 +g_2^2)\cdot I -(5g_1
  +4g_2)\cdot\Gamma_h +
  3 \Gamma_{h^2}, \quad\text{on }F\times F\backslash\Delta_F.
\end{equation}

Our next step is to relate $Z'$ to $\Gamma_\varphi$. Let $e_0$ (resp.
$e_1$ and $e_2$) be a local generator of $q_0^*\mathscr{E}_1$ (resp.
$\mathcal{Q}_1$ and $\mathcal{Q}_2$) and $T_0$ (resp. $T_1$ and $T_2$)
be the dual generator of $q_0^*\mathscr{E}_1^{-1}$ (resp.
$\mathcal{Q}_1^{-1}$ and $\mathcal{Q}_2^{-1}$). Then $(T_0,T_1,T_2)$
can be viewed as homogeneous coordinates of the plane $\Pi_t$ spanned
by $l_1$ and $l_2$, where $t=([l_1],[l_2])\in I_0$ is a point at which
the $T_i$'s are defined. Furthermore, $l_1$ is defined by $T_2=0$ and
$l_2$ is defined by $T_1=0$. Let $\phi_t$ be the restriction of $\phi$
to $\Pi_t$. The fact that $\phi=0$ along $l_1$ and $l_2$ implies that
$\phi_t=aT_0T_1T_2 +b T_1^2T_2 +cT_1T_2^2$. The condition $\phi_3=0$
is equivalent to $a=0$ and the condition $\phi_{4,1}=0$ is equivalent
to $b=0$. Thence $t\in Z'$ if and only if $\phi_t=cT_1T_2^2$ for some
constant $c$. Equivalently, we have
$$
Z'=\{t\in I_0 : \Pi_t\cdot X =2l_1 + l_2 \text{ or }\Pi_t\subset X\}.
$$
It follows that $(i_0)_*Z' = \Gamma_\varphi + I_1$ on $F\times
F\backslash\Delta_F$. Using the localization sequence for Chow groups
and equation \eqref{eq push forward of Z'}, we get
$$
\Gamma_\varphi + I_1 = \alpha\Delta_F + (2g_1^2 +3g_1g_2 +g_2^2)\cdot
I -(5g_1 +4g_2)\cdot\Gamma_h +3 \Gamma_{h^2},\quad\text{in
}\CH^4(F\times F),
$$
for some integer $\alpha$. By Proposition \ref{prop class of Gamma_h},
the terms involving $\Gamma_h$ and $\Gamma_{h^2}$ are polynomials in
$g_i$ and $c_i$. To determine the value of $\alpha$, we note that all
the terms in the above formula act trivially on $\HH^{4,0}(F)$ except
for $\Gamma_\varphi$ and $\alpha\Delta_F$. We already know that
$\varphi^*=4$ on $\HH^{4,0}(F)$~; see Proposition \ref{prop varphi
  action on cohomology}. Hence we get $\alpha=4$.
\end{proof}

\subsection{The action of $\varphi$ on homologically trivial cycles}
Let $\Gamma_\varphi\subset F\times F$ be the closure of the graph of
$\varphi$ as before. Then, for any element $\sigma\in\CH^i(F)$, we
define
\begin{center}
  $ \varphi^*\sigma=p_{1*}(\Gamma_\varphi\cdot p_2^*\sigma)$ \quad and
  \quad $\varphi_*\sigma=p_{2*}(\Gamma_\varphi\cdot p_1^*\sigma) ;$
\end{center}
See Appendix \ref{app rat chow} for a general discussion on the action
of rational maps on Chow groups.  The main goal of this section is to
get explicit descriptions of $\varphi^*$ and $\varphi_*$.

\begin{prop}\label{prop varphi action on homologically trivial
cycles} The following are true.
\begin{enumerate}[(i)]
\item  The action of $\varphi^*$ on homologically trivial cycles can be
described by
\begin{align*}
  &\varphi^*\sigma=4\sigma+2g^2\cdot I_*\sigma,
  \qquad\sigma\in\CH^4(F)_{{\mathrm{hom}}}\\
  &\varphi^*\sigma=4\sigma +3g\cdot
  I_*(g\cdot\sigma),\qquad\sigma\in\CH^3(F)_{\mathrm{hom}}\\
  &\varphi^*\sigma=4\sigma+ I_*(g^2\cdot\sigma),\qquad\sigma\in
  \CH^2(F)_{\mathrm{hom}} ;
\end{align*}

\item The action of $\varphi_*$ on homologically trivial cycles can be
described by
\begin{align*}
&\varphi_*\sigma =4\sigma + g^2\cdot
I_*\sigma,\qquad \sigma\in\CH^4(F)_\mathrm{hom}\\
&\varphi_*\sigma=
4\sigma+3g\cdot
 I_*(g\cdot\sigma),\qquad\sigma\in\CH^3(F)_{\mathrm{hom}}\\
&\varphi_*\sigma=4\sigma+2
I_*(g^2\cdot\sigma),\qquad\sigma\in
 \CH^2(F)_{\mathrm{hom}}.
\end{align*}
\end{enumerate}
\end{prop}

\begin{proof}
  In view of Proposition \ref{prop key identity of varphi}, \emph{(i)} and
  \emph{(ii)} are direct consequences of the following three facts. First,
  since $I_1$ is supported on the union of the $\Pi^\ast_j
  \times \Pi^\ast_j$, the action of $I_1$ on $\CH^i(F)$ factors
  through both $\bigoplus_j \CH_i(\Pi^\ast_j)$ and $\bigoplus_j
  \CH^i(\Pi^\ast_j)$. Therefore, $I_1$ acts as zero on $\CH^i(F)$ for
  $i \neq 2$ and since $\CH^*(\Pi^\ast_j)_{\hom}=0$ it also acts as zero
  on $\CH^*(F)_{\hom}$. Secondly, the action of $p_1^*\alpha\cdot
  p_2^*\beta \cdot \Gamma$ on $\sigma \in \CH^*(F)$, where $\alpha,
  \beta \in \CH^*(F)$ and $\Gamma \in \CH^*(F \times F)$, is given by
  $$(p_1^*\alpha\cdot p_2^*\beta \cdot \Gamma)^*\sigma = \alpha\cdot
  \Gamma^*(\beta\cdot \sigma).$$ Thirdly, if $\Gamma_2' \in \CH^*(F
  \times F)$ is a polynomial in the variables $c_1, c_2, g_1, g_2$,
  then for $\sigma \in \CH^*(F)_{\hom}$ $(\Gamma_2')^*\sigma \in
  \CH^*(F)$ is a polynomial in the variables $c,g$ which is
  homologically trivial. The main result of \cite{voisin2} recalled in
  Theorem \ref{thm voisin} implies that $(\Gamma_2')^*\sigma = 0$ in
  $\CH^*(F)$.
 \end{proof}

\subsection{The action of $\varphi$ on $g$, $c$ and their products}
Amerik \cite{amerik} computed, for $X$ generic, the action of
$\varphi^*$ on the powers of $g$ and on $c$. We use Proposition
\ref{prop key identity of varphi} to complement Amerik's result.
\begin{prop} \label{prop cubic g c} We have $\varphi^*g=7g$,
  $\varphi_*g=28g$, $\varphi^*g^2= \varphi_*g^2= 4g^2 + 45c-I_1^*g^2$,
  $\varphi^*c= \varphi_*c = 31c-I_1^*c$, $\varphi^*g^3 = 28g^3$ and
  $\varphi_*g^3=7g^3$.
\end{prop}
\begin{proof} First note, as in the proof of Proposition \ref{prop key
    identity of varphi}, that $I_1$ acts as zero on $\CH^i(F)$ for $i
  \neq 2$ but also that the correspondence $I_1$ is self-transpose.
  To prove $\varphi_*g=28g$, we apply the identity in Proposition
  \ref{prop key identity of varphi} and Lemmas \ref{lem special zero
    cycle} and \ref{lem identities self intersection of g}. We have
\begin{align*}
\varphi_*g &=4g +2I_*g^3 +3g\cdot I_*g^2
-4g\cdot(\Gamma_h)_*g -5(\Gamma_h)_*g^2+3(\Gamma_{h^2})_*g\\
&=4g+ 72g  +63g  -4\cdot 6g-5\cdot 21g+3\cdot 6g
\end{align*}
One simplifies the above expression and gets $\varphi_*g=28g$. Let us
now compute $\varphi_* c$. By Proposition \ref{prop key identity of
  varphi}, we have $$\varphi_*c + I_1^*c = 4c + 2I_*(g^2\cdot c) +
3g\cdot I_*(g\cdot c) + g^2\cdot I_*c - 5(\Gamma_h)_*(g\cdot c) -
4g\cdot (\Gamma_h)_*c + 3(\Gamma_{h^2})_*c.$$ Lemma \ref{lem special
  zero cycle} and Proposition \ref{prop class of Gamma_h} give $
(\Gamma_h)_*c = 6g$ and $ (\Gamma_{h^2})_*c = 6(g^2-c)$. On the other
hand, \cite[Lemma 3.5]{voisin2} shows that $g\cdot c$ and $g^3$ are
proportional. Since by Lemma \ref{lem special zero cycle}
$\deg(g^2\cdot c) = 45$ and $\deg(g^4) = 108$, we get $12g\cdot c =
5g^3$. In order to finish off the computation of $\varphi_*c$ using
Lemma \ref{lem identities self intersection of g}, it remains to
compute $I_*c$. Since $6l = 25g^2 -40c$ and $L_*l = 0$, we get $8L_*c
= 5L_*g^2$. Now by \eqref{eq L Cubic}, $L =
\frac{1}{3}(g_1^2+\frac{3}{2}g_1g_2+g_2^2-c_1-c_2) - I$. A
straightforward calculation yields $I_*c = 6[F]$. Putting all together
gives the required $\varphi_*c = 31c-I_1^*c$. The other identities can
be proved by using the same basic arguments and are therefore omitted.
\end{proof}

\subsection{Eigenspace decomposition of the cohomology groups}
Let us denote $\tau  : \Gamma_{\varphi} \rightarrow F$ and
$\widetilde{\varphi}  : \Gamma_{\varphi} \rightarrow F$ the first and
second projection restricted to $\Gamma_{\varphi} \subset F \times F$.
Thus $\varphi^*\sigma = \tau_* \widetilde{\varphi}^*\sigma$ and
$\varphi_*\sigma =  \widetilde{\varphi}_*  \tau^* \sigma$ for all $\sigma
\in \CH^*(F)$. Let us also denote $\Pi_1, \ldots, \Pi_n$ the planes
contained in $X$, $\Pi^\ast_1, \ldots, \Pi^\ast_n$ the corresponding
dual planes contained in $F$ and $E_1, \ldots, E_n$ their pre-images
under $\tau$. Finally, we write $E$ for the ruled divisor which is
the pre-image of the surface $\Sigma_2$ of lines of second type on
$F$~; see Lemma \ref{lem class of Sigma2}.  The Chow group $\CH^1(F)$
splits as
$$\CH^1(F) = \langle g \rangle \oplus \CH^1(F)_{\mathrm{prim}},$$
where $\CH^1(F)_{\mathrm{prim}}  := p_*q^*\CH_2(X)_{\mathrm{prim}}$ and
$\CH_2(X)_{\mathrm{prim}}  := \{Z \in \CH_2(X)  : [Z] \in
\HH^4(X,\Q)_{\mathrm{prim}}\}$. An easy adaptation of
\cite[Proposition 6]{amerik} and \cite[Lemma 3.11]{voisin2} is the
following.

\begin{lem} \label{lem cubic divisor phi} We have
  $\widetilde{\varphi}^*g = 7\tau^*g -3E + \sum_i a_{g,i}E_i$ for some
  numbers $a_{g,i} \in \Z$, and for $D \in \CH^1(F)_{\mathrm{prim}}$
  we have $\widetilde{\varphi}^*D = -2\tau^*D + \sum_i a_{D,i}E_i$ for
  some numbers $a_{D,i} \in \Z$.\qed
\end{lem}
\begin{rmk}
  Note that by Proposition \ref{prop cubic g c}, $\varphi_*g=28g$, so
  that if $X$ does not contain any plane we get $7 \cdot 28g = 7 \cdot
  \varphi_*g = \widetilde{\varphi}_* \widetilde{\varphi}^*g + 3
  \widetilde{\varphi}_*E = 16 g + 3 \widetilde{\varphi}_*E$ and thus
  $\widetilde{\varphi}_*E =4g$. Likewise, one may compute
  $3\widetilde{\varphi}_*E^2 = 52g^2+ 735c$ and
  $4\widetilde{\varphi}_*E^3 =8059g^3$.
\end{rmk}

First we have the following cohomological description of
$\Gamma_{\varphi}$. Note that Proposition \ref{prop key identity of
  varphi}, together with \eqref{eq L Cubic}, already fully computed
the cohomology class of $\Gamma_{\varphi}$ in terms of $g$, $c$ and
the Beauville--Bogomolov class $\mathfrak{B}$.

\begin{prop} \label{prop cubic hom phi} Let $X$ be a smooth cubic
  fourfold and let $F$ be its variety of lines. Then
  \begin{itemize}
  \item $(\varphi^*+2)(\varphi^*-7)$ vanishes on $\HH^2(F,\Q)$~;
  \item $(\varphi^*-28)(\varphi^*+8)$ vanishes on $\HH^6(F,\Q)$.
  \end{itemize}
  Assume
  moreover that $X$ does not contain any plane. Then
  \begin{itemize}
  \item $(\varphi^*-31)(\varphi^*+14)(\varphi^*-4)$ vanishes on
    $\HH^4(F,\Q)$.
  \end{itemize}
 \end{prop}
\begin{proof}
  According to Proposition \ref{prop key identity of varphi}, we know
  that the cohomology class of $\Gamma_{\varphi} + I_1$ is constant in
  the family, where $I_1 = 0$ if $X$ does not contain any plane. As in
  the proof of Proposition \ref{prop varphi action on homologically
    trivial cycles}, we see that $I_1$ acts trivially on $\HH^i(F,\Q)$
  unless possibly when $i=4$. Therefore it is enough to prove the
  proposition for generic $X$. In particular $X$ does not contain any
  plane and, as explained in \cite[Remark 9]{amerik}, $\HH^2(F,\Q)$
  decomposes as $\langle g \rangle \oplus
  \HH^2(F,\Q)_{\mathrm{prim}}$, where $\HH^2(F,\Q)_{prim}$ is a simple
  Hodge structure spanned by a nowhere degenerate global $2$-form
  $\omega$, and $\HH^4(F,\Q) = \mathrm{Sym}^2\HH^2(F,\Q)$ decomposes
  as a direct sum of simple Hodge structures
\begin{equation} \label{eq H4 dec simple}
\HH^4(F,\Q) = \langle g^2 \rangle \oplus \langle c \rangle \oplus
(g\cdot \HH^2(F,\Q)_{\mathrm{prim}}) \oplus V,
\end{equation}
 where $V$ is spanned by $\omega^2$.

 By Proposition \ref{prop varphi action on cohomology} and Proposition
 \ref{prop varphi action on homologically trivial cycles}
 respectively,
  we see that $\varphi^*$ acts by multiplication by $-2$ on
 $\HH^2(F,\Q)_{prim}$ and by multiplication by $7$ on $\langle g
 \rangle$. Thus $(\varphi^*+2)(\varphi^*-7)$ vanishes on
 $\HH^2(F,\Q)$. By Proposition \ref{prop varphi action on cohomology},
 $\varphi^*-4$ vanishes on $V$.  By Proposition \ref{prop cubic g c},
 $\varphi^*g^2 = 4g^2 + 45c$ and $\varphi^*c=31c$. Since clearly
 $\widetilde{\varphi}^*\omega = -2 \tau^* \omega$, it is
 straightforward from the projection formula to see that
 $\varphi^*([g] \cup \omega) = -14 \, [g]\cup \omega$.  Therefore
 $(\varphi^*-31)(\varphi^*+14)(\varphi^*-4)$ vanishes on
 $\HH^4(F,\Q)$.  Finally, by the Lefschetz isomorphism, we have
 $\HH^6(F,\Q) = g^2 \cdot \HH^2(F,\Q) = \langle g^3 \rangle \oplus g^2
 \cdot \HH^2(F,\Q)_{\mathrm{prim}}.$ By \cite[Theorem 8]{amerik},
 $\varphi^*g^3 = 28g^3$. It remains to compute $\varphi^*([g]^2\cup
 \omega)$. We have
$$
\varphi^*([g]^2\cup \omega) = \tau_*(\widetilde{\varphi}^* [g]^2 \cup
(-2)\tau^*\omega)= \varphi^*[g]^2 \cup (-2\omega) = (4[g]^2 +45[c])
\cup (-2\omega) = -8\, [g]^2\cup \omega - 90\, [c] \cup \omega.
$$
But then, for generic $X$, $c$ can be represented by a singular
rational surface~; see the proof of Lemma 3.2 in \cite{voisin2}.
Therefore $[c] \cup \omega = 0$ and hence $\varphi^*([g]^2\cup \omega)
= -8\, [g]^2\cup \omega$.
\end{proof}

\subsection{Eigenspace decomposition of the Chow groups} \label{sec
  eigenspace} In this paragraph we show that for a cubic fourfold $X$
not containing any plane, the Chow groups of its variety of lines $F$
split under the action of $\varphi^*$. Let us fix some notations.
\begin{defn}\label{defn V_i^n}
  For $\lambda \in \Q$, we set $$V^i_\lambda  := \{ \sigma \in \CH^i(F)
   : \varphi^*\sigma=\lambda \sigma\}.$$
\end{defn}
\noindent For instance, $\CH^0(F) = V_1^4$.  According to Proposition
\ref{prop cubic hom phi} we have
\begin{center}
  $\CH^1(F) = V_7^1 \oplus V_{-2}^1$, with $V_7^1 = \langle g \rangle$
  and $V_{-2}^1 = \CH^1(F)_{\mathrm{prim}}$.
\end{center}\medskip

The following is the main theorem of this section.

\begin{thm}\label{thm eigenspace decomposition}
  Let $X$ be a smooth cubic fourfold and let $F$ be its variety of
  lines.  Let $\varphi :F\dashrightarrow F$ be the rational map of
  Definition \ref{defn varphi}. Then the following are true.
  \begin{enumerate}[(i)]
  \item The action of $\varphi^*$ on $\CH^4(F)$ satisfies $
  (\varphi^*-16)(\varphi^*+8)(\varphi^*-4)=0$ and induces an
  eigenspace decomposition
  \begin{center}
    $\CH^4(F)=V_{16}^4 \oplus V_{-8}^4\oplus V_4^4, $

    with \quad $\CH^4(F)_0 = V_{16}^4 $, \quad $ \CH^4(F)_2 = V_{-8}^4 $,
    \quad $\CH^4(F)_4 = V_4^4$.
  \end{center}
  \item The action of $\varphi^*$ on $\CH^3(F)$ satisfies $
  (\varphi^*-28)(\varphi^*+8)(\varphi^*-4)(\varphi^*+14)=0$ and
  induces a decomposition
  \begin{center}
    $ \CH^3(F) = V_{28}^3 \oplus V_{-8}^3\oplus V_4^3\oplus V_{-14}^3,$

    with \quad $V_{28}^3 = \langle{g^3} \rangle $, \quad $V_{-8}^3 =
    g^2\cdot V_{-2}^1$, \quad $V_{-14}^3 = g\cdot V_{-2}^2$,

    $\CH^3(F)_0 = V_{28}^3 \oplus V_{-8}^3 $, \quad $\CH^3(F)_2 =
    V_{4}^3 \oplus V_{-14}^3$.
  \end{center}
  \item The action of $\varphi^*$ on $\CH^2(F)$ satisfies $
  (\varphi^*-31)(\varphi^*+14)(\varphi^*-4)(\varphi^*+2)=0$, provided
  that $X$ does not contain any plane, and induces a decomposition
   \begin{center}
     $ \CH^2(F)=V_{31}^2 \oplus V_{-14}^2 \oplus V_4^2\oplus V_{-2}^2
     $,

     with \quad $V_{31}^2 = \langle c \rangle$, \quad $V_{-14}^2 = g\cdot
     V_{-2}^1$,

      $ \CH^2(F)_2 = V_{-2}^2 $, \quad  $ \CH^2(F)_0 = V_{31}^2
     \oplus V_{-14}^2 \oplus V_4^2$.
   \end{center}
   \end{enumerate}
\end{thm}

That $\CH^2(F)_2 $ is an eigenspace for the action of $\varphi^*$ on
$\CH^2(F)$ is singled out in the following proposition. Its proof uses
the triangle relation introduced in Definition \ref{dfn triangle}, and
its link to the filtration $\F^\bullet$ on $\CH^4(F)$ as shown in
Theorem \ref{prop F3 triangle}.

\begin{prop} \label{prop eigenA} Let $F$ be the variety of lines on a
  cubic fourfold. Then, inside $\CH^2(F)$, the sub-group $V_{-2}^2$
  coincides with the sub-group $\mathcal{A}_{\hom}$ introduced in
  Definition \ref{defn the filtration}. In particular, $V_{-2}^2 =
  \CH^2(F)_2$.
\end{prop}
\begin{proof}
  The group $\mathcal{A}_{\hom} $ is spanned by cycles of the form
  $S_{l_1} -S_{l_2}$. Let us compute $\varphi^*\sigma$ for $\sigma =
  S_{l_1} - S_{l_2} = I_*x$ where $x = [l_1] - [l_2] \in \CH_0(F)$.
  \begin{eqnarray}
    \varphi^*\sigma &=& 4\sigma + I_*(g^2\cdot I_*x) \nonumber\\
    &=& 4\sigma + I_*(\varphi_*x -4x) \nonumber\\
    &=& 4\sigma + I_*(\varphi_*x + 2x) - I_*(6x) \nonumber\\
    &=& 4\sigma - 6\sigma \nonumber \\
    &=& -2\sigma. \nonumber
  \end{eqnarray} Thus $\mathcal{A}_{\hom} \subseteq V_{-2}^2$.
  Here, the first and second equalities are Proposition
  \ref{prop varphi action on homologically trivial cycles}.
  The fourth equality follows simply from the fact that $\varphi_*x + 2x$
  satisfies the triangle relation of Definition \ref{dfn triangle},
  \emph{i.e.}, belongs to  $\mathcal{R}_{\hom}$, and that $I_*$ acts as
  zero on $\mathcal{R}_{\hom}$ by Theorem \ref{prop F3 triangle}.
  \medskip

  Consider now $\sigma \in \CH^2(F)$ such that $\varphi^*\sigma =
  -2\sigma$. Because $-2$ is not an eigenvalue of $\varphi^*$ acting
  on $\HH^4(F,\Q)$, we see that $\sigma \in \CH^2(F)_{\hom}$. By
  Proposition \ref{prop varphi action on homologically trivial
    cycles}, we have $\varphi^* \sigma = 4\sigma + I_*(g^2\cdot
  \sigma)$ so that $I_*(g^2\cdot \sigma) = -6\sigma$.  Thus $6V_{-2}^2
  \subseteq I_*\CH_0(F)_{\hom}  := \mathcal{A}_{\hom}$.
\end{proof}

\begin{proof}[Proof of Theorem \ref{thm eigenspace
    decomposition}]
  We first show \emph{(i)}, the decomposition for $\CH^4(F)$. Let
  $x\in\CH^4(F)_\mathrm{hom}$, then
  $\varphi_*x+2x\in\mathcal{R}_\mathrm{hom}$. Proposition \ref{prop F3
    triangle} yields $I_*(\varphi_*x+2x)=0$. Hence by Proposition
  \ref{prop varphi action on homologically trivial cycles} we have
$$
(\varphi_*-4)(\varphi_*+2)x=g^2\cdot I_*(\varphi_*x+2x)=g^2\cdot 0=0.
$$
This implies that $(\varphi^*-4)(\varphi^*+8)=0$ on
$\CH^4(F)_\mathrm{hom}$. The rational map $\varphi$ has degree $16$
and hence, by Proposition \ref{prop proj formula},
$\varphi_*\varphi^*=16$ on $\CH^4(F)$. If $x\in\CH^4(F)$, then
$\varphi^*x-16x\in\CH^4(F)_\mathrm{hom}$ and hence
$(\varphi^*-4)(\varphi^*+8)(\varphi^*-16)x=0$. This establishes the
polynomial equation that $\varphi^*$ satisfies on $\CH^4(F)$. Let us
prove that $\varphi^*\mathfrak{o}_F=16\mathfrak{o}_F$, or equivalently
$\varphi_*\mathfrak{o}_F=\mathfrak{o}_F$. This will show that
$V_0^{16} = \CH^4(F)_0$ and that $\CH^4(F)_\mathrm{hom}=V_{-8}^4
\oplus V_4^4$.  By Lemma \ref{lem I [o]}, we have
$I_*\mathfrak{o}_F=\frac{1}{3}(g^2-c).$ Meanwhile, having in mind that
whenever $\varphi$ is defined at $l$ the triple $(l,l,\varphi(l))$ is
a triangle, Proposition \ref{prop intersection identities}\emph{(i)} yields
$(I_*\mathfrak{o}_F)^2=\varphi_*\mathfrak{o}_F+4\mathfrak{o}_F$.
Hence using the identities in \emph{(i)} of Lemma \ref{lem special zero
  cycle}, we get
$$\varphi_*\mathfrak{o}_F+4\mathfrak{o}_F =
\frac{1}{9}(g^4-2g^2c+c^2) = 5\mathfrak{o}_F.$$ It follows that
$\varphi_*\mathfrak{o}_F=\mathfrak{o}_F$.

Finally, we have $g^2\cdot I_*=\varphi_*-4=-6$ on $V_{-8}^4$~; see
Proposition \ref{prop varphi action on homologically trivial cycles}.
This implies that $I_*$ induces an isomorphism
$V_{-8}^4\stackrel{\cong}{\longrightarrow} \mathcal{A}_\mathrm{hom}$
and that $\mathrm{F}^4\CH^4(F) = \ker(I_*) = V_4^4$. This shows that
$V_{-8}^4 = \CH^4(F)_2$ and $V_0^4 = \CH^4(F)_4$. \medskip

Now we show \emph{(iii)}, the decomposition for $\CH^2(F)$ when $X$ does not
contain any plane.
For $\sigma\in\CH^2(F)$, Proposition \ref{prop key identity of varphi}
gives
\begin{equation} \label{eq varphi CH2 4}
  \varphi^*\sigma-4\sigma=I_*(g^2\cdot\sigma) + P(g,D_i,c) \quad
  \mathrm{in} \ \CH^2(F),
\end{equation}
for some weighted homogeneous polynomial $P$ of degree $2$, $D_i \in
\CH^1(F)_{\mathrm{prim}}$.
The cycle $g^2\cdot \sigma$ can be written as $g^2\cdot \sigma =
\deg(g^2\cdot \sigma)\, [\mathfrak{o}_F] + \sigma'$, where $\sigma'
\in \CH^4(F)_{\hom}$. By Lemma \ref{lem I [o]},
$I_*\mathfrak{o}_F=\frac{1}{3}(g^2-c)$. Thus
$\varphi^*\sigma-4\sigma=I_*(\sigma') + Q(g,D_i,c)$, for some weighted
homogeneous polynomial $Q$ of degree $2$. Now $I_*(\sigma')$ belongs
to $\mathcal{A}_{\hom}$ so that by Proposition \ref{prop eigenA} we
get $(\varphi^*+2)I_*\sigma' = 0$. Recall that $\varphi^*c = 31c$ and
$\varphi^*g^2 = 4g^2+45c$~; see \cite[Theorem 8]{amerik}. Lemma
\ref{lem cubic divisor phi} gives, when $X$ does not contain any
plane, $\varphi^*D^2 = 4D^2$ and $\varphi^*(g\cdot D) = -14g\cdot D$
for all $D \in \CH^1(F)_{\mathrm{prim}}$. We can then conclude that
$(\varphi^*+2)(\varphi^*-4)\sigma$ is a polynomial in $g$, $c$ and
$D_i$, $D_i \in \CH^1(F)_{\mathrm{prim}}$. On the other hand, by
Proposition \ref{prop cubic hom phi}, we know that $
(\varphi^*-31)(\varphi^*+14)(\varphi^*-4)$ maps $\CH^2(F)$ to
$\CH^2(F)_{\hom}$. Therefore $
(\varphi^*-31)(\varphi^*+14)(\varphi^*+2)(\varphi^*-4)\sigma$ is a
polynomial in the variables $g$, $c$ and $D_i$ which is homologically
trivial. By the main result of \cite{voisin2}, it must be zero modulo
rational equivalence. This establishes the polynomial equation of
$\varphi^*$ on $\CH^2(F)$.

Let us now show that $V_{31}^2 = \langle c \rangle$. First note that
for $\sigma \in \CH^2(F)_{\hom}$ Proposition \ref{prop varphi action
  on homologically trivial cycles} gives
$\varphi^*\sigma-4\sigma=I_*(g^2\cdot\sigma)$, and hence
$(\varphi^*+2)(\varphi^*-4)\sigma = 0$. Consider then $\sigma \in
V_{31}^2$. By Proposition \ref{prop cubic hom phi}, or rather its
proof, we see that $[\sigma]$ is proportional to $[c]$. But then, by
the above, $31$ is not an eigenvalue of $\varphi^*$ on
$\CH^2(F)_{\hom}$. Hence $\sigma$ is proportional to $c$, which proves
that $V_{31}^2 = \langle c \rangle$.

We now show that $V_{-14}^2 = g \cdot V_{-2}^1$. Recall that $V_{-2}^1
= \CH^1(F)_{\mathrm{prim}}$. Lemma \ref{lem cubic divisor phi} gives
the inclusion $V_{-14}^2 \supseteq g \cdot V_{-2}^1$. Consider then
$\sigma \in V_{-14}^2$. A quick look at the proof of Proposition
\ref{prop cubic hom phi} shows that $[\sigma]$ belongs to $[g] \cdot
\HH^2(F,\Q)_{\mathrm{prim}}$. Since $-14$ is not an eigenvalue of
$\varphi^*$ acting on $\CH^2(F)_{\hom}$, we get that $\sigma \in g
\cdot V_{-2}^1$.

Finally let us show that $\CH^2(F)_0 = V_{31}^2\oplus V_{-14}^2 \oplus
V_4^2.$ Consider $\sigma \in V_4^2$. We know that $0=\varphi^*\sigma -
4\sigma = I_*(g^2\cdot \sigma - \deg(g^2\cdot \sigma)[\mathfrak{o}_F])
+ Q(g,D_i,c)$, for some weighted homogeneous polynomial $Q$ of degree
$2$. Since $ I_*(g^2\cdot \sigma - \deg(g^2\cdot
\sigma)[\mathfrak{o}_F])$ belongs to $\CH^2(F)_2$ and $Q(g,D_i,c)$
belongs to $\CH^2(F)_0$, we get that $ I_*(g^2\cdot \sigma -
\deg(g^2\cdot \sigma)[\mathfrak{o}_F])=0$. It immediately follows from
\eqref{eq L Cubic} and from the identity
$I_*\mathfrak{o}_F=\frac{1}{3}(g^2-c)$ of Lemma \ref{lem I [o]} that
$L_*(g^2\cdot \sigma)=0$. Since by \emph{(iii)} of Lemma \ref{lem special zero
  cycle} $c\cdot \sigma$ is a multiple of $[\mathfrak{o}_F]$ for all
$\sigma \in \CH^2(F)$, we see that $L_*(l\cdot \sigma) = 0$.
Therefore by Theorem \ref{thm cubic L conjecture} $(L^2)_*\sigma =
2\sigma$, {\textit{i.e.}}, $\sigma \in W_1^2 \subset
\CH^2(F)_0$. Since $V_{31}^2\oplus V^2_{-14} \subseteq \CH^2(F)_0$ and
$V^2_{-2}=\CH^2(F)_2$, it is straightforward to conclude that
$\CH^2(F)_0 = V^2_{31}\oplus V^2_{-14} \oplus V^2_4.$ \medskip

Now we prove \emph{(ii)}, the decomposition for $\CH^3(F)$. First consider
$\sigma_1\in g\cdot V^2_{-2}$. Since
$V^2_{-2}=\mathcal{A}_\mathrm{hom}$ and $I_* :
V^4_{-8}\stackrel{\simeq}{\longrightarrow} \mathcal{A}_\mathrm{hom}$
is an isomorphism, we can write $\sigma_1=g\cdot I_*x$ for some $x\in
V^4_{-8}$. Then by Proposition \ref{prop varphi action on
  homologically trivial cycles} we have
$$
\varphi^*\sigma_1=4\sigma_1+ 3g\cdot
I_*(g\cdot\sigma_1)=4\sigma+3g\cdot I_*(g^2\cdot
I_*x)=4\sigma_1+3g\cdot I_*(-6x)=-14\sigma_1.
$$
Here, we used the fact mentioned earlier in the proof that $g^2\cdot
I_*$ acts as multiplication by $-6$ on $V^4_{-8}$. Thus we have shown
that $\varphi^*=-14$ on $g\cdot V^2_{-2}$. The formula for the action
of $\varphi^*$ on $\CH^3(F)_\mathrm{hom}$ obtained in Proposition
\ref{prop varphi action on homologically trivial cycles} shows that
for all $\sigma\in\CH^3(F)_\mathrm{hom}$, we have
$\varphi^*\sigma-4\sigma\in g\cdot V^2_{-2}$. It follows that
$(\varphi^* + 14)(\varphi^* - 4)\sigma=0$. We may then conclude that $
(\varphi^*-28)(\varphi^*+8)(\varphi^*-4)(\varphi^*+14)$ vanishes on
$\CH^3(F)$ by Proposition \ref{prop cubic hom phi}.

Finally, that $V^3_{28} = \langle{g^3} \rangle, V^3_{-8} = g^2\cdot
V^1_{-2}$ and $\CH^3(F)_0 = V^3_{28} \oplus V^3_{-8}$ follows from the
proof of Proposition \ref{prop cubic hom phi} and from the fact proved
above that $28$ and $-8$ are not eigenvalues of $\varphi^*$ acting on
$\CH^3(F)_{\hom}$.
\end{proof}

\begin{rmk} \label{rmk lagrangian} It is interesting to note the
  following. Let $F$ be a hyperk\"ahler variety of dimension $2d$
  which is birational to a Lagrangian fibration $\tilde{F} \rightarrow
  \mathds{P}^d$. Let $B \rightarrow \mathds{P}^d$ be a generically
  finite map with $B$ smooth projective such that the pull-back
  $\tilde{F}_B \rightarrow B$ admits a section. In that case,
  $\tilde{F}_B$ is an abelian scheme over $B$. Consider then the
  relative multiplication-by-$n$ map on $\tilde{F}_B$. By the work of
  Deninger--Murre \cite{dm}, we see that $\CH_{2d}(\tilde{F}_B)$
  splits into eigenspaces for the action of the multiplication-by-$n$
  map and that the eigenvalues are $n^{2d}, n^{2d-1}, \ldots,
  n^d$. This multiplication-by-$n$ map defines naturally a
  self-correspondence $\Gamma_n \in \CH_{2d}(F\times F)$ which induces
  a similar decomposition on $\CH^{2d}(F)$, \emph{i.e.}, we have
  $\CH^{2d}(F) = \CH^{2d}(F)_0 \oplus \CH^{2d}(F)_{2} \oplus \ldots
  \oplus \CH^{2d}(F)_{2d}$ where $\CH^{2d}(F)_{2i} =\{\sigma \in
  \CH^{2d}(F)  : \Gamma_n^*\sigma = n^{2d-i}\sigma \}$.

  Thus, what we see here, is that the rational map $\varphi : F
  \dashrightarrow F$ induces a map $\varphi^*  : \CH^4(F)
  \rightarrow\CH^4(F)$ which splits $\CH^4(F)$ in the same way a
  relative multiplication-by-$(-2)$ map would on a Lagrangian
  fibration. It would be interesting to show that when $F$ admits a
  Lagrangian fibration the two splittings are identical.
\end{rmk}

Let us end this section by the following complement to Theorem
\ref{thm eigenspace decomposition}, which deals with the action of
$\varphi^*$ on $\CH^2(F)$ when $X$ may contain some planes.

\begin{lem} \label{lem varphi CH24} Let $F$ be the variety of lines on
  a cubic fourfold $X$. If $\sigma \in \CH^2(F)_0$, then
  $$\varphi^*\sigma = 4\sigma +P(g,c,D_i),$$ where $P(g,c,D_i)$, $D_i
  \in \CH^1(F)$, is a weighted homogeneous polynomial of degree $2$.
  In particular, if $\sigma \in \CH^2(F)_{0,\hom}$, then
  $\varphi^*\sigma=4\sigma$.
\end{lem}
\begin{proof} Let $\sigma$ be a $2$-cycle on $F$. By Proposition
  \ref{prop key identity of varphi}, we have $$\varphi^*\sigma =
  4\sigma + I_*(g^2\cdot \sigma) + Q(D_i,g,c,\Pi^\ast_j),$$ where
  $Q(D_i,g,c,\Pi^\ast_j)$ is a weighted homogeneous polynomial of
  degree $2$. By Lemma \ref{lem dual plane} below, $\Pi^\ast_j$ is a
  polynomial in $g$, $c$ and $D_{\Pi_j}$, where $D_{\Pi_j}$ is the
  divisor of lines meeting $\Pi_j$. Therefore $Q(D_i,g,c,\Pi^\ast_j)$
  is actually a weighted homogeneous polynomial of degree $2$ in the
  variables $D_i, g, c$.  Now assume that $\sigma \in \CH^2(F)_0$. We
  know from Proposition \ref{prop l CH2 4} that $l \cdot \sigma$ is a
  multiple of $\mathfrak{o}_F$. Since $g^2$ is a linear combination of
  $l$ and $c$ by \eqref{eq l cubic fourfold} and since $c \cdot \tau$
  is a multiple of $\mathfrak{o}_F$ by Lemma \ref{lem special zero
    cycle} for all $\tau \in \CH^2(F)$, we see that $g^2\cdot \sigma$
  is a multiple of $\mathfrak{o}_F$. Thus, by Lemma \ref{lem I [o]},
  $I_*(g^2\cdot \sigma)$ is a linear combination of $g^2$ and $c$.
  All in all, this yields the first part of
  the lemma.

  Now, if $\sigma$ is homologically trivial, then $P(g,c,D_i)$ is
  homologically trivial. By \cite{voisin2}, it follows that
  $P(g,c,D_i) = 0$.
\end{proof}

\begin{lem}\label{lem dual plane}
  Let $\Pi\subset X$ be a plane and $\Pi^\ast\subset F$ the dual
  plane. Then the class of $\Pi^\ast$ in $\CH^2(F)$ is a linear
  combination of $g^2$, $c$, $(D_{\Pi})^2$ and $g\cdot D_{\Pi}$, where
  $D_{\Pi}$ is the divisor of all lines meeting $\Pi$.
\end{lem}
\begin{proof}
  This is implicitly contained in \cite[p.646]{voisin2} and we repeat
  the proof for completeness. Let $\widetilde{D}=q^{-1}\Pi$ then
  $\widetilde{D}$ is isomorphic to $D=D_{\Pi}$ away from $\Pi^\ast$.
  Thus under this identification, we get
$$
D|_D = \det (q^\ast \mathscr{N}_{\Pi/X}) -
\mathscr{T}_{P/F}|_{\widetilde{D}} = 0 - (2q^*h - p^*g) =p^*g-2q^*h,
$$
away from $\Pi^\ast$. By pushing forward to $F$ and noting that the
class $h|_{\Pi}$ is represented by a line on $\Pi$, we get
$$
(D_{\Pi})^2 = a\Pi^\ast + g\cdot D_{\Pi} -2 S_l
$$
for some integer $a\in\Z$, where $l\subset\Pi$ is a line. Note that
$[l]=\mathfrak{o}_F$ (see Lemma \ref{lem special zero cycle}) and
$S_l=I_*[l] =I_*\mathfrak{o}_F =\frac{1}{3}(g^2-c)$ (see Lemma
\ref{lem I [o]}). To conclude, we still have to show that $a\neq 0$.
But this is clear since if $a=0$ then $c$ would be a linear
combination of intersections of divisors and we know this is not the
case. Hence $a\neq 0$ and the Lemma follows.
\end{proof}

\subsection{Compatibility of the action of $\varphi$ with the Fourier
  transform} \label{sec varphi fourier compatible}

\begin{prop} \label{prop varphi fourier compatible} The rational map
  $\varphi  : F \dashrightarrow F$ preserves the Fourier decomposition,
  \emph{i.e.}, $\varphi^*$ maps $\CH^i(F)_s$ into $\CH^i(F)_s$ for all
  $i$ and all $s$.
\end{prop}
\begin{proof}
  By Theorem \ref{thm eigenspace decomposition}, the only thing to
  check is that $\varphi^*\sigma \in \CH^2(F)_0$ for all $\sigma \in
  \CH^2(F)_0$. But this is a consequence of Lemma \ref{lem varphi
    CH24}, together with Theorem \ref{thm reformulation voisin} which
  shows that any weighted homogeneous polynomial $P(g,c,D_i)$ of
  degree $2$ belongs to $\CH^2(F)_0$ (recall that $c$ is a linear
  combination of $l$ and $g^2$).
\end{proof}

Actually we have the following more precise result which implies that
$\varphi^*$ commutes with $\FF\circ \FF$ on $\CH^i(F)_s$ as long as
$(i,s) \neq (2,0)$. On $\CH^2(F)_0$, it is not true that $
\varphi_*\mathcal{F} =\mathcal{F}\varphi^*$ because both $\varphi_* =
\varphi^*$ and $\FF$ are diagonalizable but $\langle l \rangle$ is
fixed by $\FF$ and $\varphi^*l=\varphi_*l \notin \langle l \rangle$.

\begin{prop} Let $F$ be the variety of lines on a cubic fourfold $X$.
  Then
  \begin{align*}
    \varphi_*\mathcal{F} &=\mathcal{F}\varphi^*  : \CH^i(F)_s
    \rightarrow \CH^{4-i+s}(F)_s \quad \mbox{for all }
    (i,s) \neq (2,0)  ;\\
    \varphi_* &= \varphi^*  : \CH^2(F)_0 \rightarrow \CH^2(F)_0.
  \end{align*}
\end{prop}
\begin{proof} By Theorem \ref{thm main splitting}, $\FF$ induces
  isomorphisms $\CH^i(F)_s \rightarrow \CH^{4-i+s}(F)_s$. By
  Theorem \ref{thm eigenspace decomposition}, the action of
  $\varphi^*$ is multiplication by $(-2)^{4-s}$ on
  $\CH^4(F)_{2s}$ and is multiplication by $-2$ on $\CH^2(F)_2$.  By
  Proposition \ref{prop proj formula}, we know that
  $\varphi_*\varphi^*$ is multiplication by $16$ on $\CH^4(F)$ and on
  $\CH^2(F)_{\hom}$. Therefore the proposition is settled for $i=0$ or
  $4$ and $(i,s) =(2,2)$.  For $i=3$, we have $\CH^3(F)_2 =
  \CH^3(F)_{\hom}$ and $\FF$ acts on $\CH^3(F)_2$ as the identity. We
  can then conclude in the case $(i,s) = (3,2)$ from Proposition
  \ref{prop varphi action on homologically trivial cycles} where it is
  shown that $\varphi_*$ and $\varphi^*$ have the same action on
  $\CH^3(F)_{\hom}.$ Let us now consider $(i,s)=(3,0)$ or $(1,0)$. Let
  $D$ be a divisor on $F$. One can check from the proof of Theorem
  \ref{prop L2} that $\FF(D) = \frac{1}{2}l\cdot D$ and that
  $\FF(l\cdot D)=25D$. If $D \in \CH^1(F)_{\mathrm{prim}}$, then by
  \cite[(3.36)]{voisin2} $c \cdot D = 0$ so that $l\cdot D$ is
  proportional to $g^2\cdot D$~; if $D \in \langle g \rangle$, then by
  \cite[Lemma 3.5]{voisin2} $c \cdot D$ is proportional to $g^3$ so
  that $l\cdot D$ is proportional to $g^3$. In the proof of
  Proposition \ref{prop cubic hom phi}, we determined the action of
  $\varphi^*$ on $\HH^2(F,\Q)$ and on $g^2\cdot \HH^2(F,\Q)$. Since
  both $\CH^3(F)_0$ and $\CH^1(F)_0$ inject into cohomology via the
  cycle class map, we need only to determine the action of $\varphi_*$ on
  $\HH^2(F,\Q)$ and on $g^2\cdot \HH^2(F,\Q)$. On the one hand, we
  have $\varphi_*g=28g$ and $\varphi_*\omega = -8\omega$ because
  $-2\tau^*\omega=\widetilde{\varphi}^*\omega$. On the other hand, we
  have $\varphi_*g^3=7g$ and $\varphi_*([g^2] \cup \omega)=
  -\frac{1}{2} \widetilde{\varphi}_*(\tau^*[g^2]\cup
  \widetilde{\varphi}^*\omega) = -\frac{1}{2} \varphi_*[g^2] \cup
  \omega =-2[g^2] \cup \omega$. This proves that $\varphi_*\mathcal{F}
  =\mathcal{F}\varphi^*$ on $\CH^3(F)_0$ and on $\CH^1(F)_0$.

  It remains to see that $ \varphi_* = \varphi^*$ on $\CH^2(F)_0$.
  Consider $\sigma \in \CH^2(F)_0$. By Lemma \ref{lem varphi CH24}, we
  know that $\varphi^*\sigma = 4\sigma +P(g,c,D_i),$ where
  $P(g,c,D_i)$, $D_i \in \CH^1(F)$, is a weighted homogeneous
  polynomial of degree $2$.  The exact same arguments as in the proof
  of Lemma \ref{lem varphi CH24} show that likewise we have
  $\varphi_*\sigma = 4\sigma +Q(g,c,D_i),$ where $Q(g,c,D_i)$, $D_i
  \in \CH^1(F)$, is a weighted homogeneous polynomial of degree $2$.
  Thus $$\varphi^*\sigma - \varphi_*\sigma = (P-Q)(g,c,D_i).$$ By
  Theorem \ref{thm reformulation voisin}, it is then enough to check
  that $\varphi_*=\varphi^*$ on $\HH^4(F,\Q)$. We know from
  Proposition \ref{prop cubic g c} that this is the case on $\langle
  g^2, c \rangle$. From the proof of Proposition \ref{prop cubic hom
    phi}, it is enough to check that $\varphi_*([g] \cup \omega)
  =-14[g] \cup \omega$ and $\varphi_*(\omega^2) =4 \omega^2$. But then
  this follows immediately from
  $-2\tau^*\omega=\widetilde{\varphi}^*\omega$ and from
  $\varphi_*g=28g$.
  \end{proof}

\vspace{10pt}
\section{The Fourier decomposition for $F$ is multiplicative}
\label{sec mult}

In this section, we use the rational map $\varphi : F \dashrightarrow
F$ to investigate the compatibility of the Fourier decomposition with
the intersection product on $\CH^*(F)$. The concatenation of all the
results of this section provides a proof of Theorem \ref{thm main
  mult} for $F$ the variety of lines on a very general cubic fourfold.

\subsection{Intersection with homologically trivial $2$-cycles}
Because $\varphi$ is only a rational map, $\varphi^*$ does not respect
the intersection product in general. Yet we have the following
observation.

\begin{lem}\label{lem varphi and intersection product}
  If $\alpha\in\CH^2(F)_\mathrm{hom}$, then
  $\varphi^*\alpha\cdot\varphi^*\beta = \varphi^*(\alpha\cdot\beta)$
  for all $\beta\in\CH^*(F)$.
\end{lem}
\begin{proof}
  This is Proposition \ref{prop cup} in view of the fact that
  $\CH^2(F)_\mathrm{hom} = T^2(F)$, where $T^2(F)$ is the kernel of
  the Abel-Jacobi map $\CH^2(F)_\mathrm{hom} \rightarrow J^2(F)$,
  because $J^2(F)$ is a sub-quotient of $\HH^3(F,\C) = 0$.
\end{proof}

\begin{prop} \label{prop cubic mult structure2} Let $F$ be the variety
  of lines on a cubic fourfold $X$. Then the Fourier decomposition
  satisfies
  $$
  \CH^4(F)_2 = \CH^1(F)_0^{\cdot 2} \cdot \CH^2(F)_2.
$$
Moreover, we have
\begin{center}
  $g\cdot D \cdot \CH^2(F)_2 = 0$, \quad for all $D \in
  \CH^1(F)_{\mathrm{prim}}$.
\end{center}
\end{prop}
\begin{proof}
  Recall that $\CH^2(F)_2 = V^2_{-2}$.  By Lemma \ref{lem cubic
    divisor phi}, we have for all $D \in \CH^1(F)_{\mathrm{prim}}$
  \begin{align*}
      \varphi^*(g\cdot D) & = \tau_*\widetilde{\varphi}^*(g\cdot D) =
  \tau_*\big((7\tau^*g-3E+\sum a_{g,i}E_i)\cdot (-2\tau^*D+ \sum
  a_{D,i}E_i) \big)\\ & = -14\, g\cdot D + \sum b_{g,D,i}\Pi^\ast_i,
  \end{align*}
  where the $b_{g,D,i}$ are rational numbers and the $\Pi^\ast_i$ are the
  dual planes to the planes contained in $X$. Likewise, we have
  $\varphi^*g^2 = 4g^2 + 45 c + \sum b_{g^2,i}\Pi^\ast_i$ and $\varphi^*D^2 =
  4D^2 + \sum b_{D^2,i}\Pi^\ast_i$ for all $D \in \CH^1(F)_{\mathrm{prim}}$.

  Since for all $\sigma \in \CH^2(F)_{\hom}$ $c\cdot \sigma =0$ by
  Lemma \ref{lem special zero cycle}\emph{(iii)} and $\Pi^\ast_i \cdot
  \sigma = 0$, we see thanks to Lemma \ref{lem varphi and intersection
    product} that for all $D \in \CH^1(F)_{\mathrm{prim}}$ and for all
  $\sigma \in \CH^2(F)_2$ we have
  \begin{align}
    \varphi^*(g^2 \cdot \sigma) & = -8\, g^2 \cdot\sigma,\nonumber \\
    \label{eq temp} \varphi^*(D^2 \cdot \sigma) & = -8\, D^2
    \cdot\sigma, \\ \varphi^*(g \cdot D \cdot \sigma) & = 28\, g\cdot
    D \cdot\sigma. \nonumber
  \end{align}
  The proposition then follows because $28$ is not an eigenvalue of
  $\varphi^*$ acting on $\CH^4(F)$ by Theorem \ref{thm eigenspace
    decomposition} and because $g^2 \cdot : \CH^2(F)_2 \rightarrow
  \CH^4(F)_2$ is an isomorphism.
\end{proof}

Since $\CH^1(F)_0^{\cdot 2} \subseteq \CH^2(F)_0$ by Theorem \ref{thm
  reformulation voisin}, the following is a refinement of the first
statement of Proposition \ref{prop cubic mult structure2}.

\begin{prop} Let $F$ be the variety of lines on a cubic fourfold. Then
$$\CH^2(F)_0 \cdot \CH^2(F)_2 = \CH^4(F)_2.$$
\end{prop}
\begin{proof}
  By Proposition \ref{prop eigenA}, we know that $\CH^2(F)_2 =
  V^2_{-2}$ consists of homologically trivial cycles. Consider $\sigma
  \in \CH^2(F)_0$ and $\tau \in \CH^2(F)_2$. By Lemmas \ref{lem varphi
    and intersection product} and \ref{lem varphi CH24}, we have
 $$
 \varphi^*(\sigma \cdot\tau) = \varphi^*\sigma \cdot\varphi^*\tau =
 (4\sigma + P(g,c,D_i)) \cdot (-2\tau) = -8\sigma \cdot \tau -
 2P(g,c,D_i)\cdot \tau.
 $$
 Since $c \cdot\tau = 0$, by Proposition \ref{prop cubic mult
   structure2} we obtain that $P(g,c,D_i)\cdot \tau$ belongs to
 $\CH^4(F)_2$. It then immediately follows from the fact that
 $\CH^4(F)_2 = V^4_{-8}$ that $\sigma \cdot \tau \in \CH^4(F)_2$.
\end{proof}

\begin{prop} \label{prop CH3624hom}
  Let $F$ be the variety of lines on a cubic fourfold.
  Then
 $$\CH^1(F)\cdot \CH^2(F)_{0,\hom} = 0.$$
\end{prop}
\begin{proof}
  Indeed, if $\sigma \in \CH^2(F)_{0,\hom}$, then by Lemma \ref{lem
    varphi CH24} $\varphi^*\sigma = 4\sigma$. It follows from Lemma
  \ref{lem varphi and intersection product} that $\varphi^*(D \cdot
  \sigma) = 4\sigma \cdot \varphi^*D$ for all $D \in \CH^1(F)$.
  Depending whether $D \in \langle g \rangle$ or $D \in
  \CH^1(F)_{\mathrm{prim}}$, we see that $\varphi^*(D \cdot \sigma) =
  28\sigma \cdot D$ or $-8\sigma \cdot D$. By Theorem \ref{thm
    eigenspace decomposition}, neither $28$ nor $-8$ are eigenvalues
  of $\varphi^*$ acting on $\CH^3(F)_{\hom} = \CH^3(F)_2$. Hence $D
  \cdot \sigma =0$.
\end{proof}

\begin{prop} \label{prop cubic mult structure1} Let $F$ be the variety
  of lines on a cubic fourfold.  Then
\begin{center}
   $\CH^2(F)_0 \cdot \CH^2(F)_{0,\hom} = 0.$
\end{center}
\end{prop}
\begin{proof} Consider $\sigma \in \CH^2(F)_0$ and $\tau \in
  \CH^2(F)_{0,\hom}$. By Lemmas \ref{lem varphi and intersection
    product} and \ref{lem varphi CH24}, we have
$$
\varphi^*(\sigma \cdot\tau) = \varphi^*\sigma \cdot\varphi^*\tau =
(4\sigma + P(g,c,D_i)) \cdot 4\tau = 16\sigma \cdot \tau +
4P(g,c,D_i)\cdot \tau.
$$
But then, as in the proof of Proposition \ref{prop cubic mult
  structure2}, we see that $\varphi^*(g^2\cdot \tau) = 16g^2\cdot
\tau$, $\varphi^*(D^2\cdot \tau) = 16D^2\cdot \tau$ and
$\varphi^*(g\cdot D\cdot \tau) = -14\cdot 4 g \cdot D\cdot \tau$ for
all $D \in \CH^1(F)_{\mathrm{prim}}$. Also, we have $\varphi^*(c\cdot \tau) =
31\cdot 4 c\cdot \tau.$ Since none of these numbers are
eigenvalues of $\varphi^*$ acting on $\CH^4(F)_{\hom}$, we see that
$P(g,c,D_i)\cdot \tau = 0$. Thus $\varphi^*(\sigma \cdot\tau) =
16\sigma \cdot \tau$. Again, $16$ is not an eigenvalue of $\varphi^*$
acting on $\CH^4(F)_{\hom}$ and hence $\sigma \cdot \tau =0$.
\end{proof}

\subsection{Intersection with divisors}
We know that $\CH^3(F)_2=\CH^3(F)_{\hom}$ and $\CH^2(F)_2 \subseteq
\CH^2(F)_{\hom}$, so that it is clear that $\CH^1(F)\cdot \CH^2(F)_2
\subseteq \CH^3(F)_2$. The following is thus a strengthening of the
equality $\CH^4(F)_2 = \CH^1(F)^{\cdot 2} \cdot \CH^2(F)_2$ of
Proposition \ref{prop cubic mult structure2}.

\begin{prop} \label{prop cubic CH361402} Let $F$ be the variety
  of lines on a cubic fourfold $X$.  Then
$$\CH^1(F) \cdot \CH^3(F)_2 = \CH^4(F)_2.$$
\end{prop}
\begin{proof}
  Recall from Theorems \ref{prop L2} \& \ref{prop main} that
  $\CH^3(F)_{\hom}= \CH^3(F)_2$.  Let us first compute
  $\varphi^*(g\cdot \sigma)$ for $\sigma \in \CH^3(F)_2$.  By Lemma
  \ref{lem cubic divisor phi}, we have $\widetilde{\varphi}^*g =
  7\tau^*g -3E + \sum a_{g,i}E_i$ for some numbers $a_{g,i} \in \Z$.
  Thus
  \begin{equation*}
    \varphi^*(g\cdot \sigma) = \tau_*\big((7\tau^*g -3E + \sum
    a_{g,i}E_i )
    \cdot \widetilde{\varphi}^*\sigma \big) = 7g\cdot \varphi^*\sigma
    - 3\tau_*(E \cdot \widetilde{\varphi}^*\sigma)
    + \sum a_{g,i}\tau_*(E_i\cdot \widetilde{\varphi}^*\sigma).
  \end{equation*}
  Now $\widetilde{\varphi}^*\sigma$ is a homologically trivial cycles
  so that $\tau_*(E \cdot \widetilde{\varphi}^*\sigma)$ is a
  homologically trivial zero-cycle supported on $\Sigma_2$. By
  Proposition \ref{prop support} we get $\tau_*(E \cdot
  \widetilde{\varphi}^*\sigma) \in \CH^4(F)_2 = V^4_{-8}$. Likewise
  $\tau_*(E_i\cdot \widetilde{\varphi}^*\sigma)$ is a homologically
  trivial cycle support on the plane $\Pi^\ast_i$ and is therefore
  zero. By Theorem \ref{thm eigenspace decomposition}, the action of
  $\varphi^*$ on $\CH^3(F)_{\hom}$ diagonalizes with eigenvalues $4$
  and $-14$. Since $7\cdot 4$ and $7\cdot (-14)$ are not eigenvalues
  of $\varphi^*$ acting on $\CH^4(F)$ we conclude that $g\cdot\sigma
  \in \CH^4(F)_2 = V_{-8}^4$.

  Consider now $D \in V^1_{-2} = \CH^1(F)_{\mathrm{prim}}$. In that case, we
  have $\widetilde{\varphi}^*D = -2\tau^*D + \sum a_{D,i}E_i$ for some
  numbers $a_{D,i} \in \Z$. Thus
\begin{equation*}
  \varphi^*(D\cdot \sigma) = \tau_*\big((-2\tau^*D + \sum a_{D,i}E_i )
  \cdot \widetilde{\varphi}^*\sigma \big) = -2D\cdot \varphi^*\sigma
  + \sum a_{D,i} \tau_*(E_i\cdot \widetilde{\varphi}^*\sigma)
  = -2D\cdot \varphi^*\sigma.
  \end{equation*}
  It follows that $\varphi^*(D\cdot \sigma) \in V^4_{-8}$ if
  $\varphi^* \sigma =4 \sigma$ and that $\varphi^*(D\cdot \sigma) =0$
  if $\varphi^*\sigma =-14\sigma$.
\end{proof}

\begin{prop} \label{prop cubic CH362412} Let $F$ be the variety of
  lines on a cubic fourfold $X$. Then
  $$\CH^1(F)_{\mathrm{prim}} \cdot \CH^2(F)_0 = \CH^3(F)_0.$$
Moreover, if $X$ is very general, then $\CH^1(F) \cdot \CH^2(F)_0 =
\CH^3(F)_0.$
\end{prop}
\begin{proof}
  Let us consider $D \in \CH^1(F)_{\mathrm{prim}}$ and $\sigma \in
  \CH^2(F)_2$. We have
\begin{align*}
  \varphi^*(\sigma \cdot D) &= \tau_*(\widetilde{\varphi}^*\sigma
  \cdot \widetilde{\varphi}^*D) \\
  &= \tau_*\big((4\tau^*\sigma + \tau^*P(g,c,D_i) + \alpha)
  \cdot (-2\tau^*D + \sum a_{D,i} E_i)   \big) \\
  &= -8\sigma \cdot D - 2P(g,c,D_i) \cdot D + \sum a_{D,i}
  \tau_*(\alpha \cdot E_i).
\end{align*}
For the second equality, we have used Lemmas \ref{lem cubic divisor
  phi} and \ref{lem varphi CH24}, and $\alpha$ is a $2$-cycle on
$\widetilde{F}$ such that $\tau_*\alpha =0$. The third equality is the
projection formula coupled with $\tau_*E_i =0$. Since $l$ is a linear
combination of $g^2$ and $c$, Theorem \ref{thm reformulation voisin}
gives $P(g,c,D_i) \cdot D \in \CH^3(F)_0$. On the other hand,
$\tau_*(\alpha \cdot E_i)$ is a $1$-cycle supported on $\Pi^\ast_i$.
Therefore, $\tau_*(\alpha \cdot E_i)$ is rationally equivalent to a
multiple of a rational curve on $F$ that contains a point rationally
equivalent to $\mathfrak{o}_F$ (for instance because $c \cdot
\Pi^\ast_i$ is a multiple of $\mathfrak{o}_F$). Now we claim that the
class of every rational curve $C$ on $F$ that contains a point
rationally equivalent to $\mathfrak{o}_F$ is in $\CH^3(F)_0$. Let us
finish the proof assuming the claim. First, we see that
$$\varphi^*(\sigma \cdot D) + 8 \, \sigma \cdot D \in \CH^3(F)_0.$$
Second, we know from Theorem \ref{thm eigenspace decomposition} that
$\CH^3(F)_0 = V^3_{28}\oplus V^3_{-8}$ and that $\CH^3(F)_2 = V^3_4
\oplus V^3_{-14}$. We immediately get that $\sigma \cdot D \in
\CH^3(F)_0$.

Let us now prove the claim. As in the proof of \cite[Corollary
3.4]{voisin2}, note that $I^2$ is the restriction of $I \times I$ to
the diagonal $\Delta_{F\times F}$ of $(F \times F)^2$. Therefore
$(I^2)_*C = ((I\times I)_*(\iota_\Delta)_*C)|_{\Delta_{F\times F}}$.
But then, because $C$ is a rational curve and because it contains a
point rationally equivalent to $\mathfrak{o}_F$, we have
$(\iota_\Delta)_*C = \mathfrak{o}_F \times C + C \times
\mathfrak{o}_F$. This immediately yields $(I^2)_*C = 2I_*C \cdot
I_*\mathfrak{o}_F$. Since $I_*\mathfrak{o}_F$ is a linear combination
of $l$ and $g^2$, Theorem \ref{thm reformulation voisin} shows that
$(I^2)_*C \in \CH^3(F)_0$.  By comparing \eqref{identity of voisin}
and \eqref{eq simplified varphi}, one may see that, for $\mu \in
\CH^3(F)_2 = \CH^3(F)_{\hom}$, $\varphi^*\mu = 4\mu$ if and only if
$(I^2)_*\mu = 2\mu$ and that $\varphi^*\mu = -14\mu$ if and only if
$(I^2)_*\mu = -4\mu$. Thus, in view of Theorem \ref{thm eigenspace
  decomposition}, $I^2$ induces an automorphism of $\CH^3(F)_{\hom}$
and we obtain that $C \in \CH^3(F)_0$.

Finally,
if $X$ is very general, then \eqref{eq H4 dec simple} shows that
$\CH^2(F)_0$ is spanned by $l$, $g^2$ and by
$\CH^2(F)_{0,\hom}$. The proposition is thus a combination of
Proposition \ref{prop CH3624hom}, Theorem \ref{thm reformulation
  voisin} and Proposition \ref{prop l CH2 4}.
\end{proof}

\subsection{Self-intersection of $2$-cycles.}
We already proved that $\CH^2(F)_2\cdot \CH^2(F)=\CH^4(F)_4$ and
that $\CH^2(F)_0 \cdot \CH^2(F)_2 = \CH^4(F)_2$. We cannot prove that
$\CH^2(F)_0 \cdot \CH^2(F)_0 = \CH^4(F)_0$ in general but we have the
following.

\begin{prop}Let $F$ be the variety of lines on a cubic fourfold $X$. Then
$$\CH^2(F)_0 \cdot \CH^2(F)_0 \subseteq \CH^4(F)_0 \oplus \CH^4(F)_2.$$
If $X$ is a very general cubic, then $\CH^2(F)_0 \cdot \CH^2(F)_0 =
\CH^4(F)_0.$
\end{prop}
\begin{proof} First we show that
\begin{equation} \label{eq CH363624} \CH^1(F) \cdot \CH^1(F)
    \cdot \CH^2(F)_0 \subseteq \CH^4(F)_0 \oplus \CH^4(F)_2.
\end{equation}
Indeed, if $D,D' \in \CH^1(F)$ and $\sigma \in \CH^2(F)_0$, then
$D'\cdot \sigma \in \CH^3(F) = \CH^3(F)_0 \oplus \CH^3(F)_2$, so that
by Theorem \ref{thm reformulation voisin} and Proposition \ref{prop
  cubic CH361402} we get $D \cdot D' \cdot \sigma \in \CH^4(F)_0
\oplus \CH^4(F)_2.$

Let us now consider $\sigma$ and $\sigma' \in \CH^2(F)_0$. By Lemma
\ref{lem varphi CH24}, we have \begin{center}
  $\widetilde{\varphi}^*\sigma = 4\tau^*\sigma + \tau^*P + \alpha$
  \quad and \quad $\widetilde{\varphi}^*\sigma' = 4\tau^*\sigma' +
  \tau^*P' + \alpha'$,
\end{center}
where $P$, $P'$ are weighted homogeneous polynomials of degree $2$ in
$g, c, D_i$, and $\alpha$, $\alpha'$ are cycles supported on $E$ such
that $\tau_*\alpha =0$ and $\tau_*\alpha'=0$. We then have
\begin{align*}
  \varphi^*(\sigma \cdot \sigma') &=
  \tau_*(\widetilde{\varphi}^*\sigma
  \cdot \widetilde{\varphi}^*\sigma')\\
  &= 16\, \sigma \cdot \sigma' + \sigma'\cdot P + \sigma \cdot P' + P
  \cdot P' + \tau_*(\alpha \cdot \alpha').
\end{align*}
By Theorem \ref{thm reformulation voisin}, $P \cdot P' \in
\CH^4(F)_0$. By \eqref{eq CH363624}, $\sigma' \cdot P$ and $\sigma
\cdot P'$ belong to $\CH^4(F)_0 \oplus \CH^4(F)_2$. The cycle
$\tau_*(\alpha \cdot \alpha')$ is supported on $\Sigma_2$ so that, by
Proposition \ref{prop support}, $\tau_*(\alpha \cdot \alpha') \in
\CH^4(F)_0 \oplus \CH^4(F)_2$. It is then immediate to conclude from
Theorem \ref{thm eigenspace decomposition} that $\sigma \cdot \sigma'
\in \CH^4(F)_0 \oplus \CH^4(F)_2.$

If $X$ is very general, then \eqref{eq H4 dec simple} shows that
$\CH^2(F)_0$ is spanned by $l$, by $g^2$ and by $\CH^2(F)_{0,\hom}$.
That $\CH^2(F)_0 \cdot \CH^2(F)_0 = \CH^4(F)_0$ is thus a combination
of Proposition \ref{prop cubic mult structure1}, Theorem \ref{thm
  reformulation voisin} and Proposition \ref{prop l CH2 4}.
\end{proof}

\begin{rmk}[About the intersection with $\CH^2(F)_0$] \label{rmk
    intersection CH24} The results of this section prove Theorem
  \ref{thm main mult} for the variety of lines on a very general cubic
  fourfold. In order to prove Theorem \ref{thm main mult} for the
  variety of lines on a (not necessarily very general) cubic fourfold,
  one would need to prove that $g \cdot \CH^2(F)_0 = \CH^3(F)_0$ and
  that $\CH^2(F)_0 \cdot \CH^2(F)_0 = \CH^4(F)_0$.

  That $g \cdot \CH^2(F)_0 = \CH^3(F)_0$ would follow from knowing
  that the image of $\CH^1(\Sigma_2) \rightarrow \CH^3(F)$ is
  contained in $\CH^3(F)_0$.  In addition, that $\CH^2(F)_0 \cdot
  \CH^2(F)_0 = \CH^4(F)_0$ would follow from knowing that the image of
  $\CH^1(\Sigma_2)\cdot \CH^1(\Sigma_2) \rightarrow \CH^4(F)$ is
  contained in $\CH^4(F)_0$.
\end{rmk}

\vspace{10pt}
\appendix
\section{Some geometry of cubic fourfolds}
In this appendix we review and prove some results on the geometry of a
smooth cubic fourfold and its variety of lines. We fix an
algebraically closed field $k$ of characteristic different from 2 or
3. Let $X\subset\PP(V)$ be a smooth cubic fourfold, where $V$ is a
6-dimensional vector space over $k$. Let $F$ be the variety of lines
on $X$, which is smooth of dimension 4 by \cite{ak}. When $k=\C$, it
is also known that $F$ is a hyperk\"ahler manifold of
$\mathrm{K3}^{[2]}$-type~; see \cite{bd}. Let $h\in\Pic(X)$ be the
class of a hyperplane section.  The variety $F$ is naturally a
sub-variety of $G(2,V)$ and hence inherits a natural rank 2 sub-bundle
$\mathscr{E}_2\subset V$. Let $g=-c_1(\mathscr{E}_2)$ be the Pl\"ucker
polarization and $c=c_2(\mathscr{E}_2)$.

\subsection{Some natural cycle classes}
As a sub-variety of $G(2,V)$, the variety $F$ can be viewed as the
zero locus of a section of $\Sym^3\mathscr{E}_2^\vee$~; see
\cite{ak}. This description also gives a short exact sequence
\[
\xymatrix@C=0.5cm{ 0 \ar[r] & \mathscr{T}_F \ar[rr] &&
  \mathscr{T}_{G(2,V)}|_F \ar[rr] &&\Sym^3\mathscr{E}_2^\vee|_F \ar[r]
  & 0. }
\]
Then a routine Chern class computation gives the following
\begin{lem}\label{lem second chern class cubic}
$c_2(F)=5g^2-8c$.\qed
\end{lem}

Over $F$, we have the universal family of lines
$$
\xymatrix{
 P\ar[r]^q\ar[d]_p &X\\
 F &
}
$$
\begin{defn}\label{defn cylinder and aj}
Associated with the above situation, we define the \textit{cylinder
homomorphism} as $\Psi=q_*p^* :\CH^i(F)\rightarrow \CH^{i-1}(X)$ and
the \textit{Abel-Jacobi homomorphism} as
$\Phi=p_*q^* :\CH^i(X)\rightarrow \CH^{i-1}(F)$.
\end{defn}

It is proved in \cite{voisin2} that $F$ contains a surface $W$ which
represents $c=c_2(\mathscr{E}_2)$ in $\CH^2(F)$. Furthermore, any two
points on $W$ are rationally equivalent on $F$. Hence a point on $W$
defines a special degree 1 element $\mathfrak{o}_F\in\CH_0(F)$.  We
summarize some results we need.

\begin{lem}\label{lem special zero cycle} The following statements
  hold true.
  \begin{enumerate}[(i)]
  \item Any weighted homogeneous polynomial of degree 4 in $g$ and $c$
    is a multiple of $\mathfrak{o}_F$ in $\CH_0(F)$. To be more
    precise, we have $g^4=108\mathfrak{o}_F$, $c^2=27\mathfrak{o}_F$
    and $g^2c=45\mathfrak{o}_F$.

  \item Let $C_x\subset F$ be the curve of all lines passing through a
  general point $x\in X$, then $g\cdot C_x=6\mathfrak{o}_F$ in
  $\CH_0(F)$.

  \item For any 2-cycle $\gamma\in\CH^2(F)$, the intersection
  $\gamma\cdot c$ is a multiple of $\mathfrak{o}_F$ in $\CH_0(F)$.

  \item If $\Pi=\PP^2\subset X$ is a plane contained in $X$, then
  $[l]=\mathfrak{o}_F$ for any line $l\subset\Pi$.

  \item Let $l\subset X$ be a line such that $[l]=\mathfrak{o}_F$, then
  $3l=h^3$ in $\CH^3(X)$.
  \end{enumerate}
\end{lem}
\begin{proof}
  For \emph{(i)}, \emph{(ii)} and \emph{(iii)}, we refer to
  \cite{voisin2}. For \emph{(iv)}, we take $\gamma$ to be the lines on
  $\Pi$ and then apply \emph{(iii)}. To prove \emph{(v)}, we note that
  the class $c^2$ can be represented by the sum of all points
  corresponding to lines on the intersection $Y=H_1\cap H_2$ of two
  general hyperplanes on $X$. Since $Y$ is a smooth cubic surface, it
  contains exactly $27$ lines. These $27$ lines can be divided into
  $9$ groups, where each group forms a triangle~; see Definition
  \ref{dfn triangle}. Statement \emph{(v)} then follows from the
  observation that the class of each triangle on $X$ is equal to $h^3$
  and that $c^2=27\mathfrak{o}_F$.
\end{proof}

Now we introduce some correspondences on $F$. The most natural one is
the incidence correspondence $I\subset F\times F$ which consists of
pairs of lines meeting each other. Let $H_1$ and $H_2$ be two general
hyperplanes sections on $X$. Then we define
\begin{equation}\label{eq Gamma_h}
\Gamma_h=\{([l_1],[l_2])\in F\times F : \exists\, x\in H_1, \text{
such that } x\in l_1\cap l_2\} ;
\end{equation}
\begin{equation}\label{eq Gamma_h^2}
\Gamma_{h^2}=\{([l_1],[l_2])\in F\times F : \exists\, x\in H_1\cap
H_2, \text{ such that } x\in l_1\cap l_2\}.
\end{equation}

\begin{lem}\label{lem identities self intersection of g} We have the
  following list of identities.
\begin{enumerate}[(i)]
\item The cylinder homomorphism $\Psi$ satisfies
$$
\Psi(g)=6[X],\quad
\Psi(g^2)=21h,\quad\Psi(g^3)=36h^2,\quad\Psi(g^4)=36 h^3.
$$
The Abel-Jacobi homomorphism $\Phi$ satisfies
$$
\Phi(h^2)=g,\quad \Phi(h^3)=g^2-c,\quad \Phi(h^4)=3C_x.
$$
\item The action of $I$ on self-intersections of $g$ is given by
$$
I_*(g^2)=21[F],\quad I_*(g^3)=36 g, \quad I_*(g^4)=36\Phi(h^3),
$$
where $I\subset F\times F$ is the incidence correspondence.
\item The action of $\Gamma_h$ on the self-intersections of $g$ is
given by
$$
(\Gamma_h)^*g= 6[F],\quad (\Gamma_h)^*g^2=21 g,\quad
(\Gamma_h)^*g^3=36\Phi(h^3),\quad(\Gamma_h)^*g^4=108C_x.
$$
\item The action of $\Gamma_{h^2}$ on the self-intersections of $g$ is
  given by
$$
(\Gamma_{h^2})^*g=6g, \quad(\Gamma_{h^2})^*g^2=21\Phi(h^3),
\quad(\Gamma_{h^2})^*g^3=108 C_x, \quad(\Gamma_{h^2})^*g^4=0.
$$
\end{enumerate}
\end{lem}
\begin{proof}
  First, for dimension reasons, we have $\Psi(g)=n_1[X]$ for some
  integer $n_1$.  Then one computes
$$
n_1\equiv x\cdot \Psi(g)\equiv\Phi(x)\cdot g =C_x\cdot g\equiv 6.
$$
Since $\Psi(g^2)$ lives on $X$, it has to be a multiple of $h$.  We
assume that $\Psi(g^2)=n_2h$, then we have
$$
n_2\equiv\Psi(g^2)\cdot l\equiv g^2\cdot\Phi(l)=g^2\cdot
S_l\equiv21.
$$
In the last step, we used the fact that the cohomology class of $S_l$
is $\frac{1}{3}(g^2-c)$~; see \cite[\S0]{voisin}. Similarly, we have
$\Psi(g^3)=n_3h^2$ and we determine $n_3=36$ by
$$
3n_3\equiv\Psi(g^3)\cdot h^2\equiv g^3\cdot\Phi(h^2)=g^3\cdot
g\equiv 108.
$$
By Lemma \ref{lem special zero cycle}, we have
$\Psi(\mathfrak{o}_F)=\frac{1}{3}h^3$. Hence
$\Psi(g^4)=\Psi(108\mathfrak{o}_F)=36h^3$. The equality
$g=\Phi(h^2)$ and $\Phi(h^4)=3C_x$ are easy. To prove the remaining
equality, we note that $P=\PP(\mathscr{E}_2)$ and $q^*h$ is the
Chern class of the relative $\calO(1)$ bundle. Hence we have the
equality
$$
q^*h^2+p^*c_1(\mathscr{E}_2)q^*h +p^*c_2(\mathscr{E}_2)=0.
$$
Then it follows
\begin{align*}
q^*h^3 & =q^*h(-p^*c_1(\mathscr{E}_2)q^*h-p^*c_2(\mathscr{E}_2))\\
 &= q^*h (p^*gq^*h-p^*c)\\
 &=p^*g(p^*gq^*h-p^*c)-q^*hp^*c\\
 &=p^*(g^2-c)q^*h -p^*(g^2c).
\end{align*}
Hence $\Phi(h^3)=p_*q^*h^3=g^2-c$. This proves \emph{(i)}. Since
$I_*=\Phi\circ\Psi$, $(\Gamma_h)^*g^i=\Phi(h\cdot \Psi(g^i))$ and
$(\Gamma_{h^2})^*g^i=\Phi(h^2\cdot \Psi(g^i))$, all the remaining
identities follow from \emph{(i)}.
\end{proof}

\begin{lem}\label{lem I [o]}
$I_*\mathfrak{o}_F=\frac{1}{3}(g^2-c)$.
\end{lem}
\begin{proof}
By \emph{(v)} of Lemma \ref{lem special zero cycle}, we see that
$\Psi(\mathfrak{o}_F)=\frac{1}{3}h^3$. Then by \emph{(i)} of Lemma \ref{lem
identities self intersection of g}, we conclude the proof.
\end{proof}

\begin{prop}\label{prop class of Gamma_h}
The following identities hold true in the Chow groups of $F\times F$
\begin{align*}
\Gamma_h &= \frac{1}{3}(g_1^3+g_1^2g_2+g_1g_2^2+g_2^3
-2g_1c_1-g_1c_2-g_2c_1-2g_2c_2) ;\\
\Gamma_{h^2}&=
\frac{1}{3}(g_1^3g_2+g_1^2g_2^2+g_1g_2^3-g_1^2c_2-2g_1g_2c_2
-g_2^2c_1-2g_1g_2c_1+c_1c_2).
\end{align*}
\end{prop}
\begin{proof}
Let $h_i=p_i^*h$ on $X\times X$. In the proof of \cite[Proposition
3.3]{voisin2}, Voisin shows that there is a K\"unneth type
decomposition
$$
\Delta_X=\Delta_1+\Delta_0
$$
in $\CH^4(X\times X)$ such that $\Delta_1=\sum_{i=0}^{4}
\alpha_ih_1^ih_2^{4-i}$ and $\Delta_0\cdot h_i=0$ in $\CH^5(X\times
X)$. One can show that $\alpha_i=\frac{1}{3}$, $i=0,\ldots,4$, as
follows. Let both side act on $h^j$, $j=1,\ldots,4$, and get
$$
h^{j}=\alpha_{4-j}(h^4)h^{j}=3\alpha_{4-j}h^{j}.
$$
Hence we get $\alpha_i=\frac{1}{3}$, $i=0,\ldots,3$. By symmetry, we
also have $\alpha_4=\frac{1}{3}$. Therefore we have
\begin{align*}
\Gamma_h &=(p\times p)_*(q\times q)^*(\Delta_X\cdot h_1)\\
 &=\frac{1}{3}(p\times p)(q\times q)^*(h_1h_2^4+h_1^2h_2^3+h_1^3h_2^2
 +h_1^4h_2)\\
 &=\frac{1}{3}\sum_{i=1}^{4}p_1^*\Phi(h^i)p_2^*\Phi(h^{5-i})\\
 &=\frac{1}{3}(3p_2^*C_x+g_1(g_2^2-c_2)+(g_1^2-c_1)g_2 +3p_1^*C_x)
\end{align*}
Note that in the Proof of \cite[Lemma 3.2]{voisin3}, Voisin shows
that $3C_x=g^3-2gc$. Put this into the above formula, we get the
equality for $\Gamma_h$. The second equality can be proved in a
similar way.
\end{proof}

\subsection{Secant lines and their specializations}
We recall some definitions from \cite{relation}. Let $C_1,
C_2\subset X$ be two disjoint curves on $X$.
\begin{defn}\label{defn secant line}
  A \textit{secant line} of the pair $(C_1,C_2)$ is a line on $X$ that
  meets both $C_1$ and $C_2$.
\end{defn}
It is shown in \cite{relation} that the total number, if finite, of
secant lines (counted with multiplicity) of $(C_1,C_2)$ is equal to
$5d_1d_2$, where $d_i=\deg(C_i)$, $i=1,2$. We will be mainly
interested in the case when both $C_1$ and $C_2$ are lines.

\begin{lem}\label{lem specialization of secant lines}
  Let $l_1,l_2$ be two lines on $X$ meeting each other in a point $x$
  and $\Pi=\PP^2$ the plane spanned by $l_1$ and $l_2$. Assume that
  $\Pi$ is not contained in $X$ and $\Pi\cdot X = l_1 +l_2 +l_3$ for a
  third line $l_3\subset X$. Let $\{l_t : t\in T\}$ be a 1-dimensional
  family of lines on $X$ with $l_{t_0}=l_1$. Assume that for a general
  point $t\in T$, the pair $(l_2,l_t)$ has five secant lines
  $\{E_{1,t},\ldots, E_{5,t}\}$. Then as $t$ specializes to $t_0$,
  four of the secant lines, say $\{E_{1,t},\ldots,E_{4,t}\}$,
  specialize to four lines $\{E_1,\ldots, E_4\}$ passing through $x$
  while the fifth secant line $E_{5,t}$ specializes to the line $l_3$,
  such that
$$
[l_1] + [l_2] + [E_1] +[E_2] + [E_3] + [E_4] = 6\mathfrak{o}_F.
$$
\end{lem}
\begin{proof}
  Note that the tangent space of $T$ at the point $t_0$ corresponds to
  a global section $v\in\HH^0(l_1,\mathscr{N}_{l_1/X})$. We may assume
  that $v(x)\neq 0$ and hence $v(x)$, $\mathscr{T}_{l_1,x}$ and
  $\mathscr{T}_{l_2,x}$ span a linear $\PP^3$ which is tangent to $X$
  at the point $x$. There are 6 lines on $X$ passing through $x$ and
  contained in the above $\PP^3$. These lines are
  $\{l_1,l_2,E_1,\ldots,E_4\}$. By construction, one easily checks
  that the cycle $[l_1]+ [l_2] + [E_1] +\ldots +[E_4]$ represents
  $C_x\cdot g=6\mathfrak{o}_F$.
\end{proof}

\subsection{The geometry of lines of second type}
If $l\subset X$ is a line contained in $X$, then either
$$
\mathscr{N}_{l/X}\cong \calO(1)\oplus\calO^2,
$$
in which case $l$ is said to be of \textit{first type}, or
$$
\mathscr{N}_{l/X}\cong \calO(1)^2\oplus \calO(-1),
$$
in which case $l$ is said to be of \textit{second type}. It is known
that for a general point $[l]\in F$, the corresponding line $l$ is
of first type. Let $\Sigma_2\subset F$ be the closed subscheme of
lines of second type.

\begin{lem}[\cite{cg, amerik}]\label{lem class of Sigma2}
$\Sigma_2$ is a surface on $F$ whose cycle class is $5(g^2-c)$.
\end{lem}
\begin{proof}
  The fact that $\Sigma_2$ is a surface is proved in \cite[Corollary
  7.6]{cg}.
  The cohomology class of $\Sigma_2$ is computed in \cite{amerik} and
  is actually valid at the level of Chow groups.
\end{proof}

\begin{prop} \label{prop intersect sigma2} For all $\sigma \in
  \CH^2(F)_{\hom}$, $[\Sigma_2] \cdot \sigma = 5g^2\cdot \sigma.$
\end{prop}
\begin{proof}
  Let $\sigma\in \CH^2(F)_{\mathrm{hom}}$, then by \emph{(iii)} of Lemma
  \ref{lem special zero cycle}, we have $c\cdot\sigma=0$. The
  proposition then follows from Lemma \ref{lem class of Sigma2}.
\end{proof}

Now we study the geometry of lines of second type. Let $l\subset X$ be
a line of second type. Then there exist a linear $\PP^3$, denoted
$\PP^3_{\langle l\rangle}$, containing $l$ which is tangent to $X$
along $l$. Let $B\cong\PP^1$ be the rational curve parameterizing all
planes containing $l$ and contained in the linear $\PP^3_{\langle
  l\rangle}$. For $t\in B$, we use $\Pi_t$ to denote the corresponding
plane. When $t$ is general, then $\Pi_t$ is not contained in $X$ and
hence we have
\[
\Pi_t\cdot X = 2l +l_t,
\]
for some line $l_t$. As $t$ moves in $B$, the points $[l_t]$ traces
out a curve $\PP^1\cong\mathcal{E}_{[l]}\subset F$. The lines
parameterized by $\mathcal{E}_{[l]}$ swipe out a cubic surface
singular along $l$, which is simply $\PP^3_{\langle l\rangle}\cap
X$.

\begin{prop}\label{prop lines of second type}
  Let $l\subset X$ be a line of second type and notation as above.
  Then there is a unique line $l'\subset X$, disjoint with $l$,
  together with a natural isomorphism $\alpha_1 :
  \mathcal{E}_{[l]}\rightarrow l'$ and a degree 2 morphism $\alpha_2 :
  \mathcal{E}_{[l]}\rightarrow l$ such that the following holds~: for
  all point $s\in \mathcal{E}_{[l]}$, the corresponding line
  $l_s\subset X$ is the line connecting the points $\alpha_1(s)$ and
  $\alpha_2(s)$. The surface $S$ is smooth away from $l$.
\end{prop}
\begin{proof}
  This is the consequence of the geometry of cubic surfaces singular
  along a curve. Let $S=\PP^3_{\langle l\rangle}\cap X$ be the cubic
  surface swept out by all lines parameterized by $\mathcal{E}_{[l]}$.
  Then $S$ is singular along $l$. Take homogeneous coordinates
  $[X_0 :X_1 :X_2 :X_3]$ on $\PP^3_{\langle l\rangle}$ such that $l$ is
  defined by $X_2=X_3=0$. Let $l'$ be the line defined by $X_2=X_3=0$.
  We easily see that $S\subset \PP^3_{\langle l\rangle}$ is defined by
  an equation
\[
 X_0Q_1(X_2,X_3) - X_1Q_0(X_2,X_3)=0,
\]
where $Q_0, Q_1$ are homogeneous polynomials of degree 2 in $X_2,X_3$.
We see that the smoothness of $X$ forces $Q_0$ and $Q_1$ to be
linearly independent. Indeed, if $Q_0$ and $Q_1$ are proportional to
each other, then we can apply a linear change of coordinates and the
equation for $S$ in $\PP^3_{\langle l\rangle}$ can be written as
$X_0X_2X_3=0$ or $X_0X_2^2=0$. In either case, one easily checks that
$X$ is singular at the point $[0 :1 :0 :0]\in\PP^{3}_{\langle l\rangle}$.
We use $[X_0 :X_1]$ as homogeneous coordinates on $l$. Pick a point
$t=[t_0 :t_1]\in l$. A line on $S$ passing through $t$ can be
parameterized by
\[ [\lambda_0 :\lambda_1]\mapsto [\lambda_0t_0 :\lambda_0t_1 : \lambda_1
a : \lambda_2 b],
\]
where $t_0Q_1(a,b)-t_1Q_0(a,b)=0$. We note that any such line always
meets the line $l'$ at the point $[0 :0 :a :b]$. For a fixed $t$, we have
two pairs (counted with multiplicities) $(a,b)$ satisfying the above
equation. Consequently through a general point $t\in l$, there are two
points $s_1,s_2\in\mathcal{E}_{[l]}$ such that both $l_{s_1}$ and
$l_{s_2}$ pass through $t$. Then we take $\alpha_1(s) = l_s \cap l'$
and $\alpha_2(s) = l_s \cap l$ for all $s \in \mathcal{E}_{[l]}$. We
still need to show that $S$ is smooth away from $l$. Note that the
above description of $S$ implies that $S$ is smooth away from $l$ and
$l'$. However the fact that $\alpha_1$ is an isomorphism forces $S$ to
be smooth along $l'$.
\end{proof}

\begin{prop}\label{prop secant line mult}
  Let $l\subset X$ be a line of second type and
  $s_1,s_2\in\mathcal{E}_{[l]}$ two points such that the corresponding
  lines $l_1=l_{s_1}$ and $l_2=l_{s_2}$ are disjoint. Then the set of
  secant lines of $(l_1,l_2)$ is given by $4l+l'$, where $l'$ is the
  line as in Proposition \ref{prop lines of second type}.
\end{prop}
\begin{proof}
  Any secant line $L$ of $(l_1,l_2)$ different from $l$ or $l'$ is in
  the smooth locus of $S$. This implies that
  $\mathscr{N}_{l/S}\cong\calO(-1)$. It follows that $L$ does not move
  on $S$. Thus $(l_1,l_2)$ has finitely many, and hence five, secant
  lines. Note that both $l_1$ and $l_2$ are pointing the the positive
  direction of $\mathscr{N}_{l/X}$. Hence the multiplicity of $l$, as
  a secant line of $(l_1,l_2)$ is 4~; see \cite[\S3]{relation}. Hence
  the multiplicity of $l'$ has to be 1. This finishes the proof.
\end{proof}


\vspace{10pt}
\section{Rational maps and Chow groups} \label{app rat chow}

We consider $\varphi  : X \dashrightarrow Y$ a rational map between
projective varieties defined over a field $k$ of characteristic
zero. Let
$$
\xymatrix{ \widetilde{X} \ar[dr]^{\widetilde{\varphi}} \ar[d]_\pi \\
  X \ar@{-->}[r]^{\varphi} & Y }
$$ be a resolution of $\varphi$, that is, $\pi  : \widetilde{X}
\rightarrow X$ is a projective birational morphism such that $\varphi$
extends to a morphism $\widetilde{\varphi}  : \widetilde{X} \rightarrow
Y$. Note that no smoothness assumption is required on $\widetilde X$.

Assume that $X$ is smooth.  Then we define a map $\varphi_*  : \CH_l(X)
\rightarrow \CH_l(Y)$ by the formula
\begin{center}
  $\varphi_* \alpha  := \widetilde{\varphi}_*\pi^* \alpha$ \quad for all
  $\alpha \in \CH_l(X)$.
\end{center}
Here, $\pi^*  : \CH_l(X) \rightarrow \CH_l(\widetilde{X})$ is the
pull-back map as defined in \cite[\S 8]{fulton}.

\begin{lem} \label{lem covariant} Assume that $X$ is smooth. The map
  $\varphi_*  : \CH_l(X) \rightarrow \CH_l(Y)$ defined above does not
  depend on a choice of resolution for $\varphi$.
\end{lem}
\begin{proof}
  Let $\pi_1  : X_1 \rightarrow X$, $\varphi_1  : X_1 \rightarrow Y$ be
  a resolution of $\varphi$, and let $\pi_2  : X_2 \rightarrow X$,
  $\varphi_2  : X_2 \rightarrow Y$ be another resolution of $\varphi$.
  There is then a projective variety $\widetilde{X}$ and birational
  morphisms $\widetilde{\pi}_1  : \widetilde{X} \rightarrow X_1$ and
  $\widetilde{\pi}_2  : \widetilde{X} \rightarrow X_2$ such that both
  $\varphi_1$ and $\varphi_2$ extend to a morphism
  $\widetilde{\varphi}  : \widetilde{X} \rightarrow Y$. It is then
  enough to see that
  $$\widetilde{\varphi}_* (\pi_1 \circ \widetilde{\pi}_1)^* =
  (\varphi_1)_* \pi_1^*  : \CH_l(X) \rightarrow \CH_l(Y). $$ The
  identity $\varphi_1 \circ \widetilde{\pi}_1 = \widetilde{\varphi}$
  gives $\widetilde{\varphi}_* (\pi_1 \circ \widetilde{\pi}_1)^* =
  (\varphi_1)_* (\widetilde{\pi}_1)_* (\pi_1 \circ
  \widetilde{\pi}_1)^*$. We may then conclude by the projection
  formula \cite[Proposition 8.1.1]{fulton} which implies that
  $(\widetilde{\pi}_1)_* (\pi_1 \circ \widetilde{\pi}_1)^* = \pi_1^*$.
\end{proof}

If $Y$, instead of $X$, is assumed to be smooth, then we define a
map $\varphi^*  : \CH^l(Y) \rightarrow \CH^l(X)$ by the formula
\begin{center}
  $\varphi^* \beta := \pi_*\widetilde{\varphi}^* \beta$\quad for all
  $\beta \in \CH^l(Y)$.
\end{center}
Likewise, we have
\begin{lem} \label{lem contravariant} Assume that $Y$ is smooth. The
  map $\varphi^*  : \CH^l(Y) \rightarrow \CH^l(X)$ defined above does
  not depend on a choice of resolution for $\varphi$. \qed
\end{lem}

Let now $p_X  : X \times Y \rightarrow X$ and $p_Y  : X \times Y
\rightarrow Y$ be the natural projections, and let $\Gamma_\varphi
\subset X \times Y$ be the closure of the graph of $\varphi$, given
with the natural projections
$$\xymatrix{\Gamma_\varphi \ar[d]_p \ar[r]^q & Y \\ X}$$ The morphism
$p$ is birational and the morphism $q$ clearly extends $\varphi$.
Therefore, as a consequence of Lemmas \ref{lem covariant} and \ref{lem
  contravariant}, we get

\begin{lem} \label{lem comp correspondences} If $X$ is smooth, then
  $\varphi_* \alpha = q_*p^*\alpha$ for all $\alpha \in \CH_l(X)$. If
  $Y$ is smooth, then $\varphi^*\beta = p_*q^* \beta$ for all $\beta
  \in \CH^l(Y)$.  If $X$ and $Y$ are both smooth, then
  \begin{center}
    $\varphi_* \alpha = (\Gamma_{\varphi})_*\alpha :=
    (p_Y)_*(\Gamma_\varphi \cdot p_X^*\alpha)$ \quad and \quad
    $\varphi^*\beta = (\Gamma_{\varphi})^* \beta :=
    (p_X)_*(\Gamma_\varphi \cdot p_Y^*\beta)$.
  \end{center}\qed
\end{lem}

We now want to understand, when $X$ and $Y$ are both smooth, to which
extent $\varphi^*  : \CH^*(Y) \rightarrow \CH^*(X)$ is compatible with
intersection product.

\begin{lem} \label{lem pushcup} Let $\pi  : \widetilde{X} \rightarrow
  X$ be a dominant morphism between smooth projective varieties and
  let $x$ and $y$ be two cycles in $\CH_*(\widetilde{X})$. Then
  $$\pi_*(x\cdot y) = \pi_*x \cdot \pi_*y +
  \pi_*\big((x-\pi^*\pi_*x)\cdot (y-\pi^*\pi_*y) \big).$$
\end{lem}
\begin{proof}
  Let's define $x'  := x - \pi^*\pi_*x$ and $y'  := y - \pi^*\pi_*y.$ By
  the projection formula, $\pi_*\pi^*$ acts as the identity on
  $\CH_*({X})$. Therefore $\pi_*x' = 0 $ and $\pi_*y' = 0$.  The
  projection formula gives $\pi_*\big((\pi^*\pi_*x)\cdot y' \big) =
  \pi_*x \cdot \pi_* y' = 0$ and $\pi_*\big(x' \cdot (\pi^*\pi_*y)
  \big) = \pi_*x' \cdot \pi_* y = 0$.  This yields
$$  \pi_* (x \cdot y) = \pi_*(\pi^*\pi_*x \cdot  \pi^*\pi_*y) +
  \pi_*(x'\cdot y') = \pi_*x \cdot \pi_*y + \pi_*(x'
  \cdot y'). $$
\end{proof}

For a smooth projective variety $X$ over $k$, we write $$T^2(X)  :=
\ker\{AJ^2  : \CH^2(X)_{\hom} \rightarrow J^2(X) \}$$ for the kernel of
Griffiths' second Abel-Jacobi map to the second intermediate Jacobian
$J^2(X)$ which is a quotient of $H^3(X,\mathds{C})$.
The group $T^2(X)$ is a birational invariant of smooth projective
varieties. Precisely, we have

\begin{lem}\label{lem birat invariant}
  Let $\pi  : \widetilde{X} \dashrightarrow X$ be a birational map
  between smooth projective varieties. Then $\pi^* \pi_*$ acts as the
  identity on $\CH_0(\widetilde{X})$,
  $\mathrm{Griff}_1(\widetilde{X})$,
  $\mathrm{Griff}^2(\widetilde{X})$, $T^2(\widetilde{X})$,
  $\CH^1(\widetilde{X})_{\hom}$ and $\CH^0(\widetilde{X})$.
 \end{lem}
\begin{proof} This is proved in \cite[Proposition 5.4]{vial} in the
  case when $\pi$ is a birational morphism. In the general case, by
  resolution of singularities, there is a smooth projective variety
  $Y$ and birational morphisms $f  : Y \rightarrow \widetilde{X}$ and
  $\widetilde{\pi}  : \widetilde{X} \rightarrow X$ such that
  $\widetilde{\pi} = \pi \circ f$. By definition of the action of
  rational maps on Chow groups, we have $\pi^*\pi_* =
  f_*\widetilde{\pi}^*\widetilde{\pi}_*f^*$ and the proof of the lemma
  reduces to the case of birational morphisms.
\end{proof}

\begin{prop} \label{prop cup} Let $\varphi  : X \dashrightarrow Y$ be a
  rational map between smooth projective varieties and let $d$ be the
  dimension of $X$. Let $\alpha \in \CH^*(Y)$ be a cycle that belongs
  to one of the sub-groups $\CH^d(Y)$, $T^2(Y)$, $CH^1(Y)_{\hom}$ or
  $\CH^0(Y)$ (respectively $\alpha \in \mathrm{Griff}^{d-1}(Y)$ or
  $\mathrm{Griff}^2(Y)$), and let $\beta \in \CH^*(Y)$ be any cycle.
  Then
  $$\varphi^*(\alpha \cdot \beta) = \varphi^*\alpha \cdot \varphi^*
  \beta \in \CH^*(X) \ (\text{resp.} \ \in \mathrm{Griff}^*(X)).$$
\end{prop}
\begin{proof}
  Let $\pi  : \widetilde{X} \rightarrow X$, $\widetilde{\varphi}  :
  \widetilde{X} \rightarrow Y$ be a resolution of $\varphi$ with
  $\widetilde{X}$ smooth.  For $x$ belonging to one of the groups
  $\CH_0(\widetilde{X})$, $\mathrm{Griff}_1(\widetilde{X})$,
  $\mathrm{Griff}^2(\widetilde{X})$, $T^2(\widetilde{X})$,
  $CH^1(\widetilde{X})_{\hom}$ or $\CH^0(\widetilde{X})$, we have by
  Lemma \ref{lem birat invariant} $x = \pi^*\pi_*x$, so that Lemma
  \ref{lem pushcup} yields
  \begin{equation} \label{eqn pi} \pi_* (x \cdot y) = \pi_*x \cdot
    \pi_*y.\end{equation} If we now set $x = \widetilde{\varphi}^*
  \alpha$ for $\alpha$ as in the statement of the proposition, then
  $x$ defines a cycle as above (if $\alpha \in T^2(Y)$, then $x \in
  T^2(\widetilde{X})$ by functoriality of the Abel-Jacobi map with
  respect to the action of correspondences), and (\ref{eqn pi}) gives~:
\begin{eqnarray}
  \varphi^*\alpha \cdot \varphi^*\beta &=& \pi_*\widetilde{\varphi}^*
  \alpha \cdot \pi_*\widetilde{\varphi}^* \beta \nonumber \\
  &=& \pi_*(\widetilde{\varphi}^* \alpha \cdot \widetilde{\varphi}^*
  \beta) \nonumber \\
  &=& \pi_*\widetilde{\varphi}^* (\alpha \cdot \beta)\nonumber \\
  &=& \varphi^*(\alpha \cdot \beta). \nonumber
\end{eqnarray}
\end{proof}

There is also a projection formula for rational maps.

\begin{prop} \label{prop proj formula} Let $\varphi  : X
  \dashrightarrow Y$ be a rational map between smooth projective
  varieties and let $d$ be the dimension of $X$. Let $\alpha \in
  \CH^*(Y)$ be a cycle that belongs to one of the sub-groups
  $\CH^d(Y)$, $T^2(Y)$, $CH^1(Y)_{\hom}$ or $\CH^0(Y)$ (respectively
  $\alpha \in \mathrm{Griff}^{d-1}(Y)$ or $\mathrm{Griff}^2(Y)$), and
  let $\beta \in \CH^*(X)$ be any cycle. Then
  $$\varphi_*(\varphi^*\alpha \cdot \beta) = \alpha \cdot \varphi_*
  \beta \in \CH^*(Y) \ (\text{resp.} \ \in \mathrm{Griff}^*(Y)).$$ In
  particular, if $\varphi  : X \dashrightarrow Y$ is a finite rational
  map of degree $\deg \varphi$,
  then $$\varphi_*\varphi^*\alpha = (\deg \varphi) \alpha \in \CH^*(Y) \
  (\text{resp.} \ \in \mathrm{Griff}^*(Y)).$$
\end{prop}
\begin{proof}
  Again consider a resolution of $\varphi$ as in the proof of
  Proposition \ref{prop cup}. By Lemmas \ref{lem covariant} and
  \ref{lem contravariant}, Lemma \ref{lem birat invariant} and the
  projection formula for morphism of smooth projective varieties, we
  have
  \begin{align*}
    \varphi_*(\varphi^*\alpha \cdot \beta) & = \widetilde{\varphi}_* \pi^*
    (\pi_*\widetilde{\varphi}^*
    \alpha \cdot \beta) \\
    & = \widetilde{\varphi}_*(\pi^*\pi_*\widetilde{\varphi}^*\alpha \cdot
    \pi^*\beta) \\
& = \widetilde{\varphi}_*(\widetilde{\varphi}^*\alpha \cdot \pi^*\beta) \\
& = \alpha \cdot \widetilde{\varphi}_*\pi^*\beta \\
& = \alpha \cdot \varphi_*\beta.
  \end{align*}
The last statement follows by taking $\beta = [X] \in \CH^0(X)$.
\end{proof}


\newpage

\end{document}